\documentclass[10pt]{book}
\usepackage{latexsym,a4,calc}
\usepackage{marvosym,stmaryrd,amsthm,amssymb,amsmath}
\usepackage{fancyhdr}
\usepackage{mathptmx,mathrsfs,newcent}
\usepackage{pstricks,pst-node,multido,pst-plot,graphicx,xypic}
\usepackage{mymultind}
\usepackage{times}

\renewcommand{\chaptermark}[1]%
  {\markboth{\chaptername\ \thechapter.\ #1}{}}
\renewcommand{\sectionmark}[1]%
  {\markright{\thesection.\ #1}}

\makeatletter
\renewcommand{\ps@plain}{\ps@empty}
\makeatother

\newtheorem{thm}{Theorem}[section]
\newtheorem{lemma}[thm]{Lemma}
\newtheorem{prop}[thm]{Proposition}
\newtheorem{cor}[thm]{Corollary}
\newtheorem{corollary}[thm]{Corollary}
\newtheorem{remark}[thm]{Remark}
\newtheorem{remarks}[thm]{Remarks}
\newtheorem{definition}[thm]{Definition}
\newtheorem{notation}[thm]{Notation}

\newtheorem{example}[thm]{Example}
\newtheorem{examples}[thm]{Examples}

\newcommand{\Remark}[1]{\begin{remark}{\rm #1}\end{remark}}

\newcommand{\Definition}[1]{\begin{definition}{\rm #1}\end{definition}}
\newcommand{\Notation}[1]{\begin{notation}{\rm #1}\end{notation}}
\newcommand{\Example}[1]{\begin{example}{\rm #1}\end{example}}
\newcommand{\Examples}[1]{\begin{examples}{\rm #1}\end{examples}}

\newcommand{\dist}{\operatorname{dist}}
\newcommand{\FB}{\mbox{\Pointinghand}}
\renewcommand{\Box}{\square}

\newcommand{\N}{{\mathbb N}}
\newcommand{\Co}{{\mathbb C}}
\newcommand{\R}{{\mathbb R}}
\newcommand{\K}{{\mathbb K}}

\newcommand{\RR}{\mathcal{R}}
\newcommand{\DD}{\mathcal{D}}

\newcommand{\JJ}{\mathcal{J}}

\newcommand{\KK}{\mathcal{K}}

\newcommand{\FF}{\mathcal{F}}
\newcommand{\LL}{\mathcal{L}}

\newcommand{\VV}{\mathcal{V}}
\newcommand{\A}{\mathfrak{A}}
\newcommand{\B}{\mathfrak{B}}
\newcommand{\RRt}{{\widetilde{\mathcal{R}}}}
\newcommand{\id}{{\rm id}}
\renewcommand{\Re}{\mathfrak{Re}}
\renewcommand{\Im}{\mathfrak{Im}}
\newcommand{\Obj}{\operatorname{Obj}}
\newcommand{\Mo}{\operatorname{Mor}}
\newcommand{\Id}{\operatorname{Id}}

\DeclareMathAlphabet{\mathpzc}{OT1}{pzc}{m}{it}
\newcommand{\SYMPL}{\mathpzc{SymplVec}}
\newcommand{\CALG}{\mathpzc{C^*}\!\!\!\!-\!\!\!\mathpzc{Alg}}
\newcommand{\GLOBHY}{\mathpzc{GlobHyp}}
\newcommand{\GLOBHYN}{\mathpzc{GlobHyp}_{\mathpzc{\footnotesize naked}}}
\newcommand{\LORFUND}{\mathpzc{LorFund}}
\newcommand{\QLA}{\mathpzc{QuasiLocAlg}}
\newcommand{\CNET}{\mathpzc{QuasiLocAlg}_{\mathpzc{\footnotesize weak}}}
\newcommand{\SET}{\mathpzc{Set}}
\newcommand{\TOP}{\mathpzc{Top}}
\newcommand{\GROUP}{\mathpzc{Groups}}
\newcommand{\AB}{\mathpzc{AbelGr}}
\newcommand{\AAA}{\mathpzc{A}}
\newcommand{\BBB}{\mathpzc{B}}

\newcommand{\CCR}{\operatorname{CCR}}
\newcommand{\solve}{\operatorname{SOLVE}}
\newcommand{\SYM}{\operatorname{SYMPL}}
\newcommand{\forget}{\operatorname{FORGET}}

\newcommand{\Mor}{\operatorname{Mor}}
\newcommand{\ext}{\operatorname{ext}}
\newcommand{\alg}{{\operatorname{alg}}}
\newcommand{\res}{\operatorname{res}}
\newcommand{\grad}{\operatorname{grad}}
\newcommand{\ord}{\operatorname{ord}}
\newcommand{\dom}{\operatorname{dom}}

\newcommand{\End}{\operatorname{End}}
\newcommand{\Hom}{\operatorname{Hom}}

\newcommand{\del}{\partial}
\newcommand{\tr}{\operatorname{tr}}

\newcommand{\vol}{\operatorname{vol}}

\newcommand{\dt}{{\,\,{dt}}}
\newcommand{\dV}{{\,\,\operatorname{dV}}}
\newcommand{\dA}{\operatorname{dA}}
\renewcommand{\div}{\operatorname{div}}
\newcommand{\ric}     {\mbox{\upshape ric}}
\newcommand{\scal}{\operatorname{scal}}
\newcommand{\supp}{\operatorname{supp}}
\newcommand{\ssupp}{\operatorname{sing\,supp}}

\newcommand{\guenterl}{\operatorname{L}}
\newcommand{\la}{\langle}       
\newcommand{\ra}{\rangle}       
\newcommand{\lala}{\langle \! \langle} 
\newcommand{\rara}{\rangle \! \rangle}

\newcommand{\Csc}{C_{\mathrm{sc}}^\infty}

\renewcommand{\phi}{\varphi}

\newcommand{\stetig}{\ensuremath{C}}

\newcommand{\lra}{\longrightarrow}

\newcommand{\ovl}{\overline}
\newcommand{\wit}{\widetilde}
\newcommand{\wih}{\widehat}
\def \be{\begin{eqnarray*}}
\def \ee{\end{eqnarray*}}
\def \ben{\begin{enumerate}}
\def \een{\end{enumerate}}
\def \beit{\begin{itemize}}
\def \eeit{\end{itemize}}
\def \bui#1#2{\mathrel{\mathop{\kern 0pt#1}\limits^{#2}}}
\def \buil#1#2{\mathrel{\mathop{\kern 0pt#1}\limits_{#2}}}
\newcommand{\pa}[1]{\left(#1\right)}
\newcommand{\ueber}[2]{\ensuremath{{\scriptstyle #1}\atop
    {\scriptstyle #2}}}
\title{Wave equations on Lorentzian manifolds and quantization} 

\author{Christian B\"ar\\ Nicolas Ginoux \\ Frank Pf\"aff\-le}

\date{\today}

\makeindex{Symbole}
\makeindex{DerIndex}
\makeindex{Abbildungen}

\newcommand{\indexn}[1]{\index{DerIndex}{#1}}  
\newcommand{\indexs}[1]{\index{Symbole}{#1}}  
\newcommand{\defidx}[1]{\defem{#1}\index{DerIndex}{#1>defemidx}} 
\newcommand{\defem}[1]{\textit{{\red#1}}}

\newcounter{Abbildung}
\newcommand{\abb}[1]{\stepcounter{Abbildung}\begin{center}{\footnotesize
      Fig.\ \theAbbildung: #1}\end{center}\index{Abbildungen}{#1}}

\begin{document}
%%%%%%%%%%%%%%%%%%%%%%%%%%%%%%%%%%%%%%%%%%%%%%%%%%%%%%%%%%%%%%%%%%%%%%%%%

\frontmatter {
%%\maketitle
\begin{titlepage}
\enlargethispage{30cm}
\pspicture(0,0)(20,47)

\newrgbcolor{farbeL}{0 0 1}        % Farbe links
\newrgbcolor{farbeO}{0.75 0.75 1}  % Farbe oben
\newrgbcolor{farbeU}{0.9 0.9 1}    % Farbe unten

\psframe[fillcolor=farbeO,fillstyle=solid,linewidth=0pt,linecolor=farbeO](-10,45)(50,55)
\psframe[fillcolor=farbeU,fillstyle=solid,linewidth=0pt,linecolor=farbeU](0,0)(50,45)
\psframe[fillcolor=farbeL,fillstyle=solid,linewidth=0pt,linecolor=farbeL](-10,0)(0,55)

%%%%%%%%%%%%%%%%%%%%%%%%%%%%%%%%%%%%%%%%%%%%%%%%%%%%%%%%%%%%%%%%%%%%%%%%%%%%
%% Autoren
\rput[l](1,47){\huge\sf\white \begin{tabular}{l} 
    Christian B\"ar\\ 
    Nicolas Ginoux\\
    Frank Pf\"affle  
\end{tabular}} 

%%%%%%%%%%%%%%%%%%%%%%%%%%%%%%%%%%%%%%%%%%%%%%%%%%%%%%%%%%%%%%%%%%%%%%%%%%%%
%% Titel
\rput[l](1,42){\Huge\sf \begin{tabular}{l}
    Wave Equations on Lorentzian \\
    Manifolds and Quantization
\end{tabular}}

%%%%%%%%%%%%%%%%%%%%%%%%%%%%%%%%%%%%%%%%%%%%%%%%%%%%%%%%%%%%%%%%%%%%%%%%%%%%

\endpspicture    
\end{titlepage}
\newpage\thispagestyle{empty}

%%%%%%%%%%%%%%%%%%%%%%%%%%%%%%%%%%%%%%%%%%%%%%%%%%%%%%%%%%%%%%%%%%%%%%%%%
\chapter*{Preface}
\addcontentsline{toc}{chapter}{Preface}
%%%%%%%%%%%%%%%%%%%%%%%%%%%%%%%%%%%%%%%%%%%%%%%%%%%%%%%%%%%%%%%%%%%%%%%%%

In General Relativity spacetime is described mathematically by a Lorentzian
manifold.
Gravitation manifests itself as the curvature of this manifold.
Physical fields, such as the electromagnetic field, are defined on this
manifold and have to satisfy a wave equation.
This book provides an introduction to the theory of linear wave equations on
Lorentzian manifolds.
In contrast to other texts on this topic \cite{FL2,Guenther} we develop the
global theory.
This means, we ask for existence and uniqueness of solutions which are defined
on all of the underlying manifold.
Such results are of great importance and are already used much in the
literature despite the fact that published proofs are missing.
Tracing back the references one typically ends at Leray's unpublished lecture
notes \cite{L} or their exposition \cite{CB}.

In this text we develop the global theory from scratch in a modern geometric
language.
In the first chapter we provide basic definitions and facts about
distributions on manifolds, Lorentzian geometry, and normally hyperbolic
operators.
We study the building blocks for local solutions, the Riesz distributions, in
some detail.
In the second chapter we show how to solve wave equations locally.
Using Riesz distributions and a formal recursive procedure one first
constructs formal fundamental solutions.
These are formal series solving the equations formally but in general they do
not converge.
Using suitable cut-offs one gets ``almost solutions'' from these formal
solutions.
They are well-defined distributions but solve the equation only up to an error
term.
This is then corrected by some further analysis which yields true local
fundamental solutions.

This procedure is similar to the construction of the heat kernel for a Laplace
type operator on a compact Riemannian manifold.
The analogy goes even further.
Similar to the short-time asymptotics for the heat kernel, the formal
fundamental solution turns out to be an asymptotic expansion of the true
fundamental solution.
Along the diagonal the coefficients of this asymptotic expansion are given by
the same algebraic expression in the curvature of the manifold, the
coefficients of the operator, and their derivatives as the heat kernel
coefficients. 

In the third chapter we use the local theory to study global solutions.
This means we construct global fundamental solutions, Green's operators, and
solutions to the Cauchy problem.
This requires assumptions on the geometry of the underlying manifold.
In Lorentzian geometry one has to deal with the problem that there is no
good analog for the notion of completeness of Riemannian manifolds.
In our context globally hyperbolic manifolds turn out to be the right class of
manifolds to consider.
Most basic models in General Relativity turn out to be globally hyperbolic but
there are exceptions such as anti-deSitter spacetime.
This is why we also include a section in which we study cases where one can
guarantee existence (but not uniqueness) of global solutions on certain
non-globally hyperbolic manifolds.

In the last chapter we apply the analytical results and describe the basic
mathematical concepts behind field quantization.
The aim of quantum field theory on curved spacetimes is to provide a partial
unification of General Relativity with Quantum Physics where the
gravitational field is left classical while the other fields are quantized.
We develop the theory of $C^*$-algebras and CCR-representations in full detail
to the extent that we need.
Then we construct the quantization functors and check that the Haag-Kastler
axioms of a local quantum field theory are satisfied.
We also construct the Fock space and the quantum field.

From a physical perspective we just enter the door to quantum field theory but
do not go very far.
We do not discuss $n$-point functions, states, renormalization, nonlinear
fields, nor physical applications such as Hawking radiation.
For such topics we refer to the corresponding literature.
However, this book should provide the reader with a firm mathematical basis to
enter this fascinating branch of physics.

In the appendix we collect background material on category theory, functional
analysis, differential geometry, and differential operators that is used
throughout the text. 
This collection of material is included for the convenience of the reader but
cannot replace a thorough introduction to these topics.
The reader should have some experience with differential geometry.
Despite the fact that normally hyperbolic operators on Lorentzian manifolds
look formally exactly like Laplace type operators on Riemannian manifolds
their analysis is completely  different.
The elliptic theory of Laplace type operators is not needed anywhere in this
text.
All results on hyperbolic equations which are relevant to the subject are
developed in full detail.
Therefore no prior knowledge on the theory of partial differential equations
is needed. 

%%%%%%%%%%%%%%%%%%%%%%%%%%%%%%%%%%%%%%%%%%%%%%%%%%%%%%%%%%%%%%%%%%%%%%%%%
\section*{Acknowledgements}
\addcontentsline{toc}{section}{Acknowledgements}
%%%%%%%%%%%%%%%%%%%%%%%%%%%%%%%%%%%%%%%%%%%%%%%%%%%%%%%%%%%%%%%%%%%%%%%%%

This book would not have been possible without the help of and the inspired
discussions with many colleagues.
Special thanks go to 
Helga Baum,
Klaus Fredenhagen,
Olaf M\"uller,
Miguel S\'anchez,
Alexander Strohmaier,
Rainer Verch, 
and all the other participants of the workshop on ``Hyperbolic operators on
Lorentzian manifolds and quantization'' in Vienna, 2005.

We would also like to thank the Schwerpunktprogramm (1154) ``Globale
Differential\-geometrie'' funded by the Deutsche Forschungsgemeinschaft for
financial support. 

\vspace{0.8cm}
{\sc Christian B\"ar, Nicolas Ginoux, and Frank Pf\"affle}, 

Potsdam, October 2006

%%%%%%%%%%%%%%%%%%%%%%%%%%%%%%%%%%%%%%%%%%%%%%%%%%%%%%%%%%%%%%%%%%%%%%%%%

\tableofcontents\newpage\thispagestyle{empty}
}
\parindent0cm 
\mainmatter
\setcounter{chapter}{0}

%%%%%%%%%%%%%%%%%%%%%%%%%%%%%%%%%%%%%%%%%%%%%%%%%%%%%%%%%%%%%%%%%%%%%%%%%
\chapter{Preliminaries}
%%%%%%%%%%%%%%%%%%%%%%%%%%%%%%%%%%%%%%%%%%%%%%%%%%%%%%%%%%%%%%%%%%%%%%%%%

We want to study solutions to wave equations on Lorentzian manifolds.
In this first chapter we develop the basic concepts needed for this task.
In the appendix the reader will find the background material on differential
geometry, functional analysis and other fields of mathematics that will be
used throughout this text without further comment.

A wave equation is given by a certain differential operator of second order
called a ``normally hyperbolic operator''.
In general, these operators act on sections in vector bundles which is the
geometric way of saying that we are dealing with systems of equations and not
just with scalar equations.
It is important to allow that the sections may have certain singularities.
This is why we work with distributional sections rather than with smooth or
continuous sections only.

The concept of distributions on manifolds is explained in the first section.
One nice feature of distributions is the fact that one can apply differential
operators to them and again obtain a distribution without any further
regularity assumption.

The simplest example of a normally hyperbolic operator on a Lorentzian
manifold is given by the d'Alembert operator on Minkowski space.
Its fundamental solution, a concept to be explained later, can be described
explicitly.
This gives rise to a family of distributions on Minkowski space, the Riesz
distributions, which will provide the building blocks for solutions in the
general case later.

After explaining the relevant notions from Lorentzian geometry we will show
how to ``transplant'' Riesz distributions from the tangent space into the
Lorentzian manifold.
We will also derive the most important properties of the Riesz distributions.

%%%%%%%%%%%%%%%%%%%%%%%%%%%%%%%%%%%%%%%%%%%%%%%%%%%%%%%%%%%%%%%%%%%%%%%%%
\section{Distributions on manifolds}
%%%%%%%%%%%%%%%%%%%%%%%%%%%%%%%%%%%%%%%%%%%%%%%%%%%%%%%%%%%%%%%%%%%%%%%%%

Let us start by giving some definitions and by fixing the
terminology for distributions on manifolds.
We will confine ourselves to those facts that we will actually need later on.
A systematic and much more complete introduction may be found e.~g.\ in
\cite{FL1}.

%%%%%%%%%%%%%%%%%%%%%%%%%%%%%%%%%%%%%%%%%%%%%%%%%%%%%%%%%%%%%%%%%%%%%%%%%
\subsection{Preliminaries on distributions}
%%%%%%%%%%%%%%%%%%%%%%%%%%%%%%%%%%%%%%%%%%%%%%%%%%%%%%%%%%%%%%%%%%%%%%%%% 

Let $M$ be a manifold equipped with a smooth volume density $\dV$. 
Later on we will use the volume density induced by a Lorentzian metric
but this is irrelevant for now.
We consider a real or complex vector bundle $E\to M$.
We will always write $\K=\R$\indexs{K*@$\protect\K$, field $\protect\R$ or $\protect\Co$} or $\K=\Co$ depending on whether $E$ is a real or complex. 
The space of compactly supported smooth sections in $E$ will be denoted by
$\mathcal{D}(M,E)$. 
\indexs{D*@ $\protect\mathcal{D}(M,E)$, compactly supported smooth sections in
  vector bundle $E$ over $M$}  
We equip $E$ and $T^*M$ with connections, both denoted by $\nabla$.
They induce connections on the tensor bundles $T^*M\otimes\cdots\otimes T^*M
\otimes E$, again denoted by $\nabla$.
For a continuously differentiable section $\phi\in C^1(M,E)$ the covariant
derivative is a continuous section in $T^*M\otimes E$, $\nabla\phi\in
C^0(M,T^*M\otimes E)$.
More generally, for $\varphi\in C^k(M,E)$ we get $\nabla^k\phi\in
C^0(M,\underbrace{T^*M\otimes\cdots\otimes T^*M}_{k\,\,
  \mathrm{factors}}\otimes E)$. 

We choose a Riemannian metric on $T^*M$ and a Riemannian or Hermitian metric
on $E$ depending on whether $E$ is real or complex.
This induces metrics on all bundles $T^*M\otimes\cdots\otimes
T^*M\otimes E$.
Hence the norm of $\nabla^k\phi$ is defined at all points of $M$.

For a subset $A \subset M$ and $\varphi\in C^k(M,E)$ we define the
\defem{$C^k$-norm}\indexn{Cknorm@$C^k$-norm>defemidx} 
by
\indexs{*@$\|\varphi\|_{C^k(A)}$, $C^k$-norm of section $\varphi$ over a set
  $A$} 
\begin{equation}
\|\varphi\|_{C^k(A)} := \max_{j=0,\ldots,k} \;
\sup_{x\in A}\,|\nabla^j\varphi(x)| .
\label{defCkNorm}
\end{equation}
If $A$ is compact, then different choices of the metrics and the
connections yield equivalent norms $\|\cdot\|_{C^k(A)}$.
For this reason there will usually be no need to explicitly specify the
metrics and the connections.

The elements of $\DD(M,E)$ are referred to as \defidx{test sections} in $E$.
We define a notion of convergence of test sections.

\Definition{
Let $\varphi,\varphi_n\in \mathcal{D}(M,E)$. 
We say that the sequence $(\varphi_n)_n$ {\it converges to} $\varphi$ {\it in} 
$\mathcal{D}(M,E)$ if the following two conditions hold:
\begin{enumerate}
 \item There is a compact set $K\subset M$ such that the supports of all
 $\varphi_n$ are contained in $K$, i.~e., $\supp(\varphi_n)\subset K$ for 
 all $n$.
 \item The sequence $(\varphi_n)_n$ converges to $\varphi$ in all
 $C^k$-norms over $K$, i.~e., for each $k\in\N$
$$
\|\varphi-\varphi_n\|_{C^k(K)} \buil{\lra}{n\to\infty} 0.
$$
%\stackrel{n\to\infty}{\longrightarrow}
\end{enumerate}
}

We fix a finite-dimensional $\K$-vector space $W$.
Recall that $\K=\R$ or $\K=\Co$ depending on whether $E$ is real or complex.

\Definition{\label{def:distr}
A $\K$-linear map $F:\mathcal{D}(M,E^*)\to W$ is called a 
\defem{distribution in $E$ with values in $W$} 
\indexn{distribution>defemidx}
if it is continuous in the sense  
that for all convergent sequences $\varphi_n\to\varphi$ in
$\mathcal{D}(M,E^*)$  one has $ F[\varphi_n]\to F[\varphi]$.
We write $\DD'(M,E,W)$ for the space of all $W$-valued distributions in $E$. 
}
\indexs{Da*@$\protect\DD'(M,E,W)$, $W$-valued distributions in $E$}

Note that since $W$ is finite-dimensional all norms $|\cdot|$
on $W$ yield the same topology on $W$.
Hence there is no need to specify a norm on $W$ for Definition~\ref{def:distr}
to make sense.
Note moreover, that distributions in $E$ act on test sections in $E^*$.

\begin{lemma}\label{endlicheordnung}
Let $F$ be a  $W$-valued distribution in $E$ and let $K\subset M$ be compact. 
Then there is a nonnegative integer $k$ and a constant $C>0$
such that for all $\varphi\in \mathcal{D}(M,E^*)$ with 
$\supp(\varphi)\subset K$ we have
\begin{equation}\label{ordnungsabsch}
  \left|F[\varphi]\right|\le C\cdot\left\|\varphi\right\|_{C^k(K)}.
\end{equation}
\end{lemma}

The smallest $k$ for which inequality (\ref{ordnungsabsch}) holds is called
the \defem{order} of $F$ over $K$.
\indexn{order of a distribution>defemidx} 

\begin{proof}
Assume (\ref{ordnungsabsch}) does not hold for any pair of $C$ and $k$.
Then for every positive integer $k$ we can find a nontrivial section
$\varphi_k \in\mathcal{D}(M,E^*)$ with $\supp(\varphi_k)\subset K$ and
$\left|F[\varphi_k]\right| \ge k\cdot\left\|\varphi_k\right\|_{C^k}$. 
We define sections $\psi_k:=\tfrac{1}{|F[\varphi_k]|}\varphi_k$.
Obviously, these $\psi_k$ satisfy $\supp(\psi_k)\subset K$ and 
\[
\|\psi_k\|_{C^k(K)}=\tfrac{1}{\left|F[\varphi_k]\right|}\|\varphi_k\|_{C^k(K)}
\le\tfrac{1}{k}. 
\]
Hence for $k\geq j$ 
$$
\|\psi_k\|_{C^j(K)} \leq \|\psi_k\|_{C^k(K)} \leq \tfrac1k.
$$
Therefore the sequence $(\psi_k)_k$ converges to $0$ in $\mathcal{D}(M,E^*)$.
Since $F$ is a distribution we get $F[\psi_k]\to F[0]=0$ for $k\to\infty$.
On the other hand, $|F[\psi_k]|=
\left|\tfrac{1}{|F[\varphi_k]|}F[\varphi_k]\right|=1$
for all $k$, which yields a contradiction.
\end{proof}

Lemma~\ref{endlicheordnung} states that the restriction of any distribution to
a (relatively) compact set is of finite order. 
We say that a distribution $F$ is of order $m$ if $m$ is the smallest integer
such that for each compact subset $K\subset M$ there exists a constant $C$
so that
$$
|F[\varphi]| \leq C\cdot \|\varphi\|_{C^m(K)}
$$
for all $\varphi\in\DD(M,E^*)$ with $\supp(\varphi)\subset K$.
Such a distribution extends uniquely to a continuous linear map on
$\DD^m(M,E^*)$, the space of $C^m$-sections in $E^*$ with compact support.
\indexs{Da*@$\protect\DD^m(M,E^*)$, compactly supported $C^m$-sections in
  $E^*$} 
Convergence in $\DD^m(M,E^*)$ is defined similarly to that of test sections.
We say that $\phi_n$ converge to $\phi$ in $\DD^m(M,E^*)$ if the supports of
the $\phi_n$ and $\phi$ are contained in a common compact subset $K\subset M$
and $\|\varphi-\varphi_n\|_{C^m(K)} \to 0$ as $n\to\infty$.

Next we give two important examples of distributions.

\Example{
Pick a  bundle $E\to M$ and a point $x\in M$.
The \defidx{delta-distribution} $\delta_x$ is an $E_x^*$-valued distribution in
$E$. 
\indexs{deltadistr@$\delta_x$,  delta-distribution}
For $\varphi\in\mathcal{D}(M,E^*)$ it is defined by
\[ \delta_x[\varphi]=\varphi(x). \]
Clearly, $\delta_x$ is a distribution of order 0.
}

\Example{\label{ex:locint}
Every locally integrable section $f\in L^1_{\mathrm{loc}}(M,E)$ can be 
interpreted as a $\K$-valued distribution in $E$ by setting for any 
$\varphi\in\mathcal{D}(M,E^*)$ 
\indexs{L*@$L^1_{\protect\mathrm{loc}}(M,E)$, locally integrable sections in
  vector  bundle $E$}
\[ 
f[\varphi]:=\int_M \varphi(f) \dV.
\]
As a distribution $f$ is of order $0$.
}

\begin{lemma}\label{zweivariablen}
Let $M$ and $N$ be differentiable manifolds equipped with smooth volume
densities.
Let $E\to {M}$ and $F\to N$ be vector bundles. 
Let $K\subset N$ be compact and
let $\varphi\in C^k({M}\times N,E\boxtimes F^*)$ be such that
$supp(\varphi)\subset M\times K$. 
Let $m\leq k$ and let $T\in\mathcal{D}'(N,F,\K)$ be a distribution of
order $m$.
Then the map
\be
&f:{M}\to E,&\\
&x\mapsto T[\varphi(x,\cdot)],&
\ee
defines a $C^{k-m}$-section in $E$ with support contained in the
projection of $\supp(\varphi)$ to the first factor, i.~e., $\supp(f)\subset
\{x\in M\,|\,\exists\, y\in K\textrm{ such that }(x,y)\in\supp(\varphi)\}$. 
In particular, if $\varphi$ is smooth with compact support, and $T$ is any
distribution in $F$, then $f$ is a smooth section in $E$ with compact support.

Moreover, $x$-derivatives up to order $k-m$ may be interchanged with $T$.
More precisely, if $P$ is a linear differential operator of order $\leq k-m$
acting on sections in $E$, then
$$
P f = T[P_x\phi(x,\cdot)].
$$
\end{lemma}

Here $E\boxtimes F^*$ 
\indexs{E*@$E\boxtimes F^*$, exterior tensor product of vector bundles $E$ and
  $F^*$} 
denotes the vector bundle over $M\times N$ whose fiber
over $(x,y)\in M\times N$ is given by $E_x\otimes F^*_y$.

\begin{proof} 
There is a canonical isomorphism
\be
E_x\otimes\DD^k(N,F^*)&\to&\DD^k(N,E_x\otimes
F^*),\\ 
v\otimes s&\mapsto&(y\mapsto v\otimes s(y)).
\ee
Thus we can apply $\id_{E_x}\otimes T$ to
$\varphi(x,\cdot)\in\DD^k(N,E_x\otimes F^*)\cong E_x\otimes\DD^k(N,F^*)$ and
we obtain $(\id_{E_x}\otimes T)[\varphi(x,\cdot)]\in E_x$.
We briefly write $T[\varphi(x,\cdot)]$ instead of $(\id_{E_x}\otimes
T)[\varphi(x,\cdot)]$.

To see that the section $x\mapsto T[\varphi(x,\cdot)]$ in $E$ is of regularity
$C^{k-m}$ we may assume that $M$ is an open ball in $\R^p$ and that the 
vector bundle $E\to{M}$ is trivialized over $M$, $E=M\times \K^n$,
because differentiability and continuity are local properties.

For fixed $y\in N$ the map $x \mapsto \phi(x,y)$ is a $C^k$-map $U \to
\K^n\otimes F_y^*$.
We perform a Taylor expansion at $x_0\in U$, see \cite[p.~38f]{FL1}.
For $x\in U$ we get
\begin{eqnarray*}
\lefteqn{\phi(x,y)}\\ 
&=&
\sum_{|\alpha|\leq k-m-1}\tfrac{1}{\alpha!}D_x^\alpha
\phi(x_0,y)(x-x_0)^\alpha
\\&&
+ \sum_{|\alpha|= k-m}\frac{k-m}{\alpha!}\int_0^1 (1-t)^{k-m-1}D_x^\alpha
\phi((1-t)x_0+tx,y)(x-x_0)^\alpha \dt\\
&=&
\sum_{|\alpha|\leq k-m}\tfrac{1}{\alpha!}D_x^\alpha \phi(x_0,y)(x-x_0)^\alpha
+ \\
&&
\sum_{|\alpha|= k-m}\tfrac{k-m}{\alpha!}\int_0^1 (1-t)^{k-m-1}\left(D_x^\alpha
\phi((1-t)x_0+tx,y)- D_x^\alpha \phi(x_0,y)\right)\dt\cdot (x-x_0)^\alpha .
\end{eqnarray*}
Here we used the usual multi-index notation,
$\alpha=(\alpha_1,\ldots,\alpha_p)\in\N^p$,
$|\alpha|=\alpha_1+\cdots+\alpha_p$,
$D_x^\alpha=\frac{\partial^{|\alpha|}}{(\partial x^1)^{\alpha_1}\cdots(\partial
  x^p)^{\alpha_p}}$, and $x^\alpha=x_1^{\alpha_1}\cdots x_p^{\alpha_p}$.
For $|\alpha|\leq k-m$ we certainly have $D_x^\alpha\phi(\cdot,\cdot)\in
C^m(U\times N,\K^n\otimes F^*)$ and, in particular, 
$D_x^\alpha \phi(x_0,\cdot)\in \DD^m(N,\K^n\otimes F^*)$.
We apply $T$ to get
\begin{eqnarray}\label{eq:TTaylor}
\lefteqn{T[\phi(x,\cdot)]}\\ 
&=&
\sum_{|\alpha|\leq k-m}\tfrac{1}{\alpha!}
T[D_x^\alpha \phi(x_0,\cdot)](x-x_0)^\alpha + \nonumber\\
&&
\sum_{|\alpha|= k-m}\tfrac{k-m}{\alpha!}T\left[\int_0^1
(1-t)^{k-m-1}\left(D_x^\alpha \phi((1-t)x_0+tx,\cdot)- D_x^\alpha
  \phi(x_0,\cdot)\right)\dt\right](x-x_0)^\alpha . \nonumber
\end{eqnarray}
Restricting the $x$ to a compact convex neighborhood $U'\subset U$ of $x_0$
the $D_x^\alpha\phi(\cdot,\cdot)$ and all their $y$-derivatives up to order
$m$ are {\em uniformly} continuous on $U'\times K$.
Given $\epsilon>0$ there exists $\delta>0$ so that
$|\nabla_y^j D_x^\alpha\phi(\tilde x,y)-\nabla_y^j
D_x^\alpha\phi(x_0,y)|\leq\frac{\epsilon}{m+1}$ whenever $|\tilde
x-x_0|\leq\delta$, $j=0,\ldots,m$.
Thus for $x$ with $|x-x_0|\leq\delta$ 
\begin{eqnarray*}
\lefteqn{\left\|\int_0^1 (1-t)^{k-m-1}\left(D_x^\alpha
  \phi((1-t)x_0+tx,\cdot)-
  D_x^\alpha\phi(x_0,\cdot)\right)\dt\right\|_{C^m(M)}}\\ 
&=& 
\left\|\int_0^1 (1-t)^{k-m-1}\left(D_x^\alpha \phi((1-t)x_0+tx,\cdot)-
  D_x^\alpha\phi(x_0,\cdot)\right)\dt\right\|_{C^m(K)}\\
&\leq&
\int_0^1 (1-t)^{k-m-1}\left\|D_x^\alpha \phi((1-t)x_0+tx,\cdot)-
  D_x^\alpha\phi(x_0,\cdot)\right\|_{C^m(K)}\dt\\
&\leq&
\int_0^1 (1-t)^{k-m-1}\epsilon\dt\\
&=& 
\frac{\epsilon}{k-m} .
\end{eqnarray*}
Since $T$ is of order $m$ this implies in (\ref{eq:TTaylor}) that 
$T[\int_0^1 \cdots \dt] \to 0$ as $x\to x_0$.
Therefore the map $x\mapsto T[\phi(x,\cdot)]$ is $k-m$ times differentiable
with derivatives $D_x^\alpha|_{x=x_0}T[\phi(x,\cdot)] = T[D_x^\alpha
\phi(x_0,\cdot)]$.
The same argument also shows that these derivatives are continuous in $x$.
\end{proof}

%%%%%%%%%%%%%%%%%%%%%%%%%%%%%%%%%%%%%%%%%%%%%%%%%%%%%%%%%%%%%%%%%%%%%%%%%
\subsection{Differential operators acting on distributions}
\label{subseq:DiffOpDist}
%%%%%%%%%%%%%%%%%%%%%%%%%%%%%%%%%%%%%%%%%%%%%%%%%%%%%%%%%%%%%%%%%%%%%%%%% 

Let $E$ and $F$ be two $\K$-vector bundles over the manifold $M$, $\K=\R$ or
$\K=\Co$. 
Consider a linear differential operator $P:C^\infty(M,E)\to
C^\infty(M,F)$.
\indexn{differential operator}
There is a unique linear differential operator $P^*:C^\infty(M,F^*)\to
C^\infty(M,E^*)$ called the \defem{formal adjoint} of 
\indexn{formal adjoint of differential operator>defemidx}
$P$ such that
for any $\varphi\in \DD(M,E)$ and $\psi\in\DD(M,F^*)$
\begin{equation}
\int_M \psi(P\varphi)\, \dV =
\int_M (P^*\psi)(\varphi)\, \dV .
\label{formaladjoint}
\end{equation}
If $P$ is of order $k$, then so is $P^*$ and (\ref{formaladjoint}) holds for 
all $\varphi\in C^k(M,E)$ and $\psi\in C^k(M,F^*)$ such that
$\supp(\varphi)\cap\supp(\psi)$ is compact.
With respect to the canonical identification $E=(E^*)^*$ we have $(P^*)^*=P$.

Any linear differential operator $P:C^\infty(M,E)\to C^\infty(M,F)$ extends 
canonically to a linear operator $P:\DD'(M,E,W)\to\DD'(M,F,W)$ by
$$
(PT)[\varphi]:=T[P^*\varphi]
$$
where $\phi\in\DD(M,F^*)$.
If a sequence $(\varphi_n)_n$ converges in $\mathcal{D}(M,F^*)$ to $0$, 
then the sequence $(P^*\varphi_n)_n$ converges to $0$ as well 
because $P^*$ is a differential operator.
Hence $(PT)[\varphi_n] = T[P^*\varphi_n] \to 0$.
Therefore $PT$ is again a distribution.

The map $P:\mathcal{D}'(M,E,W)\rightarrow \mathcal{D}'(M,F,W)$
is $\K$-linear.
If $P$ is of order $k$ and $\varphi$ is a $C^k$-section in $E$, seen as a 
$\K$-valued distribution in $E$, then the distribution $P\varphi$ coincides 
with the continuous section obtained by applying $P$ to $\varphi$ classically.

An important special case occurs when $P$ is of order $0$, i.~e., 
$P\in C^\infty(M,\Hom(E,F))$.
Then $P^*\in C^\infty(M,\Hom(F^*,E^*))$ is the pointwise adjoint.
In particular, for a function $f\in C^\infty(M)$ we have 
$$
(fT)[\varphi] = T[f\varphi].
$$

%%%%%%%%%%%%%%%%%%%%%%%%%%%%%%%%%%%%%%%%%%%%%%%%%%%%%%%%%%%%%%%%%%%%%%%%%
\subsection{Supports}
%%%%%%%%%%%%%%%%%%%%%%%%%%%%%%%%%%%%%%%%%%%%%%%%%%%%%%%%%%%%%%%%%%%%%%%%%

\Definition{
The \defem{support} of a distribution $T\in\DD'(M,E,W)$ is defined as the set
\indexn{support of a distribution>defemidx}
\indexs{su*@$\protect\supp(T)$, support of distribution $T$}
\begin{eqnarray*}
\lefteqn{\supp(T)} \\
&:=& 
\{x\in M\,|\, \forall\mbox{ neighborhood $U$ of $x$ }\exists \,
\varphi\in \DD(M,E)\mbox{ with }\supp(\varphi)\subset U
\mbox{ and } T[\varphi]\not= 0\}.
\end{eqnarray*}
}
It follows from the definition that the support of $T$ is a closed subset of
$M$.
In case $T$ is a $L^1_{\mathrm{loc}}$-section this notion of support coincides
with the usual one for sections.

If for $\phi\in\DD(M,E^*)$ the supports of $\phi$ and $T$ are disjoint, then
$T[\phi]=0$.
Namely, for each $x\in\supp(\phi)$ there is a neighborhood $U$ of $x$ such
that $T[\psi]=0$ whenever $\supp(\psi)\subset U$.
Cover the compact set $\supp(\phi)$ by finitely many such open sets
$U_1,\ldots,U_k$. 
Using a partition of unity one can write $\phi=\psi_1 + \cdots + \psi_k$ with
$\psi_j\in\DD(M,E^*)$ and $\supp(\psi_j)\subset U_j$.
Hence 
$$
T[\phi]=T[\psi_1 + \cdots + \psi_k] = T[\psi_1] + \cdots + T[\psi_k]=0.
$$
Be aware that it is not sufficient to assume that $\phi$ vanishes on
$\supp(T)$ in order to ensure $T[\phi]=0$.
For example, if $M=\R$ and $E$ is the trivial $\K$-line bundle let
$T\in\DD'(\R,\K)$ be given by $T[\phi]=\phi'(0)$.
Then $\supp(T)=\{0\}$ but $T[\phi]=\phi'(0)$ may well be nonzero while
$\phi(0)=0$.

If $T\in\DD'(M,E,W)$ and $\varphi\in C^\infty(M,E^*)$, then the evaluation
$T[\varphi]$ can be defined if $\supp(T)\cap\supp(\varphi)$ is compact
even if the support of $\phi$ itself is noncompact.
To do this pick a function $\sigma\in \DD(M,\R)$ that is constant $1$ on
a neighborhood of $\supp(T)\cap\supp(\varphi)$ and put
$$
T[\varphi] := T[\sigma\varphi].
$$
This definition is independent of the choice of $\sigma$ since for another
choice $\sigma'$ we have 
$$
T[\sigma\varphi] - T[\sigma'\varphi] =
T[(\sigma-\sigma')\varphi] = 0
$$
because $\supp((\sigma-\sigma')\varphi)$ and $\supp(T)$ are disjoint.

Let $T\in\DD'(M,E,W)$ and let $\Omega\subset M$ be an open subset.
Each test section $\varphi\in\DD(\Omega,E^*)$ can be extended by $0$ and yields
a test section $\varphi \in \DD(M,E^*)$.
This defines an embedding $\DD(\Omega,E^*)\subset\DD(M,E^*)$.
By the restriction of $T$ to $\Omega$ we mean its restriction from 
$\DD(M,E^*)$ to $\DD(\Omega,E^*)$.

\Definition{
The \defem{singular support}
\indexn{singular support of a distribution>defemidx}
 $\ssupp(T)$ 
\indexs{ss@$\protect\ssupp(T)$, singular support of distribution $T$}
of a distribution $T\in\DD'(M,E,W)$ 
is the set of points which do not have a neighborhood restricted to which
$T$ coincides with a smooth section.
}

The singular support is also closed and we always have $\ssupp(T)
\subset\supp(T)$.

\Example{
For the delta-distribution $\delta_x$ we have
$\supp(\delta_x)=\ssupp(\delta_x)=\{x\}$.
}

\subsection{Convergence of distributions}

The space $\DD'(M,E)$ of distributions in $E$ will always be given the
\defem{weak topology}.
\indexn{weak topology for distributions>defemidx}
This means that $T_n \to T$ in $\DD'(M,E,W)$ if and only if $T_n[\varphi]
\to T[\varphi]$ for all $\varphi\in\DD(M,E^*)$.
Linear differential operators $P$ are always continuous with respect to the
weak topology. 
Namely, if $T_n \to T$, then we have for every $\varphi\in\DD(M,E^*)$
$$
PT_n[\varphi] = T_n[P^*\varphi] \to T[P^*\varphi] = PT[\varphi].
$$
Hence
$$
PT_n \to PT.
$$

\begin{lemma}\label{distributionlimesvertausch}
Let $T_n,T\in C^0(M,E)$ and suppose $\|T_n - T\|_{C^0(M)} \to 0$.
Consider $T_n$ and $T$ as distributions.

Then $T_n \to T$ in $\DD'(M,E)$.
In particular, for every linear differential operator $P$
we have $PT_n \to PT$.
\end{lemma}

\begin{proof}
Let $\varphi\in\DD(M,E)$.
Since $\|T_n - T\|_{C^0(M)} \to 0$ and $\varphi\in L^1(M,E)$, it follows from
Lebesgue's dominated convergence theorem: 
\begin{eqnarray*}
\lim_{n\to\infty}T_n[\varphi] 
&=& 
\lim_{n\to\infty}\int_M T_n(x)\cdot\varphi(x) \dV(x)\\
&=&
\int_M \lim_{n\to\infty}(T_n(x)\cdot\varphi(x)) \dV(x)\\
&=&
\int_M (\lim_{n\to\infty}T_n(x))\cdot\varphi(x) \dV(x)\\
&=&
\int_M T(x)\cdot\varphi(x) \dV(x)\\
&=&
T[\varphi] .
\end{eqnarray*}
\end{proof}

%%%%%%%%%%%%%%%%%%%%%%%%%%%%%%%%%%%%%%%%%%%%%%%%%%%%%%%%%%%%%%%%%%%%%%%%%
\subsection{Two auxiliary lemmas}
%%%%%%%%%%%%%%%%%%%%%%%%%%%%%%%%%%%%%%%%%%%%%%%%%%%%%%%%%%%%%%%%%%%%%%%%%

The following situation will arise frequently.
Let $E$, $F$, and $G$ be $\K$-vector bundles over $M$ equipped with metrics
and with connections which we all denote by $\nabla$.
We give $E\otimes F$ and $F^*\otimes G$ the induced metrics and connections.
Here and henceforth $F^*$ will denote the dual bundle to $F$.
The natural pairing $F\otimes F^* \to \K$ given by evaluation of the second
factor on the first yields a vector bundle homomorphism
$E\otimes F\otimes F^*\otimes G \to E\otimes G$ which we write as
$\varphi\otimes\psi \mapsto \varphi\cdot\psi$.
\footnote{If one identifies $E\otimes F$ with $\Hom(E^*,F)$ and $F^*\otimes G$
with $\Hom(F,G)$, then $\varphi\cdot\psi$ corresponds to $\psi\circ\varphi$.}

\begin{lemma}\label{CkNormvonProdukt}
For all $C^k$-sections $\varphi$ in $E\otimes F$ and $\psi$ in $F^*\otimes G$
and all $A\subset M$ we have
$$
\|\varphi\cdot\psi\|_{C^k(A)}
\leq
2^k\cdot \|\varphi\|_{C^k(A)}\cdot\|\psi\|_{C^k(A)}.
$$
\end{lemma}

\begin{proof}
The case $k=0$ follows from the Cauchy-Schwarz inequality.
Namely, for fixed $x\in M$ we choose an orthonormal basis $f_i$, $i=1,\ldots,
r$, for $F_x$.
Let $f_i^*$ be the basis of $F_x^*$ dual to $f_i$.
We write $\phi(x)=\sum_{i=1}^r e_i\otimes f_i$ for suitable $e_i\in E_x$ and
similarly  $\psi(x)=\sum_{i=1}^r f_i^*\otimes g_i$, $g_i\in G_x$.
Then $\phi(x)\cdot\psi(x)=\sum_{i=1}^r e_i \otimes g_i$ and we see
\begin{eqnarray*}
|\phi(x)\cdot\psi(x)|^2
&=&
|\sum_{i=1}^r e_i \otimes g_i|^2\\
&=&
\sum_{i,j=1}^r \langle e_i\otimes g_i,e_j\otimes g_j\rangle\\
&=&
\sum_{i,j=1}^r \langle e_i,e_j\rangle \langle g_i,g_j\rangle\\
&\leq&
\sqrt{\sum_{i,j=1}^r \langle e_i,e_j\rangle^2}\cdot
\sqrt{\sum_{i,j=1}^r \langle g_i,g_j\rangle^2}\\
&\leq&
\sqrt{\sum_{i,j=1}^r |e_i|^2|e_j|^2}\cdot
\sqrt{\sum_{i,j=1}^r |g_i|^2|g_j|^2}\\
&=&
\sqrt{\sum_{i=1}^r |e_i|^2\sum_{j=1}^r|e_j|^2}\cdot
\sqrt{\sum_{i=1}^r |g_i|^2\sum_{j=1}^r|g_j|^2}\\
&=&
\sum_{i=1}^r |e_i|^2 \cdot \sum_{i=1}^r |g_i|^2\\
&=&
|\phi(x)|^2\cdot |\psi(x)|^2.
\end{eqnarray*}
Now we proceed by induction on $k$.
\begin{eqnarray*}
\|\nabla^{k+1}(\varphi\cdot\psi)\|_{C^{0}(A)}
&\leq&
\|\nabla(\varphi\cdot\psi)\|_{C^{k}(A)}\\
&=&
\|(\nabla\varphi)\cdot\psi + \varphi\cdot\nabla\psi\|_{C^{k}(A)}\\
&\leq&
\|(\nabla\varphi)\cdot\psi\|_{C^{k}(A)} 
+ \|\varphi\cdot\nabla\psi\|_{C^{k}(A)}\\
&\leq&
2^k\cdot\|\nabla\varphi\|_{C^{k}(A)} \cdot\|\psi\|_{C^{k}(A)} 
+ 2^k\cdot\|\varphi\|_{C^{k}(A)}\cdot\|\nabla\psi\|_{C^{k}(A)}\\
&\leq&
2^k\cdot\|\varphi\|_{C^{k+1}(A)} \cdot\|\psi\|_{C^{k+1}(A)} 
+ 2^k\cdot\|\varphi\|_{C^{k+1}(A)}\cdot\|\psi\|_{C^{k+1}(A)}\\
&=&
2^{k+1} \cdot\|\varphi\|_{C^{k+1}(A)}\cdot\|\psi\|_{C^{k+1}(A)} .
\end{eqnarray*}
Thus 
\begin{eqnarray*}
\|\varphi\cdot\psi\|_{C^{k+1}(A)}
&=&
\max\{\|\varphi\cdot\psi\|_{C^{k}(A)},
\|\nabla^{k+1}(\varphi\cdot\psi)\|_{C^{0}(A)}\}\\
&\leq&
\max\{2^k\cdot\|\varphi\|_{C^{k}(A)}\cdot\|\psi\|_{C^{k}(A)},
2^{k+1} \cdot\|\varphi\|_{C^{k+1}(A)}\cdot\|\psi\|_{C^{k+1}(A)}\}\\
&=&
2^{k+1} \cdot\|\varphi\|_{C^{k+1}(A)}\cdot\|\psi\|_{C^{k+1}(A)} .
\end{eqnarray*}
\end{proof}

This lemma allows us to estimate the $C^k$-norm of products of sections in
terms of the $C^k$-norms of the factors.
The next lemma allows us to deal with compositions of functions.
We recursively define the following universal constants:
$$
\alpha(k,0) := 1 ,
$$
$$
\alpha(k,j) := 0
$$
for $j> k$ and for $j<0$ and
\begin{equation}
\alpha(k+1,j) := \max\{\alpha(k,j),\, 2^k \cdot \alpha(k,j-1)\}
\label{alphaconst}  
\end{equation}
if $1 \leq j \leq k$.
The precise values of the $\alpha(k,j)$ are not important.
The definition was made in such a way that the following lemma holds.

\begin{lemma}\label{CkNormvonVerkettung}
Let $\Gamma$ be a real valued $C^k$-function on a Lorentzian manifold $M$
and let $\sigma :\R\to\R$ be a $C^k$-function.
Then for all $A\subset M$ and $I \subset\R$ such that $\Gamma(A)\subset I$
we have
$$
\|\sigma\circ\Gamma\|_{C^k(A)}
\leq
\|\sigma\|_{C^k(I)} \cdot \max_{j=0,\ldots,k}\alpha(k,j)\|\Gamma\|_{C^k(A)}^j .$$
\end{lemma}

\begin{proof}
We again perform an induction on $k$.
The case $k=0$ is obvious.
By Lemma~\ref{CkNormvonProdukt}
\begin{eqnarray*}
\|\nabla^{k+1}(\sigma\circ\Gamma)\|_{C^0(A)}
&=&
\|\nabla^{k}[(\sigma'\circ\Gamma)\cdot\nabla\Gamma]\|_{C^0(A)}\\
&\leq&
\|(\sigma'\circ\Gamma)\cdot\nabla\Gamma\|_{C^k(A)}\\
&\leq&
2^k\cdot \|\sigma'\circ\Gamma\|_{C^k(A)}\cdot \|\nabla\Gamma\|_{C^k(A)}\\
&\leq&
2^k\cdot \|\sigma'\circ\Gamma\|_{C^k(A)} \cdot\|\Gamma\|_{C^{k+1}(A)}\\
&\leq&
2^k\cdot \|\sigma'\|_{C^k(I)} \cdot
\max_{j=0,\ldots,k}\alpha(k,j)\|\Gamma\|_{C^{k+1}(A)}^j \cdot
\|\Gamma\|_{C^{k+1}(A)}\\
&\leq&
2^k\cdot \|\sigma\|_{C^{k+1}(I)} \cdot
\max_{j=0,\ldots,k}\alpha(k,j)\|\Gamma\|_{C^{k+1}(A)}^{j+1}\\
&=&
2^k\cdot \|\sigma\|_{C^{k+1}(I)} \cdot
\max_{j=1,\ldots,k+1}\alpha(k,j-1)\|\Gamma\|_{C^{k+1}(A)}^{j} .
\end{eqnarray*}
Hence
\begin{eqnarray*}
\|\sigma\circ\Gamma\|_{C^{k+1}(A)}
&=&
\max\{\|\sigma\circ\Gamma\|_{C^{k}(A)},
\|\nabla^{k+1}(\sigma\circ\Gamma)\|_{C^0(A)}\}\\
&\leq&
\max\{\|\sigma\|_{C^k(I)} \cdot
\max_{j=0,\ldots,k}\alpha(k,j)\|\Gamma\|_{C^k(A)}^j,\\
&&\quad\quad
2^k\cdot \|\sigma\|_{C^{k+1}(I)} \cdot
\max_{j=1,\ldots,k+1}\alpha(k,j-1)\|\Gamma\|_{C^{k+1}(A)}^{j}
\}\\
&\leq&
\|\sigma\|_{C^{k+1}(I)}\cdot\max_{j=0,\ldots,k+1}
\max\{\alpha(k,j),2^k\alpha(k,j-1)\}\|\Gamma\|_{C^{k+1}(A)}^{j}\\ 
&=&
\|\sigma\|_{C^{k+1}(I)}\cdot\max_{j=0,\ldots,k+1}
\alpha(k+1,j)\|\Gamma\|_{C^{k+1}(A)}^{j} . 
\end{eqnarray*}
\end{proof}

%%%%%%%%%%%%%%%%%%%%%%%%%%%%%%%%%%%%%%%%%%%%%%%%%%%%%%%%%%%%%%%%%%%%%%%%%
\section{Riesz distributions on Minkowski space}
%%%%%%%%%%%%%%%%%%%%%%%%%%%%%%%%%%%%%%%%%%%%%%%%%%%%%%%%%%%%%%%%%%%%%%%%%

The distributions $R_+(\alpha)$ and $R_-(\alpha)$ to be defined below were
introduced by M.~Riesz in the first half of the 20$^\mathrm{th}$ century in
order to find solutions to certain differential equations.
He collected his results in \cite{Riesz}.
We will derive all relevant facts in full detail.

Let $V$ be an $n$-dimensional real vector space, let $\la\cdot,\cdot\ra$
be a nondegenerate symmetric bilinear form of index $1$ on $V$.
Hence $(V,\la\cdot,\cdot\ra)$ is isometric to $n$-dimensional Minkowski
space $(\R^n,\la\cdot,\cdot\ra_0)$ where $\la x,y\ra_0= -x_1y_1 +
x_2y_2 + \cdots + x_ny_n$.
Set 
\indexs{g*@$\gamma$, quadratic form associated to Minkowski product}
\begin{equation}
\gamma: V \to \R,\quad \gamma(X):=-\la X,X\ra.
\label{gammadef}
\end{equation}
A nonzero vector $X\in V\setminus \{0\}$ is called \defem{timelike} (or
\defem{lightlike} or \defem{spacelike}) if and only if $\gamma(X)>0$ (or
$\gamma(X)=0$ or $\gamma(X)<0$ respectively).\indexn{timelike
  vector>defemidx}\indexn{lightlike vector>defemidx}\indexn{spacelike vector>defemidx}
The zero vector $X=0$ is considered as spacelike.
The set $I(0)$ of timelike vectors consists of two connected components.
We choose a \defem{timeorientation} on $V$ by picking one of these two connected
components.
\indexn{timeorientation>defemidx}
Denote this component by $I_+(0)$ and call its elements \defem{future directed}.
\indexn{future-directed vector>defemidx}
Put $J_+(0) := \overline{I_+(0)}$, $C_+(0) := \partial I_+(0)$, 
$I_-(0):= -I_+(0)$, $J_-(0):= -J_+(0)$, and $C_-(0):= -C_+(0)$.
\indexs{I*@$I_+(0)$, chronological future in Minkowski space}
\indexs{I*@$I_-(0)$, chronological past in Minkowski space}
\indexs{C*@$C_+(0)$, future light cone in Minkowski space}
\indexs{C*@$C_-(0)$, past light cone in Minkowski space}
\indexs{J*@$J_+(0)$ causal future in Minkowski space}
\indexs{J*@$J_-(0)$ causal past in Minkowski space}
\begin{center}
\input{fig-lichtkegel}
\end{center}

\Definition{ 
For any complex number $\alpha$ with $\Re(\alpha)>n$ let $R_+(\alpha)$ 
and $R_-(\alpha)$ be the complex-valued continuous functions on $V$
defined by
\[
R_{\pm}(\alpha)(X)
:=
\left\{\begin{array}{cl}
  C(\alpha,n)\gamma(X)^{\frac{\alpha-n}{2}},& \textrm{ if }X\in J_\pm(0),\\ 
0, &  \textrm{ otherwise,}\end{array}\right.
\] 
  where $C(\alpha,n):=\frac{2^{1-\alpha}\pi^{\frac{2-n}{2}}}{
    (\frac{\alpha}{2}-1)!(\frac{\alpha-n}{2})!}$  
and $z \mapsto (z-1)!$ is the Gamma function.
\indexs{C*@$ C(\alpha,n)$, coefficients involved in definition of Riesz
    distributions} 
}

For $\alpha\in\Co$ with $\Re(\alpha)\leq n$ this definition no longer yields
continuous functions due to the singularities along $C_\pm(0)$.
This requires a more careful definition of $R_\pm(\alpha)$ as a distribution
which we will give below.
Even for $\Re(\alpha)>n$ we will from now on consider the continuous functions
$R_\pm(\alpha)$ as distributions as explained in Example~\ref{ex:locint}.

Since the Gamma function has no zeros the map $\alpha\mapsto C(\alpha,n)$ is
holomorphic on $\Co$. 
Hence for each fixed testfunction $\varphi\in\DD(V,\Co)$ the map $\alpha\mapsto
R_{\pm}(\alpha)[\varphi]$ yields a holomorphic
function on $\{\Re(\alpha)>n\}$.

There is a natural differential operator $\Box$ acting on functions on $V$,
$\Box f:= \partial_{e_1}\partial_{e_1}f - \partial_{e_2}\partial_{e_2}f -
\cdots - \partial_{e_n}\partial_{e_n}f$ where $e_1,\ldots,e_n$ is any basis of
$V$ such that $-\la e_1,e_1 \ra = \la e_2,e_2 \ra = \cdots =\la
e_n,e_n\ra = 1$ and $\la e_i,e_j\ra=0$ for $i\not= j$.
Such a basis $e_1,\ldots,e_n$ is called \defem{Lorentzian orthonormal}.
\indexn{Lorentzian orthonormal basis>defemidx}
The operator $\Box$ is called the {\em d'Alembert operator}.
The formula in Minkowski space with respect to the standard basis may look
more familiar to the reader,
\indexn{d'Alembert operator}
$$
\Box = \frac{\partial^2}{(\partial x^1)^2} - \frac{\partial^2}{(\partial
  x^2)^2} - 
\cdots - \frac{\partial^2}{(\partial x^n)^2}.
$$
The definition of the d'Alembertian on general Lorentzian manifolds can be
found in the next section.
In the following lemma the application of differential operators such as
$\Box$ to the $R_\pm(\alpha)$ is to be taken in the distributional sense.

\begin{lemma}\label{erstvorbereitung}
For all $\alpha\in\Co$ with $\Re(\alpha)>n$ we have
\begin{enumerate}
\item\label{zweitens}
 $\gamma\cdot R_\pm(\alpha)=\alpha(\alpha-n+2) R_{\pm}(\alpha+2)$,
\item\label{drittens} 
 $(\grad\gamma)\cdot R_\pm(\alpha)=2\alpha\;\grad\, R_{\pm}(\alpha+2)$,
\item\label{viertens}
 $\Box R_{\pm}(\alpha+2)=R_{\pm}(\alpha)$.  
\item\label{d}
The map $\alpha \mapsto R_{\pm}(\alpha)$ extends uniquely to $\Co$ as a
holomorphic family of distributions.
In other words, for each $\alpha \in \Co$ there exists a unique distribution
$R_{\pm}(\alpha)$ on $V$ such that for each testfunction $\varphi$ the map
$\alpha 
\mapsto R_{\pm}(\alpha)[\varphi]$ is holomorphic.
\end{enumerate}
\end{lemma}
\begin{proof}
Identity (\ref{zweitens}) follows from
\begin{equation*}
\frac{C(\alpha,n)}{C(\alpha+2,n)}=  
\frac{2^{(1-\alpha)}\;(\frac{\alpha+2}{2}-1)!\;(\frac{\alpha+2-n}{2})!
}{2^{(1-\alpha-2)}\;(\frac{\alpha}{2}-1)!\;(\frac{\alpha-n}{2})!}
=\alpha\,(\alpha-n+2).
\end{equation*}
To show (\ref{drittens}) we choose a Lorentzian orthonormal basis $e_1,
\ldots, e_n$ of $V$ and we denote differentiation in direction $e_i$ by
$\del_i$.
We fix a testfunction $\varphi$ and integrate by parts: 
\begin{eqnarray*}
\del_i\gamma\cdot R_\pm(\alpha)[\varphi]
&=&
C(\alpha,n)\int_{J_\pm(0)}
\gamma(X)^{\frac{\alpha-n}{2}}\,\del_i\gamma(X)\varphi(X)\,\,dX\\
&=& 
\frac{2C(\alpha,n)}{\alpha+2-n}\int_{J_\pm(0)}
\del_i(\gamma(X)^{\frac{\alpha-n+2}{2}})\varphi(X)\,\,dX\\ 
&=&
-2\alpha C(\alpha+2,n)\int_{J_\pm(0)}
\gamma(X)^{\frac{\alpha-n+2}{2}}\del_i\varphi(X)\,\,dX\\ 
&=&
-2\alpha R_\pm(\alpha+2)[\del_i\varphi]\\
&=&
2\alpha\del_i R_\pm(\alpha+2)[\varphi], 
\end{eqnarray*}
which proves (\ref{drittens}).
Furthermore, it follows from (\ref{drittens}) that
\be
\del_i^2 R_\pm(\alpha+2)
&=&
\del_i\left(\frac{1}{2\alpha} \del_i\gamma\cdot R_\pm(\alpha)\right)\\ 
&=&\frac{1}{2\alpha}\left(\del_i^2\gamma\cdot R_\pm(\alpha)+\del_i\gamma\cdot
  \left(\frac{1}{2(\alpha-2)} \del_i\gamma \cdot
  R_\pm(\alpha-2)\right)\right)\\  
&=&\frac{1}{2\alpha}\del_i^2\gamma\cdot
R_\pm(\alpha)+\frac{1}{4\alpha(\alpha-2)}
(\del_i\gamma)^2\frac{(\alpha-2)(\alpha-n)}{\gamma}\cdot R_\pm(\alpha)\\ 
&=&\left(\frac{1}{2\alpha}\del_i^2\gamma+ \frac{\alpha-n}{4\alpha}\cdot
  \frac{(\del_i\gamma)^2}{\gamma}\right)\cdot R_\pm(\alpha), 
\ee
so that
\be
\Box\, R_\pm(\alpha+2)&=&\left(\frac{n}{\alpha}+
  \frac{\alpha-n}{4\alpha}\cdot\frac{4\gamma}{\gamma}\right)R_\pm(\alpha)\\ 
&=&R_\pm(\alpha).
\ee
To show {(\ref{d})} we first note that for fixed $\phi\in\DD(V,\Co)$ the map
$\{\Re(\alpha)>n\} \to \Co$, $\alpha \mapsto R_\pm(\alpha)[\phi]$, is
holomorphic. 
For $\Re(\alpha)>n-2$ we set 
\begin{equation}\label{defwitR}
\wit{R}_\pm(\alpha):=\Box\, R_\pm(\alpha+2).
\end{equation}
This defines a distribution on $V$.
The map $\alpha\mapsto \wit{R}_\pm(\alpha)$ is then holomorphic on
$\{\Re(\alpha)>n-2\}$.
By (\ref{viertens}) we have $\wit{R}_\pm(\alpha)=R_\pm(\alpha)$ for
$\Re(\alpha)>n$, so that $\alpha\mapsto \wit{R}_\pm(\alpha)$ extends
$\alpha\mapsto 
R_\pm(\alpha)$ holomorphically to $\{\Re(\alpha)>n-2\}$.
We proceed inductively and construct a holomorphic extension of 
$\alpha\mapsto R_\pm(\alpha)$ on $\{\Re(\alpha)>n-2k\}$ (where
$k\in\mathbb{N}\setminus\{0\}$) from that on $\{\Re(\alpha)>n-2k+2\}$ just as
above. 
Note that these extensions necessarily coincide on their common domain since
they are holomorphic and they coincide on an open subset of $\mathbb{C}$.  
We therefore obtain a holomorphic extension of $\alpha\mapsto R_\pm(\alpha)$
to the  whole of $\mathbb{C}$, which is necessarily unique.
\end{proof}

Lemma~\ref{erstvorbereitung} (\ref{d}) defines $R_\pm(\alpha)$ for all
$\alpha\in\Co$, not as functions but as distributions.

\Definition{
We call $R_+(\alpha)$ the \defem{advanced Riesz distribution} and 
$R_-(\alpha)$ the \defem{retarded Riesz distribution} on $V$ for
$\alpha\in\Co$.  
\indexn{Riesz distributions on Minkowski space>defemidx}
\indexn{advanced Riesz distributions}
\indexn{retarded Riesz distributions}
}

The following illustration shows the graphs of Riesz distributions
$R_+(\alpha)$ for $n=2$
and various values of $\alpha$.
%(The set $I_+(0)$ is yellow-colored, the graph of $R_+(\alpha)$ over $I_+(0)$
%is blue-colored.)  
In particular, one sees the singularities along $C_+(0)$ for $\Re(\alpha)\leq 
2$.

\begin{minipage}{6cm}
\begin{center}
\includegraphics*[width=3.5cm]{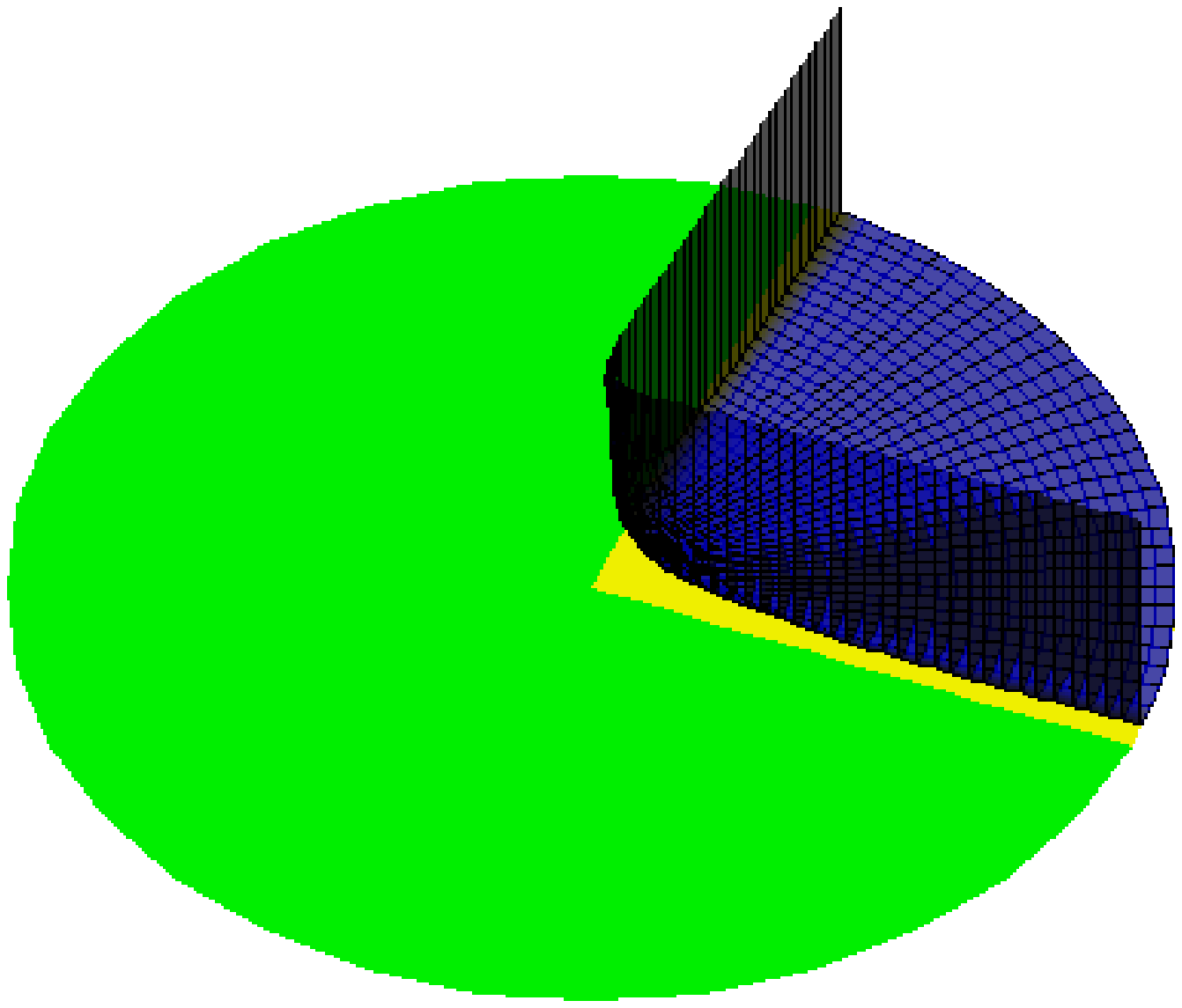}

$\alpha=0.1$
\end{center}
\end{minipage}
\hfill
\begin{minipage}{6cm}
\begin{center}
\includegraphics*[width=3.5cm]{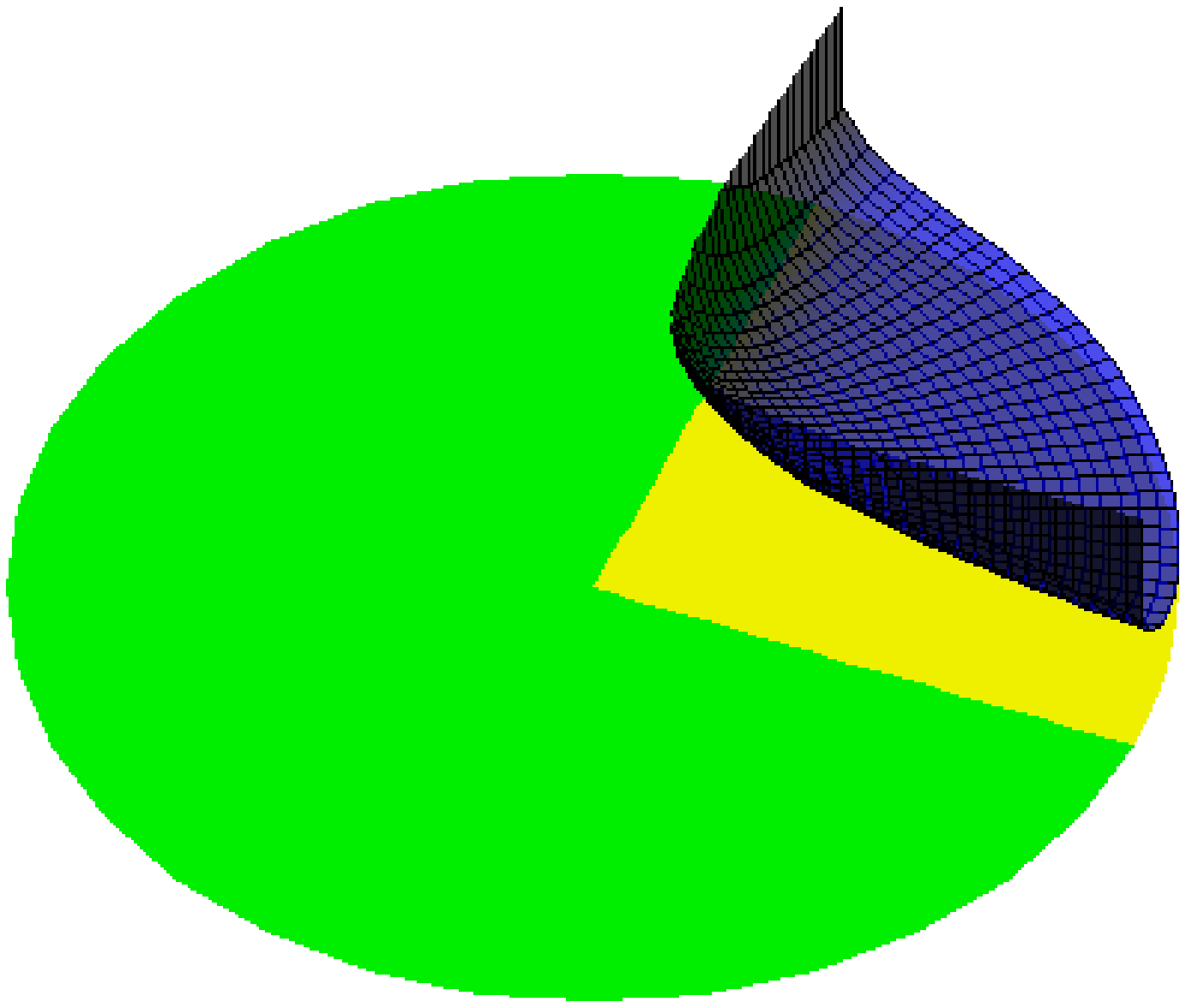}

$\alpha=1$
\end{center}
\end{minipage}

\begin{minipage}{6cm}
\begin{center}
\includegraphics*[width=3.5cm]{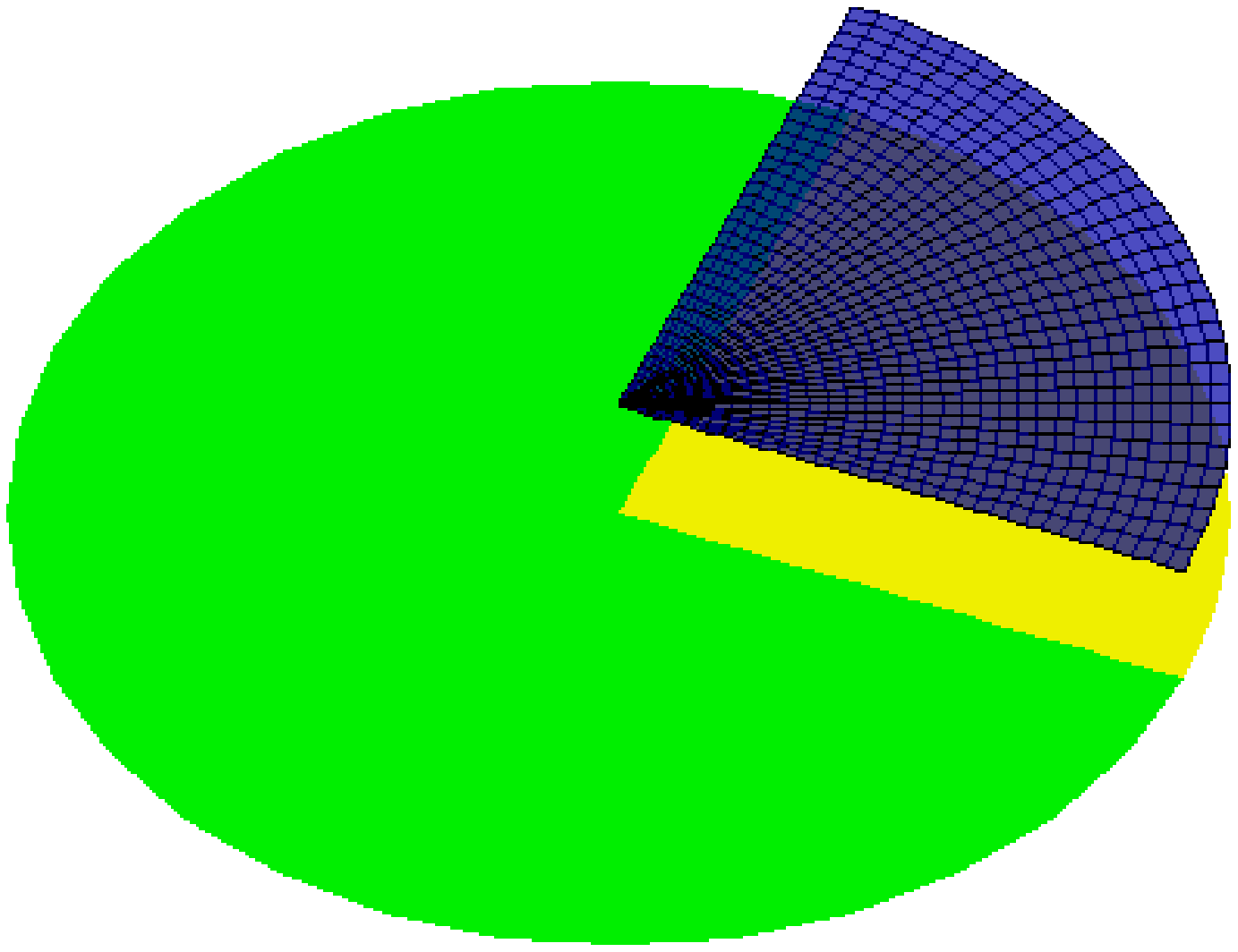}

$\alpha=2$
\end{center}
\end{minipage}
\hfill
\begin{minipage}{6cm}
\begin{center}
\includegraphics*[width=3.5cm]{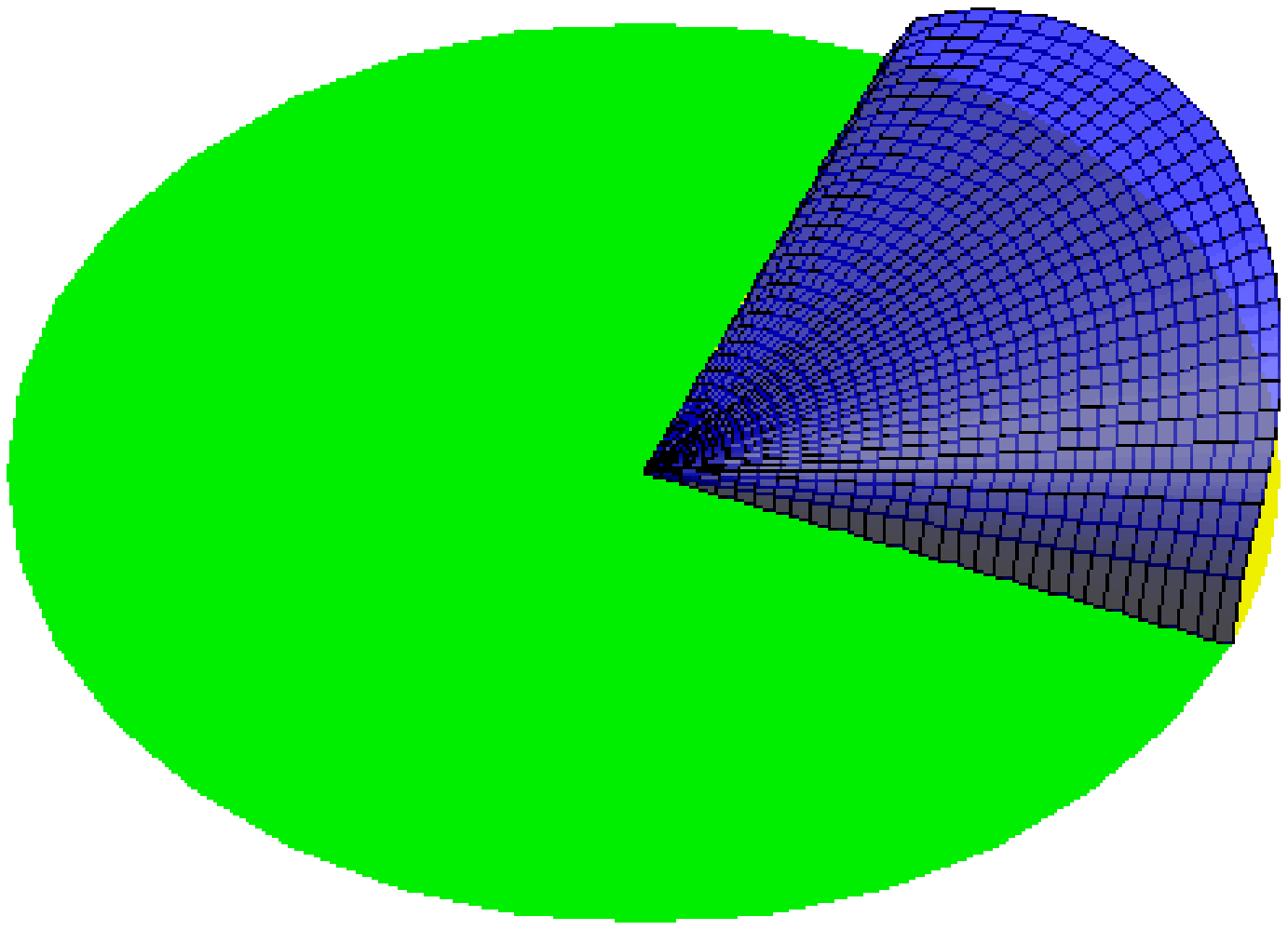}

$\alpha=3$
\end{center}
\end{minipage}

\begin{minipage}{6cm}
\begin{center}
\includegraphics*[width=3.5cm]{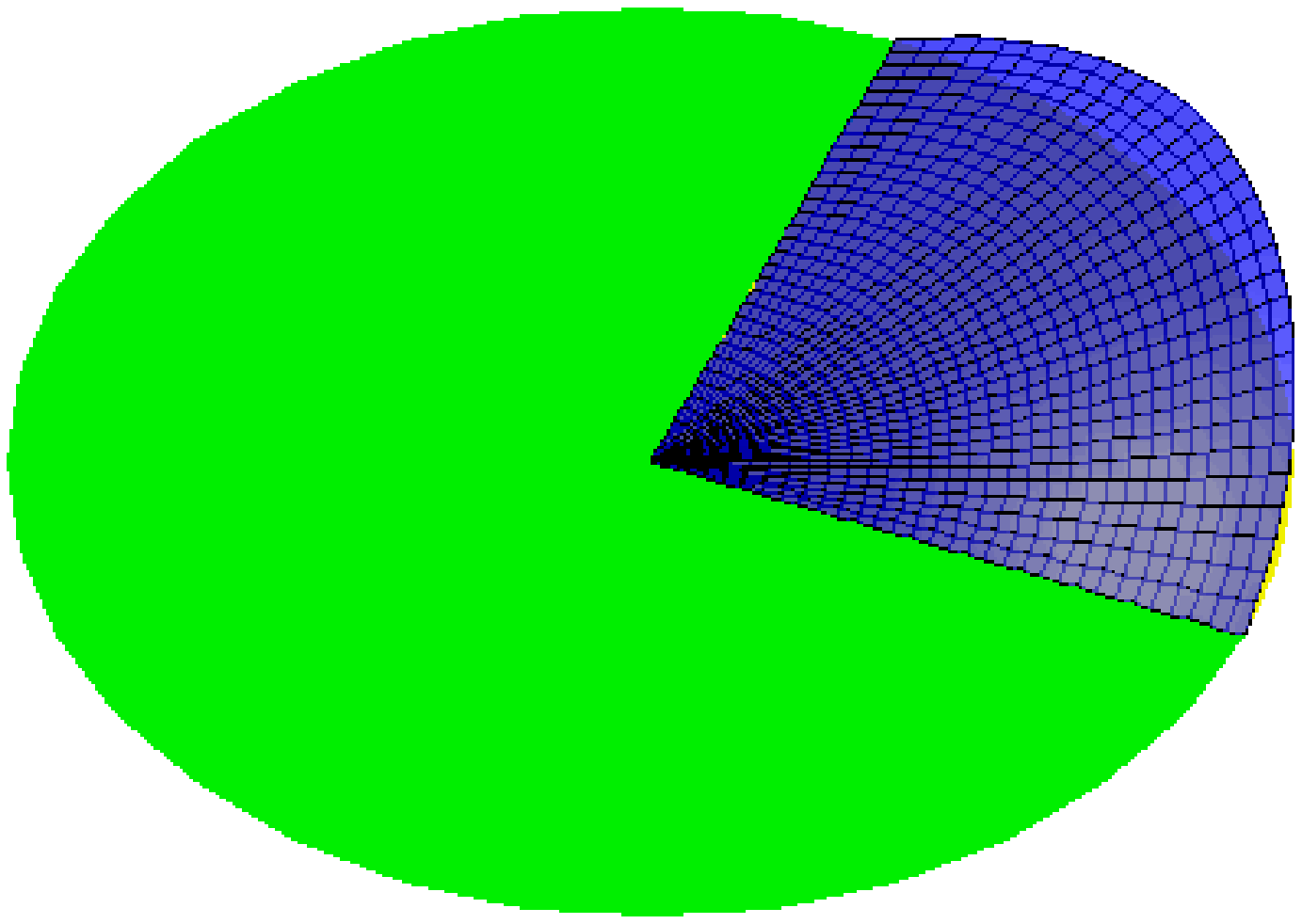}

$\alpha=4$
\end{center}
\end{minipage}
\hfill
\begin{minipage}{6cm}
\begin{center}
\includegraphics*[width=3.5cm]{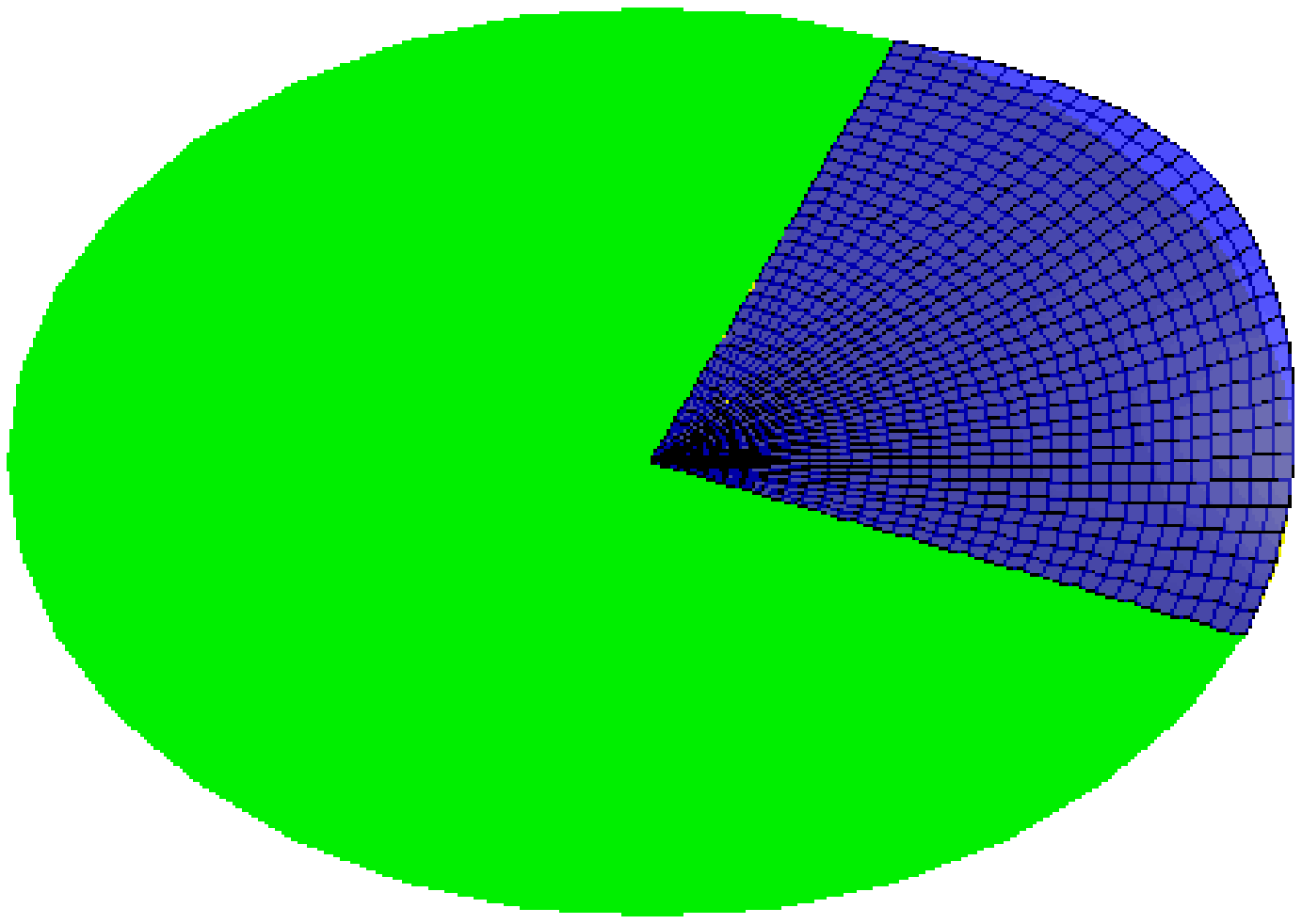}

$\alpha=5$
\end{center}
\end{minipage}

\abb{Graphs of Riesz distributions $R_+(\alpha)$ in two dimensions}

We now collect the important facts on Riesz distributions.

\begin{prop}\label{MinkRiesz}
The following holds for all $\alpha\in\Co$:
\begin{enumerate}
\item\label{a} 
$\gamma\cdot R_{\pm}(\alpha)=\alpha(\alpha-n+2)\, R_{\pm}(\alpha+2)$,
\item\label{b} 
$(\grad\gamma) R_{\pm}(\alpha)=2\alpha\grad\left(R_{\pm}(\alpha+2)\right)$,
\item\label{c} 
$\Box R_{\pm}(\alpha+2)=R_{\pm}(\alpha)$,
\item\label{e} 
For every
  $\alpha\in\mathbb{C}\setminus\left(\{0,-2,-4,\ldots\}
    \cup\{n-2,n-4,\ldots\}\right)$,   
  we have\\ $\supp\left(R_{\pm}(\alpha)\right)=J_\pm(0)$ and
  $\ssupp\left(R_{\pm}(\alpha)\right)\subset C_\pm(0)$.
\item\label{f} 
For every
  $\alpha\in\{0,-2,-4,\ldots\}\cup\{n-2,n-4,\ldots\}$, we have\\ 
$\supp\left(R_{\pm}(\alpha)\right)=\ssupp\left(R_{\pm}(\alpha)\right)\subset
C_\pm(0)$. 
\item\label{g} For $n\geq 3$ and $\alpha=n-2,n-4,\ldots,1$ or $2$
  respectively, we have\\  
$\supp\left(R_{\pm}(\alpha)\right)=
  \ssupp\left(R_{\pm}(\alpha)\right)=C_\pm(0)$.  
\item\label{h} 
$R_{\pm}(0)=\delta_0$.
\item\label{i}
For $\Re(\alpha) > 0$ the order of $R_{\pm}(\alpha)$ is bounded from above by
$n+1$. 
\item\label{j}
If $\alpha\in\R$, then $R_\pm(\alpha)$ is real, i.~e.,
$R_\pm(\alpha)[\phi]\in\R$ for all $\phi\in\DD(V,\R)$.
\end{enumerate}
\end{prop}

\begin{proof}
Assertions {(\ref{a})}, {(\ref{b})}, and {(\ref{c})} hold for 
$\Re(\alpha)>n$ by Lemma~\ref{erstvorbereitung}.
Since, after insertion of a fixed $\phi\in\DD(V,\Co)$, all expressions in these
equations are holomorphic in $\alpha$ they hold for all $\alpha$.

Proof of {(\ref{e})}. 
Let $\varphi\in\DD(V,\Co)$ with $\supp(\varphi)\cap J_\pm(0)=\emptyset$.
Since $\supp(R_\pm(\alpha))\subset J_\pm(0)$ for $\Re(\alpha)>n$, it follows
for those $\alpha$ that
$$
R_\pm(\alpha)[\varphi]=0,
$$
and then for all $\alpha$ by Lemma~\ref{erstvorbereitung} (\ref{d}).
Therefore $\supp(R_\pm(\alpha))\subset J_\pm(0)$ for all $\alpha$.

On the other hand, if $X\in\, I_{\pm}(0)$, then $\gamma(X)>0$ and
the map $\alpha\mapsto C(\alpha,n)\gamma(X)^{\frac{\alpha-n}{2}}$ is
well-defined 
and holomorphic on all of $\mathbb{C}$.
By Lemma~\ref{erstvorbereitung} (\ref{d}) we have for $\phi\in\DD(V,\Co)$ with
$\supp(\varphi)\subset\,I_{\pm}(0)$ 
\[R_\pm(\alpha)[\varphi]=
\int_{\supp(\varphi)}C(\alpha,n)\gamma(X)^{\frac{\alpha-n}{2}}\varphi(X)dX\]  
for \emph{every} $\alpha\in\mathbb{C}$.
Thus $R_\pm(\alpha)$ coincides on $I_{\pm}(0)$ with the smooth function
$C(\alpha,n)\gamma(\cdot)^{\frac{\alpha-n}{2}}$ and therefore
$\ssupp(R_\pm(\alpha))\subset C_\pm(0)$.
Since furthermore the function $\alpha\mapsto C(\alpha,n)$ vanishes only on
$\{0,-2,-4,\ldots\}\cup\{n-2,n-4,\ldots\}$ (caused by the poles of the Gamma
function), we have $I_{\pm}(0)\subset\supp(R_\pm(\alpha))$ for every
$\alpha\in\mathbb{C}\setminus
\left(\{0,-2,-4,\ldots\}\cup\{n-2,n-4,\ldots\}\right)$. 
Thus $\supp(R_\pm(\alpha))=J_\pm(0)$. 
This proves {(\ref{e})}.

Proof of {(\ref{f})}. 
For $\alpha\in\{0,-2,-4,\ldots\}\cup\{n-2,n-4,\ldots\}$ we have
$C(\alpha,n)=0$ and therefore $I_{\pm}(0)\cap\supp(R_\pm(\alpha))=\emptyset$.
Hence $\ssupp(R_\pm(\alpha))\subset\supp(R_\pm(\alpha))\subset C_\pm(0)$.
It remains to show $\supp(R_\pm(\alpha))\subset\ssupp(R_\pm(\alpha))$.
Let $X\not\in\ssupp(R_\pm(\alpha))$.
Then $R_\pm(\alpha)$ coincides with a smooth function $f$ on a neighborhood of
$X$.
Since $\supp(R_\pm(\alpha))\subset C_\pm(0)$ and since $C_\pm(0)$ has a dense
complement in $V$, we have $f\equiv 0$.
Thus $X\not\in\supp(R_\pm(\alpha))$.
This proves {(\ref{f})}.

Before we proceed to the next point we derive a more explicit formula for
the Riesz distributions evaluated on testfunctions of a particular form.
Introduce linear coordinates $x^1,\ldots,x^n$ on $V$ such that $\gamma(x)
= -(x^1)^2 + (x^2)^2 + \cdots + (x^n)^2$ and such that the $x^1$-axis is
future directed.
Let $f\in\DD(\R,\Co)$ and $\psi\in\DD(\R^{n-1},\Co)$ and put $\varphi(x) :=
f(x^1)\psi(\hat x)$ where $\hat x = (x^2,\ldots,x^n)$.
Choose the function $\psi$ such that on $J_+(0)$ we have $\varphi(x)=f(x^1)$.

\begin{center}
\input{fig-suppphi}
\end{center}

{\em Claim:}
If $\Re(\alpha) > 1$, then
$$
R_+(\alpha)[\varphi]=\frac{1}{(\alpha-1)!}\int_0^\infty r^{\alpha-1}f(r)dr.
$$
Proof of the Claim.
Since both sides of the equation are holomorphic in $\alpha$ for $\Re(\alpha)
> 1$ it suffices to show it for $\Re(\alpha)>n$. 
In that case we have by the definition of $R_+(\alpha)$ 
\be
R_+(\alpha)[\varphi]&=&C(\alpha,n)\int_{J_+(0)}
\varphi(X)\gamma(X)^{\frac{\alpha-n}{2}}dX\\  
&=&C(\alpha,n)\int_0^\infty \int_{\{|\hat{x}|<x^1\}}
\varphi(x^1,\hat{x})((x^1)^2-|\hat{x}|^2)^{\frac{\alpha-n}{2}}d\hat{x}\;dx^1\\ 
&=&C(\alpha,n)\int_0^\infty f(x^1)
\int_{\{|\hat{x}|<x^1\}}
((x^1)^2-|\hat{x}|^2)^{\frac{\alpha-n}{2}}d\hat{x}\;dx^1\\  
&=&C(\alpha,n)\int_0^\infty f(x^1)
\int_0^{x^1}\int_{S^{n-2}}((x^1)^2-t^2)^{\frac{\alpha-n}{2}}t^{n-2}
d\omega\;dt\;dx^1, 
\ee
where $S^{n-2}$ is the $(n-2)$-dimensional round sphere and $d\omega$ its
standard volume element. 
Renaming $x^1$ we get
\[R_+(\alpha)[\varphi]=\vol(S^{n-2})\,C(\alpha,n) \int_0^\infty
f(r)\int_0^r(r^2-t^2)^{\frac{\alpha-n}{2}}t^{n-2}dt\;dr.\] 
Using $\int_0^r(r^2-t^2)^{\frac{\alpha-n}{2}}t^{n-2}dt = \frac12 
r^{\alpha-1}\frac{(\frac{\alpha-n}{2})!
  (\frac{n-3}{2})!}{(\frac{\alpha-1}{2})!}$ we obtain
\begin{eqnarray*}
R_+(\alpha)[\varphi]
&=&
\frac{\vol(S^{n-2})}{2}\,C(\alpha,n)\int_0^\infty f(r)r^{\alpha-1}
\frac{(\frac{\alpha-n}{2})!(\frac{n-3}{2})!}{(\frac{\alpha-1}{2})!}dr\\
&=&
\frac12 \frac{2\pi^{(n-1)/2}}{(\tfrac{n-1}{2}-1)!}\cdot
\frac{2^{1-\alpha}\pi^{1-n/2}}{(\alpha/2-1)!(\tfrac{\alpha-n}{2})!}\cdot
\frac{(\tfrac{\alpha-n}{2})!(\tfrac{n-3}{2})!}{(\tfrac{\alpha-1}{2})!}\cdot
\int_0^\infty f(r)r^{\alpha-1}dr\\
&=&
\frac{\sqrt{\pi}\cdot 2^{1-\alpha}}{(\alpha/2-1)!(\tfrac{\alpha-1}{2})!}\cdot
\int_0^\infty f(r)r^{\alpha-1}dr.
\end{eqnarray*}
Legendre's duplication formula (see \cite[p.~218]{Jeff})
\begin{equation}\label{legendreformel}
\left(\frac{\alpha}{2}-1\right)!\left(\frac{\alpha+1}{2}-1\right)! =
2^{1-\alpha}\sqrt{\pi}\,(\alpha-1)!
\end{equation}
yields the Claim.

To show {(\ref{g})} recall first from {(\ref{f})} that we know already
$$
\ssupp(R_\pm(\alpha))=\supp(R_\pm(\alpha))\subset C_\pm(0)
$$
for $\alpha=n-2,n-4,\ldots, 2$ or $1$ respectively.
Note also that the distribution $R_\pm(\alpha)$ is invariant under
timeorientation-preserving Lorentz transformations, that is, for any such
transformation $A$ of $V$ we have
\[R_\pm(\alpha)[\varphi\circ A]=R_\pm(\alpha)[\varphi]\]
for every testfunction $\varphi$.
Hence $\supp(R_\pm(\alpha))$ as well as $\ssupp(R_\pm(\alpha))$ are also
invariant under the group of those transformations.
Under the action of this group the orbit decomposition of $C_\pm(0)$ is given
by
\[C_\pm(0)=\{0\}\cup(C_\pm(0)\setminus\{0\}).\]
Thus $\supp(R_\pm(\alpha))=\ssupp(R_\pm(\alpha))$ coincides either with $\{0\}$
or with $C_\pm(0)$.

The Claim shows for the testfunctions $\varphi$ considered there 
$$
R_+(2)[\varphi]=\int_0^\infty rf(r)dr.
$$
Hence the support of $R_+(2)$ cannot be contained in $\{0\}$.
If $n$ is even, we conclude $\supp(R_+(2))=C_+(0)$ and then also
$\supp(R_+(\alpha))=C_+(0)$ for $\alpha = 2,4,\ldots,n-2$.

Taking the limit $\alpha \searrow 1$ in the Claim yields
$$
R_+(1)[\varphi]=\int_0^\infty f(r)dr.
$$
Now the same argument shows for odd $n$ that $\supp(R_+(1))=C_+(0)$ and then
also $\supp(R_+(\alpha))=C_+(0)$ for $\alpha = 1,3,\ldots,n-2$.
This concludes the proof of {(\ref{g})}.

Proof of {(\ref{h})}. 
Fix a compact subset $K\subset V$. 
Let $\sigma_K\in\DD(V,\R)$ be a function such that $\sigma_{|_K}\equiv 1$.
For any $\varphi\in\DD(V,\Co)$ with $\supp(\varphi)\subset K$ write
\[
\varphi(x)=\varphi(0)+\sum_{j=1}^nx^j\varphi_j(x)
\]
with suitable smooth functions $\phi_j$.
Then
\be
R_\pm(0)[\varphi]&=&R_\pm(0)[\sigma_K\varphi]\\
&=&R_\pm(0)[\varphi(0)\sigma_K+\sum_{j=1}^n x^j\sigma_K\varphi_j]\\
&=&\varphi(0)\underbrace{R_\pm(0)[\sigma_K]}_{=:c_K}+\sum_{j=1}^n
\underbrace{(x^jR_\pm(0))}_{=0\textrm{ by }(\ref{b})}[\sigma_K\varphi_j]\\ 
&=&c_K \varphi(0).
\ee
The constant $c_K$ actually does not depend on $K$ since for $K'\supset K$ 
and $\supp(\varphi)\subset K(\subset K')$,
\[c_{K'}\varphi(0)=R_+(0)[\varphi]=c_K\varphi(0),\]
so that $c_K=c_{K'}=:c$.
It remains to show $c=1$.

We again look at testfunctions $\varphi$ as in the Claim and compute using 
{(\ref{c})}
\be
c\cdot\varphi(0) &=& R_+(0)[\varphi]\\
&=&R_+(2)[\Box\varphi]\\
&=&\int_0^\infty rf''(r)dr\\
&=&-\int_0^\infty f'(r)dr\\
&=&f(0)\\
&=&\varphi(0).
\ee
This concludes the proof of {(\ref{h})}.

Proof of {(\ref{i})}. 
By its definition, the distribution $R_\pm(\alpha)$ is a continuous function
if $\Re(\alpha)>n$, therefore it is of order $0$.
Since $\Box$ is a differential operator of order $2$, the order of 
$\Box R_\pm(\alpha)$ is at most that of $R_\pm(\alpha)$ plus $2$.
It then follows from {(\ref{c})} that:\\
$\bullet$ If $n$ is even: for every $\alpha$ with $\Re(\alpha)>0$ we have
$\Re(\alpha)+n=\Re(\alpha)+2\cdot\frac{n}{2}>n$, so that the order of
$R_\pm(\alpha)$ is not greater than $n$ (and so $n+1$).\\ 
$\bullet$ If $n$ is odd: for every $\alpha$ with $\Re(\alpha)>0$ we have
$\Re(\alpha)+n+1=\Re(\alpha)+2\cdot\frac{n+1}{2}>n$, so that the order of
$R_\pm(\alpha)$ is not greater than $n+1$.\\ 
This concludes the proof of (\ref{i}).

Assertion (\ref{j}) is clear by definition whenever $\alpha>n$.
For general $\alpha\in\R$ choose $k\in\N$ so large that $\alpha + 2k >n$.
Using (\ref{c}) we get for any $\phi\in\DD(V,\R)$
$$
R_\pm(\alpha)[\phi] \quad=\quad \Box^k R_\pm(\alpha+2k)[\phi] \quad=\quad 
R_\pm(\alpha+2k)[\Box^k \phi] \in \R
$$
because $\Box^k \phi\in\DD(V,\R)$ as well.
\end{proof}
 
In the following we will need a slight generalization of
Lemma~\ref{erstvorbereitung} (\ref{d}):

\begin{cor}\label{Ckholo}
For $\varphi\in \DD^k(V,\Co)$ the map 
$\alpha\mapsto R_{\pm}(\alpha)[\varphi]$
defines a holomorphic function on $\{\alpha\in\Co\; | 
\;\Re(\alpha) > n-2[\frac{k}{2}]\}$.
\end{cor} 

\begin{proof}
Let $\phi\in\DD^k(V,\Co)$.
By the definition of $R_{\pm}(\alpha)$ the map $\alpha\mapsto
R_{\pm}(\alpha)[\varphi]$ is clearly holomorphic on $\{\Re(\alpha) > n\}$.
Using {(\ref{c})} of Proposition~\ref{MinkRiesz} we get the holomorphic
extension to the set $\{\Re(\alpha) > n-2[\frac{k}{2}] \}$.
\end{proof}

%%%%%%%%%%%%%%%%%%%%%%%%%%%%%%%%%%%%%%%%%%%%%%%%%%%%%%%%%%%%%%%%%%%%%%%%%
\section{Lorentzian geometry}
\label{sec:lorgeo}
%%%%%%%%%%%%%%%%%%%%%%%%%%%%%%%%%%%%%%%%%%%%%%%%%%%%%%%%%%%%%%%%%%%%%%%%%

We now summarize basic concepts of Lorentzian geometry.
We will assume familiarity with semi-Riemannian manifolds, geodesics, the
Riemannian exponential map etc.
A summary of basic notions in differential geometry can be found in
Appendix~\ref{app:diffgeo}.
A thorough introduction to Lorentzian geometry can e.~g.\ be found in
\cite{BEE} or in \cite{ONeill}.
Further results of more technical nature which could distract the reader at a
first reading but which will be needed later are collected in
Appendix~\ref{app:lorgeo}.

Let $M$ be a timeoriented Lorentzian manifold.
A piecewise $C^1$-curve in $M$ is called \defem{timelike, lightlike, causal, 
spacelike, future directed},\indexn{causal curve>defemidx}\indexn{timelike
curve>defemidx}\indexn{lightlike curve>defemidx}\indexn{spacelike
  curve>defemidx} 
or \defem{past directed} if its tangent vectors are timelike, lightlike,
causal, spacelike, future directed, or past directed respectively.
A piecewise $C^1$-curve in $M$ is called
\defem{inextendible},\indexn{inextendible curve>defemidx} if no piecewise
$C^1$-reparametrization of the curve can be continuously extended to any of
the end points of the parameter interval.

The \defidx{chronological future} $I_+^M(x)$ 
\indexs{I*@ $I_+^M(x)$, chronological future of point $x$ in $M$}
\indexs{I*@ $I_-^M(x)$, chronological past of point $x$ in $M$}of a point $x\in M$ is the set of
points that can be reached from $x$ by future directed timelike curves. 
Similarly, the \defidx{causal future} $J_+^M(x)$ 
\indexs{J*@ $J_+^M(x)$, causal future of point $x$ in $M$}
\indexs{J*@ $J_-^M(x)$, causal past of point $x$ in $M$}of a point $x\in M$ consists 
of those points that can be reached from $x$ by causal curves and of $x$
itself. 
In the following, the notation $x<y$ (or $x\leq
y$)\indexs{x*@$x\protect\textless y$, (strict) causality
  relation}\indexs{x*@$x\protect\leq y$, causality relation} will mean $y\in
I_+^M(x)$ (or $y\in J_+^M(x)$ respectively). 
The {\em chronological future} of a subset
$A\subset M$ is defined to be $I_+^M(A):=\buil{\cup}{x\in A}I_+^M(x)$. 
Similarly, the {\em causal future} of $A$ is $J_+^M(A):=\buil{\cup}{x\in
  A}J_+^M(x)$.\indexs{Ia*@$I_+^M(A)$, chronological future of subset $A$ of
  $M$}\indexs{Ia*@$I_-^M(A)$, chronological past of subset $A$ of
  $M$}\indexs{Ja*@$J_+^M(A)$, causal future of subset $A$ of
  $M$}\indexs{Ja*@$J_-^M(A)$, causal past of subset $A$ of $M$} The
\defidx{chronological past} $I_-^M(A)$ and the \defidx{causal past} $J_-^M(A)$ 
are defined by replacing future directed curves by past directed curves.
One has in general that $I_\pm^M(A)$ is the interior of $J^M_\pm(A)$ and
that $J_\pm^M(A)$ is contained in the closure of $I_\pm^M(A)$.
The chronological future and past are open subsets but the causal
future and past are not always closed even if $A$ is closed (see also
Section~\ref{app:lorgeo} in Appendix).

\begin{center}
\input{fig-futurepast}
\end{center}

We will also use the notation $J^M(A) := J_-^M(A) \cup J_+^M(A)$.
\indexs{Jaa*@$J^M(A) := J_-^M(A) \cup J_+^M(A)$}
A subset $A\subset M$ is called \defem{past compact} 
\indexn{past compact subset>defemidx}
if $A\cap J_-^M(p)$
is compact for all $p\in M$.
Similarly, one defines \defem{future compact} subsets.
\indexn{future compact subset>defemidx}

\begin{center}
\input{fig-pastcompact1}
\end{center}

\Definition{
A subset $\Omega \subset M$ in a timeoriented Lorentzian manifold
is called \defem{causally compatible}\indexn{causally compatible
  subset>defemidx} if for all points $x \in \Omega$ 
$$
J_\pm^\Omega(x) = J_\pm^M(x) \cap \Omega
$$
holds.
}

Note that the inclusion ``$\subset$'' always holds.
The condition of being causally compatible means that whenever two
points in $\Omega$ can be joined by a causal curve in $M$
this can also be done inside $\Omega$.

\begin{center}
\input{fig-causalcompatible}
\end{center}

\begin{center}
\input{fig-nonconvex}
\end{center}

If $\Omega \subset M$ is a causally compatible domain in a timeoriented
Lorentzian manifold, then we immediately see that for each subset $A \subset
\Omega$ we have
$$
J_\pm^\Omega(A) = J_\pm^M(A) \cap \Omega .
$$
Note also that being causally compatible is transitive: 
If $\Omega\subset\Omega'\subset\Omega''$, if $\Omega$ is causally compatible
in $\Omega'$, and if  $\Omega'$ is causally compatible in $\Omega''$, then
so is $\Omega$ in $\Omega''$.

\Definition{\label{ddefgeodstarshaped}
A domain $\Omega \subset M$ in a Lorentzian manifold
is called
\beit\item \defem{geodesically starshaped} 
\indexn{geodesically starshaped domain>defemidx}
with respect to a fixed point $x\in \Omega$
if there exists an open subset $\Omega' \subset T_xM$,
starshaped with respect to $0$, such that the Riemannian exponential map
$\exp_x$ 
maps $\Omega'$ diffeomorphically onto $\Omega$.
\item \defem{geodesically convex}
\indexn{geodesically convex domain>defemidx}
(or simply \defem{convex}\indexn{convex domain>defemidx}) if it is
geodesically starshaped with respect to all of its 
points.
\eeit
}

\begin{center}
\input{fig-sternfoermig}
\end{center}

If $\Omega$ is geodesically starshaped with respect to $x$, then
$\exp_x(I_\pm(0)\cap\Omega') = I_\pm^\Omega(x)$ and
$\exp_x(J_\pm(0)\cap\Omega') = J_\pm^\Omega(x)$.
We put $C_\pm^\Omega(x):=\exp_x(C_\pm(0)\cap\Omega')$.

On a geodesically starshaped domain $\Omega$ we define the smooth positive 
function $\mu_x : \Omega \to \R$ by
\begin{equation}\label{mudef}
\dV=\mu_x \cdot (\exp_x^{-1})^*\left(dz\right),
\end{equation}

where $\dV$ is the Lorentzian volume density and $dz$ is the standard volume
density on $T_x\Omega$.
In other words, $\mu_x = \det(d\exp_x)\circ \exp_x^{-1}$.
\indexs{m*@$\mu_x$, local density function}
In normal coordinates about $x$, $\mu_x=\sqrt{|\det(g_{ij})|}$.\\

For each open covering of a Lorentzian manifold there exists a refinement
consisting of convex open subsets, see \cite[Chap.~5, Lemma 10]{ONeill}.

\Definition{
A domain $\Omega$ is called \defem{causal}\indexn{causal domain>defemidx} if
$\overline\Omega$ is contained in 
a convex domain $\Omega'$ and if for any $p,q\in\overline\Omega$ the
intersection $J^{\Omega'}_+(p)\cap J^{\Omega'}_-(q)$ is compact and contained
in $\overline\Omega$.
}

%\mnote{im Bild $\Omega'$ einzeichnen und im kausalen Fall den Punkt $p$ auf
%  den Rand von $\Omega$ verlegen}

\begin{center}
\input{fig-convexcausal}
\end{center}

\Definition{
A subset $S$ of a connected timeoriented Lorentzian manifold is called 
\defem{achronal} (or \defem{acausal}) if and only if each timelike
(respectively causal) curve meets $S$ at most once. 
\indexn{achronal subset>defemidx}
\indexn{acausal subset>defemidx}

A subset $S$ of a connected timeoriented Lorentzian manifold is a
\defidx{Cauchy hypersurface} if each inextendible timelike curve in $M$  
meets $S$ at exactly one point.   
}

\begin{center}
\input{fig-cauchy}
\end{center}

Obviously every acausal subset is achronal, but the reverse is wrong. However,
every achronal spacelike hypersurface is acausal (see Lemma 42 from Chap.~14
in \cite{ONeill}).\\ 
Any Cauchy hypersurface is achronal.
Moreover, it is a closed topological hypersurface and it is hit
by each inextendible causal curve in at least one point.
Any two Cauchy hypersurfaces in $M$ are homeomorphic.
Furthermore, the causal future and past of a Cauchy hypersurface is past-
and future-compact respectively.
This is a consequence of e.~g.\ \cite[Ch.~14, Lemma~40]{ONeill}. 

\Definition{\label{ddefCauchydev}
The \defem{Cauchy development} of a subset $S$ 
\indexn{Cauchy development of a subset>defemidx}
\indexs{D*@$D(S)$, Cauchy development of a subset $S$}
of a timeoriented Lorentzian manifold $M$ is the set $D(S)$ of points of $M$
through which every inextendible causal curve in $M$ meets $S$.  
}

\begin{center}
\input{fig-defcauchyentwick}
\end{center}

\Remark{\label{rem:Cauchydev}
It follows from the definition that $D(D(S))=D(S)$ for every subset $S\subset
M$.
Hence if $T\subset D(S)$, then $D(T) \subset D(D(S))=D(S)$.

Of course, if $S$ is achronal, then every inextendible causal curve in $M$
meets $S$ at most once.
The Cauchy development $D(S)$ of every \emph{acausal} hypersurface $S$ is open,
see \cite[Chap.~14, Lemma~43]{ONeill}. 
}

\Definition{
A Lorentzian manifold is said to satisfy the \defidx{causality condition} if
it does not contain any closed causal curve.  

A Lorentzian manifold is said to satisfy the \defidx{strong causality
  condition} 
if there are no almost closed causal curves.
More precisely, for each point $p\in M$ and for each open neighborhood 
$U$ of $p$ there exists an open neighborhood $V\subset U$ of $p$ such
that each causal curve in $M$ starting and ending in $V$ is entirely
contained in $U$.
}

\begin{center}
\input{fig-strongcausal}
\end{center}

Obviously, the strong causality condition implies the causality condition.
Convex open subsets of a Lorentzian manifold satisfy the strong causality
condition.
 
\Definition{
A connected timeoriented Lorentzian manifold is called \defem{globally
hyperbolic} 
\indexn{globally hyperbolic manifold>defemidx}
if it satisfies the strong causality condition and if for all
$p,q\in M$ the intersection $J_+^M(p)\cap J_-^M(q)$ is compact.
}

\Remark{\label{rem:globhypOmega}
If $M$ is a globally hyperbolic Lorentzian manifold, then a nonempty open
subset $\Omega\subset M$ is itself globally hyperbolic if and only if for any
$p,q\in \Omega$ the intersection $J_+^\Omega(p)\cap
J_-^\Omega(q)\subset\Omega$ is compact.
Indeed non-existence of almost closed causal curves in $M$ directly implies
non-existence of such curves in $\Omega$.
}

We now state a very useful characterization of globally hyperbolic manifolds.

\begin{thm}\label{globhyp}
Let $M$ be a connected timeoriented Lorentzian manifold.
Then the following are equivalent:
\begin{itemize}
\item[(1)]
$M$ is globally hyperbolic.
\item[(2)]
There exists a Cauchy hypersurface in $M$.
\item[(3)]
$M$ is isometric to $\R\times S$ with metric $-\beta dt^2 + g_t$ where
$\beta$ is a smooth positive function, $g_t$ is a Riemannian metric
on $S$ depending smoothly on $t\in\R$ and each $\{t\}\times S$ is a smooth 
spacelike Cauchy hypersurface in $M$.
\end{itemize}
\end{thm}

\begin{proof}
That (1) implies (3) has been shown by Bernal and S\'anchez in 
\cite[Thm.~1.1]{BS} using work of Geroch \cite[Thm.~11]{Geroch}.
See also \cite[Prop.~6.6.8]{EH} and \cite[p.~209]{Wa1} for earlier mentionings
of this fact.
That (3) implies (2) is trivial and that (2) implies (1) is well-known,
see e.~g.\ \cite[Cor.~39, p.~422]{ONeill}.
\end{proof}

\Examples{\label{ex:globhyp}
Minkowski space is globally hyperbolic.
Every spacelike hyperplane is a Cauchy hypersurface.
One can write Minkowski space as $\R\times\R^{n-1}$ with the metric 
$-\dt^2+g_t$ where $g_t$ is the Euclidean metric on $\R^{n-1}$ and does not
depend on $t$. 

Let $(S,g_0)$ be a connected Riemannian manifold and $I\subset\R$ an interval.
The manifold $M=I\times S$ with the metric $g=-\dt^2 + g_0$ is
globally hyperbolic if and only if $(S,g_0)$ is complete.
This applies in particular if $S$ is compact.

More generally, if $f:I\to\R$ is a smooth positive function we may equip
$M=I\times S$ with the metric $g=-\dt^2 + f(t)^2\cdot g_0$.
Again, $(M,g)$ is globally hyperbolic if and only if $(S,g_0)$ is complete,
see Lemma~\ref{lem:globhypcyl}.
{\em Robertson-Walker spacetimes} and, in particular, {\em Friedmann
  cosmological models}, are of this type. 
They are used to discuss big bang, expansion of the universe, and cosmological
redshift, compare \cite[Ch.~5 and 6]{Wa1} or \cite[Ch.~12]{ONeill}.
Another example of this type is {\em deSitter spacetime}, where $I=\R$,
$S=S^{n-1}$, $g_0$ is the canonical metric of $S^{n-1}$ of constant sectional
curvature $1$, and $f(t)=\cosh(t)$.
{\em Anti-deSitter spacetime} which we will discuss in more detail in
Section~\ref{seq:nonglobhyp} is not globally hyperbolic.

The interior and exterior {\em Schwarzschild spacetimes} are globally
hyperbolic.
They model the universe in the neighborhood of a massive static rotationally
symmetric body such as a black hole.
They are used to investigate perihelion advance of Mercury, the bending of
light near the sun and other astronomical phenomena, see \cite[Ch.~6]{Wa1}
and \cite[Ch.~13]{ONeill}.
}

\begin{cor}\label{cexisttimefctn}
On every globally hyperbolic Lorentzian manifold $M$ there exists a smooth
function $h:M\to\R$ whose gradient is past directed timelike at every point
and all of whose level-sets are spacelike Cauchy hypersurfaces.
\end{cor}

\begin{proof}
Define $h$ to be the composition $t\circ\Phi$ where $\Phi:M\rightarrow
\R\times S$ is the isometry given in Theorem~\ref{globhyp} and $t:\R\times S
\to \R$ is the projection onto the first factor.
\end{proof}

Such a function $h$ on a globally hyperbolic Lorentzian manifold will be
referred to as 
a \defem{Cauchy time-function}\indexn{Cauchy time-function>defemidx}.  
Note that a Cauchy time-function is strictly monotonically increasing along 
any future directed causal curve.

We quote an enhanced form of Theorem~\ref{globhyp}, due to A.~Bernal and
M.~S\'anchez (see \cite[Theorem~1.2]{BS2}), which will be needed in
Chapter~\ref{chapglobaltheorie}.

\begin{thm}\label{globhyp2}
Let $M$ be a globally hyperbolic manifold and $S$ be a spacelike smooth Cauchy
hypersurface in $M$.
Then there exists a Cauchy time-function $h:M\rightarrow\R$ such that
$S=h^{-1}(\{0\})$.
\hfill$\Box$
\end{thm}

Any given smooth spacelike Cauchy hypersurface in a (necessarily globally
hyperbolic) Lorentzian manifold is therefore the leaf of a foliation by smooth
spacelike Cauchy hypersurfaces. 

Recall that the \defem{length} \indexn{length of a curve>defemidx}
 $L[c]$ \indexs{L*@$L[c]$, length of curve $c$}
of a piecewise $C^1$-curve
$c:[a,b]\to M$ on a Lorentzian manifold $(M,g)$ is defined by 
\[L[c]:=\int_a^b\sqrt{|g(\dot{c}(t),\dot{c}(t))|}dt.\]

\Definition{\label{def:timesep}
The \defidx{time-separation} on a Lorentzian manifold $(M,g)$ is the
  function $\tau:M\times M\to \R\cup\{\infty\}$ defined by 
\[\tau(p,q):=\left\{\begin{array}{cl}\sup\{L[c]\,|\,c\textrm{\small{
  future directed causal curve from }}p\textrm{\small{ to
  }}q,&\textrm{\small{if }}p<q\\ 
0,&\textrm{\small{otherwise},}\end{array}\right.\]
for all $p$, $q$ in $M$.
}
\indexs{t*@$\tau(p,q)$, time-separation of points $p,q$}

The properties of $\tau$ which will be needed later are the following:

\begin{prop}\label{prop:timesep}
Let $M$ be a timeoriented Lorentzian manifold.
Let $p$, $q$, and $r\in M$.
Then
\ben
\item 
$\tau(p,q)>0$ if and only if $q\in I_+^M(p)$. 
\item 
The function $\tau$ is lower semi-continuous on $M\times M$.
If $M$ is convex or globally hyperbolic, then $\tau$ is finite and continuous. 
\item 
The function $\tau$ satisfies the \emph{inverse triangle inequality}:
If $p\leq q\leq r$, then
\begin{equation}\label{idu}
\tau(p,r)\geq\tau(p,q)+\tau(q,r).
\end{equation}
\een
\end{prop}

See e.~g.\ Lemmas~16, 17, and 21 from Chapter~14 in \cite{ONeill} for a proof. 
\hfill$\Box$

Now let $M$ be a Lorentzian manifold.
For a differentiable function $f:M\to \R$, the \defem{gradient} of $f$ is the
vector field \indexn{gradient of a function}
\begin{equation}
\grad f := (df)^\sharp .
\label{defgrad}
\end{equation}
Here $\omega\mapsto \omega^\sharp$ denotes the canonical isomorphism $T^*M \to
TM$ induced by the Lorentzian metric, i.~e., for $\omega\in T_x^*M$ the
vector $\omega^\sharp\in T_xM$ is characterized by the fact that
$\omega(X)=\langle \omega^\sharp,X\rangle$ for all $X\in T_xM$.
\indexs{*@$\sharp$, isomorphism $T^*M\to TM$ induced by Lorentzian metric}
The inverse isomorphism $TM \to T^*M$ is denoted by $X \mapsto X^\flat$.
\indexs{b@$\flat$, isomorphism $TM\to T^*M$ induced by Lorentzian metric}
One easily checks that for differentiable functions $f,g:M\to\R$
\begin{equation}
\grad(fg)=g\grad f + f\grad g.
\label{eq:gradproduct}
\end{equation}
Locally, the gradient of $f$ can be written as
$$
\grad f = \sum_{j=1}^n \epsilon_j\, df(e_j)\, e_j
$$
where $e_1,\ldots,e_n$ is a local Lorentz orthonormal frame of
$TM$, $\varepsilon_j = \la e_j, e_j \ra = \pm 1$.
For a differentiable vector field $X$ on $M$ the \defem{divergence} is the
function \indexn{divergence of a vector field>defemidx}
\indexs{di*@$\protect\div X$, divergence of vector field $X$} 
$$
\div X := \tr(\nabla X) = \sum_{j=1}^n \varepsilon_j \la e_j, \nabla_{e_j}X
\ra
$$
If $X$ is a differentiable vector field and $f$ a differentiable function on
$M$, then one immediately sees that
\begin{equation}
\div(fX) = f\div X + \langle \grad f,X\rangle.
\label{eq:divproduct}
\end{equation}
There is another way to characterize the divergence.
Let $\dV$ be the volume form induced by the Lorentzian metric.
Inserting the vector field $X$ yields an $(n-1)$-form
$\dV(X,\cdot,\ldots,\cdot)$.
Hence $d(\dV(X,\cdot,\ldots,\cdot))$ is an $n$-form and can therefore be
written as a function times $\dV$, namely
\begin{equation}
d(\dV(X,\cdot,\ldots,\cdot) = \div X \cdot \dV.
\label{eq:defdiv2}
\end{equation}
This shows that the divergence operator depends only mildly on the Lorentzian
metric.
If two Lorentzian (or more generally, semi-Riemannian) metrics have
the same volume form, then they also have the same divergence operator.
This is certainly not true for the gradient.

The divergence is important because of Gauss' divergence theorem:
\indexn{Gauss' divergence theorem>defemidx}

\begin{thm}\label{thm:gauss}
Let $M$ be a Lorentzian manifold and let $D\subset M$ be a domain with
piecewise smooth boundary. 
We assume that the induced metric on the smooth part of the boundary is
non-degenerate, i.~e., it is either Riemannian or Lorentzian on each connected
component.
Let $\mathfrak{n}$ \indexs{n@$\protect{\mathfrak{n}}$, unit normal
  field} denote the 
exterior normal field along $\partial D$, normalized to
$\la\mathfrak{n},\mathfrak{n}\ra =: \epsilon_\mathfrak{n} = \pm 1$. 

Then for every smooth vector field $X$ on $M$ such that $\supp(X)\cap
\overline{D}$ is compact we have
$$
\int_D \div(X)\,\dV = \int_{\partial D}\epsilon_\mathfrak{n}\la
X,\mathfrak{n}\ra\,\dA . 
$$
\hfill$\Box$
\end{thm}

Let $e_1,\ldots,e_n$ be a Lorentz orthonormal basis of $T_xM$.
Then $(\xi^1,\ldots,\xi^n) \mapsto \exp_x(\sum_j\xi^j e_j)$ is a local
diffeomorphism of a neighborhood of $0$ in $\R^n$ onto a neighborhood of $x$
in $M$.
This defines coordinates $\xi^1,\ldots,\xi^n$ on any open neighborhood of $x$
which is geodesically starshaped with respect to $x$.
Such coordinates are called \defidx{normal coordinates} about the point $x$.

We express the vector $X$ in normal coordinates about $x$ and write $X =
\sum_j \eta^j \frac{\partial}{\partial \xi^j}$.
From (\ref{eq:defdiv2}) we conclude, using $\dV=\mu_x\cdot d\xi^1 \wedge
\ldots \wedge d\xi^n$
\begin{eqnarray*}
\div(\mu_x^{-1}X)\cdot\dV
&=&
d(\dV(\mu_x^{-1}X,\cdot,\ldots,\cdot)\\
&=&
d\left(\sum_j (-1)^{j-1}\, \eta^j\, d\xi^1\wedge\ldots\wedge
\widehat{d\xi^j}\wedge\ldots\wedge d\xi^n\right)\\
&=&
\sum_j (-1)^{j-1}\, d\eta^j\wedge d\xi^1\wedge\ldots\wedge
\widehat{d\xi^j}\wedge\ldots\wedge d\xi^n\\
&=&
\sum_j  \frac{\partial\eta^j}{\partial\xi^j}\, d\xi^1\wedge\ldots\wedge
d\xi^n\\
&=&
\sum_j  \frac{\partial\eta^j}{\partial\xi^j}\,\mu_x^{-1}\,\dV.
\end{eqnarray*}
Thus
\begin{equation}
\mu_x\,\div(\mu_x^{-1}X) = 
\sum_j  \frac{\partial\eta^j}{\partial\xi^j}.
\label{eq:defdiv3}
\end{equation}
For a $C^2$-function $f$ the \defem{Hessian} at $x$ is the symmetric bilinear
form
\indexn{Hessian of a function>defemidx}
\indexs{H*@$\protect\mathrm{Hess}(f)|_x$, Hessian of function $f$ at point $x$}
$$
\mathrm{Hess}(f)|_x:T_xM\times T_xM \to \R,\quad
\mathrm{Hess}(f)|_x(X,Y) := \la \nabla_X\grad f, Y\ra.
$$
The \defidx{d'Alembert operator} is defined by
\indexs{*@$\protect\Box$, d'Alembert operator}
$$
\Box f := -\tr(\mathrm{Hess}(f)) = -\div\grad f.
$$
If $f:M\to \R$ and $F:\R\to\R$ are $C^2$ a straightforward computation yields
\begin{equation}
\Box(F\circ f) = - (F''\circ f)\la df,df \ra + (F'\circ f)\Box f .
\label{eq:BoxFf}
\end{equation}
%% In normal coordinates about $x$ the d'Alembertian can be written
%% $$
%% \Box f =-\mu_x^{-1}\sum_{i=1}^n\frac{\partial}{\partial x^i}
%% \left(\mu_x\frac{\partial f}{\partial x^i}\right).
%% $$
%% \mnote{Brauchen wir diese Formel? Wenn ja, Referenz!}

\begin{lemma}\label{taylormu}
Let $\Omega$ be a domain in $M$, geodesically starshaped with respect to
$x\in\Omega$.
Then the function $\mu_x$ defined in (\ref{mudef}) satisfies
$$
\mu_x(x)=1,\quad
d\mu_x|_x = 0, \quad
\mathrm{Hess}(\mu_x)|_x = -\frac13 \mathrm{ric}_x, \quad
(\Box\mu_x)(x)=\frac{1}{3}\mathrm{scal}(x),
$$
where $ \mathrm{ric}_x$ denotes the Ricci curvature considered as a bilinear
form on $T_x\Omega$ and $\mathrm{scal}$ is the scalar curvature.
\end{lemma}

\begin{proof}
Let $X\in T_x\Omega$ be fixed.
Let $e_1,\ldots,e_n$ be a Lorentz orthonormal basis of $T_x\Omega$. 
Denote by $J_1,\ldots,J_n$ the Jacobi fields along $c(t) = \exp_x(tX)$
satisfying $J_j(0)=0$ and $\frac{\nabla J_j}{dt}(0)=e_j$ for every $1\leq
j\leq n$.
The differential of $\exp_x$ at $tX$ is, for every $t$ for which it is
defined, given by
\[d_{tX}\exp_x(e_j)=\frac{1}{t}J_j(t),\]
$j=1,\ldots,n$.
From the definition of $\mu_x$ we have
\begin{eqnarray*}
{\mu_x}(\exp_x(tX)) e_1\wedge\ldots\wedge e_n 
&=&
\det(d_{tX}\exp_x) e_1\wedge\ldots\wedge e_n\\
&=&
(d_{tX}\exp_x(e_1))\wedge\ldots\wedge (d_{tX}\exp_x(e_n))\\
&=&
\frac{1}{t}J_1(t)\wedge\ldots\wedge \frac{1}{t}J_n(t) .
\end{eqnarray*}
Jacobi fields $J$ along the geodesic $c(t)=\exp_x(tX)$ satisfy the Jacobi field
equation $\frac{\nabla^2}{dt^2}J(t)=-R(J(t),\dot{c}(t))\dot{c}(t)$, where $R$
denotes the curvature tensor of the Levi-Civita connection $\nabla$. 
Differentiating this once more yields
$\frac{\nabla^3}{dt^3}J(t)=-\frac{\nabla R}{dt}(J(t),\dot{c}(t))\dot{c}(t)
- R(\frac{\nabla}{dt}J(t),\dot{c}(t))\dot{c}(t)$.
For $J=J_j$ and $t=0$ we have $J_j(0)=0$, 
$\frac{\nabla J_j}{dt}(0)=e_j$, $\frac{\nabla^2 J_j}{dt^2}(0)=-
R(0,\dot{c}(0))\dot{c}(0)=0$, and $\frac{\nabla^3
  J_j}{dt^3}(0)=-R(e_j,X)X$ where $X=\dot{c}(0)$. 
Identifying $J_j(t)$ with its parallel translate to $T_x\Omega$ along $c$
the Taylor expansion of $J_j$ up to order $3$ reads as
\[J_j(t)=te_j-\frac{t^3}{6}R(e_j,X)X+{\mathrm{O}}(t^4).\]
This implies
\be
\frac{1}{t}J_1(t)\wedge\ldots\wedge \frac{1}{t}J_n(t)&=&e_1\wedge\ldots\wedge
e_n\\ 
& &-\frac{t^2}{6}\sum_{j=1}^n e_1\wedge\ldots\wedge
R(e_j,X)X\wedge\ldots\wedge e_n+{\mathrm{O}}(t^3)\\ 
&=&e_1\wedge\ldots\wedge e_n\\
& &-\frac{t^2}{6}\sum_{j=1}^n\varepsilon_j\la R(e_j,X)X,e_j\ra
e_1\wedge\ldots\wedge e_n+{\mathrm{O}}(t^3)\\ 
&=&\Big(1-\frac{t^2}{6}\ric(X,X)+{\mathrm{O}}(t^3)\Big)e_1\wedge\ldots\wedge
e_n. 
\ee
Thus 
$$
{\mu_x}(\exp_x(tX)) = 1-\frac{t^2}{6}\ric(X,X)+{\mathrm{O}}(t^3)
$$
and therefore
\[{\mu_x}(x)=1,\quad d{\mu_x}(X)=0,\quad 
\mathrm{Hess}(\mu_x)(X,X)=-\frac{1}{3}\ric(X,X).\]
Taking a trace yields the result for the d'Alembertian.
\end{proof}

Lemma~\ref{taylormu} and (\ref{eq:BoxFf}) with $f=\mu_x$ and $F(t)=t^{-1/2}$
yield:

\begin{cor}\label{boxmuscal}
Under the assumptions of Lemma~\ref{taylormu} one has
$$
(\Box \mu_x^{-1/2})(x)=-\frac{1}{6}\mathrm{scal}(x).
$$
\hfill$\Box$
\end{cor}

Let $\Omega$ be a domain in a Lorentzian manifold $M$, geodesically starshaped
with respect to $x\in\Omega$.
Set 
\indexs{G*@$\Gamma_x=\gamma\circ\exp_x^{-1}$}
\begin{equation}
\Gamma_x:=\gamma\circ\exp_x^{-1}:\Omega\to\R
\label{eq:GammaDef}
\end{equation}
where $\gamma$ is defined as in (\ref{gammadef}) with $V=T_x\Omega$.

\begin{lemma}\label{Gammalemma}
Let $M$ be a timeoriented Lorentzian manifold.
Let the domain $\Omega\subset M$ be geodesically starshaped
with respect to $x\in\Omega$.
Then the following holds on $\Omega$:
\begin{itemize}
\item[(1)]
$\la \grad\Gamma_x,\grad\Gamma_x \ra = -4\Gamma_x$.
\item[(2)]
On $I_+^\Omega(x)$ (or on $I_-^\Omega(x)$)
the gradient $\grad\Gamma_x$ is a past directed (or future directed
respectivel) 
timelike vector field.
\item[(3)]
$\Box\Gamma_x -2n = -\la\grad\Gamma_x,\grad(\log(\mu_x)) \ra$.
\end{itemize}
\end{lemma}

\begin{proof}
Proof of (1). 
Let $y\in \Omega$ and $Z\in T_y\Omega$. 
The differential of $\gamma$ at a point $p$ is given by $d_p\gamma=-2\la
p,\cdot\ra$.
Hence
\be
d_y\Gamma_x(Z)&=& d_{\exp_x^{-1}(y)}\gamma\circ d_y\exp_x^{-1}(Z)\\
&=&-2\la \exp_x^{-1}(y),d_y\exp_x^{-1}(Z)\ra.
\ee
Applying the Gauss Lemma \cite[p.~127]{ONeill}, we obtain
\[
d_y\Gamma_x(Z)= -2\la d_{\exp_x^{-1}(y)}\exp_x(\exp_x^{-1}(y)),Z\ra.
\]
Thus
\begin{equation}
\label{gradGamma}
\grad_y\Gamma_x=-2d_{\exp_x^{-1}(y)}\exp_x(\exp_x^{-1}(y)).
\end{equation}
It follows again from the Gauss Lemma that
\be
\la\grad_y\Gamma_x,\grad_y\Gamma_x\ra
&=&
4\la d_{\exp_x^{-1}(y)}
\exp_x(\exp_x^{-1}(y)),d_{\exp_x^{-1}(y)}\exp_x(\exp_x^{-1}(y))\ra\\   
&=&
4\la\exp_x^{-1}(y),\exp_x^{-1}(y)\ra\\
&=&
-4\Gamma_x(y).
\ee

Proof of (2).
On $I_+^\Omega(x)$ the function $\Gamma_x$ is positive, hence $\la
\grad\Gamma_x,\grad\Gamma_x\ra = -4\Gamma_x < 0$.
Thus $\grad\Gamma_x$ is timelike.
For a future directed timelike tangent vector $Z\in T_x\Omega$ the curve
$c(t) := \exp_x(tZ)$ is future directed timelike and $\Gamma_x$ increases along
$c$.
Hence $0\leq \tfrac{d}{dt} (\Gamma_x\circ c) = \la \grad\Gamma_x,\dot{c}\ra$.
Thus $\grad\Gamma_x$ is past directed along $c$.
Since every point in $I_+^\Omega(x)$ can be written in the form $\exp_x(Z)$
for a future directed timelike tangent vector $Z$ this proves the assertion
for $I_+^\Omega(x)$.
The argument for $I_-^\Omega(x)$ is analogous.

Proof of (3).
Using (\ref{eq:divproduct}) with $f=\mu_x^{-1}$ and $X=\grad\Gamma_x$ we get
$$
\div(\mu_x^{-1}\,\grad\Gamma_x)
=
\mu_x^{-1}\,\div\grad\Gamma_x + \la \grad(\mu_x^{-1}),\grad\Gamma_x\ra
$$
and therefore
\begin{eqnarray}
\Box\Gamma_x 
&=&
\la \grad(\log(\mu_x^{-1})),\grad\Gamma_x\ra -
\mu_x\,\div(\mu_x^{-1}\,\grad\Gamma_x)\nonumber\\
&=&
-\la \grad(\log(\mu_x)),\grad\Gamma_x\ra -
\mu_x\,\div(\mu_x^{-1}\,\grad\Gamma_x) .\nonumber
\label{eq:BoxGamma}
\end{eqnarray}
It remains to show $\mu_x\,\div(\mu_x^{-1}\,\grad\Gamma_x)=-2n$.
We check this in normal coordinates $\xi^1,\ldots,\xi^n$ about $x$.
By (\ref{gradGamma}) we have $\grad\Gamma_x = -2 \sum_j
\xi^j\frac{\partial}{\partial\xi^j}$ so that (\ref{eq:defdiv3}) implies
$$
\mu_x\,\div(\mu_x^{-1}\,\grad\Gamma_x) = -2 \sum_j
\frac{\partial\xi^j}{\partial\xi^j} = -2n.
$$
\end{proof}

%% \begin{cor}\label{gammawaechst}
%% Let $\Omega$ be a domain in a timeoriented Lorentzian manifold,
%% geodesically starshaped with respect to $x$.  
%% Let $c:[a,b]\to\Omega$ be a future directed causal curve starting in
%% $I_+^M(x)$. 
%% Then the function $\Gamma_x\circ c:[a,b]\to\R$ is monotonically increasing.
%% \end{cor}
%% \mnote{Wo brauchen wir das?}
%% \begin{proof}
%% By assumption $\dot{c}$ is future directed causal and by part (2) of Lemma~
%%~\ref{Gammalemma} $\grad\Gamma_x$ is past directed timelike.
%% Hence
%% \[ 
%% \tfrac{d}{dt}(\Gamma_x\circ c) =\la \grad\Gamma_x,\dot{c} \ra >0.
%% \]
%% \end{proof}

\Remark{
If $\Omega$ is convex and $\tau$ is the time-separation function of $\Omega$,
then one can check that 
$$
\tau(p,q) = 
\left\{
  \begin{array}{cl}
  \sqrt{\Gamma(p,q)}, & \mbox{if $p < q$}\\
  0, & \mbox{otherwise.}
  \end{array}
\right.
$$
}

%%%%%%%%%%%%%%%%%%%%%%%%%%%%%%%%%%%%%%%%%%%%%%%%%%%%%%%%%%%%%%%%%%%%%%%%
\section{Riesz distributions on a domain} 
%%%%%%%%%%%%%%%%%%%%%%%%%%%%%%%%%%%%%%%%%%%%%%%%%%%%%%%%%%%%%%%%%%%%%%%%

Riesz distributions have been defined on all spaces isometric to Minkowski
space.
They are therefore defined on the tangent spaces at all points of a Lorentzian
manifold.
We now show how to construct Riesz distributions defined in small open subsets
of the Lorentzian manifold itself.
The passage from the tangent space to the manifold will be provided by the
Riemannian exponential map. 

Let $\Omega$ be a domain in a timeoriented $n$-dimensional Lorentzian
manifold, $n\geq 2$. 
Suppose $\Omega$ is geodesically starshaped with respect to some point
$x\in\Omega$.  
In particular, the Riemannian exponential function $\exp_x$ is a
diffeomorphism from $\Omega':=\exp^{-1}(\Omega)\subset T_x\Omega$ to $\Omega$.
Let $\mu_x: \Omega\to\R$ be defined as in (\ref{mudef}).
Put 
$$
R_{\pm}^\Omega(\alpha,x):=\mu_x\, 
\exp_x^*R_{\pm}(\alpha),
$$
that is, for every testfunction $\varphi\in\DD(\Omega,\Co)$,  
$$
R_{\pm}^\Omega(\alpha,x)[\varphi]:=
R_{\pm}(\alpha)[\left(\mu_x\varphi\right)\circ\exp_x].
$$
Note that $\supp(\left(\mu_x\varphi\right)\circ\exp_x)$ is contained in
$\Omega'$.
Extending the function $\left(\mu_x\varphi\right)\circ\exp_x$ by zero we can
regard it as a testfunction on $T_x\Omega$ and thus apply $R_\pm(\alpha)$ to
it.

%% In case there is no ambiguity about the point $x$ we briefly write
%% $R_{\pm}^\Omega(\alpha)$ instead of $R_{\pm}^\Omega(\alpha,x)$.

\Definition{
We call $R_+^\Omega(\alpha,x)$ the \defem{advanced Riesz distribution} and 
$R_-^\Omega(\alpha,x)$ the \defem{retarded Riesz distribution} on $\Omega$ at
$x$ 
for $\alpha\in\Co$. 
}
\indexn{Riesz distributions on a domain>defemidx}
\indexn{advanced Riesz distributions>defemidx}
\indexn{retarded Riesz distributions>defemidx}
\indexs{R*@$R_+^\Omega(\alpha,x)$, advanced Riesz distribution on domain
  $\Omega$ at point $x$}
\indexs{R*@$R_-^\Omega(\alpha,x)$, retarded Riesz distribution on domain
  $\Omega$ at point $x$} 

The relevant properties of the Riesz distributions are collected in the
following proposition.

\begin{prop}\label{OmegaRiesz}
The following holds for all $\alpha\in\Co$ and all $x\in\Omega$:
\begin{enumerate}
\item\label{Rexplic} 
If $\Re(\alpha)>n$, then $R_\pm^\Omega(\alpha,x)$ is the continuous function
\[
R_{\pm}^\Omega(\alpha,x)=
\left\{\begin{array}{cl}C(\alpha,n)\,\Gamma_x^{\frac{\alpha-n}{2}}
    & \textrm{ on }\;J_{\pm}^\Omega(x),\\ 
   0& \textrm{ elsewhere.}\end{array}\right.
\]
\item\label{Ranalyt} 
For every fixed testfunction $\varphi$ the map $\alpha\mapsto
R_{\pm}^\Omega(\alpha,x)[\varphi] $ is holomorphic on $\mathbb{C}$. 
\item\label{gammaR} 
$\Gamma_x \cdot R_{\pm}^\Omega(\alpha,x)=\alpha
(\alpha-n+2)\,R_{\pm}^\Omega(\alpha+2,x)$ 
\item\label{gradR} 
$\grad\left(\Gamma_x\right)\cdot R_{\pm}^\Omega(\alpha,x)
=2\alpha\grad R_{\pm}^\Omega(\alpha+2,x)$
\item\label{boxR} 
If $\alpha\not=0$, then $\Box R_{\pm}^\Omega(\alpha+2,x)
=\Big(\frac{\Box\Gamma_x-2n}{2\alpha}+1\Big) R_{\pm}^\Omega(\alpha,x)$
\item\label{R0} 
$R_{\pm}^\Omega(0,x)=\delta_x$
\item\label{suppR1} 
For every
$\alpha\in\mathbb{C}\setminus\left(\{0,-2,-4,\ldots\}
  \cup\{n-2,n-4,\ldots\}\right)$ 
we have $\supp\left(R_{\pm}^\Omega(\alpha,x)\right)=J_{\pm}^\Omega(x)$ and 
  $\ssupp\left(R_{\pm}^\Omega(\alpha,x)\right)\subset C_{\pm}^\Omega(x).$  
\item\label{suppR2} 
For every $\alpha\in\{0,-2,-4,\ldots\}\cup\{n-2,n-4,\ldots\}$ we have
$\supp\left(R_{\pm}^\Omega(\alpha,x)\right)=\ssupp
\left(R_{\pm}^\Omega(\alpha,x)\right)\subset C_{\pm}^\Omega(x).$  
\item\label{suppR3} 
For $n\geq 3$ and $\alpha=n-2,n-4,\ldots,1$ or $2$ respectively we have
$\supp\left(R_{\pm}^\Omega(\alpha,x)\right)=
\ssupp\left(R_{\pm}^\Omega(\alpha,x)\right)=C_{\pm}^\Omega(x).$ 
\item\label{orderR} 
For $\Re(\alpha)>0$ we have $\ord(R_{\pm}^\Omega(\alpha,x))\leq n+1$.
Moreover, there exists a neighborhood $U$ of $x$ and a constant 
$C>0$ such that
$$
|R^\Omega_\pm(\alpha,x')[\varphi]| \leq 
C\cdot\|\varphi\|_{C^{n+1}(\Omega)}
$$
for all $\varphi\in\DD(\Omega,\Co)$ and all $x'\in U$.
\item\label{Rglattinx}
If $U \subset \Omega$ is an open neighborhood of $x$ such that $\Omega$
is geodesically starshaped with respect to all $x'\in U$ and if
$V\in\DD(U\times\Omega,\Co)$, then the function 
$U \to \Co$, $x' \mapsto R_{\pm}^\Omega(\alpha,x')[y\mapsto V(x',y)]$, 
is smooth.
\item\label{RCkinx}
If $U \subset \Omega$ is an open neighborhood of $x$ such that $\Omega$
is geodesically starshaped with respect to all $x'\in U$, if $\Re(\alpha) > 0$, and if
$V\in\DD^{n+1+k}(U\times\Omega,\Co)$, then the function 
$U \to \Co$, $x' \mapsto R_{\pm}^\Omega(\alpha,x')[y\mapsto V(x',y)]$, 
is $C^k$.
\item\label{Ckanalyt}
For every $\varphi\in\DD^k(\Omega,\Co)$ the map $\alpha\mapsto
R_{\pm}^\Omega(\alpha,x)[\varphi]$ is a holomorphic function on
$\{\alpha\in\Co\;|\; \Re(\alpha)>n-2[\frac{k}{2}]\}$. 
\item\label{Rreell}
If $\alpha\in\R$, then $R_\pm^\Omega(\alpha,x)$ is real, i.~e.,
$R_\pm^\Omega(\alpha,x)[\phi]\in\R$ for all $\phi\in\DD(\Omega,\R)$.
\end{enumerate}
\end{prop}

\begin{proof}
It suffices to prove the statements for the advanced Riesz distributions.

Proof of {(\ref{Rexplic})}.
Let $\Re(\alpha)>n$ and $\varphi\in\DD(\Omega,\Co)$. 
Then
\begin{eqnarray*}
R_+^\Omega(\alpha,x)[\varphi]
&=&
R_+^\Omega(\alpha,x)[\left(\mu_x\circ\exp_x\right)\cdot
\left(\varphi\circ\exp_x\right)]\\   
&=&
C(\alpha,n)\int_{J_+(0)}\gamma^{\frac{\alpha-n}{2}}\cdot(\varphi\circ\exp_x)
\cdot\mu_x\, dz\\ 
&=&
C(\alpha,n)\int_{J_+^\Omega(x)}\Gamma_x^{\frac{\alpha-n}{2}}\cdot\varphi\, \dV.
\end{eqnarray*}
Proof of {(\ref{Ranalyt})}.
This follows directly from the definition of $R_+^\Omega(\alpha,x)$ and from
Lemma~\ref{erstvorbereitung}~(\ref{d}).

Proof of {(\ref{gammaR})}.
By (\ref{Rexplic}) this obviously holds for $\Re(\alpha)>n$ since $C(\alpha,n)
= \alpha(\alpha-n+2)C(\alpha+2,n)$. 
By analyticity of $\alpha\mapsto R_+^\Omega(\alpha,x)$ it must hold
for all $\alpha$.

Proof of {(\ref{gradR})}.
Consider $\alpha$ with $\Re(\alpha)>n$. 
By {(\ref{Rexplic})} the function $R_+^\Omega(\alpha+2,x)$ is then $C^1$.
On $J_+^\Omega(x)$ we compute 
\begin{eqnarray*}
2\alpha\grad R_+^\Omega(\alpha+2,x)
&=&
2\alpha C(\alpha+2,n)\grad\left(\Gamma_x^{\frac{\alpha+2-n}{2}}\right)\\ 
&=&\underbrace{2\alpha C(\alpha+2,n)
  \frac{\alpha+2-n}{2}}_{C(\alpha,n)}\Gamma_x^{\frac{\alpha-n}{2}}
\grad\Gamma_x\\   
&=&R_+^\Omega(\alpha,x)\grad\Gamma_x.
\end{eqnarray*}
For arbitrary $\alpha\in\Co$ assertion {(\ref{gradR})} follows from
analyticity of $\alpha\mapsto R_+^\Omega(\alpha,x)$.

Proof of {(\ref{boxR})}.
Let $\alpha\in\mathbb{C}$ with $\Re(\alpha)>n+2$.
Since $R_+^\Omega(\alpha+2,x)$ is then $C^2$, we can compute $\Box
R_+^\Omega(\alpha+2,x)$ classically.
This will show that (\ref{boxR}) holds for all $\alpha$ with
$\Re(\alpha)>n+2$.
Analyticity then implies (\ref{boxR}) for all $\alpha$. 
\begin{eqnarray*}
\Box R_+^\Omega(\alpha+2,x)
&=&
-\div\left(\grad R_+^\Omega(\alpha+2,x)\right)\\
&\stackrel{(\ref{gradR})}{=}&
-\frac{1}{2\alpha}\,\div\left(R_+^\Omega(\alpha,x)
\cdot\grad(\Gamma_x)\right)\\ 
&\stackrel{\mathrm{(\ref{eq:divproduct})}}{=}&
\frac{1}{2\alpha}\,\Box \Gamma_x\cdot R_+^\Omega(\alpha,x)-
\frac{1}{2\alpha}\, 
\la\grad\Gamma_x,\grad R_+^\Omega(\alpha,x)\ra\\ 
&\stackrel{(\ref{gradR})}{=}&
\frac{1}{2\alpha}\,\Box \Gamma_x\cdot R_+^\Omega(\alpha,x)- \frac{1}{2\alpha 
\cdot 2(\alpha-2)}\,\la\grad\Gamma_x,\grad\Gamma_x\cdot
R_+^\Omega(\alpha-2,x)\ra \\ 
&\stackrel{\mathrm{Lemma~\ref{Gammalemma}(1)}}{=}&
\frac{1}{2\alpha}\,\Box \Gamma_x\cdot
R_+^\Omega(\alpha,x)+\frac{1}{\alpha(\alpha-2)}\,\Gamma_x\cdot
R_+^\Omega(\alpha-2,x)\\ 
&\stackrel{(\ref{gammaR})}{=}&
\frac{1}{2\alpha}\,\Box \Gamma_x\cdot
R_+^\Omega(\alpha,x)+\frac{(\alpha-2)(\alpha-n)}{\alpha(\alpha-2)}
\,R_+^\Omega (\alpha,x)\\   
&=&
\left(\frac{\Box\Gamma_x-2n}{2\alpha}+1\right)R_+^\Omega(\alpha,x).
\end{eqnarray*}

Proof of {(\ref{R0})}.
Let $\varphi$ be a testfunction on $\Omega$.
Then by Proposition~\ref{MinkRiesz}~(\ref{h})
\begin{eqnarray*}
R_+^\Omega(0,x)[\varphi]
&=&R_+(0)[(\mu_x\varphi)\circ\exp_x]\\
&=&\delta_0[(\mu_x\varphi)\circ\exp_x]\\
&=&((\mu_x\varphi)\circ\exp_x)(0)\\
&=&\mu_x(x)\varphi(x)\\
&=&\varphi(x)\\
&=&\delta_x[\varphi].
\end{eqnarray*}

Proof of {(\ref{Rglattinx})}.
Let $A(x,x') : T_x\Omega \to T_{x'}\Omega$ be a timeorientation preserving
linear isometry.
Then
$$
R_+^\Omega(\alpha,x')[V(x',\cdot)] = R_+(\alpha)[(\mu_{x'}\cdot V(x',\cdot))
\circ \exp_{x'}\circ A(x,x')]
$$
where $R_+(\alpha)$ is, as before, the Riesz distribution on $T_x\Omega$.
Hence if we choose $A(x,x')$ to depend smoothly on $x'$, then
$(\mu_{x'}\cdot V(x',y))\circ\exp_{x'}\circ A(x,x')$ is smooth in $x'$ and $y$
and the assertion follows from Lemma~\ref{zweivariablen}.

Proof of {(\ref{orderR})}.
Since $\ord(R_\pm(\alpha)) \leq n+1$ by Proposition~\ref{MinkRiesz}~(\ref{i})
we have $\ord(R^\Omega_\pm(\alpha,x)) \leq n+1$ as well.
From the definition $R_{\pm}^\Omega(\alpha,x)=\mu_x\,\exp_x^*R_{\pm}(\alpha)$
it is clear that the constant $C$ may be chosen locally uniformly in $x$.

Proof of {(\ref{RCkinx})}.
By {(\ref{orderR})} we can apply $R_\pm^\Omega(\alpha,x')$ to
$V(x',\cdot)$. 
Now the same argument as for {(\ref{Rglattinx})} shows that the assertion
follows from Lemma~\ref{zweivariablen}.

The remaining assertions follow directly from the corresponding properties
of the Riesz distributions on Minkowski space.
For example (\ref{Ckanalyt}) is a consequence of Corollary~\ref{Ckholo}.
\end{proof}
 
Advanced and retarded Riesz distributions are related as follows.

\begin{lemma}\label{xytausch}
Let $\Omega$ be a convex timeoriented Lorentzian manifold.
Let $\alpha\in\Co$.
Then for all $u\in\DD(\Omega\times\Omega,\Co)$ we have
$$ 
\int_\Omega R_+^\Omega(\alpha,x)\left[y\mapsto u(x,y)\right] \dV(x)
 = \int_\Omega R_-^\Omega(\alpha,y)\left[x\mapsto u(x,y)\right] \dV(y). 
$$  
\end{lemma}

\begin{proof} 
The convexity condition for $\Omega$ ensures that the Riesz distributions 
$R_\pm^\Omega(\alpha,x)$ are defined for all $x \in \Omega$.
By Proposition~\ref{OmegaRiesz}~{(\ref{Rglattinx})} the integrands are smooth.
Since $u$ has compact support contained in $\Omega\times\Omega$ the integrand
$R_+^\Omega(\alpha,x)\left[y\mapsto u(x,y)\right]$ (as a function in $x$) has
compact support contained in $\Omega$.
A similar statement holds for the integrand of the right hand side.
Hence the integrals exist.
By Proposition~\ref{OmegaRiesz}~{(\ref{Ckanalyt})} they are holomorphic in
$\alpha$.
Thus it suffices to show the equation for $\alpha$ with 
$\Re(\alpha)>n$.

For such an $\alpha\in\Co$ the Riesz distributions $R_+(\alpha,x)$ and
$R_-(\alpha,y)$  are continuous functions.
From the explicit formula {(\ref{Rexplic})} in Proposition~\ref{OmegaRiesz} we
see
\begin{equation*} R_+(\alpha,x)(y)= R_-(\alpha,y)(x)
\end{equation*}
for all $x,y\in\Omega$.
By Fubini's theorem we get
\begin{eqnarray*}
\int_\Omega R_+^\Omega(\alpha,x)\left[ y\mapsto u(x,y) \right] \dV(x)
&=& 
\int_\Omega \left( \int_{\Omega} R_+^\Omega (\alpha,x)(y)\;u(x,y)\,
\dV(y) \right)\, \dV(x)\\
&=& 
\int_\Omega \left( \int_{\Omega} R_-^\Omega(\alpha,y)(x)\;u(x,y)\,
\dV(x)\right)\,\dV(y)\\
&=&
\int_\Omega R_-^\Omega(\alpha,y)\left[x\mapsto u(x,y)\right] \dV(y).
\end{eqnarray*}
\end{proof}

As a technical tool we will also need a version of Lemma~\ref{xytausch}
for certain nonsmooth sections.

\begin{lemma}\label{xytauschCk}
Let $\Omega$ be a causal domain in a timeoriented Lorentzian manifold of
dimension $n$.
Let $\Re(\alpha)>0$ and let $k \geq n+1$.
Let $K_1$, $K_2$ be compact subsets of $\ovl{\Omega}$ and let
$u\in C^k(\ovl{\Omega}\times\ovl{\Omega})$ so that $\supp(u)\subset
J_+^\Omega(K_1)\times J_-^\Omega(K_2)$.   
Then
$$ 
\int_\Omega R_+^\Omega(\alpha,x)\left[y\mapsto u(x,y)\right] \dV(x)
 = \int_\Omega R_-^\Omega(\alpha,y)\left[x\mapsto u(x,y)\right] \dV(y). 
$$  
\end{lemma}

\begin{proof}
For fixed $x$, the support of the function $y\mapsto u(x,y)$ is contained in
$J_-^\Omega(K_2)$. 
Since $\Omega$ is causal, it follows from Lemma~\ref{pastcompact} that the
subset $J_-^\Omega(K_2)\cap J_+^\Omega(x)$ is relatively compact in
$\ovl{\Omega}$. 
Therefore the intersection of the supports of $y\mapsto u(x,y)$ and
$R_+^\Omega(\alpha,x)$ is compact and contained in $\ovl{\Omega}$.
By Proposition~\ref{OmegaRiesz} {(\ref{orderR})} one can then apply
$R_+^\Omega(\alpha,x)$ to the $C^k$-function $y\mapsto u(x,y)$. 
Furthermore, the support of the continuous function $x\mapsto
R_+^\Omega(\alpha,x)\left[y\mapsto u(x,y)\right]$ is contained in
$J_+^\Omega(K_1)\cap J_-^\Omega(\supp(y\mapsto u(x,y)))\subset
J_+^\Omega(K_1)\cap J_-^\Omega(J_-^\Omega(K_2))=J_+^\Omega(K_1)\cap
J_+^\Omega(K_2)$, which is relatively compact in $\ovl{\Omega}$, again by
Lemma~\ref{pastcompact}.
Hence the function $x\mapsto R_+^\Omega(\alpha,x)\left[y\mapsto u(x,y)\right]$
has compact support in $\ovl{\Omega}$, so that the left-hand-side makes sense. 
Analogously the right-hand-side is well-defined.
Our considerations also show that the integrals depend only on the
values of $u$ on $\left(J_+^\Omega(K_1)\cap
  J_-^\Omega(K_2)\right)\times\left(J_+^\Omega(K_1)\cap
  J_-^\Omega(K_2)\right)$ which is a relatively compact set.
Applying a cut-off function argument we may assume without loss of generality
that $u$ has 
compact support.
Proposition~\ref{OmegaRiesz}~{(\ref{Ckanalyt})} says that the integrals
depend holomorphically on $\alpha$ on the domain $\{\Re(\alpha)>0\}$. 
Therefore it suffices to show the equality for $\alpha$ with sufficiently large
real part, which can be done exactly as in the proof of Lemma~\ref{xytausch}. 
\end{proof}

%%%%%%%%%%%%%%%%%%%%%%%%%%%%%%%%%%%%%%%%%%%%%%%%%%%%%%%%%%%%%%%%%%%%%%%%% 
\section{Normally hyperbolic operators} 
%%%%%%%%%%%%%%%%%%%%%%%%%%%%%%%%%%%%%%%%%%%%%%%%%%%%%%%%%%%%%%%%%%%%%%%%%
\indexn{normally hyperbolic operator>defemidx}

Let ${M}$ be a Lorentzian manifold and let $E\to {M}$ be a real or complex 
vector bundle.
For a summary on basics concerning linear differential operators see
Appendix~\ref{app:diffops}.
A linear differential operator $P:\,C^{\infty}({M},E)\to 
C^{\infty}({M},E)$ of second order will be called
\defem{normally hyperbolic}  if its principal symbol is given by the metric,
$$
\sigma_P(\xi) = -\la\xi,\xi\ra\cdot\id_{E_x}
$$
for all $x\in{M}$ and all $\xi \in T^*_x{M}$.
In other words, if we choose local coordinates $x^1,\ldots,x^n$ on $M$ and a
local trivialization of $E$, then 
$$
P = -\sum_{i,j=1}^n g^{ij}(x)\frac{\partial^2}{\partial x^i \partial x^j}
+ \sum_{j=1}^n A_j(x) \frac{\partial}{\partial x^j} + B_1(x)
$$
where $A_j$ and $B_1$ are matrix-valued coefficients depending smoothly on
$x$ and $(g^{ij})_{ij}$ is the inverse matrix of  $(g_{ij})_{ij}$ with
$g_{ij}=\la\frac{\partial}{\partial x^i},\frac{\partial}{\partial x^j}\ra$.

\Example{
Let $E$ be the trivial line bundle so that sections in $E$ are just functions.
The d'Alembert operator $P=\Box$ is normally hyperbolic because
$$
\sigma_{\grad}(\xi)f = f\xi^\sharp, \quad
\sigma_{\div}(\xi)X = \xi(X)
$$
and so
$$
\sigma_{\Box}(\xi)f = - \sigma_{\div}(\xi)\circ\sigma_{\grad}(\xi)f
= - \xi(f\xi^\sharp) = -\la\xi,\xi\ra\, f.
$$
Recall that $\xi\mapsto \xi^\sharp$ denotes the isomorphism $T_x^*M \to T_xM$
induced by the Lorentzian metric, compare (\ref{defgrad}).
}
\Example{
Let $E$ be a vector bundle and let $\nabla$ be a connection on $E$.
This connection together with the Levi-Civita connection on $T^*M$ induces a
connection on $T^*M\otimes E$, again denoted $\nabla$.
We define the \defidx{connection-d'Alembert operator} $\Box^\nabla$ 
\indexs{*@$\protect\Box^\nabla$, connection-d'Alembert operator}
to be minus
the composition of the following three maps
$$
C^\infty(M,E) \stackrel{\nabla}{\longrightarrow}
C^\infty(M,T^*M\otimes E) \stackrel{\nabla}{\longrightarrow}
C^\infty(M,T^*M\otimes T^*M\otimes E)
\xrightarrow{\tr\otimes\id_E} 
C^\infty(M,E)
$$
where $\tr:T^*M\otimes T^*M \to \R$ denotes the metric trace,
$\tr(\xi\otimes\eta)=\la\xi,\eta\ra$. 
We compute the principal symbol,
$$
\sigma_{\Box^\nabla}(\xi)\phi
=
-(\tr\otimes\id_E)\circ\sigma_{\nabla}(\xi)\circ\sigma_{\nabla}(\xi)(\phi)
=
-(\tr\otimes\id_E)(\xi\otimes\xi\otimes\phi)
=
-\la\xi,\xi\ra\,\phi.
$$
Hence $\Box^\nabla$ is normally hyperbolic.
}
\Example{
Let $E=\Lambda^kT^*M$ be the bundle of $k$-forms.
Exterior differentiation $d:C^\infty(M,\Lambda^kT^*M) \to
C^\infty(M,\Lambda^{k+1}T^*M)$ increases the degree by one while
the codifferential $\delta:C^\infty(M,\Lambda^{k}T^*M) \to
C^\infty(M,\Lambda^{k-1}T^*M)$ decreases the degree by one, see
\cite[p.~34]{Bes} for details.
While $d$ is independent of the metric, the codifferential $\delta$ does
depend on the Lorentzian metric.
The operator $P=d\delta + \delta d$ is normally hyperbolic.
}
\Example{
If $M$ carries a Lorentzian metric and a spin structure, then one can define
the spinor bundle $\Sigma M$ and the Dirac operator
\indexn{Dirac operator}
$$
D : C^\infty(M,\Sigma M) \to C^\infty(M,\Sigma M) ,
$$
see \cite{BGM} or \cite{Baum} for the definitions.
The principal symbol of $D$ is given by Clifford multiplication,
$$
\sigma_D(\xi)\psi = \xi^\sharp \cdot \psi.
$$
Hence
$$
\sigma_{D^2}(\xi)\psi = \sigma_D(\xi)\sigma_D(\xi)\psi = \xi^\sharp \cdot
\xi^\sharp \cdot \psi = -\la\xi,\xi\ra\,\psi.
$$
Thus $P=D^2$ is normally hyperbolic.
}

The following lemma is well-known, see e.~g.\ \cite[Prop.~3.1]{BK}.
It says that each normally hyperbolic operator is a connection-d'Alembert
operator up to a term of order zero.

\begin{lemma}\label{canonicalconnection}
Let $P:\,C^{\infty}({M},E)\to C^{\infty}({M},E)$ be a normally hyperbolic
operator on a Lorentzian manifold $M$.
Then there exists a unique connection $\nabla$ on $E$ and a unique
endomorphism field $B\in C^{\infty}({M},\Hom(E,E))$ such that
$$
P = \Box^\nabla + B.
$$
\end{lemma}

%% Here $\Box^\nabla = -\tr(\nabla^2) = - \sum_{i,j=1}^n g^{ij}\left(
%% \nabla_{\frac{\partial}{\partial x^i}}\nabla_{\frac{\partial}{\partial x^j}} -
%% \nabla_{\nabla_{\frac{\partial}{\partial x^i}}\frac{\partial}{\partial
%% x^j}}\right).$

\begin{proof} 
First we prove uniqueness of such a connection.
Let $\nabla'$ be an arbitrary connection on $E$.
For any section $s\in C^\infty(M,E)$ and any function $f\in C^\infty(M)$ we
get
\begin{equation}\label{nablaleipniz}
  \Box^{\nabla'}(f\cdot s)= f\cdot(\Box^{\nabla'}s)-2\nabla'_{\grad f}s
  +(\Box   f)\cdot s .
\end{equation}
Now suppose that $\nabla$ satisfies the condition in Lemma~
\ref{canonicalconnection}. 
Then $B=P-\Box^{\nabla}$ is an endomorphism field and we obtain
\[ f\cdot\left( P(s)-\Box^\nabla s \right)=P(f\cdot
s)-\Box^\nabla(f\cdot s). \] 
By (\ref{nablaleipniz}) this yields
\begin{equation}\label{supernabla}
\nabla_{\grad f}s=\tfrac{1}{2}\left\{f\cdot P(s)-P(f\cdot s)+(\Box f)\cdot s
\right\}. 
\end{equation}
At a given point $x\in M$ every tangent vector $X\in T_xM$ can be written in
the form $X=\grad_x f$ for some suitably chosen function $f$. 
Thus (\ref{supernabla}) shows that $\nabla$ is determined by $P$ and
$\Box$ (which is determined by the Lorentzian metric).

To show existence one could use (\ref{supernabla}) to define a
connection $\nabla$ as in the statement. 
We follow an alternative path.
Let $\nabla'$ be some connection on $E$.
Since $P$ and $\Box^{\nabla'}$ are both normally hyperbolic operators acting
on sections in $E$, the difference $P-\Box^{\nabla'}$ is a differential
operator of first order and can therefore be written in the form
\[P-\Box^{\nabla'}=A'\circ\nabla'+B',\]
for some $A'\in C^{\infty}({M},\mathrm{Hom}(T^*M\otimes E,E))$ and $B'\in
C^{\infty}({M},\Hom(E,E))$.
Set for every vector field $X$ on $M$ and section $s$ in $E$
\[\nabla_X s:=\nabla'_X s-\frac{1}{2}A'(X^\flat\otimes s).\]
This defines a new connection $\nabla$ on $E$.
Let $e_1,\ldots,e_n$ be a local Lorentz orthonormal basis of $TM$.
Write as before $\varepsilon_j=\la e_j,e_j\ra=\pm 1$. 
We may assume that at a given point $p\in M$ we have $\nabla_{e_j}e_j(p)=0$.
Then we compute at $p$
\be 
\Box^{\nabla'}s+A'\circ\nabla's
&=&
\sum_{j=1}^n\varepsilon_j
\Big\{-\nabla'_{e_j}\nabla'_{e_j}s+A'(e_j^\flat
\otimes\nabla'_{e_j}s)\Big\}\\  
&=&
\sum_{j=1}^n\varepsilon_j\Big\{-(\nabla_{e_j}+
\frac{1}{2}A'(e_j^\flat\otimes\cdot))(\nabla_{e_j}s+
\frac{1}{2}A'(e_j^\flat\otimes s)) \\  
&&\phantom{\sum_{j=1}^n\varepsilon\Big\{}+A'(e_j^\flat\otimes\nabla_{e_j}s)+
\frac{1}{2}A'(e_j^\flat\otimes A'(e_j^\flat\otimes s))\Big\}\\ 
&=&
\sum_{j=1}^n\varepsilon_j\Big\{-\nabla_{e_j}\nabla_{e_j}s 
-\frac{1}{2}\nabla_{e_j}(A'(e_j^\flat\otimes s))\\ 
&&\phantom{\sum_{j=1}^n\varepsilon_j\Big\{}+
\frac{1}{2}A'(e_j^\flat\otimes\nabla_{e_j}s)+ \frac{1}{4}A'(e_j^\flat\otimes
A'(e_j^\flat\otimes s))\Big\}\\ 
&=&
\Box^\nabla s+\frac{1}{4}\sum_{j=1}^n\varepsilon_j
\Big\{A'(e_j^\flat\otimes A'(e_j^\flat\otimes s))-
2(\nabla_{e_j}A')(e_j^\flat\otimes s)\Big\}, 
\ee
where $\nabla$ in $\nabla_{e_j}A'$ stands for the induced connection on
$\mathrm{Hom}(T^*M\otimes E,E)$. 
We observe that $Q(s) := \Box^{\nabla'}s+A'\circ\nabla's - \Box^\nabla s =
\frac{1}{4}\sum_{j=1}^n\varepsilon_j
\Big\{A'(e_j^\flat\otimes A'(e_j^\flat\otimes s))-
2(\nabla_{e_j}A')(e_j^\flat\otimes s)\Big\}$ is of order zero.
Hence 
$$
P = \Box^{\nabla'}+A'\circ\nabla'+B'
= \Box^\nabla s+Q(s)+B'(s)
$$
is the desired expression with $B=Q+B'$.
\end{proof}

The connection in Lemma~\ref{canonicalconnection} will be called the
\defem{$P$-compatible} connection.  
We shall henceforth always work with the $P$-compatible connection.
\indexn{Pcom*@$P$-compatible connection>defemidx}
We restate (\ref{supernabla}) as a lemma.

\begin{lemma}\label{produktregel}
Let $ P=\Box^\nabla+B$ be normally hyperbolic.
For $f\in C^{\infty}({M})$ and $s\in C^{\infty}({M},E)$ one gets
$$ 
P(f\cdot s)=f\cdot P(s)-2\,{\nabla}_{\grad f}\,s +\Box f\cdot s.
$$  
\hfill$\Box$
\end{lemma}

%%%%%%%%%%%%%%%%%%%%%%%%%%%%%%%%%%%%%%%%%%%%%%%%%%%%%%%%%%%%%%%%%%%%%%%%% 
\chapter{The local theory}
%%%%%%%%%%%%%%%%%%%%%%%%%%%%%%%%%%%%%%%%%%%%%%%%%%%%%%%%%%%%%%%%%%%%%%%%% 

Now we start with our detailed study of wave equations.
\indexn{wave equation>defemidx}
By a wave equation we mean an equation of the form $Pu=f$ where $P$ is a 
normally hyperbolic operator acting on sections in a vector bundle.
The right-hand-side $f$ is given and the section $u$ is to be found.
In this chapter we deal with local problems, i.~e., we try to find solutions
defined on sufficiently small domains.
This can be understood as a preparation for the global theory which
we postpone to the third chapter.
Solving wave equations on all of the Lorentzian manifold is, in general,
possible only under the geometric assumption of the manifold being globally
hyperbolic. 

There are various techniques available in the theory of partial differential
equations that can be used to settle the local theory.
We follow an approach based on Riesz distributions and Hadamard coefficients
as in \cite{Guenther}.
The central task is to construct fundamental solutions. 
This means that one solves the wave equation where the right-hand-side $f$
is a delta-distribution.

The construction consists of three steps.
First one writes down a formal series in Riesz distributions with unknown
coefficients.
The wave equation yields recursive relations for these Hadamard coefficients
known as transport equations.
Since the transport equations are ordinary differential equations along
geodesics they can be solved uniquely.
There is no reason why the formal solution constructed in this way should be
convergent.

In the second step one makes the series convergent by introducing certain
cut-off functions.
This is similar to the standard proof showing that each formal power series is
the Taylor series of some smooth function.
Since there are error terms produced by the cut-off functions the result is
convergent but no longer solves the wave equation.
We call it an approximate fundamental solution.

Thirdly, we turn the approximate fundamental solution into a true one using
certain integral operators.
Once the existence of fundamental solutions is established one can find
solutions to the wave equation for an arbitrary smooth $f$ with compact
support.
The support of these solutions is contained in the future or in the past of
the support of $f$.

Finally, we show that the formal fundamental solution constructed in the first
step is asymptotic to the true fundamental solution.
This implies that the singularity structure of the fundamental solution is
completely determined by the Hadamard coefficients which are in turn
determined by the geometry of the manifold and the coefficients of the
operator.

%%%%%%%%%%%%%%%%%%%%%%%%%%%%%%%%%%%%%%%%%%%%%%%%%%%%%%%%%%%%%%%%%%%%%%%%%
\section{The formal fundamental solution}
%%%%%%%%%%%%%%%%%%%%%%%%%%%%%%%%%%%%%%%%%%%%%%%%%%%%%%%%%%%%%%%%%%%%%%%%%

In this chapter the underlying Lorentzian manifold will typically be denoted
by $\Omega$.
Later, in Chapter~3, when we apply the local results $\Omega$ will play the
role of a small neighborhood of a given point.

\Definition{\label{deffundsol}
Let $\Omega$ be a timeoriented Lorentzian manifold, let $E\to \Omega$ be a
vector bundle and let $P:\,C^{\infty}(\Omega,E)\to C^{\infty}(\Omega,E)$ be
normally hyperbolic. 
Let $x\in\Omega$.
A \defidx{fundamental solution} of $P$ at $x$ is a distribution 
$F\in\DD'(\Omega,E,E_x^*)$ such that
$$
PF=\delta_x.
$$
In other words, for all $\varphi\in\DD(\Omega,E^*)$ we have
$$
F[P^*\varphi] = \varphi(x).
$$
If $\supp(F(x))\subset J_+^\Omega(x)$, then we call $F$ an \defem{advanced
fundamental solution}, if $\supp(F(x))\subset J_-^\Omega(x)$, then we call $F$
a \defem{retarded fundamental solution}.
\indexn{advanced fundamental solution>defemidx}
\indexn{retarded fundamental solution>defemidx}
}

For flat Minkowski space with $P=\Box$ acting on functions 
Proposition~\ref{MinkRiesz} {(\ref{c})} and {(\ref{h})} show that the Riesz 
distributions $R_\pm(2)$ are fundamental solutions at $x=0$.
More precisely, $R_+(2)$ is an advanced fundamental solution because its
support is contained in $J_+(0)$ and $R_-(2)$ is a retarded fundamental
solution.

On a general timeoriented Lorentzian manifold $\Omega$ the situation is more
complicated even if $P=\Box$. 
The reason is the factor $\frac{\Box\Gamma_x-2n}{2\alpha}+1$ in
Proposition~\ref{OmegaRiesz} {(\ref{boxR})} which cannot be evaluated for 
$\alpha=0$ unless $\Box\Gamma_x-2n$ vanished identically.
It will turn out that $R_\pm^\Omega(2,x)$ does not suffice to construct
fundamental solutions.
We will also need Riesz distributions $R_\pm^\Omega(2+2k,x)$ for $k\ge 1$.

Let $\Omega$ be geodesically starshaped with respect to some fixed
$x\in\Omega$ so  
that the Riesz distributions $R_\pm^\Omega(\alpha,x)=R_\pm^\Omega(\alpha,x)$ 
are defined.
Let $E\to\Omega$ be a real or complex vector bundle and let $P$ be a normally 
hyperbolic operator $P$ acting on $C^\infty(\Omega,E)$.
In this section we start constructing fundamental solutions.
We make the following formal ansatz: 
$$ 
\RR_\pm(x):= \sum_{k=0}^{\infty} V_x^k \; R_\pm^{\Omega}(2+2k,x) 
$$
where $V_x^k\in C^{\infty}(\Omega,E\otimes E_x^*)$ are smooth sections yet to 
be found. 
For $\varphi\in\DD(\Omega,E^*)$ the function
$V_x^k\cdot\varphi$ is an $E_x^*$-valued testfunction and we have
$(V_x^k\cdot R_\pm^{\Omega}(2+2k,x))[\varphi] =
R_\pm^{\Omega}(2+2k,x)[V_x^k\cdot\varphi]\in E_x^*$. 
Hence each summand $V_x^k\cdot R_\pm^{\Omega}(2+2k,x)$ is a distribution in 
$\DD'(\Omega,E,E_x^*)$.

By formal termwise differentiation using Lemma~\ref{produktregel} and 
Proposition~\ref{OmegaRiesz} we translate the condition of $\RR_\pm(x)$ being a
fundamental solution at $x$ into conditions on the $V_x^k$.
To do this let $\nabla$ be the $P$-compatible connection on $E$, that is,
$P=\Box^\nabla+B$ where $B\in C^\infty(\Omega,\mathrm{End}(E))$, compare
Lemma~\ref{canonicalconnection}.  
We compute
\begin{eqnarray} 
R_{\pm}^{\Omega}(0,x)&=&\delta_x \ =\ P\RR_{\pm}(x) 
\ =\ \sum_{k=0}^{\infty}P(V_x^kR_\pm^\Omega(2+2k,x)) \nonumber\\
&=&
\sum_{k=0}^{\infty}\lbrace V_x^k\cdot\Box R_{\pm}^{\Omega}(2+2k,x)-2\,
\nabla_{\grad\,R_{\pm}^{\Omega}(2+2k,x)}V_x^k+ PV_x^k\cdot
R_{\pm}^{\Omega}(2+2k,x)\rbrace \nonumber\\
&=& 
V_x^0\cdot\Box R_{\pm}^{\Omega}(2,x)-2\,\nabla_{\grad R_{\pm}^{\Omega}(2,x)
}V_x^0\nonumber\\  
&&+\sum_{k=1}^{\infty}\Big\lbrace V_x^k\cdot\left(
\tfrac{\tfrac{1}{2}\Box\Gamma_x -n}{2k}+1 \right) R_{\pm}^{\Omega}(2k,x)
-\tfrac{2}{4k}\,\nabla_{\grad\Gamma_x\,R_{\pm}^{\Omega}(2k,x)}V_x^k\nonumber\\
&&\qquad\quad+PV_x^{k-1}\cdot R_{\pm}^{\Omega}(2k,x) \Big\rbrace \nonumber\\
&=& 
V_x^0\cdot\Box R_{\pm}^{\Omega}(2,x)-2\,\nabla_{\grad
         R_{\pm}^{\Omega}(2,x)}V_x^0\nonumber\\ 
&&+\sum_{k=1}^{\infty}\tfrac{1}{2k}\Big\lbrace \left(\tfrac{1}{2}\Box
\Gamma_x -n +2k\right)V_x^k-\nabla_{\grad\Gamma_x}
V_x^k+2k\,PV_x^{k-1}\Big\rbrace 
R_{\pm}^{\Omega}(2k,x) .
\label{termdiff}
\end{eqnarray}
Comparing the coefficients of $R_{\pm}^{\Omega}(2k,x)$ we get the conditions
\begin{align}
   2\,\nabla_{\grad R_{\pm}^{\Omega}(2,x)}V_x^0 - \Box
  R_{\pm}^{\Omega}(2,x)\cdot 
  V_x^0+ R_{\pm}^{\Omega}(0,x)&=0\quad\mbox{ and} \label{Bostrich}\\ 
  \nabla_{\grad\Gamma_x} V_x^k-\left(\tfrac{1}{2}\Box\Gamma_x -n
  +2k\right)V_x^k &= 2k\,PV_x^{k-1}\quad\mbox{ for }k\ge 1.\label{Beka}
\end{align}
We take a look at what condition (\ref{Beka}) would mean for $k=0$.
We multiply this equation by $R_{\pm}^{\Omega}(\alpha,x)$:
$$ 
\nabla_{\grad\Gamma_x\, R_{\pm}^{\Omega}(\alpha,x)} V_x^0-
\left(\tfrac{1}{2}\Box\Gamma_x -n \right)\,V_x^0\cdot
R_{\pm}^{\Omega}(\alpha,x) = 0.
$$
By Proposition~\ref{OmegaRiesz} (\ref{gradR}) and (\ref{boxR}) we obtain
$$ 
\nabla_{2 \alpha \grad R_{\pm}^{\Omega}(\alpha+2,x)} V_x^0-
\left(\alpha\Box R_{\pm}^{\Omega}(\alpha+2,x)-\alpha
R_{\pm}^{\Omega}(\alpha,x) \right)\,V_x^0 = 0.
$$  
Division by $\alpha$ and the limit $\alpha\to 0$ yield
$$ 
2\nabla_{\grad R_{\pm}^{\Omega}(2,x)} V_x^0-
\left(\Box R_{\pm}^{\Omega}(2,x)- R_{\pm}^{\Omega}(0,x) \right)\,V_x^0 = 0. 
$$  
Therefore we recover condition (\ref{Bostrich}) if and only if
$V_x^0(x)=\id_{E_x}$. \medskip

To get formal fundamental solutions $\RR_\pm(x)$ for $P$ we hence need
$V_x^k\in C^{\infty}(\Omega,E\otimes E_x^*)$ satisfying
\begin{equation}
\nabla_{\grad\Gamma_x} V_x^k-\left(\tfrac{1}{2}\Box\Gamma_x -n
+2k\right)V_x^k = 2k\,PV_x^{k-1}
\label{eq:transport}
\end{equation}
for all $k\ge 0$ with ``initial condition'' $V_x^0(x)=\id_{E_x}$.
In particular, we have the same conditions on $V_x^k$ for $\RR_+(x)$ and for 
$\RR_-(x)$.
Equations (\ref{eq:transport}) are known as {\em transport equations}.

%%%%%%%%%%%%%%%%%%%%%%%%%%%%%%%%%%%%%%%%%%%%%%%%%%%%%%%%%%%%%%%%%%%%%%%%%
\section{Uniqueness of the Hadamard coefficients}
%%%%%%%%%%%%%%%%%%%%%%%%%%%%%%%%%%%%%%%%%%%%%%%%%%%%%%%%%%%%%%%%%%%%%%%%%

This and the next section are devoted to uniqueness and existence of solutions
to the transport equations.

\Definition{
Let $\Omega$ be timeoriented and geodesically starshaped with respect to
$x\in\Omega$. 
Sections $V_x^k\in C^{\infty}(\Omega,E\otimes E_x^*)$ are called 
\defem{Hadamard coefficients} 
\indexn{Hadamard coefficients>defemidx}
\indexs{V*@$V^k_x$, Hadamard coefficient at point $x$}
for $P$ at $x$ if they satisfy the transport
equations (\ref{eq:transport}) for all $k\ge 0$ and $V_x^0(x)=\id_{E_x}$.
Given Hadamard coefficients $V_x^k$ for $P$ at $x$ we call the formal series
\indexn{formal advanced fundamental solution>defemidx}
\indexs{R*@$\protect\RR_+(x)$, formal advanced fundamental solution}
$$
\RR_+(x) = \sum_{k=0}^{\infty} V_x^k \cdot R_+^{\Omega}(2+2k,x)
$$ 
a \defem{formal advanced fundamental solution for $P$ at $x$} and 
\indexn{formal retarded fundamental solution>defemidx}
\indexs{R*@$\protect\RR_-(x)$, formal retarded fundamental solution}
$$
\RR_-(x) = \sum_{k=0}^{\infty} V_x^k \cdot R_-^{\Omega}(2+2k,x)
$$ 
a \defem{formal retarded fundamental solution for $P$ at $x$}.
}

In this section we show uniqueness of the Hadamard coefficients (and hence
of the formal fundamental solutions $\RR_\pm(x)$) by
deriving explicit formulas for them.
These formulas will also be used in the next section to prove existence.

For $y\in\Omega$ we denote the $\nabla$-parallel translation along the
(unique) geodesic from $x$ to $y$ by
\indexs{Pixy@$\Pi^x_y:E_x\to E_y$, parallel translation along the geodesic from
  $x$ to $y$} 
$$ 
\Pi^x_y:E_x\to E_y.
$$
We have $\Pi^x_x = \id_{E_x}$ and $(\Pi^x_y)^{-1} = \Pi^y_x$.
Note that the map $\Phi:\Omega\times [0,1] \to \Omega$, 
$\Phi(y,s)=\exp_x(s\cdot\exp_x^{-1}(y))$, is well-defined and smooth since
$\Omega$ is geodesically starshaped with respect to $x$.
\indexs{Phiys@$\Phi(y,s)=\exp_x(s\cdot\exp_x^{-1}(y))$}

\begin{lemma}
Let $V_x^k$ be Hadamard coefficients for $P$ at $x$.
Then they are given by
\begin{equation}\label{vaunull}
V_x^0(y) = \mu_x^{-1/2}(y)\Pi^x_y 
\end{equation}
and for $k\ge 1$
\begin{equation}\label{vaukaglatt}
V_x^k(y)=
-k\,\mu_x^{-1/2}(y) \,\Pi^x_{y}\int_0^1 \mu_x^{1/2}(\Phi(y,s)) s^{k-1}\,
\Pi^{\Phi(y,s)}_x(PV_x^{k-1}(\Phi(y,s)))\,ds .
\end{equation}
\end{lemma}

\begin{proof}
We put $\rho := \sqrt{|\Gamma_x|}$.
On $\Omega\setminus C(x)$ where $C(x) = \exp_x(C(0))$ is the light cone of $x$
we have $\Gamma_x(y) = - \varepsilon \rho^2(y)$ where $\varepsilon = 1$ if 
$\exp_x^{-1}(y)$ is spacelike and $\varepsilon = -1$ if $\exp_x^{-1}(y)$ is
timelike. 
Using the identities  $\tfrac{1}{2}\Box\Gamma_x
  -n=-\tfrac{1}{2}\del_{\grad \Gamma_x}\log\mu_x= -\del_{\grad
   \Gamma_x}\log(\mu_x^{1/2})$ from Lemma~\ref{Gammalemma}~(3) and 
$\del_{\grad \Gamma_x}(\log\rho^k)=k\del_{-2\varepsilon\rho\grad\rho}\log\rho=
-2\varepsilon k\rho \frac{\del_{\grad\rho\,}\rho}{\rho}=-2k $
we reformulate (\ref{eq:transport}):
\begin{equation*}
 \nabla_{\grad\Gamma_x} V_x^k
+\del_{\grad\Gamma_x}\log\left(\mu_x^{1/2}\cdot\rho^k\right)V_x^k 
= 2k\,PV_x^{k-1}. 
\end{equation*}
This is equivalent to
\begin{align} 
\nabla_{\grad\Gamma_x}\left(\mu_x^{1/2}\cdot\rho^k\cdot V_x^k\right)&=
\mu_x^{1/2}\cdot\rho^k\nabla_{\grad\Gamma_x}V_x^k
+\del_{\grad\Gamma_x}(\mu_x^{1/2}\cdot\rho^k) V_x^k \nonumber\\ 
&= \mu_x^{1/2}\cdot\rho^k\cdot 2k\cdot PV_x^{k-1}. 
\label{defining} 
\end{align}
For $k=0$ one has $\nabla_{\grad\Gamma_x}(\mu_x^{1/2}V_x^0)=0$. 
Hence $\mu_x^{1/2}V_x^0$ is $\nabla$-parallel along the
timelike and spacelike geodesics starting in $x$. 
By continuity it is $\nabla$-parallel along any geodesic starting at 
$x$. 
Since $\mu_x^{1/2}(x)V_x^0(x)=1\cdot \id_{E_x}=\Pi^x_x$ we conclude
$\mu_x^{1/2}(y)V_x^0(y)=\Pi^x_y$ for all $y\in\Omega$.
This shows (\ref{vaunull}).

Next we determine $V_x^k$ for $k\ge 1$. 
We consider some point $y\in\Omega\setminus C(x)$ outside the light cone of 
$x$.
We put $\eta:=\exp_x^{-1}(y)$. 
Then $c(t):=\exp_x(e^{2t}\cdot\eta)$ gives a reparametrization of the 
geodesic $\beta(t)=\exp_x(t\eta)$ from $x$ to $y$ 
such that $\dot{c}(t) = 2e^{2t}\dot{\beta}(e^{2t})$.
By Lemma~\ref{Gammalemma}~(1) 
\begin{eqnarray*}
\la \dot{c}(t),\dot{c}(t) \ra 
&=&
4e^{4t}\la \dot{\beta}(e^{2t}), \dot{\beta}(e^{2t}) \ra \\
&=&
4e^{4t}\la\eta,\eta\ra 
\,\,\,=\,\,\,
-4\gamma(e^{2t}\eta)\\
&=&
-4\Gamma_x(c(t))
\,\,\,=\,\,\,
\la\grad\Gamma_x,\grad\Gamma_x\ra .
\end{eqnarray*}
Thus $c$ is an integral curve of the vector field $-\grad\,\Gamma_x$. 
Equation~(\ref{defining}) can be rewritten as 
$$ 
-\frac{\nabla}{dt}\left(\mu_x^{1/2}\cdot\rho^k\cdot V_x^k
\right)(c(t))= \left(\mu_x^{1/2}\cdot\rho^k\cdot 2k\cdot
  PV_x^{k-1}\right)(c(t)),
$$ 
which we can solve explicitly:
\begin{eqnarray*}
\lefteqn{\left(\mu_x^{1/2}\cdot\rho^k\cdot V_x^k\right)(c(t))}\\
&=&
-\Pi^x_{c(t)}\left(\int_{-\infty}^t\Pi^{c(\tau)}_x
  \left(\mu_x^{1/2}\cdot\rho^k\cdot 2k\cdot 
  PV_x^{k-1} \right)(c(\tau))d\tau \right)\\
&=&
-2k\,\Pi^x_{c(t)}\left(\int_{-\infty}^t\mu_x^{1/2}
  (c(\tau))\rho(c(\tau))^k\Pi^{c(\tau)}_x
  \left(PV_x^{k-1}(c(\tau))\right)d\tau\right).  
\end{eqnarray*}
We have $\rho(c(\tau))^k= \rho(\exp_x(e^{2\tau}\eta))^k=
|\gamma(e^{2\tau}\eta)|^{k/2}=
|e^{4\tau}\gamma(\eta)|^{k/2}=
e^{2k\tau}|\gamma(\eta)|^{k/2}$. 
Since $y\not\in C(x)$ we can divide by $|\gamma(\eta)|\ne 0\,$:
\begin{eqnarray*}
\lefteqn{e^{2kt}\left(\mu_x^{1/2} V_x^k\right)(c(t))}\\
&=&
-2k\,\Pi^x_{c(t)}\left(\int_{-\infty}^t\mu_x^{1/2}(c(\tau))\, e^{2k\tau}\,
  \Pi^{c(\tau)}_x\left(PV_x^{k-1}(c(\tau))\right)\,d\tau\right) \\
&=&
-2k\,\Pi^x_{c(t)}%\left(
 \int_{0}^{e^{2t}}\mu_x^{1/2}(\exp_x(s\cdot\eta))\,s^k\,
  \Pi^{\exp_x(s\cdot\eta)}_x
   (PV_x^{k-1}(\exp_x(s\cdot\eta)))\,\,\frac{ds}{2s}
\end{eqnarray*}
where we used the substitution $s=e^{2\tau}$.
For $t=0$ this yields (\ref{vaukaglatt}).
\end{proof}

\begin{corollary}
Let $\Omega$ be timeoriented and geodesically starshaped with respect to
$x\in\Omega$.
Let $P$ be a normally hyperbolic operator acting on sections in a vector
bundle $E$ over $\Omega$.

Then the Hadamard coefficients $V_x^k$ for $P$ at $x$ are unique for all
$k\geq0$. 
\hfill$\Box$
\end{corollary}

%%%%%%%%%%%%%%%%%%%%%%%%%%%%%%%%%%%%%%%%%%%%%%%%%%%%%%%%%%%%%%%%%%%%%%%%%
\section{Existence of the Hadamard coefficients}
%%%%%%%%%%%%%%%%%%%%%%%%%%%%%%%%%%%%%%%%%%%%%%%%%%%%%%%%%%%%%%%%%%%%%%%%%

Let $\Omega$ be timeoriented and geodesically starshaped with respect to
 $x\in\Omega$.
Let $P$ be a normally hyperbolic operator acting on sections in a real or
complex vector bundle $E$ over $\Omega$.
To construct Hadamard coefficients for $P$ at $x$ we use formulas 
(\ref{vaunull}) and (\ref{vaukaglatt}) obtained in the previous section 
as definitions:
\begin{equation*}\label{Vaunulldef} 
 V_x^0(y):=\mu_x^{-1/2}(y)\cdot\Pi^x_y      \mbox{\hspace{3cm} and}
\end{equation*}
\begin{equation*}\label{Vaukadef}
V_x^k(y):=
-k\,\mu_x^{-1/2}(y) \,\Pi^x_{y}\int_0^1 \mu_x^{1/2}(\Phi(y,s)) s^{k-1}\,
\Pi^{\Phi(y,s)}_x(PV_x^{k-1}(\Phi(y,s)))\,ds .
\end{equation*}
We observe that this defines smooth sections $V_x^k\in
C^\infty(\Omega,E\otimes E_x^*)$. 
We have to check for all $k\ge 0$
\begin{equation}\label{check}
  \nabla_{\grad\,\Gamma_x}\left(\mu_x^{1/2}\,\rho^k\,V_x^k\right)= 
  \mu_x^{1/2}\,\rho^k\cdot 2k\cdot PV_x^{k-1},
\end{equation} 
from which Equation~(\ref{eq:transport}) follows as we have already seen.
For $k=0$ Equation~(\ref{check}) obviously holds:
$$ 
\nabla_{\grad\,\Gamma_x}\left(\mu_x^{1/2}\,V_x^0\right)=
  \nabla_{\grad\,\Gamma_x}\Pi=0.
$$
For $k\ge 1$ we check:
\begin{eqnarray*}
\lefteqn{ \nabla_{\grad\,\Gamma_x}\left(\mu_x^{1/2}\,\rho^k\,V_x^k\right)(y)}\\
&=&
-k \nabla_{\grad\,\Gamma_x}\,\Pi^x_{y}\int_0^1\mu_x^{1/2}  
   (\Phi(y,s))\,
   \underbrace{\rho(y)^k\cdot s^k}_{=\rho(\Phi(y,s))^k}\,
    \Pi^{\Phi(y,s)}_xPV_x^{k-1}(\Phi(y,s))\,\frac{ds}{s} \\
&=&
-k \nabla_{\grad\,\Gamma_x}\,\Pi^x_{y}\int_0^1
    \left( \mu_x^{1/2}\rho^k\,\Pi^x_{\Phi(y,s)}(PV_x^{k-1})\right)(\Phi(y,s))
     \,\frac{ds}{s} \\
&\stackrel{s=e^{2\tau}}{=}& 
     -2k \nabla_{\grad\,\Gamma_x}\,\Pi^x_{y}\int_{-\infty}^0
      \left(\mu_x^{1/2} \rho^k\,\Pi^x_{\Phi(y,s)}(PV_x^{k-1})\right)
      (\underbrace{\Phi(y,e^{2\tau})}_{{\rm integral\;
            curve}\atop {\rm for\; }-\grad\Gamma_x})
      \, d\tau\\
&=& 
2k \;\Pi^x_{y}\; \frac{d}{dt}\Big|_{t=0} \int_{-\infty}^0
      \left( \mu_x^{1/2}\rho^k\,\Pi^x_{\Phi(y,s)}(PV_x^{k-1})\right)
      \left(\Phi\left(\Phi(y,e^{2t}),e^{2\tau}\right)\right)\,d\tau\\
&=& 
2k \;\Pi^x_{y}\; \frac{d}{dt}\Big|_{t=0} \int_{-\infty}^0
      \left( \mu_x^{1/2}\rho^k\,\Pi^x_{\Phi(y,s)}(PV_x^{k-1})\right)
      \left(\Phi\left(y,e^{2(\tau+t)}\right)\right)\,d\tau\\
&\stackrel{\tau'=\tau+t}{=}&
  2k \;\Pi^x_{y}\; \frac{d}{dt}\Big|_{t=0} \int_{-\infty}^t
      \left( \mu_x^{1/2}\rho^k\,\Pi^x_{\Phi(y,s)}(PV_x^{k-1})\right)
      \left(\Phi\left(y,e^{2\tau'}\right)\right)\,d\tau'\\
&=&
2k \;\Pi^x_{y}\;\left( \mu_x^{1/2}\rho^k\,\Pi^x_{\Phi(y,s)}(PV_x^{k-1})\right)
   (\underbrace{\Phi(y,e^0)}_{=y})\\
&=&
2k\,\mu_x^{1/2}(y)\,\rho^k(y)\,\big(PV_x^{k-1}\big)(y)
\end{eqnarray*}
which is (\ref{check}).
This shows the existence of the Hadamard coefficients and,
therefore, we have found formal fundamental solutions $\RR_\pm(x)$ for $P$ 
at fixed $x\in\Omega$. 

Now we let $x$ vary. 
We assume there exists an open subset $U\subset\Omega$ such that $\Omega$
is geodesically starshaped with respect to all $x\in U$.
This ensures that the Riesz distributions $R_\pm^\Omega(\alpha,x)$ are defined
for all $x\in U$.
We write $V_k(x,y):=V_x^k(y)$ for the Hadamard coefficients at $x$.
Thus $V_k(x,y) \in \Hom(E_x,E_y)=E_x^*\otimes E_y$.
\indexn{Hadamard coefficients}
\indexs{V*@$V_k(x,y)=V_x^k(y)$}
The explicit formulas (\ref{vaunull}) and (\ref{vaukaglatt}) show that the 
Hadamard coefficients $V_k$ also depend smoothly on $x$, i.~e.,
$$
V_k \in C^\infty(U\times\Omega,E^* \boxtimes E) .
$$
Recall that $E^* \boxtimes E$ is the bundle with fiber $(E^* \boxtimes
E)_{(x,y)} = E_x^*\otimes E_y$.
We have formal fundamental solutions for $P$ at all $x\in U$:
$$ 
\RR_\pm(x)=\sum_{k=0}^\infty V_k(x,\cdot)\,R_\pm^\Omega(2+2k,x) .
$$
%% We call $\RR_+(x)$ the formal advanced fundamental solution for $P$ at $x$
%% and  $\RR_-(x)$ the formal retarded fundamental solution.
We summarize our results about Hadamard coefficients obtained so far.

\begin{prop}
Let $\Omega$ be a Lorentzian manifold, let $U \subset \Omega$ be a nonempty 
open subset such that $\Omega$ is geodesically starshaped with respect to all
points $x\in U$. 
Let $P=\Box^\nabla+B$ be a normally hyperbolic operator acting on sections in
a real or complex vector bundle over $\Omega$.
Denote the $\nabla$-parallel transport by $\Pi$.

Then at each $x\in U$ there are unique Hadamard coefficients $V_k(x,\cdot)$
for $P$, $k\geq 0$.
They are smooth, $V_k\in C^\infty(U\times\Omega,E^*\boxtimes E)$, and are
given by 
$$
V_0(x,y)=\mu_x^{-1/2}(y)\cdot\Pi^x_y    
$$
and for $k\geq 1$
$$
V_k(x,y)=
-k\,\mu_x^{-1/2}(y) \,\Pi^x_{y}\int_0^1 \mu_x^{1/2}(\Phi(y,s)) s^{k-1}\,
\Pi^{\Phi(y,s)}_x(P_{(2)}V_{k-1})(x,\Phi(y,s))\,ds
$$
where $P_{(2)}$ denotes the action of $P$ on the second variable of $V_{k-1}$.
\indexs{P*@$P_{(2)}$, operator $P$ applied w.~r.~t.\ second variable}
\hfill$\Box$
\end{prop}

These formulas become particularly simple along the diagonal, i.~e., for
$x=y$.
We have for any normally hyperbolic operator $P$
$$
V_0(x,x) = \mu_x(x)^{-1/2}\Pi^x_x=\id_{E_x}.
$$
For $k\geq 1$ we get 
\begin{eqnarray*}
V_k(x,x) 
&=&
-k\underbrace{\mu_x^{-1/2}(x)}_{=1}\cdot 1\cdot
\underbrace{{\Pi}^x_x}_{=\mathrm{id}}\int_0^1
s^{k-1}\underbrace{{\Pi}_x^x}_{=\mathrm{id}}
(P_{(2)}V_{k-1})(x,x)\mu_x^{-1/2}(x)ds\\  
&=&
-(P_{(2)}V_{k-1})(x,x).
\end{eqnarray*}
We compute $V_1(x,x)$ for $P=\Box^\nabla+B$.
By (\ref{vaukaglatt}) and Lemma~\ref{produktregel} we have
\begin{eqnarray*}
V_1(x,x)
&=&
-(P_{(2)}V_0)(x,x)\\
&=&
-P(\mu_x^{-1/2}{\Pi}^x_\bullet)(x)\\
&=&
-\mu_x^{-1/2}(x)\cdot P({\Pi}^x_\bullet)(x)
+2{\nabla}_{\underbrace{\grad\mu_x(x)}_{=0}}{\Pi}^x_\bullet(x)
-(\Box\mu_x^{-1/2})(x)\cdot\id_{E_x}\\
&=&
-(\Box^\nabla+B)({\Pi}^x_\bullet)(x)
-(\Box\mu_x^{-1/2})(x)\cdot\id_{E_x}\\
&=&
-B|_x-(\Box\mu_x^{-1/2})(x)\cdot\id_{E_x}.
\end{eqnarray*}
From Corollary~\ref{boxmuscal} we conclude
\[
V_1(x,x)=\frac{\mathrm{scal}(x)}{6}\id_{E_x}-B|_x.
\]

\Remark{
We compare our definition of Hadamard coefficients with the definition used
in \cite{Guenther} and in \cite{BK}.
In \cite[Chap.~3, Prop.~1.3]{Guenther} Hadamard coefficients $U_k$ are
solutions of the differential equations
\begin{equation}\label{GueHada}
\guenterl[\Gamma_x,U_k(x,\cdot)]+\left(M(x,\cdot)+2k\right)
U_k(x,\cdot)
= -PU_{k-1}(x,\cdot)  
\end{equation}
with initial conditions $U_0(x)=\id_{E_x}$,
where, in our terminology, $\guenterl[f,\cdot]$ denotes $-\nabla_{\grad
  f}(\cdot)$ for the $P$-compatible connection $\nabla$, and  
$M(x,\cdot)=\tfrac{1}{2}\Box\Gamma_x -n$.
Hence (\ref{GueHada}) reads as
$$
-\nabla_{\grad\Gamma_x}U_k(x,\cdot)+\left(\tfrac{1}{2}\Box\Gamma_x-n+2k\right)
U_k(x,\cdot)
=
-PU_{k-1}(x,\cdot).
$$
We recover our defining equations (\ref{Beka}) after the substitution 
\[U_k=\tfrac{1}{2^k\cdot k!} V_k.\]
}

%%%%%%%%%%%%%%%%%%%%%%%%%%%%%%%%%%%%%%%%%%%%%%%%%%%%%%%%%%%%%%%%%%%%%%%%%
\section{True fundamental solutions on small domains} 
\label{sec:trueffsmall} 
%%%%%%%%%%%%%%%%%%%%%%%%%%%%%%%%%%%%%%%%%%%%%%%%%%%%%%%%%%%%%%%%%%%%%%%%%

In this section we show existence of ``true'' fundamental solutions in the 
sense of Definition~\ref{deffundsol} on sufficiently small causal domains in a 
timeoriented Lorentzian manifold $M$.
Assume that $\Omega'\subset M$ is a geodesically convex open subset.
We then have the Hadamard coefficients $V_j \in C^\infty(\Omega'\times\Omega',
E^*\boxtimes E)$ and for all $x\in \Omega'$ the 
formal fundamental solutions 
$$
\RR_\pm(x) = \sum_{j=0}^\infty V_j(x,\cdot)\,R_\pm^{\Omega'}(2+2j,x) .
$$
Fix an integer $N \geq \frac n2$ where $n$ is the dimension of the manifold
$M$. 
Then for all $j \geq N$ the distribution $R_\pm^{\Omega'}(2+2j,x)$ is a 
continuous function on $\Omega'$.
Hence we can split the formal fundamental solutions
\indexn{formal advanced fundamental solution}
\indexn{formal retarded fundamental solution}
$$
\RR_\pm(x) = \sum_{j=0}^{N-1} V_j(x,\cdot)\,R_\pm^{\Omega'}(2+2j,x) 
+
\sum_{j=N}^{\infty} V_j(x,\cdot)\,R_\pm^{\Omega'}(2+2j,x) 
$$
where $\sum_{j=0}^{N-1} V_j(x,\cdot)\,R_\pm^{\Omega'}(2+2j,x)$ is a
well-defined $E_x^*$-valued distribution in $E$ over $\Omega'$ and
$\sum_{j=N}^{\infty} V_j(x,\cdot)\,R_\pm^{\Omega'}(2+2j,x)$ is a formal sum of 
continuous sections, $V_j(x,\cdot)\,R_\pm^{\Omega'}(2+2j,x)\in
C^0(\Omega',E_x^*\otimes E)$ for $j \geq N$.

Using suitable cut-offs we will now replace the infinite formal part of the
series by a convergent series.
Let $\sigma: \R \to \R$ be a smooth function vanishing outside $[-1,1]$,
such that $\sigma\equiv 1$ on $[-\tfrac12,\tfrac12]$ and $0\leq \sigma \leq 1$
everywhere.
We need the following elementary lemma.

\begin{lemma}\label{sigmadoof}
For every $l\in\N$ and every $\beta\geq l+1$ there exists a constant
$c(l,\beta)$ such that for all $0<\varepsilon \leq 1$ we have
$$
\left\|\frac{d^l}{dt^l}(\sigma(t/\varepsilon)t^{\beta})\right\|_{C^0(\R)}
\quad\leq\quad
\varepsilon\cdot c(l,\beta)\cdot \|\sigma\|_{C^l(\R)}.
$$
\end{lemma}

\begin{proof}
 \begin{eqnarray*}
\lefteqn{\left\|\frac{d^l}{dt^l}(\sigma(t/\varepsilon)t^{\beta})
  \right\|_{C^0(\R)}}\\  
&\leq&
\sum_{m=0}^l \begin{pmatrix}l \cr m \end{pmatrix}
\left\|\frac{1}{\varepsilon^m} \sigma^{(m)}(t/\varepsilon)
\cdot {\beta}({\beta}-1)\cdots
({\beta}-l+m+1)t^{{\beta}-l+m}\right\|_{C^0(\R)}\\ 
&=&
\sum_{m=0}^l \begin{pmatrix}l \cr m \end{pmatrix}
\cdot {\beta}({\beta}-1)\cdots ({\beta}-l+m+1)
\varepsilon^{{\beta}-l}
\left\|(t/\varepsilon)^{{\beta}-l+m}\sigma^{(m)}
(t/\varepsilon)\right\|_{C^0(\R)}.  
\end{eqnarray*}
Now $\sigma^{(m)}(t/\varepsilon)$ vanishes for $|t|/\varepsilon\geq 1$
and thus $\|(t/\varepsilon)^{{\beta}-m+l}\sigma^{(m)}
(t/\varepsilon)\|_{C^0(\R)} \leq \|\sigma^{(m)}\|_{C^0(\R)}$.
Moreover, ${\beta}-l \geq 1$, hence
$\varepsilon^{{\beta}-l}\leq \varepsilon$. 
Therefore
\begin{eqnarray*}
\left\|\frac{d^l}{dt^l}(\sigma(t/\varepsilon)t^{\beta})\right\|_{C^0(\R)}
&\leq&
\varepsilon
\sum_{m=0}^l \begin{pmatrix}l \cr m \end{pmatrix}
\cdot {\beta}({\beta}-1)\cdots ({\beta}-l+m+1)
\left\|\sigma^{(m)}\right\|_{C^0(\R)}\\ 
&\leq&
\varepsilon\, c(l,\beta)\, \|\sigma\|_{C^l(\R)}.
\end{eqnarray*}
\end{proof}

We define $\Gamma\in C^\infty(\Omega'\times\Omega',\R)$ by
$\Gamma(x,y):=\Gamma_x(y)$ 
\indexs{Ga*@$\Gamma(x,y)=\Gamma_x(y)$}
where $\Gamma_x$ is as in (\ref{eq:GammaDef}).
Note that $\Gamma(x,y)=0$ if and only if the geodesic joining $x$ and $y$ in
$\Omega'$ is lightlike.
In other words, $\Gamma^{-1}(0)=\bigcup_{x\in\Omega'}(C_+^{\Omega'}(x)\cup
C_-^{\Omega'}(x))$. 

\begin{lemma}\label{RtKonvergenz}
Let $\Omega \subset\subset \Omega'$ be a relatively compact open subset.
Then there exists a sequence of $\varepsilon_j \in (0,1]$, $j\geq N$, such that
for each $k\geq 0$ the series
\begin{eqnarray*}
(x,y) 
&\mapsto& 
\sum_{j=N+k}^{\infty}
\sigma(\Gamma(x,y)/\varepsilon_j)V_j(x,y)\,R_\pm^{\Omega'}(2+2j,x)(y)\\
&&=
\left\{\begin{array}{cl}
\sum_{j=N+k}^{\infty} C(2+2j,n)
\sigma(\Gamma(x,y)/\varepsilon_j)V_j(x,y)\,\Gamma(x,y)^{j+1-n/2} & 
\mbox{ if }y\in J_\pm^{\Omega'}(x)\\
0 & \mbox{ otherwise}
\end{array}\right.
\end{eqnarray*}
converges in $C^k(\overline\Omega\times\overline\Omega,E^*\boxtimes E)$.
In particular, the series 
$$
(x,y) \mapsto 
\sum_{j=N}^{\infty}
\sigma(\Gamma(x,y)/\varepsilon_j)V_j(x,y)\,R_\pm^{\Omega'}(2+2j,x)(y)
$$
defines a continuous section over
$\overline\Omega\times\overline\Omega$ and a smooth section over
$(\overline\Omega\times\overline\Omega) \setminus \Gamma^{-1}(0)$.
\end{lemma}

\begin{proof}
For $j\geq N \geq \tfrac n2$ the exponent in $\Gamma(x,y)^{j+1-n/2}$ is
positive.
Therefore the piecewise definition of the $j$-th summand yields a continuous
section over $\Omega'$.

The factor $\sigma(\Gamma(x,y)/\varepsilon_j)$ vanishes whenever
$\Gamma(x,y)\geq \varepsilon_j$.
Hence for $j\geq N\geq\tfrac n2$ and $0<\varepsilon_j\leq 1$ 
\begin{eqnarray*}
\lefteqn{\hspace{-3cm}\left\| (x,y) \mapsto
\sigma(\Gamma(x,y)/\varepsilon_j)V_j(x,y)\,R_\pm^{\Omega'}(2+2j,x)(y)
\right\|_{C^0(\overline\Omega\times\overline\Omega)}}\\
&\leq&
C(2+2j,n)\,\,\|V_j\|_{C^0(\overline\Omega\times\overline\Omega)}\,\,
\varepsilon_j^{j+1-n/2}\\
&\leq&
C(2+2j,n)\,\,\|V_j\|_{C^0(\overline\Omega\times\overline\Omega)}\,\,
\varepsilon_j .
\end{eqnarray*}
Hence if we choose $\varepsilon_j \in (0,1]$ such that 
$$
 C(2+2j,n)\,\,\|V_j\|_{C^0(\overline\Omega\times\overline\Omega)}\,\,
\varepsilon_j < 2^{-j},
$$
then the series converges in the $C^0$-norm and therefore defines a continuous 
section. 

For $k\geq 0$ and $j\geq N+k \geq \frac{n}{2} + k$ the function
$\Gamma^{j+1-\tfrac{n}{2}}$ vanishes to $(k+1)$-st order along 
$\Gamma^{-1}(0)$.
Thus the $j$-th summand in the series is of regularity $C^k$.
Writing $\sigma_j(t) := \sigma(t/\varepsilon_j)t^{j+1-n/2}$ we know from
Lemma~\ref{sigmadoof} that
$$
\|\sigma_j\|_{C^k(\R)} \leq \varepsilon_j \cdot c_1(k,j,n) \cdot
\|\sigma\|_{C^k(\R)}
$$
where here and henceforth $c_1, c_2, \ldots$ denote certain universal
positive constants whose precise values are of no importance.
Using Lemmas~\ref{CkNormvonProdukt} and \ref{CkNormvonVerkettung} we obtain
\begin{eqnarray*}
\lefteqn{\hspace{-1cm}\left\| (x,y) \mapsto
\sigma(\Gamma(x,y)/\varepsilon_j)V_j(x,y)\,R_\pm^{\Omega'}(2+2j,x)(y)
\right\|_{C^k(\overline\Omega\times\overline\Omega)}}\\
&\leq&
C(2+2j,n)\| (\sigma_j\circ\Gamma) \cdot V_j
\|_{C^k(\overline\Omega\times\overline\Omega)}\hspace{2cm}\\ 
&\leq&
c_2(k,j,n)\cdot \|
\sigma_j\circ\Gamma\|_{C^k(\overline\Omega\times\overline\Omega)} \cdot\| V_j
\|_{C^k(\overline\Omega\times\overline\Omega)}\\ 
&\leq&
c_3(k,j,n)\cdot \| \sigma_j\|_{C^k(\R)}\cdot
\max_{\ell=0,\ldots,k}\|\Gamma\|_{C^k(\overline\Omega
  \times\overline\Omega)}^\ell  
\cdot\| V_j \|_{C^k(\overline\Omega\times\overline\Omega)}\\
&\leq&
c_4(k,j,n)\cdot \varepsilon_j\cdot\| \sigma\|_{C^k(\R)}\cdot
\max_{\ell=0,\ldots,k}\|\Gamma\|_{C^k(\overline\Omega\times
  \overline\Omega)}^\ell  
\cdot\| V_j \|_{C^k(\overline\Omega\times\overline\Omega)} .
\end{eqnarray*}
Hence if we add the (finitely many) conditions on $\varepsilon_j$ that 
$$
c_4(k,j,n)\cdot \varepsilon_j 
\cdot\| V_j \|_{C^k(\overline\Omega\times\overline\Omega)}
\leq 2^{-j}
$$
for all $k\leq j-N$, then we have for fixed $k$
\begin{eqnarray*}
\lefteqn{\hspace{-2cm}\left\| (x,y) \mapsto
\sigma(\Gamma(x,y)/\varepsilon_j)V_j(x,y)\,R_\pm^{\Omega'}(2+2j,x)(y)
\right\|_{C^k(\overline\Omega\times\overline\Omega)}}\\ 
&\leq& 
2^{-j}\cdot\| \sigma\|_{C^k(\R)}\cdot
\max_{\ell=0,\ldots,k}\|\Gamma\|_{C^k(\overline\Omega
  \times\overline\Omega)}^\ell 
\end{eqnarray*}
for all $j\geq N+k$.
Thus the series
$$
(x,y) \mapsto 
\sum_{j=N+k}^{\infty}
\sigma(\Gamma(x,y)/\varepsilon_j)V_j(x,y)\,R_\pm^{\Omega'}(2+2j,x)(y)
$$
converges in $C^k(\overline\Omega\times\overline\Omega,E^*\boxtimes E)$.
All summands $\sigma(\Gamma(x,y)/\varepsilon_j)V_j(x,y)\,
R_\pm^{\Omega'}(2+2j,x)(y)$ are smooth on
$\overline\Omega\times\overline\Omega \setminus \Gamma^{-1}(0)$, thus
\begin{eqnarray*}
\lefteqn{(x,y) \mapsto \sum_{j=N}^{\infty}
\sigma(\Gamma(x,y)/\varepsilon_j)V_j(x,y)\,R_\pm^{\Omega'}(2+2j,x)(y)}\\
&=&
\sum_{j=N}^{N+k-1}
\sigma(\Gamma(x,y)/\varepsilon_j)V_j(x,y)\,R_\pm^{\Omega'}(2+2j,x)(y)\\
&&+
\sum_{j=N+k}^{\infty}
\sigma(\Gamma(x,y)/\varepsilon_j)V_j(x,y)\,R_\pm^{\Omega'}(2+2j,x)(y)
\end{eqnarray*}
is $C^k$ for all $k$, hence smooth on
$(\overline\Omega\times\overline\Omega) \setminus \Gamma^{-1}(0)$.
\end{proof}

Define distributions $\RRt_+(x)$ and $\RRt_-(x)$ by
$$
\RRt_\pm(x) := 
\sum_{j=0}^{N-1} V_j(x,\cdot)\,R_\pm^{\Omega'}(2+2j,x) 
+
\sum_{j=N}^{\infty} \sigma(\Gamma(x,\cdot)/\varepsilon_j)V_j(x,\cdot)\,
R_\pm^{\Omega'}(2+2j,x) .
$$
By Lemma~\ref{RtKonvergenz} and the properties of Riesz distributions we know
that 
\begin{equation}\label{suppRt}
\supp(\RRt_\pm(x)) \subset J_\pm^{\Omega'}(x),
\end{equation}
\begin{equation}\label{ssuppRt}
\ssupp(\RRt_\pm(x)) \subset C_\pm^{\Omega'}(x),
\end{equation}
and that $\ord(\RRt_\pm(x)) \leq n+1$.

\begin{lemma}\label{Rtistfastfundsol}
The $\varepsilon_j$ in Lemma~\ref{RtKonvergenz} can be chosen such that in
addition to the assertion in Lemma~\ref{RtKonvergenz} we have on $\Omega$
\begin{equation}
P_{(2)}\RRt_\pm(x) = \delta_x + K_\pm(x,\cdot)
\label{Rtffs}
\indexs{K*@$K_\pm(x,y)$, error term for approximate fundamental solution}
\end{equation}
with smooth $K_\pm\in
C^\infty(\overline\Omega\times\overline\Omega,E^*\boxtimes E)$.
\end{lemma}

\begin{proof}
From properties (\ref{Bostrich}) and (\ref{Beka}) of the Hadamard coefficients
we know 
\begin{equation}
P_{(2)}\left(\sum_{j=0}^{N-1} V_j(x,\cdot)\,R_\pm^{\Omega'}(2+2j,x)\right) 
=
\delta_x + (P_{(2)}V_{N-1}(x,\cdot))R_\pm^{\Omega'}(2N,x).
\label{Rteq1}
\end{equation}
Moreover, by Lemma~\ref{distributionlimesvertausch} we may interchange $P$
with the infinite sum and we get
\begin{eqnarray*}
\lefteqn{P_{(2)}\left(
\sum_{j=N}^{\infty} \sigma(\Gamma(x,\cdot)/\varepsilon_j)V_j(x,\cdot)\,
R_\pm^{\Omega'}(2+2j,x) 
\right)}\nonumber\\
&=&
\sum_{j=N}^{\infty} P_{(2)}\left(\sigma(\Gamma(x,\cdot)/
\varepsilon_j)V_j(x,\cdot)\, 
R_\pm^{\Omega'}(2+2j,x) 
\right)\nonumber\\
&=&
\sum_{j=N}^{\infty}\Big({\Box_{(2)}}(\sigma(\Gamma(x,\cdot)/\varepsilon_j))
V_j(x,\cdot)R_\pm^{\Omega'}(2+2j,x) 
-2\nabla^{(2)}_{{\grad_{(2)}} \sigma(\Gamma(x,\cdot)/\varepsilon_j)}
(V_j(x,\cdot)R_\pm^{\Omega'}(2+2j,x))\nonumber\\
&&\quad\quad
+ \sigma(\Gamma(x,\cdot)/\varepsilon_j) 
P_{(2)}(V_j(x,\cdot)R_\pm^{\Omega'}(2+2j,x))\Big)\nonumber\\
\end{eqnarray*}
Here and in the following $\Box_{(2)}$, $\grad_{(2)}$, and $\nabla^{(2)}$
indicate that the operators are applied with respect to the $y$-variable just
as for $P_{(2)}$.

Abbreviating $\Sigma_1 :=
\sum_{j=N}^{\infty}{\Box_{(2)}}(\sigma(\Gamma(x,\cdot)/\varepsilon_j))
V_j(x,\cdot)R_\pm^{\Omega'}(2+2j,x)$ and 
$\Sigma_2 := -2 \sum_{j=N}^{\infty}\nabla^{(2)}_{{\grad_{(2)}}
  \sigma(\Gamma(x,\cdot)/\varepsilon_j)}
(V_j(x,\cdot)R_\pm^{\Omega'}(2+2j,x))$
we have
\begin{eqnarray*}
\lefteqn{P_{(2)}\left(
\sum_{j=N}^{\infty} \sigma(\Gamma(x,\cdot)/\varepsilon_j)V_j(x,\cdot)\,
R_\pm^{\Omega'}(2+2j,x) 
\right)}\nonumber\\
&=&
\Sigma_1 + \Sigma_2 + 
\sum_{j=N}^{\infty}\sigma(\Gamma(x,\cdot)/\varepsilon_j) 
P_{(2)}(V_j(x,\cdot)R_\pm^{\Omega'}(2+2j,x))\\
&=&
\Sigma_1 + \Sigma_2 + 
\sum_{j=N}^{\infty}\sigma(\Gamma(x,\cdot)/\varepsilon_j)
\Big((P_{(2)}V_j(x,\cdot))R_\pm^{\Omega'}(2+2j,x)
-2\nabla^{(2)}_{{\grad_{(2)}} R_\pm^{\Omega'}(2+2j,x)}V_j(x,\cdot)\\
&&\quad\quad\quad\quad\quad\quad
+ V_j(x,\cdot){\Box_{(2)}} R_\pm^{\Omega'}(2+2j,x)\Big) .
\end{eqnarray*}
Properties (\ref{Bostrich}) and (\ref{Beka}) of the Hadamard coefficients
tell us
$$
V_j(x,\cdot){\Box_{(2)}} R_\pm^{\Omega'}(2+2j,x)-2\nabla^{(2)}_{{\grad_{(2)}}
  R_\pm^{\Omega'}(2+2j,x)}V_j(x,\cdot)=-P_{(2)}(V_{j-1}(x,\cdot)
R_\pm^{\Omega'}(2+2j,x))  
$$
and hence
\begin{eqnarray*}
\lefteqn{P_{(2)}\left(
\sum_{j=N}^{\infty} \sigma(\Gamma(x,\cdot)/\varepsilon_j)V_j(x,\cdot)\,
R_\pm^{\Omega'}(2+2j,x) 
\right)}\nonumber\\
&=&
\Sigma_1 + \Sigma_2 + 
\sum_{j=N}^{\infty}\sigma(\Gamma(x,\cdot)/\varepsilon_j)
\left((P_{(2)}V_j(x,\cdot))R_\pm^{\Omega'}(2+2j,x) -
  P_{(2)}V_{j-1}R_\pm^{\Omega'}(2j,x)\right)\\
&=&
\Sigma_1 + \Sigma_2 -
\sigma(\Gamma(x,\cdot)/\varepsilon_N)P_{(2)}V_{N-1}R_\pm^{\Omega'}(2N,x)\\ 
&&
+\sum_{j=N}^{\infty}\left(\sigma(\Gamma(x,\cdot)/\varepsilon_j) 
- \sigma(\Gamma(x,\cdot)/\varepsilon_{j+1})\right) 
(P_{(2)}V_j(x,\cdot))R_\pm^{\Omega'}(2+2j,x) .
\end{eqnarray*}
Putting $\Sigma_3 :=
\sum_{j=N}^{\infty}\left(\sigma(\Gamma(x,\cdot)/\varepsilon_j) -
  \sigma(\Gamma(x,\cdot)/\varepsilon_{j+1})\right)
(P_{(2)}V_j(x,\cdot))R_\pm^{\Omega'}(2+2j,x)$ and combining with (\ref{Rteq1})
yields 
\begin{equation}
P_{(2)}\RRt_\pm(x) - \delta_x
=
(1-\sigma(\Gamma(x,\cdot)/\varepsilon_{N-1}))
P_{(2)}V_{N-1}(x,\cdot)R_\pm^{\Omega'}(2N,x) +\Sigma_1 +\Sigma_2+\Sigma_3 .
\label{Rteq3}
\end{equation}
We have to show that the right hand side is actually smooth in both
variables.
Since 
$$
P_{(2)}V_{N-1}(x,y)R_\pm^{\Omega'}(2N,x)(y) =
\left\{
  \begin{array}{cl}
  C(2N,n)\,P_{(2)}V_{N-1}(x,y)\,\Gamma(x,y)^{N-n/2}, & 
  \mbox{ if }y\in J_\pm^{\Omega'}(x)\\
  0, & \mbox{ otherwise}
  \end{array}
\right.
$$
is smooth on $(\Omega'\times\Omega') \setminus \Gamma^{-1}(0)$ 
and since $1-\sigma(\Gamma(x,\cdot)/\varepsilon_j)$ vanishes on a neighborhood
of $\Gamma^{-1}(0)$ we have that 
$$
(x,y) \mapsto (1-\sigma(\Gamma(x,y)/\varepsilon_j))\cdot
P_{(2)}V_{N-1}(x,y)R_\pm^{\Omega'}(2N,x)(y)
$$ 
is smooth.
Similarly, the individual terms in the three infinite sums are smooth sections
because 
$\sigma(\Gamma/\varepsilon_{j})-\sigma(\Gamma/\varepsilon_{j+1})$,
${\grad_{(2)}}(\sigma\circ\tfrac{\Gamma}{\varepsilon_j})$, and
${\Box_{(2)}}(\sigma\circ\tfrac{\Gamma}{\varepsilon_j})$ all vanish on a
neighborhood 
of $\Gamma^{-1}(0)$. 
It remains to be shown that the three series in (\ref{Rteq3}) converge in all
$C^k$-norms. 

We start with $\Sigma_2$.
Let  $S_j := \{(x,y)\in\Omega'\times\Omega'\,|\,
\tfrac{\varepsilon_j}{2}\leq\Gamma(x,y) \leq\varepsilon_j\}$.

\begin{center}
\input{fig-Sj}
\end{center}

Since ${\grad_{(2)}}(\sigma\circ\tfrac{\Gamma}{\varepsilon_j})$ vanishes
outside the ``strip'' $S_j$, there exist constants $c_1(k,n)$, $c_2(k,n)$ and
$c_3(k,n,j)$ such that 
\begin{eqnarray*}
\lefteqn{\left\|
\nabla^{(2)}_{{\grad_{(2)}}(\sigma\circ\tfrac{\Gamma}{\varepsilon_j})}
\left(V_j(\cdot,\cdot)\,R_\pm^{\Omega'}(2+2j,\cdot)\right)
\right\|_{C^k(\overline\Omega\times\overline\Omega)}}\\
&=&
\left\|
\nabla^{(2)}_{{\grad_{(2)}}(\sigma\circ\tfrac{\Gamma}{\varepsilon_j})}
\left(V_j(\cdot,\cdot)\,R_\pm^{\Omega'}(2+2j,\cdot)\right)
\right\|_{C^k(\overline\Omega\times\overline\Omega\cap S_j)}\\
&\leq&
c_1(k,n)\cdot
\left\|\sigma\circ\tfrac{\Gamma}{\varepsilon_j}
\right\|_{C^{k+1}(\overline\Omega\times\overline\Omega\cap S_j)}
\cdot
\left\|
V_j(\cdot,\cdot)\,R_\pm^{\Omega'}(2+2j,\cdot)
\right\|_{C^{k+1}(\overline\Omega\times\overline\Omega\cap S_j)}\\
&\leq&
c_2(k,n)\cdot\|\sigma\|_{C^{k+1}(\R)}\cdot\max_{\ell=0,\dots,k+1}
\|\tfrac{\Gamma}{\varepsilon_j}
\|_{C^{k+1}(\overline\Omega\times\overline\Omega\cap S_j)}^\ell\\
&&
\cdot\|V_j\|_{C^{k+1}(\overline\Omega\times\overline\Omega\cap S_j)}
\cdot\|R_\pm^{\Omega'}(2+2j,\cdot)\|_{C^{k+1}(\overline\Omega\times
  \overline\Omega\cap S_j)} \\
&\leq&
c_2(k,n)\cdot
\frac{1}{\varepsilon_j^{k+1}}\cdot\|\sigma\|_{C^{k+1}(\R)}
\cdot\max_{\ell=0,\dots,k+1}
\|\Gamma\|_{C^{k+1}(\overline\Omega\times\overline\Omega)}^\ell\\
&&
\cdot\|V_j\|_{C^{k+1}(\overline\Omega\times\overline\Omega)}
\cdot\|R_\pm^{\Omega'}(2+2j,\cdot)\|_{C^{k+1}(\overline\Omega\times
  \overline\Omega\cap S_j)}\\
&\leq&
c_3(k,n,j)\cdot
\frac{1}{\varepsilon_j^{k+1}}\cdot\|\sigma\|_{C^{k+1}(\R)}
\cdot\max_{\ell=0,\dots,k+1}
\|\Gamma\|_{C^{k+1}(\overline\Omega\times\overline\Omega)}^\ell\\
&&
\cdot\|V_j\|_{C^{k+1}(\overline\Omega\times\overline\Omega)}
\cdot\|\Gamma^{1+j-n/2}\|_{C^{k+1}(\overline\Omega\times
  \overline\Omega\cap S_j)}.
\end{eqnarray*}
By Lemma~\ref{CkNormvonVerkettung} we have 
\begin{eqnarray*}
\lefteqn{\|\Gamma^{1+j-n/2}\|_{C^{k+1}(\overline\Omega\times\overline\Omega\cap
    S_j)}}\\ 
&\leq&
c_4(k)\cdot\|t\mapsto t^{1+j-n/2}\|_{C^{k+1}([\varepsilon_j/2,\varepsilon_j])}
\cdot \max_{\ell=0,\ldots,k+1}\|\Gamma\|_{C^{k+1}(\overline\Omega\times
  \overline\Omega\cap S_j)}^\ell\\
&\leq&
c_5(k,j,n)\cdot\varepsilon_j^{j-n/2-k}
\cdot \max_{\ell=0,\ldots,k+1}\|\Gamma\|_{C^{k+1}(\overline\Omega\times
  \overline\Omega\cap S_j)}^\ell .
\end{eqnarray*}
Thus
\begin{eqnarray*}
\lefteqn{\left\|
\nabla^{(2)}_{{\grad_{(2)}}(\sigma\circ\tfrac{\Gamma}{\varepsilon_j})}
\left(V_j(\cdot,\cdot)\,R_\pm^{\Omega'}(2+2j,\cdot)\right)
\right\|_{C^k(\overline\Omega\times\overline\Omega)}}\\
&\leq&
c_6(k,j,n)\cdot\|\sigma\|_{C^{k+1}(\R)}\cdot\left(\max_{\ell=0,\ldots,k+1}
\|\Gamma\|_{C^{k+1}(\overline\Omega\times\overline\Omega)}^\ell\right)^2
\cdot\|V_j\|_{C^{k+1}(\overline\Omega\times\overline\Omega)}
\cdot \varepsilon_j^{j-2k-n/2-1}\\
&\leq&
c_6(k,j,n)\cdot\|\sigma\|_{C^{k+1}(\R)}\cdot\max_{\ell=0,\ldots,k+1}
\|\Gamma\|_{C^{k+1}(\overline\Omega\times\overline\Omega)}^{2\ell}
\cdot\|V_j\|_{C^{k+1}(\overline\Omega\times\overline\Omega)}
\cdot \varepsilon_j
\end{eqnarray*}
if $j\geq 2k+n/2+2$.
Hence if we require the (finitely many) conditions
$$
c_6(k,j,n)\cdot\|V_j\|_{C^{k+1}(\overline\Omega\times\overline\Omega)}
\cdot \varepsilon_j
\leq 2^{-j}
$$
on $\varepsilon_j$ for all $k\leq j/2-n/4-1$, then almost all $j$-th terms of
the series $\Sigma_2$ are bounded in the $C^k$-norm by
$2^{-j}\cdot\|\sigma\|_{C^{k+1}(\R)}\cdot\max_{\ell=0,\ldots,k+1}\|\Gamma\|_{C^{k+1} 
  (\overline\Omega\times\overline\Omega)}^{2\ell}$.    
Thus $\Sigma_2$ converges in the $C^k$-norm for any $k$ and defines a smooth
section in $E^*\boxtimes E$ over $\overline\Omega\times\overline\Omega$.

The series $\Sigma_1$ is treated similarly.
To examine $\Sigma_3$ we observe that for $j \geq k+\tfrac{n}{2}$
\begin{eqnarray}
\lefteqn{\left\|\left((\sigma\circ\tfrac{\Gamma}{\varepsilon_j}) -
  (\sigma\circ\tfrac{\Gamma}{\varepsilon_{j+1}})\right)\cdot
(P_{(2)}V_j)\cdot R_\pm^{\Omega'}(2+2j,\cdot)
  \right\|_{C^k(\overline\Omega\times\overline\Omega)}}\nonumber\\ 
&\leq&
c_7(j,n)\cdot \left\|\left((\sigma\circ\tfrac{\Gamma}{\varepsilon_j}) -
  (\sigma\circ\tfrac{\Gamma}{\varepsilon_{j+1}})\right)\cdot
(P_{(2)}V_j)\cdot\Gamma^{1+j-n/2}
\right\|_{C^k(\overline\Omega\times\overline\Omega)}\nonumber\\
&\leq&
c_8(k,j,n)\cdot
\left\|\left((\sigma\circ\tfrac{\Gamma}{\varepsilon_j}) -
  (\sigma\circ\tfrac{\Gamma}{\varepsilon_{j+1}})\right)\cdot\Gamma^{k+1} 
\right\|_{C^k(\overline\Omega\times\overline\Omega)}\nonumber\\
&&
\cdot \left\|P_{(2)}V_j
\right\|_{C^k(\overline\Omega\times\overline\Omega)}
\cdot \left\|\Gamma^{j-k-n/2}
\right\|_{C^k(\overline\Omega\times\overline\Omega)}\nonumber\\
&\leq&
c_8(k,j,n)\cdot
\left(\left\|(\sigma\circ\tfrac{\Gamma}{\varepsilon_j})\cdot\Gamma^{k+1}
\right\|_{C^k(\overline\Omega\times\overline\Omega)}  
+ \left\|(\sigma\circ\tfrac{\Gamma}{\varepsilon_{j+1}})\cdot\Gamma^{k+1}
\right\|_{C^k(\overline\Omega\times\overline\Omega)}\right)\nonumber\\
&&
\cdot \left\|P_{(2)}V_j
\right\|_{C^k(\overline\Omega\times\overline\Omega)}
\cdot \left\|\Gamma^{j-k-n/2}
\right\|_{C^k(\overline\Omega\times\overline\Omega)} .
\label{Rteq4}
\end{eqnarray}
Putting $\sigma_j(t) := \sigma(t/\varepsilon_j)\cdot t^{k+1}$ we have
$(\sigma\circ\tfrac{\Gamma}{\varepsilon_j})\cdot\Gamma^{k+1} = \sigma_j \circ
\Gamma$. 
Hence by Lemmas~\ref{CkNormvonVerkettung} and \ref{sigmadoof}
\begin{eqnarray*}
\left\|(\sigma\circ\tfrac{\Gamma}{\varepsilon_j})\cdot\Gamma^{k+1}
\right\|_{C^k(\overline\Omega\times\overline\Omega)} 
&=&
\left\|\sigma_j \circ \Gamma
\right\|_{C^k(\overline\Omega\times\overline\Omega)}\\
&\leq&
c_9(k,n)\cdot \left\|\sigma_j\right\|_{C^k(\R)}
\cdot \max_{\ell=0,\ldots,k}
\left\|\Gamma\right\|_{C^k(\overline\Omega\times\overline\Omega)}^\ell\\
&\leq&
c_{10}(k,n)\cdot\varepsilon_j\cdot\left\|\sigma\right\|_{C^k(\R)}
\cdot \max_{\ell=0,\ldots,k}
\left\|\Gamma\right\|_{C^k(\overline\Omega\times\overline\Omega)}^\ell .
\end{eqnarray*}
Plugging this into (\ref{Rteq4}) yields
\begin{eqnarray*}
\lefteqn{\left\|\left((\sigma\circ\tfrac{\Gamma}{\varepsilon_j}) -
  (\sigma\circ\tfrac{\Gamma}{\varepsilon_{j+1}})\right)\cdot
(P_{(2)}V_j)\cdot R_\pm^{\Omega'}(2+2j,\cdot)
  \right\|_{C^k(\overline\Omega\times\overline\Omega)}}\\ 
&\leq&
c_{11}(k,j,n)\cdot(\varepsilon_j + \varepsilon_{j+1})
\cdot\left\|\sigma\right\|_{C^k(\R)}
\cdot \max_{\ell=0,\ldots,k}
\left\|\Gamma\right\|_{C^k(\overline\Omega\times\overline\Omega)}^\ell\\
&&
\cdot \left\|P_{(2)}V_j
\right\|_{C^k(\overline\Omega\times\overline\Omega)}
\cdot \left\|\Gamma^{j-k-n/2}
\right\|_{C^k(\overline\Omega\times\overline\Omega)} .
\end{eqnarray*}
Hence if we add the conditions on $\varepsilon_j$ that
$$
c_{11}(k,j,n)\cdot\varepsilon_j
\cdot \left\|P_{(2)}V_j
\right\|_{C^k(\overline\Omega\times\overline\Omega)}
\cdot \left\|\Gamma^{j-k-n/2}
\right\|_{C^k(\overline\Omega\times\overline\Omega)}
\leq 2^{-j-1}
$$
for all $k\leq j-\tfrac{n}{2}$ and
$$
c_{11}(k,j-1,n)\cdot\varepsilon_j
\cdot \left\|P_{(2)}V_{j-1}
\right\|_{C^k(\overline\Omega\times\overline\Omega)}
\cdot \left\|\Gamma^{j-1-k-n/2}
\right\|_{C^k(\overline\Omega\times\overline\Omega)}
\leq 2^{-j-2}
$$
for all $k\leq j-1-\tfrac{n}{2}$, then we have that almost all $j$-th terms in 
$\Sigma_3$ are bounded in the $C^k$-norm by
$2^{-j}\cdot\left\|\sigma\right\|_{C^k(\R)} \cdot \max_{\ell=0,\ldots,k}
\left\|\Gamma\right\|_{C^k(\overline\Omega\times\overline\Omega)}^\ell$.
Thus $\Sigma_3$ defines a smooth section as well.
\end{proof}

\begin{lemma}\label{Rtglattinx}
The $\varepsilon_j$ in Lemmas~\ref{RtKonvergenz} and \ref{Rtistfastfundsol}
can be chosen such that in addition there is a constant $C>0$ so that
$$
|\RRt_\pm(x)[\varphi]| \leq C\cdot \|\varphi\|_{C^{n+1}(\Omega)}
$$
for all $x\in\overline\Omega$ and all $\varphi\in\DD(\Omega,E^*)$.
In particular, $\RRt(x)$ is of order at most $n+1$.
Moreover, the map $x \mapsto \RRt_\pm(x)[\varphi]$ is for every fixed
$\varphi\in\DD(\Omega,E^*)$ a smooth section in $E^*$,
$$
\RRt_\pm(\cdot)[\varphi] \in C^\infty(\overline\Omega,E^*).
$$
\end{lemma}

We know already that for each $x\in\overline\Omega$ the distribution $\RRt(x)$
is of order at most $n+1$.
The point of the lemma is that the constant $C$ in the estimate
$|\RRt_\pm(x)[\varphi]| \leq C\cdot \|\varphi\|_{C^{n+1}(\Omega)}$ can be
chosen independently of $x$.

\begin{proof}
Recall the definition of $\RRt_\pm(x)$,
$$
\RRt_\pm(x) = 
\sum_{j=0}^{N-1} V_j(x,\cdot)\,R_\pm^{\Omega'}(2+2j,x) 
+
\sum_{j=N}^{\infty} \sigma(\Gamma(x,\cdot)/\varepsilon_j)V_j(x,\cdot)\,
R_\pm^{\Omega'}(2+2j,x) .
$$
By Proposition~\ref{OmegaRiesz}~(\ref{orderR}) there are constants $C_j>0$ such
that 
$|R^{\Omega'}_\pm(2+2j,x)[\varphi]| \leq C_j \cdot
\|\varphi\|_{C^{n+1}(\Omega)}$ for all $\varphi$ and all
$x\in\overline\Omega$.
Thus there is a constant $C'>0$ such that 
$$
\left|\sum_{j=0}^{N-1} V_j(x,\cdot)\,R_\pm^{\Omega'}(2+2j,x) [\varphi]\right| 
\leq C' \cdot \|\varphi\|_{C^{n+1}(\Omega)}
$$ 
for all $\varphi$ and all $x\in\overline\Omega$.
The remainder term $\sum_{j=N}^{\infty}
\sigma(\Gamma(x,y)/\varepsilon_j)V_j(x,y)\, R_\pm^{\Omega'}(2+2j,x)(y) =:
f(x,y)$ is a continuous section, hence
\begin{eqnarray*}
\left|\sum_{j=N}^{\infty}
\sigma(\Gamma(x,\cdot)/\varepsilon_j)V_j(x,\cdot)\, 
R_\pm^{\Omega'}(2+2j,x)[\varphi]\right| 
&\leq& 
\|f\|_{C^0(\overline\Omega\times\overline\Omega)}\cdot
\vol(\overline\Omega)\cdot\|\varphi\|_{C^{0}(\Omega)}\\
&\leq& 
\|f\|_{C^0(\overline\Omega\times\overline\Omega)}\cdot
\vol(\overline\Omega)\cdot\|\varphi\|_{C^{n+1}(\Omega)}
\end{eqnarray*}
for all $\varphi$ and all $x\in\overline\Omega$.
Therefore $C:= C' + \|f\|_{C^0(\overline\Omega\times\overline\Omega)}\cdot
\vol(\overline\Omega)$ does the job.

To see smoothness in $x$ we fix $k\geq 0$ and we write
\begin{eqnarray*}
\RRt_\pm(x)[\varphi] 
&=& 
\sum_{j=0}^{N-1} V_j(x,\cdot)\,R_\pm^{\Omega'}(2+2j,x)[\varphi] 
+
\sum_{j=N}^{N+k-1} \sigma(\Gamma(x,\cdot)/\varepsilon_j)V_j(x,\cdot)\,
R_\pm^{\Omega'}(2+2j,x)[\varphi]\\
&&+
\sum_{j=N+k}^{\infty} \sigma(\Gamma(x,\cdot)/\varepsilon_j)V_j(x,\cdot)\,
R_\pm^{\Omega'}(2+2j,x)[\varphi] .
\end{eqnarray*}
By Proposition~\ref{OmegaRiesz}~(\ref{Rglattinx}) the summands 
$V_j(x,\cdot)\,R_\pm^{\Omega'}(2+2j,x)[\varphi]$ and $\sigma(\Gamma(x,\cdot)/
\varepsilon_j)V_j(x,\cdot)\, R_\pm^{\Omega'}(2+2j,x)[\varphi]$ depend smoothly
on $x$.
By Lemma~\ref{RtKonvergenz} the remainder $\sum_{j=N+k}^{\infty}
\sigma(\Gamma(x,\cdot)/\varepsilon_j)V_j(x,\cdot)\, 
R_\pm^{\Omega'}(2+2j,x)[\varphi]$ is $C^k$.
Thus $x\mapsto \RRt_\pm(x)[\varphi]$ is $C^k$ for every $k$, hence smooth.
\end{proof}

 \Definition{
 If $M$ is a timeoriented Lorentzian manifold, then we call
 a subset $S \subset M\times M$ \defem{future-stretched} with respect to $M$ if
 $y\in J_+^M(x)$ whenever $(x,y)\in S$.
\indexn{future-stretched subset>defemidx}
 We call it \defem{strictly future-stretched} with respect to $M$ if $y\in
 I_+^M(x)$ 
\indexn{strictly future-stretched subset>defemidx}
 whenever $(x,y)\in S$.
 Analogously, we define \defem{past-stretched} and \defem{strictly
 past-stretched} 
 subsets.
\indexn{past-stretched subset>defemidx}
\indexn{strictly past-stretched subset>defemidx}
 }

We summarize the results obtained so far.

\begin{prop}\label{propRt}
Let $M$ be an $n$-dimensional timeoriented Lorentzian manifold and let $P$ be
a normally hyperbolic operator acting on sections in a vector bundle $E$ over
$M$.
Let $\Omega'\subset M$ be a convex open subset.
Fix an integer $N\geq \tfrac{n}{2}$ and fix a smooth function $\sigma:\R\to\R$
satisfying $\sigma\equiv 1$ outside $[-1,1]$, $\sigma \equiv 0$ on 
$[-\tfrac12,\tfrac12]$, and $0\leq\sigma\leq1$ everywhere.

Then for every relatively compact open subset $\Omega \subset\subset \Omega'$
there exists a sequence $\varepsilon_j > 0$, $j\geq N$, such that for every 
$x\in\overline\Omega$
$$
\RRt_\pm(x) = 
\sum_{j=0}^{N-1} V_j(x,\cdot)\,R_\pm^{\Omega'}(2+2j,x) 
+
\sum_{j=N}^{\infty} \sigma(\Gamma(x,\cdot)/\varepsilon_j)V_j(x,\cdot)\,
R_\pm^{\Omega'}(2+2j,x) 
$$
defines a distribution on $\Omega$ satisfying
\begin{enumerate}
\item 
$\supp(\RRt_\pm(x)) \subset J_\pm^{\Omega'}(x)$,
\item 
$\ssupp(\RRt_\pm(x)) \subset C_\pm^{\Omega'}(x)$,
\item
$P_{(2)}\RRt_\pm(x) = \delta_x + K_\pm(x,\cdot)$
with smooth $K_\pm\in
C^\infty(\overline\Omega\times\overline\Omega,E^*\boxtimes E)$,
\item \label{Rt:support}
$\supp(K_+)$ is future-stretched and $\supp(K_-)$ is
past-stretched with respect to $\Omega'$,
\item
$\RRt_\pm(x)[\varphi]$ depends smoothly on $x$ for every fixed
$\varphi\in\DD(\Omega,E^*)$,
\item
there is a constant $C>0$ such that $|\RRt_\pm(x)[\varphi]| \leq C\cdot
\|\varphi\|_{C^{n+1}(\Omega)}$ for all $x\in\overline\Omega$ and all 
$\varphi\in\DD(\Omega,E^*)$.
\end{enumerate}
\end{prop}

\begin{proof}
The only thing that remains to be shown is the statement (\ref{Rt:support}).
Recall from (\ref{Rteq3}) that in the notation of the proof of 
Lemma~\ref{Rtistfastfundsol}
$$
K_\pm(x,y) = (1-\sigma(\Gamma(x,y)/\varepsilon_{N-1}))\cdot
P_{(2)}V_{N-1}(x,y)\cdot R_\pm^{\Omega'}(2N,x)(y) +\Sigma_1 +\Sigma_2+\Sigma_3
. 
$$
The first term as well as all summands in the three infinite series 
$\Sigma_1$, $\Sigma_2$, and $\Sigma_3$ contain a factor
$R_\pm^{\Omega'}(2j,x)(y)$ for some $j\geq N$.
Hence if $K_+(x,y)\not=0$, then $y\in \supp(R_\pm^{\Omega'}(2j,x))\subset
J_+^{\Omega'}(x)$. 
In other words, $\{(x,y)\in\Omega\times\Omega\, |\, K_+(x,y) \not=0\}$ is
future-stretched with respect to $\Omega'$.
Since $\Omega'$ is geodesically convex causal futures are closed.
Hence $\supp(K_+) = \overline{\{(x,y)\in\Omega\times\Omega\, |\, K_+(x,y)
  \not=0\}}$ is future-stretched with respect to $\Omega'$ as well.
In the same way one sees that $\supp(K_-)$ is past-stretched.
\end{proof}

\Definition{
If the $\varepsilon_j$ are chosen as in Proposition~\ref{propRt}, then we call
$$
\RRt_\pm(x) = 
\sum_{j=0}^{N-1} V_j(x,\cdot)\,R_\pm^{\Omega'}(2+2j,x) 
+
\sum_{j=N}^{\infty} \sigma(\Gamma(x,\cdot)/\varepsilon_j)V_j(x,\cdot)\,
R_\pm^{\Omega'}(2+2j,x) 
\indexs{R*@$\protect\RRt_+(x)$, approximate advanced fundamental solution}
\indexs{R*@$\protect\RRt_-(x)$, approximate retarded fundamental solution}
\indexn{approximate fundamental solution>defemidx}
$$
an \defem{approximate advanced} or \defem{retarded fundamental solution}
respectively. 
}

From now on we assume that $\Omega\subset\subset\Omega'$ is a 
relatively compact {\em causal} subset.
Then for every $x\in\overline\Omega$ we have $J_\pm^{\overline\Omega}(x)
= J_\pm^{\Omega'}(x) \cap \overline\Omega$.
We fix approximate fundamental solutions $\RRt_\pm(x)$.

We use the corresponding $K_\pm$ as an integral kernel to define an integral
operator. 
Set for $u \in C^0(\overline\Omega,E^*)$ and $x \in \overline\Omega$
\begin{eqnarray}
(\KK_{\pm} u)(x) := \int_{\overline\Omega} K_{\pm}(x,y)u(y) {\dV}(y).
\label{Kdef}
\indexs{K*@$\protect\KK_{\pm}$, integral operator with kernel $K_{\pm}$}
\end{eqnarray}
Since $K_{\pm}$ is $C^{\infty}$ so is
$\KK_{\pm} u$, i.~e., $\KK_{\pm} u \in C^{\infty}(\overline\Omega, E^*)$.
By the properties of the support of $K_\pm$ the integrand
$K_{\pm}(x,y)u(y)$ vanishes unless $y \in J^{\overline\Omega}_\pm(x) \cap
\supp(u)$.
Hence $(\KK_{\pm} u)(x)=0$ if $J^{\overline\Omega}_\pm(x) \cap
\supp(u)=\emptyset$. 
In other words,
\begin{equation}
\supp(\KK_{\pm} u) \subset J_\mp^{\overline\Omega}(\supp(u)) .
\label{Ktraeger}
\end{equation}

If we put $C_k := \int_{\overline\Omega}\|K_\pm(\cdot,y)
\|_{C^k(\overline\Omega)}{\dV}(y)$, then 
$$
\|\KK_{\pm} u\|_{C^{k}(\overline\Omega)} \leq
C_k \cdot \|u\|_{C^{0}(\overline\Omega)}.
$$
Hence (\ref{Kdef}) defines a bounded linear map
$$
\KK_{\pm} : \quad C^0(\overline\Omega,E^*) \to 
C^{k}(\overline\Omega,E^*)
$$
for all $k\geq 0$.

\begin{lemma}\label{Klemma}
Let $\Omega\subset\subset\Omega'$ be causal.
Suppose $\overline\Omega$ is so small that 
\begin{equation}
\vol(\overline\Omega)\cdot
\|K_\pm\|_{C^0(\overline\Omega\times\overline\Omega)}<1.
\label{Omegaklein}
\end{equation}
Then 
$$
\id + \KK_{\pm} : \quad C^k(\overline\Omega,E^*) \to 
C^{k}(\overline\Omega,E^*)
$$
is an isomorphism with bounded inverse for all $k=0,1,2,\ldots$.
The inverse is given by the series
$$
(\id + \KK_{\pm})^{-1} = \sum_{j=0}^\infty (-\KK_{\pm})^j
$$
which converges in all $C^k$-operator norms.
The operator $(\id + \KK_{+})^{-1}\circ \KK_{+}$ has a smooth integral 
kernel with future-stretched support (with respect to $\overline\Omega$). 
The operator $(\id + \KK_{-})^{-1}\circ \KK_{-}$ has a smooth integral 
kernel with past-stretched support (with respect to $\overline\Omega$). 
\end{lemma}

\begin{proof}
The operator $\KK_{\pm}$ is bounded as an operator
$C^0(\overline\Omega,E^*) \to C^{k}(\overline\Omega,E^*)$.
Thus $\id + \KK_{\pm}$ defines a bounded operator $C^k(\overline\Omega,E^*)
\to C^{k}(\overline\Omega,E^*)$ for all $k$.
Now 
\begin{eqnarray*}
\|\KK_{\pm} u\|_{C^{0}(\overline\Omega)} 
&\leq&
\vol(\overline\Omega)\cdot \|K_\pm\|_{C^{0}
(\overline\Omega\times\overline\Omega)}\cdot \|u\|_{C^{0}(\overline\Omega)}\\
&=&
(1-\eta)\cdot\|u\|_{C^{0}(\overline\Omega)}
\end{eqnarray*}
where $\eta := 1- \vol(\overline\Omega)\cdot \|K_\pm\|_{C^{0}
(\overline\Omega\times\overline\Omega)} > 0$.
Hence the $C^0$-operator norm of $\KK_{\pm}$ is less than $1$ so that
the Neumann series $\sum_{j=0}^\infty (-\KK_{\pm})^j$ converges in the
$C^0$-operator norm and gives the inverse of $\id + \KK_{\pm}$ on
$C^0(\overline\Omega,E^*)$.

Next we replace the $C^k$-norm $\|\cdot\|_{C^k(\overline\Omega)}$ on 
$C^k(\overline\Omega,E^*)$ as defined in (\ref{defCkNorm}) by the
equivalent norm
$$
\interleave u\interleave_{C^k(\overline\Omega)} :=
\|u\|_{C^0(\overline\Omega)} + \frac{\eta}{2\vol(\overline\Omega)
\|K_\pm\|_{C^k(\overline\Omega\times\overline\Omega)}+1}
\|u\|_{C^k(\overline\Omega)}.
$$
Then 
\begin{eqnarray*}
\lefteqn{\interleave\KK_{\pm} u\interleave_{C^{k}(\overline\Omega)}}\\
&=&
\|\KK_{\pm}u\|_{C^{0}(\overline\Omega)} +
\frac{\eta}{2\vol(\overline\Omega)
\|K_\pm\|_{C^k(\overline\Omega\times\overline\Omega)}+1}
\|\KK_{\pm} u\|_{C^{k}(\overline\Omega)}\\
&\leq&
(1-\eta)\cdot\|u\|_{C^{0}(\overline\Omega)} +
\frac{\eta}{2\vol(\overline\Omega)
\|K_\pm\|_{C^k(\overline\Omega\times\overline\Omega)}+1}
\vol(\overline\Omega)\|K_\pm\|_{C^k(\overline\Omega\times\overline\Omega)}
\|u\|_{C^{0}(\overline\Omega)}\\
&\leq&
(1-\frac\eta2) \|u\|_{C^{0}(\overline\Omega)}\\
&\leq&
(1-\frac\eta2)\interleave u\interleave_{C^{k}(\overline\Omega)}  .
\end{eqnarray*}
This shows that with respect to
$\interleave\cdot\interleave_{C^k(\overline\Omega)}$
the $C^k$-operator norm of $\KK_{\pm}$ is less than $1$.
Thus the Neumann series $\sum_{j=0}^\infty (-\KK_{\pm})^j$ converges in all
$C^k$-operator norms and $\id+\KK_\pm$ is an isomorphism with bounded inverse
on all $C^k(\overline\Omega,E^*)$.

For $j\geq 1$ the integral kernel of $(\KK_{\pm})^j$ is given by
$$
K_\pm^{(j)}(x,y) :=
\int_{\overline\Omega} \cdots \int_{\overline\Omega}
K_\pm(x,z_1)K_\pm(z_1,z_2) \cdots K_\pm(z_{j-1},y) \dV(z_1) \cdots
\dV(z_{j-1}). 
$$
Thus $\supp(K_\pm^{(j)})\subset \left\{ (x,y)\in \overline\Omega\times
\overline\Omega\, |\, y\in J_\pm^{\overline\Omega}(x) \right\}$ and 
$$
\|K_\pm^{(j)}\|_{C^k(\overline\Omega\times\overline\Omega)} 
\leq
\|K_\pm\|_{C^k(\overline\Omega\times\overline\Omega)}^2 
\cdot\vol(\overline\Omega)^{j-1}\cdot 
\|K_\pm\|_{C^0(\overline\Omega\times\overline\Omega)}^{j-2} \leq
\delta^{j-2}\cdot\vol(\overline\Omega) \cdot
\|K_\pm\|_{C^k(\overline\Omega\times\overline\Omega)}^2
$$
where $\delta := \vol(\overline\Omega)\cdot 
\|K_\pm\|_{C^0(\overline\Omega\times\overline\Omega)} < 1$.
Hence the series 
$$
\sum_{j=1}^\infty (-1)^{j-1} K_\pm^{(j)}
$$
converges in all $C^k(\overline\Omega\times\overline\Omega,E^*\boxtimes E)$.
Since this series yields the integral kernel of $(\id + \KK_{\pm})^{-1}\circ
\KK_{\pm}$ it is smooth and its support is contained in
$\left\{ (x,y)\in \overline\Omega\times \overline\Omega\, |\, y\in
  J_\pm^{\overline\Omega}(x) \right\}$.
\end{proof}

\begin{cor}
Let $\Omega\subset\subset\Omega'$ be as in Lemma~\ref{Klemma}.
Then for each $u\in C^0(\overline\Omega,E)$
$$
\supp((\id+\KK_{\pm})^{-1}u) \subset J_\mp^{\overline\Omega}(\supp(u)) .
$$
\end{cor}

\begin{proof}
We observe that
$$
(\id+\KK_{\pm})^{-1}u = u - (\id+\KK_{\pm})^{-1}\KK_{\pm}u.
$$
Now $\supp(u) \subset J_\mp^{\overline\Omega}(\supp(u))$
and $\supp((\id+\KK_{\pm})^{-1}\KK_{\pm}u)\subset
J_\mp^{\overline\Omega}(\supp(u))$ by the properties of the integral kernel of
$(\id+\KK_{\pm})^{-1}\KK_{\pm}$. 
\end{proof}

Fix $\varphi\in \DD(\Omega,E^*)$.
Then $x \mapsto \RRt_\pm(x)[\varphi]$ defines a smooth section in $E^*$ over
$\overline\Omega$ with support contained in $J_\mp^{\Omega'}(\supp(\varphi))
\cap \overline\Omega = J_\mp^{\overline\Omega}(\supp(\varphi))$.
Hence 
\begin{equation}
F_\pm^{\Omega}(\cdot)[\varphi] := 
(\id + \KK_{\pm})^{-1}(\RRt_\pm(\cdot)[\varphi])
\label{Fdef}
\indexs{F*@$F_\pm^{\Omega}(\cdot)$, fundamental solution for domain
  $\Omega$ }
\end{equation}
defines a smooth section in $E^*$ with
\begin{equation}
\supp(F_\pm^{\Omega}(\cdot)[\varphi])
\subset
J_\mp^{\overline\Omega}(\supp(\RRt_\pm(\cdot)[\varphi]))
\subset
J_\mp^{\overline\Omega}(J_\mp^{\overline\Omega}(\supp(\varphi)))
=
J_\mp^{\overline\Omega}(\supp(\varphi)) .
\label{suppF}
\end{equation}

\begin{lemma}
For each $x\in\Omega$ the map $\DD(\Omega,E^*) \mapsto E^*_x$, $\varphi
\mapsto F_+^{\Omega}(x)[\varphi]$, is an advanced fundamental
solution at $x$ on $\Omega$ and $\varphi
\mapsto F_-^{\Omega}(x)[\varphi]$ is a retarded fundamental
solution at $x$ on $\Omega$.
\indexn{advanced fundamental solution}
\indexn{retarded fundamental solution}
\end{lemma}

\begin{proof}
We first check that $\varphi \mapsto F_\pm^{\Omega}(x)[\varphi]$ defines
a distribution for any fixed $x\in \Omega$.
Let $\varphi_m \to \varphi$ in $\DD(\Omega,E^*)$.
Then $\varphi_m \to \varphi$ in $C^{n+1}(\Omega,E^*)$ and by 
the last point of Proposition~\ref{propRt} $\RRt_\pm(\cdot)[\varphi_m] \to
\RRt_\pm(\cdot)[\varphi]$ in $C^{0}(\overline\Omega,E^*)$.
Since $(\id + \KK_{\pm})^{-1}$ is bounded on $C^0$ we have
$F_\pm^{\Omega}(\cdot)[\varphi_m] \to F_\pm^{\Omega}(\cdot)[\varphi]$ in 
$C^0$.
In particular, $F_\pm^{\Omega}(x)[\varphi_m] \to
F_\pm^{\Omega}(x)[\varphi]$.

Next we check that $F_\pm^{\Omega}(x)$ are fundamental solutions.
We compute
\begin{eqnarray*}
P_{(2)}F_\pm^{\Omega}(\cdot)[\varphi] 
&=&
F_\pm^{\Omega}(\cdot)[P^*\varphi]\\
&=&
(\id + \KK_{\pm})^{-1}(\RRt_\pm(\cdot)[P^*\varphi])\\
&=&
(\id + \KK_{\pm})^{-1}(P_{(2)}\RRt_\pm(\cdot)[\varphi])\\
&\stackrel{(\ref{Rtffs})}{=}&
(\id + \KK_{\pm})^{-1}(\varphi + \KK_{\pm}\varphi)\\
&=&
\varphi .
\end{eqnarray*}
Thus for fixed $x\in\Omega$,
$$
PF_\pm^{\Omega}(x)[\varphi] = \varphi(x) = \delta_x[\varphi].
$$

Finally, to see that $\supp(F_\pm^{\Omega}(x)) \subset J^\Omega_\pm(x)$ let
$\varphi\in\DD(\Omega,E^*)$ such that $\supp(\varphi) \cap J_\pm^\Omega(x) =
\emptyset$.
Then $x\not\in J_\mp^\Omega(\supp(\varphi))$ and thus
$F_\pm^{\Omega}(x)[\varphi] = 0$ by (\ref{suppF}).
\end{proof}

We summarize the results of this section.

\begin{prop}\label{fundamentalexist}
\indexn{fundamental solution}
Let $M$ be a timeoriented Lorentzian manifold.
Let $P$ be a normally hyperbolic operator acting on sections in a
vector bundle $E$ over $M$.
Let $\Omega\subset\subset M$ be a relatively compact causal domain.
Suppose that $\Omega$ is sufficiently small in the sense that 
(\ref{Omegaklein}) holds.

Then for each $x\in\Omega$
\begin{enumerate}
\item
the distributions $F_+^{\Omega}(x)$ and $F_-^{\Omega}(x)$ defined in
(\ref{Fdef}) are fundamental solutions for $P$ at $x$ over $\Omega$,
\item
$\supp(F_\pm^{\Omega}(x))\subset J_\pm^\Omega(x)$,
\item
for each $\varphi\in\DD(\Omega,E^*)$ the maps $x'\mapsto
F_\pm^{\Omega}(x')[\varphi]$ are smooth sections in $E^*$ over $\Omega.
\hfill\Box$
\end{enumerate}
\end{prop}

\begin{cor}\label{fundlocalexist}
Let $M$ be a timeoriented Lorentzian manifold.
Let $P$ be a normally hyperbolic operator acting on sections in a 
vector bundle $E$ over $M$.

Then each point in $M$ possesses an arbitrarily small causal
neighborhood $\Omega$ such that for each $x\in\Omega$ there exist fundamental
solutions $F_\pm^\Omega(x)$ for $P$ over $\Omega$ at $x$.
They satisfy
\begin{enumerate}
\item
$\supp(F_\pm^{\Omega}(x))\subset J_\pm^\Omega(x)$,
\item
for each $\varphi\in\DD(\Omega,E^*)$ the maps $x\mapsto
F_\pm^{\Omega}(x)[\varphi]$ are smooth sections in $E^*.\hfill\Box$
\end{enumerate}
\end{cor}

%%%%%%%%%%%%%%%%%%%%%%%%%%%%%%%%%%%%%%%%%%%%%%%%%%%%%%%%%%%%%%%%%%%%%%%%% 
\section{The formal fundamental solution is asymptotic} 
%%%%%%%%%%%%%%%%%%%%%%%%%%%%%%%%%%%%%%%%%%%%%%%%%%%%%%%%%%%%%%%%%%%%%%%%%

Let $M$ be a timeoriented Lorentzian manifold. 
Let $P$ be a normally hyperbolic operator acting on sections in a
vector bundle $E$ over $M$.
Let $\Omega'\subset M$ be a convex domain and let $\Omega\subset \Omega'$ be a
relatively compact causal domain with $\overline\Omega \subset \Omega'$.
We assume that $\Omega$ is so small that Corollary~\ref{fundlocalexist}
applies.
Using Riesz distributions and Hadamard coefficients we have constructed the
formal fundamental solutions at $x\in\Omega$
\indexn{approximate fundamental solution}
\indexn{formal advanced fundamental solution}
\indexn{formal retarded fundamental solution}
$$
\RR_{\pm}(x) =
\sum_{j=0}^{\infty} V_j(x,\cdot)\,R_\pm^{\Omega'}(2+2j,x),
$$
the approximate fundamental solutions
$$
\RRt_\pm(x) = 
\sum_{j=0}^{N-1} V_j(x,\cdot)\,R_\pm^{\Omega'}(2+2j,x) 
+
\sum_{j=N}^{\infty} \sigma(\Gamma(x,\cdot)/\varepsilon_j)V_j(x,\cdot)\,
R_\pm^{\Omega'}(2+2j,x),
$$
where $N\geq\frac{n}{2}$ is fixed, and the true fundamental solutions
$F^{\Omega}_{\pm}(x)$, 
$$
F_\pm^{\Omega}(\cdot)[\varphi] = 
(\id + \KK_{\pm})^{-1}(\RRt_\pm(\cdot)[\varphi])
.
$$
The purpose of this section is to show that, in a suitable sense, the formal
fundamental solution is an asymptotic expansion of the true fundamental
solution.
For $k\geq 0$ we define the \defem{truncated formal fundamental solution} 
\indexn{truncated formal fundamental solution>defemidx}
$$
\RR_{\pm}^{N+k}(x) :=
\sum_{j=0}^{N-1+k} V_j(x,\cdot)\,R_\pm^{\Omega'}(2+2j,x).
\indexs{R*@$\protect\RR_{\pm}^{N+k}(x)$, truncated formal fundamental solution}
$$
Hence we cut the formal fundamental solution at the $(N+k)$-th term.
The truncated formal fundamental solution is a well-defined distribution on
$\Omega'$, $\RR_{\pm}^{N+k}(x)\in\DD'(\Omega',E,E_x^*)$.
We will show that the true fundamental solution coincides with the truncated
formal fundamental solution up to an error term which is very regular along
the light cone.
The larger $k$ is, the more regular is the error term.

\begin{prop}\label{formalasymptotisch}
For every $k\in \N$ and every $x\in\Omega$ the difference of distributions
$F_\pm^{\Omega}(x)-\RR_{\pm}^{N+k}(x)$ is a $C^k$-section in $E$.
In fact,
$$
(x,y) \mapsto \left(F_\pm^{\Omega}(x)-\RR_{\pm}^{N+k}(x)\right)(y)
$$
is of regularity $C^k$ on $\Omega\times\Omega$.
\end{prop}

\begin{proof}
We write
$$
\left(F_\pm^{\Omega}(x)-\RR_{\pm}^{N+k}(x)\right)(y) =
\left(F_\pm^{\Omega}(x)-\RRt_{\pm}(x)\right)(y) +
\left(\RRt_{\pm}(x)-\RR_{\pm}^{N+k}(x)\right)(y)
$$
and we show that $\left(\RRt_{\pm}(x)-\RR_{\pm}^{N+k}(x)\right)(y)$ and
$\left(F_\pm^{\Omega}(x)-\RRt_{\pm}(x)\right)(y)$ are both $C^k$ in $(x,y)$.
Now
\begin{eqnarray*}
\left(\RRt_{\pm}(x)-\RR_{\pm}^{N+k}(x)\right)(y) 
&=&
\sum_{j=N}^{N+k-1} \left(\sigma(\Gamma(x,y)/\varepsilon_j)-1\right)
V_j(x,y)\,R_\pm^{\Omega'}(2+2j,x)(y)\\
&&
+ \sum_{j=N+k}^{\infty} \sigma(\Gamma(x,y)/\varepsilon_j)V_j(x,y)\,
R_\pm^{\Omega'}(2+2j,x)(y).
\end{eqnarray*}
From Lemma~\ref{RtKonvergenz} we know that the infinite part
$(x,y)\mapsto \sum_{j=N+k}^{\infty} \sigma(\Gamma(x,y)/\varepsilon_j)V_j(x,y)\,
R_\pm^{\Omega'}(2+2j,x)(y)$ is $C^k$.
The finite part $(x,y)\mapsto \sum_{j=N}^{N+k-1}
\left(\sigma(\Gamma(x,y)/\varepsilon_j)-1\right)
V_j(x,y)\,R_\pm^{\Omega'}(2+2j,x)(y)$ is actually smooth since
$\sigma(\Gamma/\varepsilon_j)-1$ vanishes on a neighborhood of
$\Gamma^{-1}(0)$ which is precisely the locus where $(x,y)\mapsto
R_\pm^{\Omega'}(2+2j,x)(y)$ is nonsmooth.
Furthermore,
\begin{eqnarray*}
F_\pm^{\Omega}(\cdot)[\varphi]-\RRt_{\pm}(\cdot)[\varphi]
&=&
\left((\id + \KK_{\pm})^{-1}-\id\right)(\RRt_\pm(\cdot)[\varphi])\\
&=&
-\left((\id + \KK_{\pm})^{-1}\circ\KK_{\pm}\right)(\RRt_\pm(\cdot)[\varphi]).
\end{eqnarray*}
By Lemma~\ref{Klemma} the operator $-\left((\id + \KK_{\pm})^{-1}\circ
\KK_{\pm}\right)$ has a smooth integral kernel $L_\pm(x,y)$ whose support is
future or past-stretched respectively.
Hence
\begin{eqnarray*}
\lefteqn{F_\pm^{\Omega}(x)[\varphi]-\RRt_{\pm}(x)[\varphi]}\\
&=&
\int_{\overline\Omega} L_\pm(x,y)\RRt_\pm(y)[\varphi] \dV(y)\\
&=&
\sum_{j=0}^{N-1}\int_{\overline\Omega} L_\pm(x,y) V_j(y,\cdot)\,
R_\pm^{\Omega'}(2+2j,y)[\varphi] \dV(y)\\
&&+ 
\sum_{j=N}^{N+k-1}\int_{\overline\Omega} L_\pm(x,y)\sigma(\Gamma(y,\cdot)/
\varepsilon_j) V_j(y,\cdot)\, R_\pm^{\Omega'}(2+2j,y)[\varphi] \dV(y)\\
&&+ 
\int_{\overline\Omega\times\overline\Omega} L_\pm(x,y)f(y,z)\varphi(z) \dV(z)
\dV(y) 
\end{eqnarray*}
where $f(y,z) = \sum_{j=N+k}^{\infty}
\sigma(\Gamma(y,z)/\varepsilon_j)V_j(y,z)\, R_\pm^{\Omega'}(2+2j,y)(z)$ is 
$C^k$ by Lemma~\ref{RtKonvergenz}.
Thus $(x,z)\mapsto \int_{\overline\Omega} L_\pm(x,y)f(y,z) \dV(y)$ is a
$C^k$-section. 
Write $\tilde V_j(y,z) := V_j(y,z)$ if $j\leq N-1$ and $\tilde V_j(y,z) :=
\sigma(\Gamma(y,z)/\varepsilon_j) V_j(y,z)$ if $j\geq N$.
It follows from Lemma~\ref{xytauschCk}
\begin{eqnarray*}
\lefteqn{\int_{\overline\Omega} L_\pm(x,y)\tilde
  V_j(y,\cdot)\,R_\pm^{\Omega'}(2+2j,y)[\varphi] \dV(y)}\\ 
&=&
\int_{\overline\Omega}R_\pm^{\Omega'}(2+2j,y)
[z\mapsto L_\pm(x,y)\tilde V_j(y,z)\varphi(z)] \dV(y)\\
&=&
\int_{\overline\Omega}R_\mp^{\Omega'}(2+2j,z)
[y\mapsto L_\pm(x,y)\tilde V_j(y,z)\varphi(z)] \dV(z)\\
&=&
\int_{\overline\Omega}R_\mp^{\Omega'}(2+2j,z)
[y\mapsto L_\pm(x,y)\tilde V_j(y,z)]\varphi(z) \dV(z)\\
&=&
\int_{\overline\Omega}W_j(x,z)\varphi(z) \dV(z)
\end{eqnarray*}
where $W_j(x,z) = R_\mp^{\Omega'}(2+2j,z)[y\mapsto L_\pm(x,y)\tilde V_j(y,z)]$
is smooth in $(x,z)$ by Proposition~\ref{OmegaRiesz}~(\ref{Rglattinx}).
Hence
$$
\left(F_\pm^{\Omega}(x)-\RRt_{\pm}(x)\right)(z) = 
\sum_{j=0}^{N+k-1}W_j(x,z) + \int_{\overline\Omega} L_\pm(x,y)f(y,z)\dV(y)
$$
is $C^k$ in $(x,z)$.
\end{proof}

The following theorem tells us that the formal fundamental solutions are 
asymptotic expansions of the true fundamental solutions near the light cone.

\begin{thm}\label{asympto}
\indexn{asymptotic expansion of fundamental solution>defemidx}
Let $M$ be a timeoriented Lorentzian manifold. 
Let $P$ be a normally hyperbolic operator acting on sections in a 
vector bundle $E$.
Let $\Omega \subset M$ be a relatively compact causal domain and let
$x\in\Omega$. 
Let $F^\Omega_\pm$ denote the fundamental solutions of $P$ at $x$ and
$\RR^{N+k}_\pm(x)$ the truncated formal fundamental solutions.

Then for each $k\in\N$ there exists a constant $C_k$ such that
\be
\left\| \left(F^\Omega_\pm(x)- \RR^{N+k}_\pm(x)\right)(y)\right\|
\le 
C_{k}\cdot \left|\Gamma(x,y)\right|^{k} 
\ee
for all $(x,y)\in\overline\Omega\times\overline\Omega$.
\end{thm}

Here $\|\cdot\|$ denotes an auxiliary norm on $E^*\boxtimes E$.
The proof requires some preparation.

\begin{lemma}\label{divisionslemma}
Let $M$ be a smooth manifold.
Let $H_1, H_2 \subset M$ be two smooth hypersurfaces globally defined by the 
equations $\varphi_1=0$ and $\varphi_2=0$ respectively, where  
$\varphi_1,\varphi_2:M\to\mathbb{R}$ are smooth functions on $M$ satisfying
$d_x\varphi_i\neq 0$ for every $x\in H_i$, $i=1,2$.
We assume that $H_1$ and $H_2$ intersect transversally.

Let $f:M\to\mathbb{R}$ be a $C^k$-function on $M$, $k\in \N$.
Let $k_1, k_2 \in \N$ such that $k_1+k_2\leq k$.
We assume that $f$ vanishes to order $k_i$ along $H_i$, i.~e., in local
coordinates $\frac{\partial^{|\alpha|}f}{\partial x^\alpha}(x)=0$ for every
$x\in H_i$ and every multi-index $\alpha$ with $|\alpha|\leq k_i-1$.
 
Then there exists a continuous function $F:M\to\mathbb{R}$ such that 
\[f=\varphi_1^{k_1}\varphi_2^{k_2} F.\]
\end{lemma}

\begin{proof}[Proof of Lemma~\ref{divisionslemma}.]
We first prove the existence of a $C^{k-k_1}$-function $F_1:M\to\mathbb{R}$
such that 
\[f=\varphi_1^{k_1}F_1 .\]
This is equivalent to saying that the function $f/\varphi_1^{k_1}$ being
well-defined and $C^k$ on $M\setminus H_1$ extends to a
 $C^{k-k_1}$-function $F_1$ on $M$.
Since it suffices to prove this locally, we introduce local coordinates
$x^1,\dots,x^n$ so that $\varphi_1(x)=x^1$. 
Hence in this local chart $H_1=\{x^1=0\}$.

Since $f(0,x^2,\ldots,x^n)=\frac{\partial^j f}
{\partial(x^1)^j}(0,x^2,\ldots,x^n)=0$ for any $(x^2,\ldots,x^n)$ and $j\leq
k_1-1$ we obtain from the Taylor expansion of $f$ in the $x^1$-direction to
the order $k_1-1$ with integral remainder term 
\[
f(x^1,x^2,\ldots,x^n)
=
\int_0^{x^1}\frac{(x^1-t)^{k_1-1}}{(k_1-1)!}\frac{\del^{k_1}f}{\del  
  (x^1)^{k_1}}(t,x^2,\ldots,x^n)dt.
\] 
In particular, for $x^1\neq 0$
\be
f(x^1,x^2,\ldots,x^n)
&=&
(x^1)^{k_1-1}\int_0^{x^1}\frac{1}{(k_1-1)!}
\left(\frac{x^1-t}{x^1}\right)^{k_1-1} 
\frac{\del^{k_1}f}{\del (x^1)^{k_1}}(t,x^2,\ldots,x^n)dt\\
&=&
\frac{(x^1)^{k_1-1}}{(k_1-1)!}\int_0^1(1-u)^{k_1-1}x^1 \frac{\del^{k_1}f}{\del
  (x^1)^{k_1}}(x^1u,x^2,\ldots,x^n)du\\
&=&
\frac{(x^1)^{k_1}}{(k_1-1)!}\int_0^1(1-u)^{k_1-1} \frac{\del^{k_1}f}{\del
  (x^1)^{k_1}}(x^1u,x^2,\ldots,x^n)du. 
\ee
Now $F_1(x^1,\ldots,x^n):=\frac{1}{(k_1-1)!}\int_0^1(1-u)^{k_1-1}
\frac{\del^{k_1}f}{\del (x^1)^{k_1}}(x^1u,x^2,\ldots,x^n)du$ yields a
$C^{k-k_1}$-function because $\frac{\del^{k_1}f}{\del (x^1)^{k_1}}$ is
$C^{k-k_1}$.
Moreover, we have
$$
f = (x^1)^{k_1} \cdot F_1 = \varphi^{k_1} \cdot F_1.
$$
On $M\setminus H_1$ we have $F_1={f}/{\varphi_1^{k_1}}$ and so $F_1$
vanishes to the order $k_2$ on $H_2\setminus H_1$ because $f$ does. 
Since $H_1$ and $H_2$ intersect transversally the subset
$H_2\setminus H_1$ is dense in $H_2$.
Therefore the function $F_1$  vanishes to the order $k_2$ on all of $H_2$. 
Applying the considerations above to $F_1$ yields a
$C^{k-k_1-k_2}$-function $F:M\to\mathbb{R}$ such that 
$F_1 = \varphi_2^{k_2}\cdot F$.
This concludes the proof.
\end{proof}

\begin{lemma}\label{factorgamma}
Let $f:\mathbb{R}^n\to\mathbb{R}$ a $C^{3k+1}$-function.
We equip $\R^n$ with its standard Minkowski product $\la\cdot,\cdot\ra$
and we assume that $f$ vanishes on all spacelike vectors.

Then there exists a continuous function $h:\mathbb{R}^n\to\mathbb{R}$ such
that 
\[f=h\cdot\gamma^k\]
where $\gamma(x) = - \la x,x\ra$.
\end{lemma}

\begin{proof}[Proof of Lemma~\ref{factorgamma}]
The problem here is that the hypersurface $\{\gamma=0\}$ is the light cone
which contains $0$ as a singular point so that Lemma~\ref{divisionslemma}
does not apply directly. 
We will get around this difficulty by resolving the singularity.

Let $\pi:M:=\mathbb{R}\times S^{n-1}\to\mathbb{R}^n$ be the map defined
by $\pi(t,x):=tx$. 
It is smooth on $M=\mathbb{R}\times S^{n-1}$ and outside
$\pi^{-1}(\{0\})=\{0\}\times S^{n-1}$ it is a two-fold covering of
$\mathbb{R}^n\setminus\{0\}$. 
The function $\wih{f}:=f\circ\pi:M\to\mathbb{R}$ is $C^{3k+1}$ since $f$ is.

Consider the functions $\wih{\gamma}:M\to\R$, $\wih{\gamma}(t,x):=\gamma(x)$,
and $\pi_1:M\to\R$, $\pi_1(t,x) := t$.
These functions are smooth and have only regular points on $M$.
For $\wih{\gamma}$ this follows from $d_x\gamma\neq 0$ for every
$x\in S^{n-1}$. 
Therefore $\wih{C}(0):=\wih{\gamma}^{-1}(\{0\})$ and $\{0\}\times
S^{n-1}=\pi_1^{-1}(\{0\})$ are smooth embedded hypersurfaces.
Since the differentials of $\wih\gamma$ and of $\pi_1$ are linearly
independent the hypersurfaces intersect transversally.
Furthermore, one obviously has $\pi(\wih{C}(0))=C(0)$ and $\pi(\{0\}\times
S^{n-1})=\{0\}$.

Since $f$ is $C^{3k+1}$ and vanishes on all spacelike vectors $f$ vanishes to
the order $3k+2$ along $C(0)$ (and in particular at $0$).
Hence $\wih{f}$ vanishes to
the order $3k+2$ along $\wih{C}(0)$ and along $\{0\}\times S^{n-1}$. 
Applying Lemma~\ref{divisionslemma} to $\wih{f}$, $\varphi_1:=\pi_1$ and
$\varphi_2:=\wih{\gamma}$, with $k_1:=2k+1$ and $k_2:=k$, yields a continuous
function $\wih{F}:\mathbb{R}\times S^{n-1}\to \mathbb{R}$ such that
\begin{equation}\label{factorfhut}
\wih{f}=\pi_1^{2k+1}\cdot\wih{\gamma}^{\,k}\cdot\wih{F}.
\end{equation}
For $y\in\mathbb{R}^n$ we set
\[h(y):=\left\{\begin{array}{cl}
\|y\|\cdot\wih{F}(\|y\|,\frac{y}{\|y\|})&\textrm{ if }y\neq 0\\ 
0&\textrm{ if }y=0,\end{array}\right.\]
where $\|\cdot\|$ is the standard Euclidean norm on $\mathbb{R}^n$.
The function $h$ is obviously continuous on $\R^n$.
It remains to show $f=\gamma^k\cdot h$.
For $y\in\mathbb{R}^n\setminus\{0\}$ we have
\be
f(y)
&=&
f\left(\|y\|\cdot\frac{y}{\|y\|}\right)\\
&=&
\wih{f}\left(\|y\|,\frac{y}{\|y\|}\right)\\
&\bui{=}{(\ref{factorfhut})}&
\|y\|^{2k+1}\cdot\gamma\left(\frac{y}{\|y\|}\right)^k\cdot
\wih{F}\left(\|y\|,\frac{y}{\|y\|}\right)\\ 
&=&
\|y\|^{2k}\cdot
\gamma\left(\frac{y}{\|y\|}\right)^k\cdot h(y)\\ 
&=&
\gamma(y)^k\cdot h(y).
\ee
For $y=0$ the equation $f(y) = \gamma(y)^k\cdot h(y)$ holds trivially.
\end{proof}

\begin{proof}[Proof of Theorem~\ref{asympto}] 
Repeatedly using Proposition~\ref{OmegaRiesz}~(\ref{gammaR}) we find constants
$C_j'$ such that
\begin{eqnarray*}
\lefteqn{\left(F^\Omega_\pm(x)- \RR^{N+k}_\pm(x)\right)(y)}\\
&=&
\left(F^\Omega_\pm(x)- \RR^{N+3k+1}_\pm(x)\right)(y)
+ \sum_{j=N+k}^{N+3k} V_j(x,y)\cdot R^{\Omega'}_\pm(2+2j,x)(y)\\
&=&
\left(F^\Omega_\pm(x)- \RR^{N+3k+1}_\pm(x)\right)(y)
+ \sum_{j=N+k}^{N+3k} V_j(x,y)\cdot  C_j'\cdot  \Gamma(x,y)^k
\cdot R^{\Omega'}_\pm(2+2(j-k),x)(y) .
\end{eqnarray*}
Now $h_j(x,y) := C_j'\cdot V_j(x,y)\cdot R^{\Omega'}_\pm(2+2(j-k),x)(y)$
is continuous since $2+2(j-k) \geq 2+2N \geq 2+n > n$.
By Proposition~\ref{formalasymptotisch} the section $(x,y)\mapsto
\left(F^\Omega_\pm(x)- \RR^{N+3k+1}_\pm(x)\right)(y)$ is of regularity
$C^{3k+1}$.
Moreover, we know $\supp(F^\Omega_\pm(x)- \RR^{N+3k+1}_\pm(x)) \subset
J_\pm^{\Omega}(x)$.
Hence we may apply Lemma~\ref{factorgamma} in normal coordinates and we obtain
a continuous section $h$ such that 
$$
\left(F^\Omega_\pm(x)- \RR^{N+3k+1}_\pm(x)\right)(y) =
\Gamma(x,y)^k\cdot h(x,y).
$$
This shows
$$
\left(F^\Omega_\pm(x)- \RR^{N+k}_\pm(x)\right)(y) =
\left(h(x,y) + \sum_{j=N+k}^{N+3k} h_j(x,y)\right) \Gamma(x,y)^k .
$$
Now $C_k := \|h + \sum_{j=N+k}^{N+3k}
h_j\|_{C^0(\overline\Omega\times\overline\Omega)}$ does the job.
\end{proof}

\Remark{
It is interesting to compare Theorem~\ref{asympto} to a similar situation
arising in the world of Riemannian manifolds.
If $M$ is an $n$-dimensional compact Riemannian manifold, then the operators
analogous to normally hyperbolic operators on Lorentzian manifolds are the
{\em Laplace type} operators. 
They are defined formally just like normally hyperbolic operators, namely
their principal symbol must be given by the metric.
Analytically however, they behave very differently because they are elliptic.

If $L$ is a nonnegative formally selfadjoint Laplace type operator on $M$,
  then it is essentially selfadjoint and one can form the semi-group
  $t\mapsto e^{-t\bar L}$ where $\bar L$ is the selfadjoint extension of
  $L$. 
For $t>0$ the operator $e^{-t\bar L}$ has a smooth integral kernel $K_t(x,y)$.
One can show that there is an asymptotic expansion of this ``heat kernel''
$$
K_t(x,x) \sim \frac{1}{(4\pi t)^{n/2}} \sum_{k=0}^\infty \alpha_k(x) t^k 
$$
as $t\searrow 0$.
%% In particular,
%% $$
%% \tr(e^{-t\bar L}) \sim \frac{1}{(4\pi t)^{n/2}} \sum_{k=0}^\infty
%% \int_M\alpha_k(x)\dV(x)\, t^k .
%% $$
The coefficients $\alpha_k(x)$ are given by a universal expression in the
coefficients of $L$ and their covariant derivatives and the curvature of $M$
and its covariant derivatives.

Even though this asymptotic expansion is very different in nature from the one
in Theorem~\ref{asympto}, it turns out that the Hadamard coefficients on the
diagonal $V_k(x,x)$ of a normally hyperbolic operator $P$ on an $n$-dimensional
Lorentzian manifold are {\em given by the same universal expression} in the
coefficients of $P$ and their covariant derivatives and the curvature of $M$
and its covariant derivatives as $\alpha_k(x)$.
This is due to the fact that the recursive relations defining $\alpha_k$
are formally the same as the transport equations (\ref{eq:transport}) for $P$.
See e.~g.\ \cite{BGV} for details on Laplace type operators.
}

%%%%%%%%%%%%%%%%%%%%%%%%%%%%%%%%%%%%%%%%%%%%%%%%%%%%%%%%%%%%%%%%%%%%%%%%% 
\section{Solving the inhomogeneous equation on small domains} 
%%%%%%%%%%%%%%%%%%%%%%%%%%%%%%%%%%%%%%%%%%%%%%%%%%%%%%%%%%%%%%%%%%%%%%%%%

In the next chapter we will show uniqueness of the fundamental solutions.
For this we need to be able to solve the inhomogeneous equation\indexn{inhomogeneous wave equation>defemidx} $Pu=v$ for
given $v$ with small support.
Let $\Omega$ be a relatively compact causal subset of $M$ as in
Corollary~\ref{fundlocalexist}. 
Let $F_\pm^{\Omega}(x)$ be the corresponding fundamental solutions for
$P$ at $x\in\Omega$ over $\Omega$.
Recall that for $\varphi\in\DD(\Omega,E^*)$ the maps $x\mapsto
F_\pm^{\Omega}(x)[\varphi]$ are smooth sections in $E^*$.
Using the natural pairing $E_x^* \otimes E_x \to \K$, $\ell\otimes e \mapsto
\ell\cdot e$, we obtain a smooth $\K$-valued function $x\mapsto
F_\pm^{\Omega}(x)[\varphi]\cdot v(x)$ with compact support.
We put
\begin{equation}
u_\pm[\varphi] := \int_\Omega  F_\pm^{\Omega}(x)[\varphi]\cdot v(x)\dV(x) .
\label{defloeseinhomo}
\indexs{u*@$u_\pm$, solution for inhomogeneous wave equation}
\end{equation}

This defines distributions $u_\pm\in\DD'(\Omega,E)$ because if $\varphi_m \to
\varphi$ in $\DD(\Omega,E^*)$, then $F_\pm^{\Omega}(\cdot)[\varphi_m] \to 
F_\pm^{\Omega}(\cdot)[\varphi]$ in $C^0(\overline\Omega,E^*)$ by
Lemma~\ref{Rtglattinx} and (\ref{Fdef}).
Hence $u_\pm[\varphi_m] \to u_\pm[\varphi]$.

\begin{lemma}\label{loeseinhomogen1}
The distributions $u_\pm$ defined in (\ref{defloeseinhomo}) satisfy
$$
Pu_\pm = v
$$
and
$$
\supp(u_\pm) \subset J_\pm^\Omega(\supp(v)).
$$
\end{lemma}

\begin{proof}
Let $\varphi\in\DD(\Omega,E^*)$.
We compute
\begin{eqnarray*}
Pu_\pm[\varphi] &=& u_\pm[P^*\varphi]\\
&=&
\int_\Omega  F_\pm^{\Omega}(x)[P^*\varphi]\cdot v(x)\dV(x)\\
&=&
\int_\Omega  P_{(2)}F_\pm^{\Omega}(x)[\varphi]\cdot v(x)\dV(x)\\
&=&
\int_\Omega  \varphi(x)\cdot v(x)\dV(x).
\end{eqnarray*}
Thus $Pu_\pm=v$.
Now assume $\supp(\varphi)\cap J^\Omega_\pm(\supp(v))=\emptyset$.
Then $\supp(v)\cap J^\Omega_\mp(\supp(\varphi))=\emptyset$.
Since $J^\Omega_\mp(\supp(\varphi))$ contains the support of $x \mapsto
F^{\Omega}_\pm(x)[\varphi]$ we have $\supp(v)\cap
\supp(F^{\Omega}_\pm(\cdot)[\varphi])=\emptyset$. 
Hence the integrand in (\ref{defloeseinhomo}) vanishes identically and
therefore $u_\pm[\varphi] = 0$.
This proves $\supp(u_\pm) \subset J^\Omega_\pm(\supp(v))$.
\end{proof}

\begin{lemma}\label{befreiungslemmaneu}
Let $\Omega$ be causal and contained in a convex domain $\Omega'$.
Let $S_1, S_2 \subset \Omega$ be compact subsets.
Let $V \in C^\infty(\overline\Omega\times\overline\Omega,E^*\boxtimes E)$.
Let $\Phi\in C^{n+1}(\ovl{\Omega},E^*)$ and $\Psi\in C^{n+1}(\ovl{\Omega},E)$
be such that $\supp(\Phi)\subset J_\mp^\Omega(S_1)$ and $\supp(\Psi)\subset
J_\pm^\Omega(S_2)$.

Then for all $j\geq 0$
$$
\int_{\ovl{\Omega}} \left(V(x,\cdot) R^{\Omega'}_\pm(2+2j,x)\right)[\Phi ]
\cdot \Psi(x)\dV(x) 
=
\int_{\ovl{\Omega}} \Phi(y)\cdot \left(V(\cdot,y) R^{\Omega'}_\mp(2+2j,y)\right)[\Psi]
\dV(y).
$$
\end{lemma}

\begin{proof}
Since $\supp(R_\pm^{\Omega'}(2+2j,x)) \cap \supp(\Phi) \subset J_\pm^\Omega(x)
\cap 
J_\mp^\Omega(S_1)$ is compact (Lemma~\ref{lJ+KJ-K'cpct}) and since the
distribution $R_\pm^{\Omega'}(2+2j,x)$ 
is of order $\leq n+1$ we may apply $V(x,\cdot) R_\pm^{\Omega'}(2+2j,x)$ to
$\Phi$. 
By Proposition~\ref{OmegaRiesz}~(\ref{RCkinx}) the section $x \mapsto 
V(x,\cdot) R_\pm^{\Omega'}(2+2j,x)[\Phi]$ is continuous.
Moreover, $\supp(x\mapsto V(x,\cdot) R_\pm^{\Omega'}(2+2j,x)[\Phi]) \cap
\supp(\Psi) \subset 
J_\mp^\Omega(\supp(\Phi)) \cap J_\pm^\Omega(S_2) \subset J_\mp^\Omega(S_1)
\cap J_\pm^\Omega(S_2)$ is also compact and contained in $\ovl{\Omega}$.
Hence the integrand of the left hand side is a compactly supported continuous
function and the integral is well-defined.
Similarly, the integral on the right hand side is well-defined.
By Lemma~\ref{xytausch}
\begin{eqnarray*}
\lefteqn{\int_\Omega \left(V(x,\cdot)R_\pm^{\Omega'}(2+2j,x)\right)[\Phi]
\cdot\Psi(x)\dV(x)}\\
&=&
\int_\Omega R_\pm^{\Omega'}(2+2j,x)[y\mapsto V(x,y)^* \Phi(y)]
\cdot\Psi(x)\,\dV(x)\\ 
&=&
\int_\Omega R_\pm^{\Omega'}(2+2j,x)[y\mapsto \Phi(y) V(x,y)\Psi(x)]\,\dV(x)\\
&=&
\int_\Omega R_\mp^{\Omega'}(2+2j,y)[x\mapsto \Phi(y) V(x,y)\Psi(x)]\,\dV(y)\\
&=&
\int_\Omega \Phi(y)\cdot\left(V(\cdot,y)R_\mp^{\Omega}(2+2j,y)[\Psi]\right)
\dV(y).
\end{eqnarray*}
\end{proof}

\begin{lemma}\label{loeseinhomogen2}
Let $\Omega\subset M$ be a relatively compact causal domain satisfying
(\ref{Omegaklein}) in Lemma~\ref{Klemma}.

Then the distributions $u_\pm$ defined in (\ref{defloeseinhomo}) are
smooth sections in $E$, i.~e., $u_\pm \in C^{\infty}(\Omega,E)$.
\end{lemma}

\begin{proof}
Let $\varphi\in\DD(\Omega,E^*)$.
Put $S:=\supp(\varphi)$.
Let $L_\pm \in C^\infty(\overline\Omega\times\overline\Omega, E^*\boxtimes E)$
be the integral kernel of $(\id + \KK_{\pm})^{-1}\circ \KK_{\pm}$.
We recall from (\ref{Fdef})
$$
F_\pm^{\Omega}(\cdot)[\varphi] 
= 
(\id + \KK_{\pm})^{-1}(\RRt_\pm(\cdot)[\varphi])
=
\RRt_\pm(\cdot)[\varphi] - (\id + \KK_{\pm})^{-1}
\KK_{\pm}(\RRt_\pm(\cdot)[\varphi]).
$$
Therefore
\begin{eqnarray*}
u_\pm[\varphi] 
&=&
\int_\Omega  F_\pm^{\Omega}(x)[\varphi]\cdot v(x)\dV(x)\\
&=&
\int_\Omega  \RRt_\pm(x)[\varphi]\cdot v(x)\dV(x) -
\int_\Omega \int_\Omega L_\pm(y,x) \cdot \RRt_\pm(x)[\varphi]\cdot v(y)
\dV(x)\dV(y)\\
&=&
\int_\Omega  \RRt_\pm(x)[\varphi]\cdot w(x)\dV(x)
\end{eqnarray*}
where $w(x) := v(x) - \int_\Omega v(y)\cdot L_\pm(y,x)\dV(y)\in E_x$.
Obviously, $w \in C^\infty(\ovl{\Omega},E)$.
By Lemma~\ref{Klemma} $\supp(L_\pm) \subset \{(y,x)\in\overline\Omega\times
\overline\Omega\, |\, x \in J_\pm^{\overline\Omega}(y)\}$.
Hence $\supp(w)\subset J_\pm^{\overline\Omega}(\supp(v))$.
We may therefore apply Lemma~\ref{befreiungslemmaneu} with $\Phi=\varphi$
and $\Psi = w$ to obtain
$$
\int_\Omega V_j(x,\cdot)R_\pm^{\Omega'}(2+2j,x)[\varphi]\cdot w(x) \dV(x)
=
\int_\Omega \varphi(y) V_j(\cdot,y)R_\mp^{\Omega'}(2+2j,y)[w] \dV(y)
$$
for $j=0,\ldots,N-1$ and 
\begin{eqnarray*}
\lefteqn{\int_\Omega \sigma(\Gamma(x,\cdot)/\varepsilon_j) V_j(x,\cdot)
R_\pm^{\Omega'}(2+2j,x)[\varphi]\cdot w(x) \dV(x)}\\
&=&
\int_\Omega \varphi(y) \sigma(\Gamma(\cdot,y)/\varepsilon_j) V_j(\cdot,y)
R_\mp^{\Omega'}(2+2j,y)[w] \dV(y) 
\end{eqnarray*}
for $j\geq N$.
Note that the contribution of the zero set $\partial{\Omega}$ in the above
integrals vanishes, hence we integrate over $\Omega$ instead of $\ovl{\Omega}$.
Summation over $j$ yields
\begin{eqnarray*}
u_\pm[\varphi] 
&=&
\int_\Omega  \RRt_\pm(x)[\varphi]\cdot w(x)\dV(x)\\
&=&
\sum_{j=0}^{N-1}
\int_\Omega \varphi(y) V_j(\cdot,y)R_\mp^{\Omega'}(2+2j,y)[w] \dV(y)\\
&&+
\sum_{j=N}^\infty 
\int_\Omega \varphi(y) \sigma(\Gamma(\cdot,y)/\varepsilon_j) V_j(\cdot,y)
R_\mp^{\Omega'}(2+2j,y)[w] \dV(y) .
\end{eqnarray*}
Thus
$$
u_\pm(y) = 
\sum_{j=0}^{N-1}\left(V_j(\cdot,y)R_\mp^{\Omega'}(2+2j,y)\right)[w]
+
\sum_{j=N}^\infty \left(\sigma(\Gamma(\cdot,y)/\varepsilon_j) V_j(\cdot,y)
R_\mp^{\Omega'}(2+2j,y)\right)[w] .
$$
Proposition~\ref{OmegaRiesz}~(\ref{Rglattinx}) shows that all summands
are smooth in $y$.
By the choice of the $\varepsilon_j$ the series converges in all $C^k$-norms.
Hence $u_\pm$ is smooth.
\end{proof}

We summarize

\begin{thm}\label{thminhomogen}
Let $M$ be a timeoriented Lorentzian manifold.
Let $P$ be a normally hyperbolic operator acting on sections in a 
vector bundle $E$ over $M$.

Then each point in $M$ possesses a relatively compact causal
neighborhood $\Omega$ such that for each $v\in\DD(\Omega,E)$
there exist $u_\pm\in C^\infty(\Omega,E)$ satisfying
\begin{enumerate}
\item 
$\int_\Omega \phi(x)\cdot u_\pm(x)\dV = \int_\Omega
F_\pm^{\Omega}(x)[\varphi]\cdot v(x)\dV$ for each $\phi\in\DD(\Omega,E^*)$,
\item
$Pu_\pm =v$,
\item
$\supp(u_\pm) \subset J_\pm^\Omega(\supp(v)).
\hfill\Box$
\end{enumerate}
\end{thm}

%%%%%%%%%%%%%%%%%%%%%%%%%%%%%%%%%%%%%%%%%%%%%%%%%%%%%%%%%%%%%%%%% %%%%%%%%
\chapter{The global theory}\label{chapglobaltheorie} 
%%%%%%%%%%%%%%%%%%%%%%%%%%%%%%%%%%%%%%%%%%%%%%%%%%%%%%%%%%%%%%%%%%%%%%%%% 

In the previous chapter we developed the local theory.
We proved existence of advanced and retarded fundamental solutions on small
domains $\Omega$ in the Lorentzian manifold.
The restriction to small domains arises from two facts.
Firstly, Riesz distributions and Hadamard coefficients are defined only in
domains on which the Riemannian exponential map is a diffeomorphism.
Secondly, the analysis in Section~\ref{sec:trueffsmall} that allows us to turn
the approximate fundamental solution into a true one requires sufficiently
good bounds on various functions defined on $\Omega$.
Consequently, our ability to solve the wave equation as in
Theorem~\ref{thminhomogen} is so far also restricted to small domains.

In this chapter we will use these local results to understand solutions to a
wave equation defined on the whole Lorentzian manifold.
To obtain a reasonable theory we have to make geometric assumptions on the
manifold.
In most cases we will assume that the manifold is globally hyperbolic.
This is the class of manifolds where we get a very complete understanding of
wave equations.

However, in some cases we get global results for more general manifolds.
We start by showing uniqueness of fundamental solutions
with a suitable condition on their support.
The geometric assumptions needed here are weaker than global hyperbolicity.
In particular, on globally hyperbolic manifolds we get uniqueness of advanced
and retarded fundamental solutions.

Then we show that the Cauchy problem is well-posed on a globally hyperbolic
manifold.
\indexn{Cauchy problem}
This means that one can uniquely solve $Pu=f$, $u|_S=u_0$ and
$\nabla_\mathfrak{n} 
u=u_1$ where $f$, $u_0$ and $u_1$ are smooth and compactly supported, $S$ is a
Cauchy hypersurface and $\nabla_\mathfrak{n}$ is the covariant normal
derivative along 
$S$.
The solution depends continuously on the given data $f$, $u_0$ and $u_1$.
It is unclear how one could set up a Cauchy problem on a non-globally
hyperbolic manifold because one needs a Cauchy hypersurface $S$ to impose the
initial conditions $u|_S=u_0$ and $\nabla_\mathfrak{n} u=u_1$.

Once existence of solutions to the Cauchy problem is established it is not
hard to show existence of fundamental solutions and of Green's operators.
In the last section we show how one can get fundamental solutions to some
operators on certain non-globally hyperbolic manifolds like anti-deSitter
spacetime.

%%%%%%%%%%%%%%%%%%%%%%%%%%%%%%%%%%%%%%%%%%%%%%%%%%%%%%%%%%%%%%%%%%%%%%%%%
\section{Uniqueness of the fundamental solution}
\indexn{fundamental solution}
%%%%%%%%%%%%%%%%%%%%%%%%%%%%%%%%%%%%%%%%%%%%%%%%%%%%%%%%%%%%%%%%%%%%%%%%%

The first global result is uniqueness of solutions to the wave equation with
future or past compact support.
For this to be true the manifold must have certain geometric properties.
Recall from Definition~\ref{def:timesep} and Proposition~\ref{prop:timesep}
the definition and properties of the time-separation function $\tau$.
\indexn{time-separation}
The relation ``$\leq$'' being closed means that $p_i\leq q_i$, $p_i \to p$,
and $q_i\to q$ imply $p\leq q$.

\begin{thm}\label{kernP}
Let $M$ be a connected timeoriented Lorentzian manifold such that
\begin{enumerate}
\item 
the causality condition holds, i.~e., there are no causal loops,
\indexn{causality condition}
\item
the relation ``$\leq$'' is closed,
\item
the time separation function $\tau$ is finite and continuous on $M\times M$.
\end{enumerate}
Let $P$ be a normally hyperbolic operator acting on sections in a 
vector bundle $E$ over $M$.

Then any distribution $u\in\DD'(M,E)$ with past or future compact 
support solving the equation $Pu=0$ must vanish identically on $M$,
\indexn{past compact support>defemidx}
\indexn{future compact support>defemidx}
$$
u\equiv 0.
$$
\end{thm}

The idea of the proof is very simple.
We would like to argue as follows:
We want to show $u[\phi]=0$ for all test sections $\phi\in\DD(M,E^*)$.
Without loss of generality let $\phi$ be a test section whose support is
contained in a 
sufficiently small open subset $\Omega\subset M$ to which
Theorem~\ref{thminhomogen} can be applied.
Solve $P^*\psi=\phi$ in $\Omega$.
Compute
$$
u[\phi]=u[P^*\psi]\stackrel{(*)}{=}
\underbrace{Pu}_{=0}[\psi]=0.
$$
The problem is that equation ($*$) is not justified because $\psi$ does not
have compact support.
The argument can be rectified in case $\supp(u) \cap \supp(\psi)$ is compact.
The geometric considerations in the proof have the purpose of getting to this
situation.

\begin{proof}[Proof of Theorem~\ref{kernP}]
Without loss of generality let $A:=\supp(u)$ be future compact.
We will show that $A$ is empty.
Assume the contrary and consider some $x\in A$. 
We fix some $y\in I_-^M(x)$.
Then the intersection $A\cap J_+^M(y)$ is compact and nonempty.

\input{fig-proofunique1}

Since the function $M \to \R$, $z \mapsto \tau(y,z)$, is continuous
it attains its maximum on the compact set $A\cap J^M_+(y)$ at some point
$z\in A\cap J^M_+(y)$.
The set $B:=  A\cap J^M_+(z)$ is compact and contains $z$.
For all $z' \in B$ we have $\tau(y,z') \geq \tau(y,z)$ from (\ref{idu}) since
$z' \geq z$ 
and hence $\tau(y,z') = \tau(y,z)$ by maximality of $\tau(y,z)$.

The relation ``$\leq$'' turns $B$ into an ordered set.
That $z_1 \leq z_2$ and $z_2 \leq z_1$ implies $z_1 = z_2$ follows from
nonexistence of causal loops.
We check that Zorn's lemma can be applied to $B$.
Let $B'$ be a totally ordered subset of $B$.
Choose\footnote{Every (infinite) subset of a manifold has a countable dense
  subset. 
This follows from existence of a countable basis of the topology.} 
a countable dense subset $B'' \subset B'$.
Then $B''$ is totally ordered as well and can be written as $B'' = \{\zeta_1,
\zeta_2, \zeta_3, \ldots\}$.
Let $z_i$ be the largest element in $\{\zeta_1\ldots,\zeta_i\}$.
This yields a monotonically increasing sequence $(z_i)_{i}$ which eventually
becomes at least as large as any given $\zeta\in B''$.

By compactness of $B$ a subsequence of $(z_i)_{i}$ converges to some $z'\in
B$ as $i \to \infty$.
Since the relation ``$\leq$'' is closed one easily sees that $z'$ is an upper
bound for $B''$. 
Since $B'' \subset B'$ is dense and ``$\leq$'' is closed $z'$ is also an upper
bound for $B'$.
Hence Zorn's lemma applies and yields a maximal element $z_0\in B$.
Replacing $z$ by $z_0$ we may therefore assume that $\tau(y,\cdot)$ attains
its maximum at $z$ and that $A\cap J^M_+(z) = \{z\}$.

\input{fig-proofunique5}

We fix a relatively compact causal neighborhood
${\Omega}\subset\subset M$ of $z$ as in Theorem~\ref{thminhomogen}. 

\input{fig-proofunique6}

Let $p_i\in\Omega\cap I^M_-(z)\cap I^M_+(y)$ such that $p_i \to z$.
We claim that for $i$ sufficiently large we have $J^M_+(p_i)\cap A \subset
\Omega$.
Suppose the contrary.
Then there is for each $i$ a point $q_i\in J^M_+(p_i)\cap A$ such that
$q_i\not\in \Omega$.
Since $q_i\in J^M_+(y) \cap A$ for all $i$ and $J^M_+(y) \cap A$ is compact
we have, after passing to a subsequence, that $q_i \to q \in J^M_+(y) \cap A$.
From $q_i \geq p_i$, $q_i \to q$, $p_i \to z$, and the fact that ``$\leq$''
is closed we conclude $q\geq z$.
Thus $q \in J^M_+(z) \cap A$, hence $q=z$.
On the other hand, $q\not\in \Omega$ since all $q_i\not\in \Omega$,
a contradiction.

\input{fig-proofunique7}

This shows that we can fix $i$ sufficiently large so that $J_+^M(p_i)\cap A
\subset\Omega$. 
We choose a cut-off function $\eta\in\DD({\Omega},\R)$ such that
$\eta|_{J^M_+(p_i)\cap A} \equiv 1$.
We put $\widetilde{\Omega}:=\Omega\cap I^M_+(p_i)$ and note that
$\widetilde{\Omega}$ is an open neighborhood of $z$.

\input{fig-proofunique8}

Now we consider some arbitrary $\varphi\in\DD(\widetilde{\Omega},E^*)$.
We will show that $u[\varphi]=0$.
This then proves that $u|_{\widetilde{\Omega}} = 0$, in particular,
$z\not\in A = \supp(u)$, the desired contradiction.

By the choice of $\Omega$ we can solve the inhomogeneous equation
$P^*\psi=\varphi$ on ${\Omega}$
with $\psi\in C^{\infty}({\Omega},E^*)$ and $\supp(\psi) \subset
J_+^{{\Omega}}(\supp(\varphi))\subset J_+^M(p_i)\cap{\Omega}$.
Then $\supp(u) \cap \supp(\psi) \subset A \cap J_+^M(p_i)\cap{\Omega}
=  A \cap J_+^M(p_i)$.
Hence $\eta|_{\supp(u) \cap \supp(\psi)} = 1$.
Thus

$$
u[\varphi]
=
u[P^*\psi]
=
u[P^*(\eta\psi)]
=
(Pu)[\eta\psi] = 0.
$$
\end{proof}

\begin{corollary}\label{fundunique}
Let $M$ be a connected timeoriented Lorentzian manifold such that
\begin{enumerate}
\item 
the causality condition holds, i.~e., there are no causal loops,
\item
the relation ``$\leq$'' is closed,
\item
the time separation function $\tau$ is finite and continuous on $M\times M$.
\end{enumerate}
Let $P$ be a normally hyperbolic operator acting on sections in a 
vector bundle $E$ over $M$.

Then for every $x\in M$ there exists at most one fundamental solution
for $P$ at $x$ with past compact support and at most one with future
compact support.
$\hfill\Box$
\end{corollary}

\Remark{
The requirement in Theorem~\ref{kernP} and Corollary~\ref{fundunique} that
$u$ have future or past compact support is crucial.
For example, on Minkowski space $u = R_+(2)-R_-(2)$ is a nontrivial
solution to $Pu=0$ despite the fact that Minkowski space satisfies the
geometric assumptions on $M$ in Theorem~\ref{kernP} and in
Corollary~\ref{fundunique}. 
}

These assumptions on $M$ hold for convex Lorentzian manifolds and for
globally hyperbolic manifolds.
On a globally hyperbolic manifold the sets $J_\pm^M(x)$ are always future
respectively past compact.
Hence we have

\begin{cor}\label{funduniqueglobhyp}
Let $M$ be a globally hyperbolic Lorentzian manifold.
Let $P$ be a normally hyperbolic operator acting on sections in a 
vector bundle $E$ over $M$.
\indexn{globally hyperbolic manifold}

Then for every $x\in M$ there exists at most one advanced and at most one
retarded fundamental solution for $P$ at $x$.
$\hfill\Box$
\end{cor}

\Remark{\label{konvexnichteind}
In convex Lorentzian manifolds uniqueness of advanced and retarded
fundamental solutions need not hold.
For example, if $M$ is a convex open subset of Minkowski space $\R^n$ such
that there exist points $x\in M$ and $y\in \R^n\setminus M$ with
$J_+^{\R^n}(y)\cap M \subset J_+^M(x)$, then the restrictions to $M$ of 
$R_+(x)$ and of $R_+(x)+R_+(y)$ are two different advanced fundamental 
solutions for $P=\Box$ at $x$ on $M$.
Corollary~\ref{fundunique} does not apply because $J_+^M(x)$ is not
past compact.
}

\begin{center}
\input{fig-fundsolnotunique.tex}
\end{center}

%%%%%%%%%%%%%%%%%%%%%%%%%%%%%%%%%%%%%%%%%%%%%%%%%%%%%%%%%%%%%%%%%%%%%%%%%
\section{The Cauchy problem}
\indexn{Cauchy problem}
%%%%%%%%%%%%%%%%%%%%%%%%%%%%%%%%%%%%%%%%%%%%%%%%%%%%%%%%%%%%%%%%%%%%%%%%%

The aim of this section is to show that the Cauchy problem on a globally
hyperbolic manifold $M$ is well-posed.
This means that given a normally hyperbolic operator $P$ and a Cauchy
hypersurface $S\subset M$ the problem
\[
\left\{\begin{array}{cccl}
Pu&=&f&\textrm{ on }M,\\
u&=&u_0&\textrm{ along }S,\\
\nabla_\mathfrak{n} u&=&u_1&\textrm{ along }S,
\end{array}\right.
\] 
has a unique solution for given $u_0, u_1 \in \DD(S,E)$ and $f\in\DD(M,E)$.
Moreover, the solution depends continuously on the data.

We will also see that the support of the solution is contained in $J^M(K)$
where $K:=\supp(u_0) \cup \supp(u_1) \cup \supp(f)$.
This is known as finiteness of propagation speed.

We start by identifying the divergence term that appears when one compares the
operator $P$ with its formal adjoint $P^*$.
This yields a local formula allowing us to control a solution of $Pu=0$ in
terms of its Cauchy data.
These local considerations already suffice to establish uniqueness of
solutions to the Cauchy problem on general globally hyperbolic manifolds.

Existence of solutions is first shown locally.
After some technical preparation we put these local solutions together to a
global one on a globally hyperbolic manifold.
This is where the crucial passage from the local to the global theory takes
place. 
Continuous dependence of the solutions on the data is an easy consequence of
the open mapping theorem from functional analysis.

\begin{lemma}\label{PminusP*}
Let $E$ be a  vector bundle over the timeoriented Lorentzian manifold 
$M$.
Let $P$ be a normally hyperbolic operator acting on sections in $E$. 
Let $\nabla$ be the $P$-compatible connection on $E$.

Then for every $\psi\in{C^\infty}(M,E^*)$ and $v\in{C^\infty}(M,E)$, 
\[
\psi\cdot(Pv)-(P^*\psi)\cdot v=\div(W),
\]
where the vector field
$W\in{C^\infty}(M,TM\otimes_{\R}\mathbb{K})$ is characterized by 
\[
\la W,X\ra=(\nabla_X\psi)\cdot v-\psi\cdot(\nabla_X v)
\]
for all $X\in{C^\infty}(M,TM)$. 
\end{lemma}

Here we have, as before, written $\K=\R$ if $E$ is a real vector bundle and
$\K=\Co$ if $E$ is complex. 

\begin{proof}
The Levi-Civita connection on $TM$ and the $P$-compatible connection $\nabla$
on $E$ induce connections on $T^*M\otimes E$ and on $T^*M\otimes E^*$ which we
also denote by $\nabla$ for simplicity.
We define a linear differential operator $L:C^\infty(M,T^*M\otimes E^*) \to
C^\infty(M,E^*)$ of first order by
$$
Ls := -\sum_{j=1}^n \epsilon_j (\nabla_{e_j}s)(e_j)
$$
where $e_1,\ldots,e_n$ is a local Lorentz orthonormal frame of $TM$ and
$\epsilon_j=\la e_j,e_j\ra$.
It is easily checked that this definition does not depend on the choice of
orthonormal frame.
Write $e_1^*,\ldots,e_n^*$ for the dual frame of $T^*M$.
The metric $\la\cdot,\cdot\ra$ on $TM$ and the natural pairing $E^*\otimes E
\to \K$, $\psi\otimes v \mapsto \psi\cdot v$, induce a pairing $(T^*M\otimes
E^*) \otimes (T^*M \otimes E) \to \K$ which we again denote by
$\la\cdot,\cdot\ra$. 
For all $\psi\in C^\infty(M,E^*)$ and $s\in{C^\infty}(M,T^*M\otimes E)$ we
obtain
\begin{eqnarray}
\la\nabla \psi,s\ra
&=&
\sum_{j,k=1}^n\la e_j^*\otimes\nabla_{e_j}\psi\, ,\, e_k^*\otimes s(e_k)\ra
\nonumber\\ 
&=&
\sum_{j,k=1}^n\la e_j^*\, ,\, e_k^*\ra\cdot(\nabla_{e_j}\psi)\cdot s(e_k)
\nonumber\\ 
&=&
\sum_{j=1}^n\varepsilon_j(\nabla_{e_j}\psi)\cdot s(e_j)\nonumber\\ 
&=&
\sum_{j=1}^n \varepsilon_j\left(\partial_{e_j}( \psi\cdot s(e_j))-\psi\cdot
  (\nabla_{e_j} s)(e_j) - \psi\cdot s(\nabla_{e_j}e_j)\right) \nonumber\\
&=&
\psi\cdot(Ls) + \sum_{j=1}^n \varepsilon_j\left(\partial_{e_j}(
  \psi\cdot s(e_j)) -\psi\cdot s(\nabla_{e_j}e_j)\right) .
\label{eq:nablaL}
\end{eqnarray}
Let $V_1$ be the unique $\K$-valued vector field characterized by
$\la V_1,X\ra = \psi\cdot s(X)$ for every $X\in{C^\infty}(M,TM)$.
Then
\begin{eqnarray*}
\div(V_1) 
&=& 
\sum_{j=1}^n \epsilon_j \la \nabla_{e_j}V_1\,,\, e_j\ra\\
&=& 
\sum_{j=1}^n \epsilon_j \left( \partial_{e_j}\la V_1\,,\, e_j\ra -
\la V_1\,,\, \nabla_{e_j}e_j\ra \right)\\
&=&
\sum_{j=1}^n \epsilon_j \left( \partial_{e_j}(\psi\cdot s(e_j)) -
\psi\cdot s(\nabla_{e_j}e_j) \right).
\end{eqnarray*}
Plugging this into (\ref{eq:nablaL}) yields
$$
\la\nabla \psi,s\ra=\psi\cdot Ls+\div(V_1).
$$
In particular, if $v\in{C^\infty}(M,E)$ we get for $s:=\nabla v \in
{C^\infty}(M,T^*M\otimes E)$
$$
\la\nabla \psi,\nabla v\ra= \psi\cdot L\nabla v+\div(V_1)
= \psi\cdot\Box^\nabla v +\div(V_1),
$$
hence 
\begin{equation}
\psi\cdot\Box^\nabla v = \la\nabla \psi,\nabla v\ra - \div(V_1)
\label{eq:nablastern}
\end{equation}
where $\la V_1,X\ra = \psi\cdot \nabla_X v$ for all $X\in C^\infty(M,TM)$.
Similarly, we obtain
$$
(\Box^\nabla \psi)\cdot v  = \la\nabla \psi,\nabla v\ra - \div(V_2)
$$
where $V_2$ is the vector field characterized by $\la V_2,X\ra =
(\nabla_X\psi)\cdot v$ for all $X\in C^\infty(M,TM)$. 
Thus 
$$
\psi\cdot\Box^\nabla v = (\Box^\nabla \psi)\cdot v  - \div(V_1) +\div(V_2)
= (\Box^\nabla \psi)\cdot v + \div(W)
$$
where $W=V_2-V_1$.
Since $\nabla$ is the $P$-compatible connection on $E$ we have $P=\Box^\nabla
+ B$ for some $B\in C^\infty(M,\End(E))$, see Lemma~\ref{canonicalconnection}.
Thus
$$
\psi\cdot Pv =\psi\cdot\Box^\nabla v + \psi \cdot Bv
=  (\Box^\nabla \psi)\cdot v + \div(W) + (B^*\psi) \cdot v.
$$
If $\psi$ or $v$ has compact support, then we can integrate $\psi\cdot Pv$ and
the divergence term vanishes.
Therefore
$$
\int_M \psi\cdot Pv\, \dV = \int_M \left((\Box^\nabla \psi)\cdot v+ (B^*\psi)
  \cdot v\right) \dV.
$$
Thus $\Box^\nabla \psi+ B^*\psi = P^*\psi$ and $\psi\cdot Pv = P^*\psi \cdot v
+ \div(W)$ as claimed.
\end{proof}

\begin{lemma}\label{greenformel}
Let $E$ be a  vector bundle over a timeoriented Lorentzian manifold $M$
and let $P$ be a normally hyperbolic operator acting on sections in $E$.
Let $\nabla$ be the $P$-compatible connection on $E$. 
Let ${\Omega}\subset M$ be a relatively compact causal domain satisfying the
conditions of Lemma~\ref{Klemma}.
Let $S$ be a smooth spacelike Cauchy hypersurface in ${\Omega}$. 
Denote by $\mathfrak{n}$ the future directed (timelike) unit normal vector
field along 
$S$. 

For every $x\in {\Omega}$ let $F_\pm^\Omega(x)$ be the fundamental solution
for $P^*$ at $x$ with support in $J_\pm^\Omega(x)$ constructed in
Proposition~\ref{fundamentalexist}.
\indexn{fundamental solution}

Let $u\in{C^\infty}({\Omega},E)$ be a solution of $Pu=0$ on ${\Omega}$.
Set $u_0:=u_{|_S}$ and $u_1:=\nabla_\mathfrak{n} u$.

Then for every $\phi\in\mathcal{D}({\Omega},E^*)$,
\[
\int_{\Omega}\phi\cdot u\dV
=
\int_S\left((\nabla_\mathfrak{n}(F^\Omega[\phi]))\cdot u_0 
-(F^\Omega[\phi])\cdot u_1\right)\dA,
\] 
where $F^\Omega[\phi]\in{C^\infty}({\Omega},E^*)$ is defined
as a distribution by 
\[ 
(F^\Omega[\phi])[w] :=
\int_\Omega\phi(x)\cdot(F_+^\Omega(x)[w]-F_ -^\Omega(x)[w])\dV(x)
\]
for every $w\in\mathcal{D}({\Omega},E)$. 

\end{lemma}

\begin{proof}
Fix $\phi\in\DD(\Omega,E^*)$.
We consider the distribution $\psi$ defined by
$\psi[w]:=\int_\Omega\phi(x)\cdot F_+^\Omega(x)[w]\dV$ for every
$w\in\mathcal{D}({\Omega},E)$. 
By Theorem~\ref{thminhomogen} we know that $\psi\in{C^\infty}({\Omega},E^*)$,
has its support contained in $J_+^\Omega(\mathrm{supp}(\phi))$ and satisfies
$P^*\psi=\phi$.

Let $W$ be the vector field from Lemma~\ref{PminusP*} with $u$ instead of $v$.
Since by Corollary~\ref{cJ+Spastcompact} the subset
$J_+^\Omega(\mathrm{supp}(\phi))\cap J_-^\Omega(S)$ of $\Omega$ is compact,
Theorem~\ref{thm:gauss} applies to $D:=I_-^\Omega(S)$ and the vector field
$W$:
\be
\int_D\left((P^*\psi)\cdot u-\psi\cdot(Pu)\right)\dV
&=&
-\int_D\div(W)\dV\\
&=&
-\underbrace{\la\mathfrak{n},\mathfrak{n}\ra}_{=-1}\int_{\partial D} \la
W,\mathfrak{n}\ra\dA\\ 
&=&
\int_{\partial D}\left((\nabla_\mathfrak{n}\psi)\cdot u-\psi\cdot
(\nabla_\mathfrak{n} 
  u)\right)\dA\\
&=&
\int_{S}\left((\nabla_\mathfrak{n}\psi)\cdot u-\psi\cdot (\nabla_\mathfrak{n}
  u)\right)\dA.
\ee
On the other hand,
$$
\int_D\left((P^*\psi)\cdot u-\psi\cdot(Pu)\right)\dV
=
\int_{I_-^\Omega(S)}((\underbrace{P^*\psi}_{=\phi})\cdot u
-\psi\cdot(\underbrace{Pu}_{=0}))\dV
=
\int_{I_-^\Omega(S)} \phi\cdot u\dV.
$$
Thus
\begin{equation}
\int_{I_-^\Omega(S)} \phi\cdot u\dV =
\int_{S}\left((\nabla_\mathfrak{n}\psi)\cdot u-\psi\cdot (\nabla_\mathfrak{n}
u)\right)\dA. 
\label{eq:green1}
\end{equation}
Similarly, using $D=I_+^\Omega(S)$ and $\psi'[w]:=\int_\Omega\phi(x)\cdot
F_-^\Omega(x)[w]\dV$ for any $w\in\mathcal{D}({\Omega},E)$ one gets
\begin{equation}
\int_{I_+^\Omega(S)}\phi\cdot u\dV
=
\int_S\left(\psi'\cdot (\nabla_\mathfrak{n} u)-(\nabla_\mathfrak{n}\psi')\cdot
u\right)\dA. 
\label{eq:green2}
\end{equation}
The different sign is caused by the fact that $\mathfrak{n}$ is the {\em
  interior} unit 
normal to $I_+^\Omega(S)$.
Adding (\ref{eq:green1}) and (\ref{eq:green2}) we get
\[
\int_{\Omega}\phi\cdot u\dV
=
\int_S\left((\nabla_\mathfrak{n}(\psi-\psi'))\cdot
u-(\psi-\psi')\cdot(\nabla_\mathfrak{n} 
u)\right)\dA,
\] 
which is the desired result.
\end{proof}

\begin{cor}\label{cor:suppu}
Let $\Omega$, $u$, $u_0$, and $u_1$ be as in Lemma~\ref{greenformel}.
Then
$$
\supp(u) \subset J^\Omega(K)
$$
where $K=\supp(u_0)\cup\supp(u_1)$.
\end{cor}

\begin{proof}
Let $\phi\in\DD(\Omega,E^*)$.
From Theorem~\ref{thminhomogen} we know that $\supp(F^\Omega[\phi])\subset
J^\Omega(\supp(\phi))$.
Hence if, under the hypotheses of Lemma~\ref{greenformel}, 
\begin{equation}
\supp(u_j)\cap J^\Omega(\supp(\phi))=\emptyset
\label{eq:suppu}
\end{equation}
for $j=0,1$, then $\int_\Omega\phi\cdot u\dV=0$. 
Equation~(\ref{eq:suppu}) is equivalent to 
$$
\supp(\phi)\cap J^\Omega(\supp(u_j))=\emptyset .
$$
Thus $\int_\Omega\phi\cdot u\dV=0$ whenever the support of the test section
$\phi$ is disjoint from $J^\Omega(K)$.
We conclude that $u$ must vanish outside $J^\Omega(K)$.
\end{proof}

\begin{cor}\label{cauchyeindeutigkeit} 
Let $E$ be a  vector bundle over a globally hyperbolic Lorentzian manifold 
$M$. 
Let $\nabla$ be a connection on $E$ and let $P=\Box^\nabla+B$ be a normally
hyperbolic operator acting on sections in $E$. 
Let $S$ be a smooth spacelike Cauchy hypersurface in $M$, and let
$\mathfrak{n}$ be the 
future directed (timelike) unit normal vector field along $S$.

If $u\in C^\infty(M,E)$ solves 
\[
\left\{\begin{array}{cccl}
Pu&=&0&\textrm{ on }M,\\
u&=&0&\textrm{ along }S,\\
\nabla_\mathfrak{n} u&=&0&\textrm{ along }S,
\end{array}\right.
\] 
then $u=0$ on $M$.
\indexn{Cauchy problem>defemidx}
\end{cor}

\begin{proof}
By Theorem~\ref{globhyp2} there is a foliation of $M$ by spacelike smooth
Cauchy hypersurfaces $S_t$ ($t\in\R$) with $S_0=S$.
Extend $\mathfrak{n}$ smoothly to all of $M$ such that
$\mathfrak{n}_{|_{S_t}}$ is the unit 
future directed (timelike) normal vector field on $S_t$ for every $t\in \R$.
Let $p\in M$.
We show that $u(p)=0$.

Let $T\in\R$ be such that $p\in S_T$.
Without loss of generality let $T>0$ and let $p$ be in the causal future of
$S$. 
Set 
\[
t_0:=\sup\Big\{t\in [0,T]\,\Big|\, u\textrm{ vanishes on
  }J_-^M(p)\cap (\buil{\cup}{0\leq\tau\leq t}S_\tau)\Big\}.
\]

\begin{center}
\input{fig-uvonpgleichnull}
\end{center}

We will show that $t_0=T$ which implies in particular $u(p)=0$.

Assume $t_0<T$.
For each $x\in J_-^M(p)\cap S_{t_0}$ we may, according to
Lemma~\ref{hyperflaechekausal}, choose a relatively compact causal
neighborhood $\Omega$ of $x$ in $M$ satisfying the hypotheses of Lemma
\ref{Klemma} and such that $S_{t_0}\cap\Omega$ is a Cauchy hypersurface of
$\Omega$.

\begin{center}
\input{fig-suppphiinJ+S}
\end{center}

Put $u_0:=u_{|_{S_{t_0}}}$ and $u_1:=(\nabla_\mathfrak{n} u)_{|_{S_{t_0}}}$.
If $t_0=0$, then $u_0=u_1=0$ on $S=S_0$ by assumption.
If $t_0>0$, then $u_0=u_1=0$ on $S_{t_0}\cap J_-^M(p)$ because $u\equiv 0$ on 
$J_-^M(p)\cap (\buil{\cup}{0\leq\tau\leq t}S_\tau)$.
Corollary~\ref{cor:suppu} implies $u=0$ on $J_-^M(p)\cap
J_+^\Omega(S_{t_0}\cap\Omega)$.

By Corollary~\ref{cJ+Spastcompact} the intersection $S_{t_0}\cap J_-^M(p)$ is
compact.
Hence it can be covered by finitely many open subsets $\Omega_i$, $1\leq i\leq
N$, satisfying the conditions of $\Omega$ above.
Thus $u$ vanishes identically on $(\Omega_1 \cup \cdots \cup \Omega_N) \cap
J_-^M(p) \cap J_+^M(S_{t_0})$.
Since $(\Omega_1 \cup \cdots \cup \Omega_N) \cap J_-^M(p)$ is an open
neighborhood of the compact set $S_{t_0}\cap J_-^M(p)$ in $J_-^M(p)$
there exists an $\varepsilon>0$ such that $S_t\cap J_-^M(p)\subset \Omega_1
\cup \cdots \cup \Omega_N$ for every $t\in [t_0,t_0+\varepsilon)$.

\begin{center}
\input{fig-Stexists}
\end{center}

Hence $u$ vanishes on $S_t\cap J_-^M(p)$ for all $t\in[t_0,t_0+\varepsilon)$.
This contradicts the maximality of $t_0$.
\end{proof}

Next we prove existence of solutions to the Cauchy problem on
small domains. 
Let $\Omega\subset M$ satisfy the hypotheses of Lemma~\ref{Klemma}.
In particular, $\Omega$ is relatively compact, causal, and has ``small
volume''.
Such domains will be referred to as \defem{RCCSV} (for ``Relatively Compact
Causal with Small Volume'').
Note that each point in a Lorentzian manifold possesses a basis of
RCCSV-neighborhoods. 
Since causal domains are contained in convex domains by definition and convex
domains are contractible, the vector bundle $E$ is trivial over any
RCCSV-domain $\Omega$.
\indexn{RCCSV-domain>defemidx}
We shall show that one can uniquely solve the Cauchy problem on every
RCCSV-domain with Cauchy data on a fixed Cauchy hypersurface in $\Omega$.

\begin{prop}\label{cauchylokalexistenz}
\indexn{Cauchy problem}
Let $M$ be a timeoriented Lorentzian manifold and let $S\subset M$ be a 
spacelike hypersurface.
Let $\mathfrak{n}$ be the future directed timelike unit normal field along $S$.

Then for each RCCSV-domain $\Omega\subset M$ such that $S\cap\Omega$ is a
(spacelike) Cauchy hypersurface in $\Omega$, the following holds: 

For each $u_0, u_1 \in \DD(S\cap\Omega,E)$ and for each $f\in\DD(\Omega,E)$
there exists a unique $u\in C^\infty(\Omega,E)$ satisfying 
\[
\left\{\begin{array}{cccl}
Pu&=&f&\textrm{ on }M,\\
u&=&u_0&\textrm{ along }S,\\
\nabla_\mathfrak{n} u&=&u_1&\textrm{ along }S.
\end{array}\right.
\] 
Moreover, $\supp(u) \subset J^M(K)$ where 
$K=\supp(u_0)\cup \supp(u_1) \cup \supp(f)$.
\end{prop}

\begin{proof}
Let $\Omega\subset M$ be an RCCSV-domain such that $S\cap\Omega$ is a Cauchy
hypersurface in $\Omega$. 
Corollary~\ref{cauchyeindeutigkeit} can then be applied on $\Omega$: 
If $u$ and $\tilde u$ are two solutions of the Cauchy problem, then $P(u-\tilde
u)=0$, $(u-\tilde u)|_S=0$, and $\nabla_\mathfrak{n}(u-\tilde u)=0$.
Corollary~\ref{cauchyeindeutigkeit} implies $u-\tilde u=0$ which shows
uniqueness. 
It remains to show existence.

Since causal domains are globally hyperbolic we may apply
Theorem~\ref{globhyp2} and find an isometry $\Omega=\R\times (S\cap\Omega)$
where the metric takes the form $-\beta dt^2+ g_t$.
Here $\beta:\Omega\rightarrow\R_+^*$ is smooth, each
$\{t\}\times(S\cap\Omega)$ is a smooth spacelike Cauchy hypersurface in
$\Omega$, and $S\cap\Omega$ corresponds to $\{0\}\times (S\cap\Omega)$. 
Note that the future directed unit normal vector field $\mathfrak{n}$ along
$\{t\}\times 
(S\cap\Omega)$ is given by
$\mathfrak{n}(\cdot)=\frac{1}{\sqrt{\beta(t,\cdot)}} \frac{\del}{\del t}$.

Now let $u_0, u_1 \in \DD(S\cap\Omega,E)$ and $f\in\DD(\Omega,E)$.
We trivialize the bundle $E$ over $\Omega$ and identify sections in $E$ with 
$\K^r$-valued functions where $r$ is the rank of $E$.

Assume for a moment that $u$ were a solution to the Cauchy problem of the form
$u(t,x) = \sum_{j=0}^\infty t^j u_j(x)$ where $x\in S\cap\Omega$.
Write $P = \frac{1}{\beta}\frac{\partial^2}{\partial t^2} + Y$ where $Y$ is a
differential operator containing $t$-derivatives only up to order 1.
Equation
\begin{equation} 
\label{cauchy1}
f = Pu = \left(\frac{1}{\beta}\frac{\partial^2}{\partial t^2} + Y\right)u 
= \frac{1}{\beta(t,\cdot)}\sum_{j=2}^\infty j(j-1)t^{j-2}u_j + Yu 
\end{equation}
evaluated at $t=0$ gives
$$
\frac{2}{\beta(0,x)}u_2(x) = -Y(u_0+tu_1)(0,x) + f(0,x)
$$
for every $x\in S\cap\Omega$.
Thus $u_2$ is determined by $u_0$, $u_1$, and $f|_S$.
Differentiating (\ref{cauchy1}) with respect to $\frac{\partial}{\partial t}$ 
and repeating the procedure shows that each $u_j$ is recursively determined by
$u_0, \ldots, u_{j-1}$ and the normal derivatives of $f$ along $S$.

Now we drop the assumption that we have a $t$-power series $u$ solving the
problem but we {\em define} the $u_j$, $j\geq 2$, by these recursive
relations.
Then $\supp(u_j) \subset \supp(u_0)\cup\supp(u_1)\cup(\supp(f)\cap S)$ for all
$j$. 

Let $\sigma : \R\to\R$ be a smooth function such that
$\sigma|_{[-1/2,1/2]}\equiv 1$ and $\sigma\equiv 0$ outside $[-1,1]$.
We claim that we can find a sequence of $\varepsilon_j \in (0,1)$ such
that 
\begin{equation}
\hat{u}(t,x) := \sum_{j=0}^\infty \sigma({t}/{\varepsilon_j})t^j u_j(x)
\label{udef}
\end{equation}
defines a smooth section that can be differentiated termwise.

By Lemma~\ref{CkNormvonProdukt} we have for $j>k$
$$
\|\sigma(t/\varepsilon_j)t^j u_j(x)\|_{C^k(\Omega)}
\leq
c(k) \cdot\left\|\sigma(t/\varepsilon_j)t^j\right\|_{C^k(\R)} 
\cdot \|u_j\|_{C^k(S)}.
$$
Here and in the following $c(k)$, $c'(k,j)$, and $c''(k,j)$ denote universal
constants depending only on $k$ and $j$.
By Lemma~\ref{sigmadoof} we have for $l\leq k$ and $0<\varepsilon_j\leq 1$
\begin{eqnarray*}
\left\|\frac{d^l}{dt^l}(\sigma(t/\varepsilon_j)t^j)\right\|_{C^0(\R)}
&\leq&
\varepsilon_j\, c'(l,j)\, \|\sigma\|_{C^l(\R)},
\end{eqnarray*}
thus
$$
\|\sigma(t/\varepsilon_j)t^j u_j(x)\|_{C^k(\Omega)}
\leq
\varepsilon_j\, c''(k,j)\, \|\sigma\|_{C^k(\R)}\, \|u_j\|_{C^k(S)} .
$$
Now we choose $0<\varepsilon_j\leq 1$ so that $\varepsilon_j\, c''(k,j)\, 
\|\sigma\|_{C^k(\R)}\, \|u_j\|_{C^k(S)} \leq 2^{-j}$ for all $k<j$.
Then the series (\ref{udef}) defining $\hat{u}$ converges absolutely in the 
$C^k$-norm for all $k$.
Hence $\hat{u}$ is a smooth section with compact support and can be
differentiated 
termwise. 
From the construction of $\hat{u}$ one sees that $\supp(\hat{u})\subset
J^M(K)$. 

By the choice of the $u_j$ the section $P\hat{u}-f$ vanishes to infinite order
along $S$.
Therefore
$$
w(t,x) := \left\{
  \begin{array}{cl}
  (P\hat{u}-f)(t,x), & \mbox{if $t\geq 0$,}\\
  0          , & \mbox{if $t< 0$,}
  \end{array}
\right.
$$
defines a smooth section with compact support.
By Theorem~\ref{thminhomogen} (which can be applied since the hypotheses of
Lemma~\ref{Klemma} are fulfilled) we can solve the equation 
$P\tilde{u}=w$ with a smooth section $\tilde{u}$ having past compact
support.
Moreover, $\supp(\tilde{u}) \subset J_+^M(\supp(w)) \subset
J_+^M(\supp(\hat{u})\cup \supp(f)) \subset J^M(K)$.

Now $u_+:= \hat{u}-\tilde{u}$ is a smooth section such that
$Pu_+ = P\hat{u} - P\tilde{u} = w+f-w = f$ on $J_+^\Omega(S\cap\Omega) =
\{t\geq 0\}$.

The restriction of $\tilde{u}$ to $I_-^\Omega(S)$ has past compact support and 
satisfies $P\tilde{u}=0$ on $I_-^\Omega(S)$, thus by Theorem~\ref{kernP} 
$\tilde{u}=0$ on $I_-^\Omega(S)$.
Thus $u_+$ coincides with $\hat{u}$ to infinite order along $S$.
In particular, $u_+|_S = \tilde{u}|_S = u_0$ and $\nabla_\mathfrak{n} u_+ =
\nabla_\mathfrak{n}\tilde{u} = u_1$. 
Moreover, $\supp(u_+) \subset \supp(\hat{u}) \cup \supp(\tilde{u}) \subset
J^M(K)$.
Thus $u_+$ has all the required properties on $J_+^M(S)$.

Similarly, one constructs $u_-$ on $J_-^M(S)$. 
Since both $u_+$ and $u_-$ coincide to infinite order with $\hat{u}$ along $S$
we obtain the smooth solution by setting
$$
u(t,x) := \left\{
  \begin{array}{cl}
  u_+(t,x), & \mbox{if $t\geq 0$,}\\
  u_-(t,x), & \mbox{if $t\leq 0$.}
  \end{array}
\right.
$$
\end{proof}

\Remark{
It follows from Lemma~\ref{hyperflaechekausal} that every point $p$ on a
spacelike hypersurface $S$ possesses a RCCSV-neighborhood $\Omega$ such that
$S\cap\Omega$ is a Cauchy hypersurface in $\Omega$.
Hence Proposition~\ref{cauchylokalexistenz} guarantees the local existence of
solutions to the Cauchy problem.
}

In order to show existence of solutions to the Cauchy problem on globally
hyperbolic manifolds we need some preparation.
Let $M$ be globally hyperbolic.
We write $M=\R\times S$ and suppose the metric is of the form $-\beta dt^2
+ g_t$ as in Theorem~\ref{globhyp}.
Hence $M$ is foliated by the smooth spacelike Cauchy hypersurfaces
$\{t\}\times S=:S_t$, $t\in\R$.
Let $p\in M$.
Then there exists a unique $t$ such that $p\in S_t$.
For any $r>0$ we denote by $B_r(p)$ the open ball
\emph{in} $S_t$ of radius $r$ about $p$ with respect to the \emph{Riemannian}
metric $g_t$ on $S_t$.
\indexs{B*@$B_r(p)$, open Riemannian ball of radius $r$ about $p$ in Cauchy
  hypersurface} 
Then $B_r(p)$ is open as a subset of $S_t$ but not as a subset of $M$.

Recall that $D(A)$ denotes the Cauchy development of a subset $A$ of $M$ (see
Definition~\ref{ddefCauchydev}). 
\indexn{Cauchy development of a subset}

\begin{lemma}\label{lrhouhs}
The function $\rho:M\to(0,\infty]$ defined by
\[
\rho(p):=\sup\{r>0\,|\, D(B_r(p))\textrm{ is RCCSV}\},
\]
is lower semi-continuous on $M$.
\end{lemma}

\begin{proof}
First note that $\rho$ is well-defined since every point has a
RCCSV-neighborhood. 
Let $p\in M$ and $r>0$ be such that $\rho(p)> r$.
Let $\epsilon>0$.
We want to show $\rho(p')>r-\epsilon$ for all $p'$ in a neighborhood of $p$.

For any point $p'\in D(B_r(p))$ consider 
\[
\lambda(p'):=\sup\{r'>0\,|\, B_{r'}(p')\subset D(B_r(p))\} .
\]
{\em Claim:
There exists a neighborhood $V$ of $p$ such that for every $p'\in V$ one has
$\lambda(p')> r-\epsilon$.}

\begin{center}
\input{fig-rhouhs}
\end{center}

Let us assume the claim for a moment.
Let $p'\in V$.
Pick $r'$ with $r-\epsilon < r' < \lambda(p')$.
Hence $B_{r'}(p')\subset D(B_r(p))$.
By Remark~\ref{rem:Cauchydev} we know $D(B_{r'}(p'))\subset D(B_r(p))$. 
Since $D(B_r(p))$ is RCCSV the subset $D(B_{r'}(p'))$ is RCCSV as well.
Thus $\rho(p') \geq r' > r-\epsilon$.
This then concludes the proof.

It remains to show the claim.
Assume the claim is false.
Then there is a sequence $(p_i)_i$ of points in $M$ converging to $p$ such
that $\lambda(p_i) \leq r-\epsilon$ for all $i$.
Hence for $r':=r-\epsilon/2$ we have $B_{r'}(p_i) \not\subset D(B_r(p))$.
Choose $x_i \in B_{r'}(p_i) \setminus D(B_r(p))$.

The closed set $\ovl{B}_r(p)$ is contained in the compact set
$\ovl{D}(B_r(p))$ and therefore compact itself.
Thus $[-1,1]\times \ovl{B}_r(p)$ is compact.
For $i$ sufficiently large $B_{r'}(p_i) \subset [-1,1]\times \ovl{B}_r(p)$
and therefore $x_i \in [-1,1]\times \ovl{B}_r(p)$.
We pass to a convergent subsequence $x_i \to x$.
Since $p_i \to p$ and $x_i\in \ovl{B}_{r'}(p_i)$ we have
$x \in \ovl{B}_{r'}(p)$.
Hence $x\in B_r(p)$.
Since $D(B_r(p))$ is an open neighborhood of $x$ we must have $x_i\in
D(B_r(p))$ for sufficiently large $i$.
This contradicts the choice of the $x_i$.
\end{proof}

For every $r>0$ and $q\in M=\R\times S$ consider
\[
\theta_r(q):=\sup\{\eta>0\,|\, J^M(\ovl{B}_{r/2}(q))\cap
([t_0-\eta,t_0+\eta]\times S) \subset D(B_{r}(q))\}.
\indexs{thetar@$\theta_r$, auxiliary function in proof of global existence of
  fundamental solutions}
\]

\begin{center}
\input{fig-deftheta}
\end{center}

\Remark{
There exist $\eta>0$ with $J^M(\ovl{B}_{r/2}(q))\cap
([t_0-\eta,t_0+\eta]\times S) \subset D(B_{r}(q))$.
Hence $\theta_r(q)>0$.

One can see this as follows.
If no such $\eta$ existed, then there would be points $x_i \in
J^M(\ovl{B}_{r/2}(q))\cap ([t_0-\frac{1}{i},t_0+\frac{1}{i}]\times S)$
but $x_i\not\in D(B_{r}(q))$, $i\in\N$.
All $x_i$ lie in the compact set $J^M(\ovl{B}_{r/2}(q))\cap
([t_0-1,t_0+1]\times S)$.
Hence we may pass to a convergent subsequence $x_i \to x$.
Then $x\in J^M(\ovl{B}_{r/2}(q))\cap (\{t_0\}\times S)=\ovl{B}_{r/2}(q)$.
Since $D(B_r(q))$ is an open neighborhood of $\ovl{B}_{r/2}(q)$ we must have
$x_i\in D(B_r(q_0))$ for sufficiently large $i$ in contradiction to the choice
of the $x_i$. 
}

\begin{lemma}\label{thetaruhs}
The function $\theta_r:M\to(0,\infty]$ is lower semi-continuous.
\end{lemma}

\begin{proof}
Fix $q\in M$.
Let $\epsilon>0$.
We need to find a neighborhood $U$ of $q$ such that for all $q'\in U$ we have
$\theta_r(q')\geq \theta_r(q)-\epsilon$.

Put $\eta:=\theta_r(q)$ and choose $t_0$ such that $q\in S_{t_0}$.
Assume no such neighborhood $U$ exists.
Then there is a sequence $(q_i)_i$ in $M$ such that $q_i \to q$ and
$\theta_r(q_i) < \eta-\epsilon$ for all $i$.
All points to be considered will be contained in the compact set
$([-T,T]\times S) \cap J^M(\ovl{B}_r(q))$ for sufficiently big $T$ and
sufficiently large $i$.
Let $q_i\in S_{t_i}$.
Then $t_i\to t_0$ as $i\to\infty$.

Choose $x_i \in J^M(\ovl{B}_{r/2}(q_i)) \cap \left([t_i-\eta+{\epsilon},
t_i+\eta-{\epsilon}] \times S\right)$ but $x_i \not\in D(B_r(q_i))$.
This is possible because of $\theta_r(q_i) < \eta-\epsilon$.
Choose $y_i\in \ovl{B}_{r/2}(q_i)$ such that $x_i\in J^M(y_i)$.

\begin{center}
\input{fig-konstruktxi}
\end{center}

After passing to a subsequence we may assume $x_i \to x$ and $y_i \to y$.
From $q_i \to q$ and $y_i\in \ovl{B}_{r/2}(q_i)$ we deduce $y \in
\ovl{B}_{r/2}(q)$. 
Since the causal relation ``$\leq$'' on a globally hyperbolic manifold is
closed we conclude from $x_i \to x$, $y_i \to y$, and $x_i\in J^M(y_i)$ that
$x\in J^M(y)$. 
Thus $x\in J^M(\ovl{B}_{r/2}(q))$.
Obviously, we also have $x\in [t_0-\eta+{\epsilon}, t_0+\eta-{\epsilon}]
\times S$.
From $\theta_r(q)=\eta>\eta-\epsilon$ we conclude $x\in D(B_r(q))$.

Since $x_i\not\in D(B_r(q_i))$ there is an inextendible causal curve $c_i$
through $x_i$ which does not intersect $B_r(q_i)$.
Let $z_i$ be the intersection of $c_i$ with the Cauchy hypersurface $S_{t_i}$.
After again passing to a subsequence we have $z_i \to z$ with $z\in S_{t_0}$.
From $z_i\not\in B_r(q_i)$ we conclude $z\not\in B_r(q)$.
Moreover, since $c_i$ is causal we have $x_i\in J^M(z_i)$.
The causal relation ``$\leq$'' is closed, hence $x\in J^M(z)$.
Thus there exists an inextendible causal curve $c$ through $x$ and $z$.
This curve does not meet $B_r(q)$ in contradiction to $x\in D(B_r(q))$.
\end{proof}

\begin{lemma}\label{lem:streifen}
For each compact subset $K\subset M$ there exists $\delta>0$ such that for
each $t\in\R$ and any $u_0, u_1\in\DD(S_t,E)$ with $\supp(u_j)\subset K$,
$j=1,2$, there is a smooth solution $u$ of $Pu=0$ defined on
$(t-\delta,t+\delta)\times S$ satisfying $u|_{S_t}=u_0$ and
$\nabla_\mathfrak{n} 
u|_{S_t}=u_1$.
Moreover, $\supp(u) \subset J^M(K\cap S_t)$.
\end{lemma}

\begin{proof}
By Lemma~\ref{lrhouhs} the function $\rho$ admits a minimum on the compact set
$K$.
Hence there is a constant $r_0>0$ such that $\rho(q)> 2r_0$ for all $q\in
K$.
Choose $\delta>0$ such that $\theta_{2r_0}>\delta$ on $K$.
This is possible by Lemma~\ref{thetaruhs}.

Now fix $t\in\R$.
Cover the compact set $S_t\cap K$ by finitely many balls $B_{r_0}(q_1),
\ldots, B_{r_0}(q_N)$, $q_j\in S_t\cap K$.
Let $u_0, u_1\in\DD(S_t,E)$ with $\supp(u_j)\subset K$.
Using a partition of unity write $u_0 = u_{0,1}+\ldots+u_{0,N}$ with
$\supp(u_{0,j})\subset B_{r_0}(q_j)$ and similarly $u_1 =
u_{1,1}+\ldots+u_{1,N}$.
The set $D(B_{2r_0}(q_j))$ is RCCSV.
By Proposition~\ref{cauchylokalexistenz} we can find a solution $w_j$
of $Pw_j=0$ on $D(B_{2r_0}(q_j))$ with $w_j|_{S_t}=u_{0,j}$ and
$\nabla_\mathfrak{n} w_j|_{S_t}=u_{1,j}$. 
Moreover, $\supp(w_j) \subset J^M(B_{r_0}(q_j))$.
From $J^M(B_{r_0}(q_j)) \cap (t-\delta,t+\delta)\times S \subset
D(B_{2r_0}(q_j))$ we see that $w_j$ is defined on $J^M(B_{r_0}(q_j))\cap
(t-\delta,t+\delta)\times S$.
Extend $w_j$ smoothly by zero to all of $(t-\delta,t+\delta)\times S$.
Now $u := w_1 + \ldots+w_N$ is a solution defined on
$(t-\delta,t+\delta)\times S$ as required.
\end{proof}

Now we are ready for the main theorem of this section.

\begin{thm}\label{cauchyglobhyp}
\indexn{Cauchy problem}
Let $M$ be a globally hyperbolic Lorentzian manifold and let
 $S\subset M$ be a spacelike Cauchy hypersurface.
Let $\mathfrak{n}$ be the future directed timelike unit normal field along $S$.
Let $E$ be a vector bundle over $M$ and let $P$ be a normally hyperbolic
operator acting on sections in $E$.

Then for each  $u_0, u_1 \in \DD(S,E)$ and for each
$f\in\DD(M,E)$ there exists a unique $u\in C^\infty(M,E)$ satisfying
$Pu=f$, $u|_S = u_0$, and $\nabla_\mathfrak{n} u|_S = u_1$.

Moreover, $\supp(u) \subset J^M(K)$ where 
$K=\supp(u_0)\cup \supp(u_1) \cup \supp(f)$.
\end{thm}

\begin{proof}
Uniqueness of the solution follows directly from
Corollary~\ref{cauchyeindeutigkeit}. 
We have to show existence of a solution and the statement on its support.

Let  $u_0, u_1 \in \DD(S,E)$ and $f\in\DD(M,E)$.
Using a partition of unity $(\chi_j)_{j=1,\ldots,m}$ we can write 
$u_0 = u_{0,1} + \ldots + u_{0,m}$,
$u_1 = u_{1,1} + \ldots + u_{1,m}$ and $f = f_{1} + \ldots + f_{m}$
where $u_{0,j} = \chi_j u_0$, $u_{1,j} = \chi_j u_1$, and $f_j = \chi_j f$.
We may assume that each $\chi_j$ (and hence each $u_{i,j}$ and $f_j$) 
have support in an open set as in Proposition~\ref{cauchylokalexistenz}.
If we can solve the Cauchy problem on $M$ for the data $(u_{0,j}, u_{1,j},
f_j)$, then we can add these solutions to obtain one for $u_0$, $u_1$, and $f$.
Hence we can without loss of generality assume that there is an $\Omega$ as in 
Proposition~\ref{cauchylokalexistenz} such that $K := \supp(u_0)\cup
\supp(u_1)\cup\supp(f) \subset \Omega$.

By Theorem~\ref{globhyp2} the spacetime $M$ is isometric to $\R\times S$ with
a Lorentzian metric of the form $-\beta dt^2+ g_t$ where $S$ corresponds to
$\{0\}\times S$, and each $S_t:=\{t\}\times S$ is a spacelike Cauchy
hypersurface in $M$.
Let $u$ be the solution on $\Omega$ as asserted by
Proposition~\ref{cauchylokalexistenz}. 
In particular, $\supp(u) \subset J^M(K)$.
By choosing the partition of unity $(\chi_j)_j$ appropriately we can assume
that $K$ is so small that there exists an $\varepsilon > 0$ such that
$((-\varepsilon,\varepsilon)\times S) \cap J^M(K) \subset
\Omega$ and $K\subset (-\varepsilon,\varepsilon)\times S$.

\begin{center}
\input{fig-cauchyglobhyp1}
\end{center}

Hence we can extend $u$ by $0$ to a smooth solution on all of
$(-\varepsilon,\varepsilon)\times S$. 
Now let $T_+$ be the supremum of all $T$ for which $u$ can be extended
to a smooth solution on $(-\varepsilon,T)\times S$ with support contained in
$J^M(K)$.
On $[\epsilon,T)\times S$ the equation to be solved is simply $Pu=0$ because
$\supp(f) \subset K$.
If we have two extensions $u$ and $\tilde{u}$ for $T<\tilde{T}$, then the
restriction of $\tilde{u}$ to $(-\varepsilon,T)\times S$ must
coincide with $u$ by uniqueness.
Note here that Corollary~\ref{cauchyeindeutigkeit} applies because
$(-\epsilon,T) \times S$ is a globally hyperbolic manifold in its own right.
Thus if we show $T_+=\infty$ we obtain a solution on
$(-\varepsilon,\infty)\times S$.
Similarly considering the corresponding infimum $T_-$ then yields a solution
on all of $M=\R\times S$.

Assume that $T_+<+\infty$.
Put $\hat{K}:=([-\epsilon,T_+]\times S)\cap J^M(K)$.
By Lemma~\ref{cJ+Spastcompact} $\hat{K}$ is compact.
Apply Lemma~\ref{lem:streifen} to $\hat K$ and get $\delta>0$ as in the Lemma.
Fix $t<T_+$ such that $T_+-t<\delta$ and still $K\subset
(-\epsilon,t)\times S$.

On $(t-\delta,t+\delta)\times S$ solve $Pw=0$ with $w|_{S_t}=u|_{S_t}$ and
$\nabla_\mathfrak{n} w|_{S_t}=\nabla_\mathfrak{n} u|_{S_t}$.
This is possible by Lemma~\ref{lem:streifen}.
On $(t-\eta,t+\delta)\times S$ the section $f$ vanishes with $\eta>0$ small
enough.
Thus $w$ coincides with $u$ on $(t-\eta,t)\times S$.
Here again, Corollary~\ref{cauchyeindeutigkeit} applies because
$(t-\eta,t+\delta)\times S$ is a globally hyperbolic manifold in its own right.
Hence $w$ extends the solution $u$ smoothly to $(-\varepsilon,t+\delta)\times
S$.
The support of this extension is still contained in $J^M(K)$ because 
$$
\supp\left(w|_{[t,t+\delta)\times S}\right)
\subset
J^M_+\left(\supp(u|_{S_t})\cup\supp(\nabla_\mathfrak{n} u|_{S_t})\right)
\subset
J^M_+(\hat K \cap S_t)
\subset
J^M_+(J^M_+(K))
=
J^M_+(K).
$$

Since $T_+ < t + \delta$ this contradicts the maximality of $T_+$.
Therefore $T_+=+\infty$.
Similarly, one sees $T_-=-\infty$ which concludes the proof.
\end{proof}

The solution to the Cauchy problem depends continuously on the data.

\begin{thm}\label{cauchystetig}
Let $M$ be a globally hyperbolic Lorentzian manifold and let 
$S\subset M$ be a spacelike Cauchy hypersurface.
Let $\mathfrak{n}$ be the future directed timelike unit normal field along $S$.
Let $E$ be vector bundle over $M$ and let $P$ be a normally hyperbolic
operator acting on sections in $E$.

Then the map $\DD(M,E) \oplus \DD(S,E) \oplus \DD(S,E) \to C^\infty(M,E)$
sending $(f,u_0,u_1)$ to the unique solution $u$ of the Cauchy
problem $Pu=f$, $u|_S=u|_0$, $\nabla_\mathfrak{n} u = u_1$ is linear
continuous. 
\end{thm}

\begin{proof}
The map $\mathcal{P} : C^\infty(M,E) \to C^\infty(M,E) \oplus C^\infty(S,E)
\oplus C^\infty(S,E)$, $u \mapsto (Pu,u|_S,\nabla_\mathfrak{n} u)$, is
obviously linear 
and continuous.
Fix a compact subset $K\subset M$.
Write $\DD_K(M,E) := \{f\in\DD(M,E)\ |\ \supp(f)\subset K \}$,
$\DD_K(S,E) := \{v\in \DD(S,E)\ |\ \supp(v)\subset K\cap S\}$, and
$\mathcal{V}_K := \mathcal{P}^{-1}(\DD_K(M,E) \oplus \DD_K(S,E) \oplus
\DD_K(S,E))$.
Since $\DD_K(M,E) \oplus \DD_K(S,E) \oplus \DD_K(S,E) \subset C^\infty(M,E)
\oplus C^\infty(S,E) \oplus C^\infty(S,E)$ is a closed subset so is
$\mathcal{V}_K \subset C^\infty(M,E)$.
Both $\mathcal{V}_K$ and $\DD_K(M,E) \oplus \DD_K(S,E) \oplus \DD_K(S,E)$ are
therefore Fr\'echet spaces and $\mathcal{P} : \mathcal{V}_K \to \DD_K(M,E)
\oplus \DD_K(S,E) \oplus \DD_K(S,E)$ is linear, continuous and bijective.
By the open mapping theorem \cite[Thm.~V.6, p.~132]{RS} the inverse mapping
$\mathcal{P}^{-1}:\DD_K(M,E) \oplus \DD_K(S,E) \oplus \DD_K(S,E) \to
\mathcal{V}_K\subset C^\infty(M,E)$ is continuous as well.

Thus if $(f_j, u_{0,j},u_{1,j})\to (f, u_{0},u_{1})$ in $\DD(M,E) \oplus
\DD(S,E) \oplus \DD(S,E)$, then we can choose a compact subset $K\subset M$
such that $(f_j, u_{0,j},u_{1,j})\to (f, u_{0},u_{1})$ in $\DD_K(M,E) \oplus 
\DD_K(S,E) \oplus \DD_K(S,E)$ and we conclude $\mathcal{P}^{-1}(f_j,
u_{0,j},u_{1,j})\to \mathcal{P}^{-1}(f, u_{0},u_{1})$. 
\end{proof}

%%%%%%%%%%%%%%%%%%%%%%%%%%%%%%%%%%%%%%%%%%%%%%%%%%%%%%%%%%%%%%%%%%%%%%%%%
\section[Fundamental solutions]{Fundamental solutions on globally hyperbolic
  manifolds} 
\indexn{fundamental solution>defemidx}
%%%%%%%%%%%%%%%%%%%%%%%%%%%%%%%%%%%%%%%%%%%%%%%%%%%%%%%%%%%%%%%%%%%%%%%%%

Using the knowledge about the Cauchy problem which we obtained in the
previous section it is now not hard to find global fundamental solutions 
on a globally hyperbolic manifold.

\begin{thm}\label{globhypexist}
Let $M$ be a globally hyperbolic Lorentzian manifold. 
Let $P$ be a normally hyperbolic operator acting on sections in a 
vector bundle $E$ over $M$.

Then for every $x\in M$ there is exactly one
fundamental solution $F_+(x)$ for $P$ at $x$ with past compact
support and exactly one fundamental solution $F_-(x)$ for $P$ at $x$ with 
future compact support.
\indexs{F*@$F_\pm(x)$, global fundamental solution}
They satisfy
\begin{enumerate}
\item
$\supp(F_\pm(x))\subset J_\pm^M(x)$,
\item
for each $\varphi\in\DD(M,E^*)$ the maps $x\mapsto F_\pm(x)[\varphi]$
are smooth sections in $E^*$ satisfying the differential equation
$P^*(F_\pm(\cdot)[\varphi])=\varphi$.
\end{enumerate}
\end{thm}

\begin{proof}
Uniqueness of the fundamental solutions is a consequence of
Corollary~\ref{fundunique}. 
To show existence fix a foliation of $M$ by spacelike Cauchy hypersurfaces
$S_t$, $t\in\R$ as in Theorem~\ref{globhyp}.
Let $\mathfrak{n}$ be the future directed unit normal field along the leaves
$S_t$. 
Let $\varphi\in\DD(M,E^*)$.
Choose $t$ so large that $\supp(\varphi)\subset I_-^M(S_t)$.
By Theorem~\ref{cauchyglobhyp} there exists a unique $\chi_\varphi\in
C^\infty(M,E^*)$ such that $P^*\chi_\varphi=\varphi$ and $\chi_\varphi|_{S_t}
= (\nabla_\mathfrak{n}\chi_\varphi)|_{S_t} = 0$.

We check that $\chi_\varphi$ does not depend on the choice of $t$.
Let $t<t'$ be such that $\supp(\varphi)\subset I_-^M(S_t)\subset
I_-^M(S_{t'})$. 
Let $\chi_\varphi$ and $\chi_\varphi'$ be the corresponding solutions.
Choose $t_- < t$ so that still $\supp(\varphi)\subset I_-^M(S_{t_-})$.
The open subset $\hat{M} := \bigcup_{\tau> t_-}S_\tau \subset M$
is a globally hyperbolic Lorentzian manifold itself.
Now $\chi_\varphi'$ satisfies $P^*\chi_\varphi'=0$ on $\hat{M}$ with vanishing
Cauchy data on $S_{t'}$.
By Corollary~\ref{cauchyeindeutigkeit} $\chi_\varphi'=0$ on $\hat{M}$. 
In particular, $\chi_\varphi'$ has vanishing Cauchy data on $S_t$ as well.
Thus $\chi_\varphi-\chi_\varphi'$ has vanishing Cauchy data on $S_t$ and
solves $P^*(\chi_\varphi-\chi_\varphi')=0$ on  
all of $M$.
Again by Corollary~\ref{cauchyeindeutigkeit} we conclude
$\chi_\varphi-\chi_\varphi'=0$ on $M$.

Fix $x\in M$.
By Theorem~\ref{cauchystetig} $\chi_\varphi$ depends continuously on
$\varphi$. 
Since the evaluation map $C^\infty(M,E) \to E_x$ is continuous, the
map $\DD(M,E^*) \to E_x^*$, $\varphi \mapsto \chi_\varphi(x)$, is also
continuous.
Thus $F_+(x)[\varphi] := \chi_\varphi(x)$ defines a distribution.
By definition $P^*(F_+(\cdot)[\varphi]) = P^*\chi_\varphi = \varphi$.

Now $P^*\chi_{P^*\varphi} = P^*\varphi$, hence
$P^*(\chi_{P^*\varphi}-\varphi)=0$.
Since both $\chi_{P^*\varphi}$ and $\varphi$ vanish along $S_t$ we conclude
from Corollary~\ref{cauchyeindeutigkeit} $\chi_{P^*\varphi}=\varphi$.
Thus
$$
(PF_+(x))[\varphi] = F_+(x)[P^*\varphi] = \chi_{P^*\varphi}(x) = \varphi(x)
= \delta_x[\varphi].
$$
Hence $F_+(x)$ is a fundamental solution of $P$ at $x$.

It remains to show $\supp(F_+(x)) \subset J_+^M(x)$.
Let $y\in M\setminus J_+^M(x)$.
We have to construct a neighborhood of $y$ such that for each test section
$\varphi\in\DD(M,E^*)$ whose support is contained in this neighborhood
we have $F_+(x)[\varphi] = \chi_\varphi(x) = 0$.
Since $M$ is globally hyperbolic $J_+^M(x)$ is closed and therefore
$J_+^M(x)\cap 
J_-^M(y') = \emptyset$ for all $y'$ sufficiently close to $y$.
We choose $y'\in I_+^M(y)$ and $y''\in I_-^M(y)$ so close that $J_+^M(x)\cap
J_-^M(y') = \emptyset$ and $\left(J_+^M(y'') \cap\bigcup_{t \leq
  t'}S_t\right)\cap 
J_+^M(x) = \emptyset$ where $t'\in\R$ is such that $y'\in S_{t'}$. 

\begin{center}
\input{fig-fundsolglobhyp}
\end{center}

Now $K:=J_-^M(y')\cap J_+^M(y'')$ is a compact neighborhood of $y$.
Let $\varphi\in\DD(M,E^*)$ be such that $\supp(\varphi)\subset K$.
By Theorem~\ref{cauchyglobhyp} $\supp(\chi_\varphi) \subset J_+^M(K)\cup
J_-^M(K) 
\subset J_+^M(y'') \cup J_-^M(y')$.
By the independence of $\chi_\varphi$ of the choice of $t>t'$ we have that
$\chi_\varphi$ vanishes on $\bigcup_{t> t'}S_t$.
Hence $\supp(\chi_\varphi) \subset \left(J_+^M(y'') \cap \bigcup_{t \leq
t'}S_t\right) \cup J_-^M(y')$ and is therefore disjoint from $J_+^M(x)$.
Thus $F_+(x)[\varphi] = \chi_\varphi(x) = 0$ as required.

\end{proof}

%%%%%%%%%%%%%%%%%%%%%%%%%%%%%%%%%%%%%%%%%%%%%%%%%%%%%%%%%%%%%%%%%%%%%%%%%
\section{Green's operators}\label{abschnitt33} 
\indexn{Green's operator>defemidx}
%%%%%%%%%%%%%%%%%%%%%%%%%%%%%%%%%%%%%%%%%%%%%%%%%%%%%%%%%%%%%%%%%%%%%%%%%

Now we want to find ``solution operators'' for a given normally hyperbolic
operator $P$.
More precisely, we want to find operators which are inverses of $P$ when
restricted to suitable spaces of sections.
We will see that existence of such operators is basically equivalent to the
existence of fundamental solutions.

\Definition{ \label{defGfunc}
Let $M$ be a timeoriented connected Lorentzian manifold.
Let $P$ be a normally hyperbolic operator acting on sections in a  
vector bundle $E$ over $M$.
A linear map $G_+ : \DD(M,E) \to C^\infty(M,E)$ satisfying
\begin{itemize}
\item[(i)]
$P\circ G_+ = \id_{\DD(M,E)}$,
\item[(ii)]
$G_+ \circ P|_{\DD(M,E)} = \id_{\DD(M,E)}$,
\item[(iii)]
$\supp(G_+\varphi) \subset J_+^M(\supp(\varphi))$ for all
$\varphi\in\DD(M,E)$,
\end{itemize}
is called an \defem{advanced Green's operator for} $P$.
\indexn{advanced Green's operator>defemidx}
\indexs{G*@$G_+$, advanced Green's operator}
Similarly, a linear map $G_- : \DD(M,E) \to C^\infty(M,E)$ satisfying
(i), (ii), and
\begin{itemize}
\item[(iii')]
$\supp(G_-\varphi) \subset J_-^M(\supp(\varphi))$ for all
$\varphi\in\DD(M,E)$
\end{itemize}
instead of (iii) is called a \defem{retarded Green's operator for} $P$.
\indexn{retarded Green's operator>defemidx}
\indexs{G*@$G_-$, retarded Green's operator}
}

Fundamental solutions and Green's operators are closely related.

\begin{prop}
Let $M$ be a timeoriented connected Lorentzian manifold.
Let $P$ be a normally hyperbolic operator acting on sections in a  
vector bundle $E$ over $M$.

If $F_\pm(x)$ is a family of advanced or retarded fundamental solutions
for the adjoint operator $P^*$ and if $F_\pm(x)$ depend smoothly on $x$ in
the sense that $x\mapsto F_\pm(x)[\varphi]$ is smooth for each test section
$\varphi$ and satisfies the differential equation 
$P(F_\pm(\cdot)[\varphi])=\varphi$, then
\begin{equation}
(G_\pm\varphi)(x) := F_\mp(x)[\varphi]
\label{Greendef}
\end{equation}
defines advanced or retarded Green's operators for $P$ respectively.
Conversely, given Green's operators $G_\pm$ for $P$, then (\ref{Greendef})
defines fundamental solutions for $P^*$ depending smoothly on $x$ and
satisfying $P(F_\pm(\cdot)[\varphi])=\varphi$ for each test section $\varphi$.
\end{prop}

\begin{proof}
Let $F_\pm(x)$ be a family of advanced and retarded fundamental solutions
for the adjoint operator $P^*$ respectively.
Let $F_\pm(x)$ depend smoothly on $x$ and suppose the
differential equation $P(F_\pm(\cdot)[\varphi])=\varphi$ holds.
By definition we have
$$
P(G_\pm\varphi) = P(F_\mp(\cdot)[\varphi]) = \varphi
$$
thus showing (i).
Assertion (ii) follows from the fact that the $F_\pm(x)$ are fundamental 
solutions,
$$
G_\pm(P\varphi)(x) = F_\mp(x)[P\varphi] = P^*F_\mp(x)[\varphi]
= \delta_x[\varphi] = \varphi(x).
$$
To show (iii) let $x\in M$ such that $(G_+\varphi)(x) \not= 0$.
Since $\supp(F_-(x))\subset J_-^M(x)$ the support of $\varphi$ must hit
$J_-^M(x)$.
Hence $x\in J_+^M(\supp(\varphi))$ and therefore $\supp(G_+\varphi)\subset
J_+^M(\supp(\varphi))$.
The argument for $G_-$ is analogous.

The converse is similar.
\end{proof}

Theorem~\ref{globhypexist} immediately yields

\begin{cor}\label{functorsolve}
Let $M$ be a globally hyperbolic Lorentzian manifold.
Let $P$ be a normally hyperbolic operator acting on sections in a  
vector bundle $E$ over $M$.

Then there exist unique advanced and retarded Green's operators $G_\pm: 
\DD(M,E) \to C^\infty(M,E)$ for $P$.
\hfill$\Box$
\end{cor}

\begin{lemma}\label{Greenadjungiert}
Let $M$ be a globally hyperbolic Lorentzian manifold.
Let $P$ be a normally hyperbolic operator acting on sections in a  
vector bundle $E$ over $M$.
Let $G_\pm$ be the Green's operators for $P$ and $G^*_\pm$ the Green's
operators for the adjoint operator $P^*$.
Then
\begin{equation}
\int_M (G^*_\pm\varphi)\cdot\psi \dV = \int_M \varphi\cdot(G_\mp\psi) \dV
\label{Gsa}
\end{equation}
holds for all $\varphi\in\DD(M,E^*)$ and $\psi\in\DD(M,E)$.
\end{lemma}

\begin{proof}
For the Green's operators we have $PG_\pm=\id_{\DD(M,E)}$ and
$P^*G_\pm^*=\id_{\DD(M,E^*)}$ and hence
\begin{eqnarray*} 
\int_M (G_\pm^*\varphi)\cdot\psi \dV 
&=& \int_M (G_\pm^*\varphi)\cdot(PG_\mp\psi) \dV  \\
&=& \int_M (P^*G_\pm^*\varphi)\cdot(G_\mp\psi) \dV \\
&=& \int_M \varphi\cdot(G_\mp\psi) \dV .
\end{eqnarray*}
Notice that $\supp(G_\pm\phi) \cap \supp(G_\mp\psi)\subset
J_\pm^M(\supp(\phi))\cap J_\mp^M(\supp(\psi))$ is compact in a globally
hyperbolic 
manifold so that the partial integration in the second equation is justified.
\end{proof}

\Notation{
We write $\Csc(M,E)$ for the set of all $\phi\in C^\infty(M,E)$ for which
there exists a compact subset $K \subset M$ such that $\supp(\phi) \subset
J^M(K)$.
Obviously, $\Csc(M,E)$ is a vector subspace of $C^\infty(M,E)$.
\indexs{C*@$\protect\Csc(M,E)$, smooth sections with spacelike compact
  support} 

The subscript ``sc'' should remind the reader of ``spacelike compact\indexn{spacelike compact support>defemidx}''.
Namely, if $M$ is globally hyperbolic and $\phi\in\Csc(M,E)$, then for every
Cauchy hypersurface $S\subset M$ the support of $\phi|_{S}$ is contained in 
$S\cap J^M(K)$ hence compact by Corollary~\ref{cJ+Spastcompact}.
In this sense sections in $\Csc(M,E)$ have spacelike compact support.
}

\Definition{
We say a sequence of elements $\phi_j\in \Csc(M,E)$ {\em converges in
 $\Csc(M,E)$ to} $\phi\in \Csc(M,E)$ if there exists a compact subset $K
\subset M$ such that 
$$
\supp(\phi_j), \supp(\phi) \subset J^M(K)
$$ 
for all $j$ and 
$$
\|\phi_j - \phi\|_{C^k(K',E)}\to 0
$$
for all $k\in\N$ and all compact subsets $K'\subset M$.
}

If $G_+$ and $G_-$ are advanced and retarded Green's operators for $P$
respectively, then we get a linear map
$$
G := G_+ - G_- : \DD(M,E) \to \Csc(M,E).
\indexs{G*@$G=G_+ - G_-$}
$$

\begin{thm}\label{thmExSeq}
Let $M$ be a connected timeoriented Lorentzian manifold.
Let $P$ be a normally hyperbolic operator acting on sections in a  
vector bundle $E$ over $M$.
Let $G_+$ and $G_-$ be advanced and retarded Green's operators for $P$
respectively. 

Then the sequence of linear maps
\begin{equation}
0 \to \DD(M,E) \stackrel{P}{\longrightarrow} \DD(M,E) 
\stackrel{G}{\longrightarrow} \Csc(M,E) \stackrel{P}{\longrightarrow}\Csc(M,E)
\label{eqExSeq}
\end{equation}
is a complex, i.~e., the composition of any two subsequent maps is zero.
The complex is exact at the first $\DD(M,E)$.
If $M$ is globally hyperbolic, then the complex is exact everywhere.
\end{thm}

\begin{proof}
Properties (i) and (ii) in Definition~\ref{defGfunc} of Green's operators
directly yield $G\circ P=0$ and $P\circ G=0$, both on $\DD(M,E)$.
Properties (iii) and (iii') ensure that $G$ maps $\DD(M,E)$ to $\Csc(M,E)$.
Hence the sequence of linear maps forms a complex.

Exactness at the first $\DD(M,E)$ means that 
$$
P : \DD(M,E) \to \DD(M,E)
$$
is injective.
To see injectivity let $\phi\in\DD(M,E)$ with $P\phi=0$.
Then $\phi= G_+P\phi=G_+0=0$.

From now on let $M$ be globally hyperbolic.
Let $\phi\in\DD(M,E)$ with $G\phi=0$, i.~e., $G_+\varphi=G_-\varphi$. 
We put $\psi:=G_+\varphi=G_-\varphi \in C^\infty(M,E)$ and we see
$\supp(\psi)=\supp(G_+\phi)\cap\supp(G_-\phi)\subset J_+^M(\supp(\varphi))\cap 
J_-^M(\supp(\varphi))$.
Since $(M,g)$ is globally hyperbolic $J_+^M(\supp(\varphi))\cap
J_-^M(\supp(\varphi))$ is compact, hence $\psi\in \DD(M,E)$.
From $P(\psi) = P(G_+(\phi)) = \phi$ we see that $\phi\in P(\DD(M,E))$.
This shows exactness at the second $\DD(M,E)$.

Finally, let $\phi\in\Csc(M,E)$ such that $P\phi=0$. 
Without loss of generality we may assume that $\supp(\phi)\subset I_+^M(K)\cup
I_-^M(K)$ for 
a compact subset $K$ of $M$. 
Using a partition of unity subordinated to the open covering
$\{I_+^M(K),I_-^M(K)\}$ write $\phi$ as $\phi=\phi_1 + \phi_2$ 
where $\supp(\phi_1)\subset I_-^M(K)\subset J_-^M(K)$ and
$\supp(\phi_2)\subset I_+^M(K)\subset J_+^M(K)$. 
For $\psi:=-P\phi_1=P\phi_2$ we see that $\supp(\psi) \subset J_-^M(K) \cap
J_+^M(K)$, hence $\psi\in\DD(M,E)$.

We check that $G_+\psi=\phi_2$.
For all $\chi\in\DD(M,E^*)$ we have
$$
\int_M \chi\cdot(G_+P\phi_2) \dV =
\int_M (G_-^{*}\chi)\cdot(P\phi_2) \dV =
\int_M (P^*G_-^{*}\chi)\cdot\phi_2 \dV =
\int_M \chi\cdot\phi_2 \dV 
$$
where $G_-^*$ is the Green's operator for the adjoint operator $P^*$ according
to Lemma~\ref{Greenadjungiert}. 
Notice that for the second equation we use the fact that $\supp(\phi_2) \cap
\supp(G^*_-\chi) \subset J^M_+(K)\cap J^M_-(\supp(\chi))$ is compact.
Similarly, one shows $G_-\psi=-\phi_1$.

Now $G\psi = G_+\psi - G_-\psi = \phi_2 + \phi_1 = \phi$, hence $\phi$ is in
the image of $G$. 
\end{proof}

\begin{prop}\label{PGstetig}
Let $M$ be a globally hyperbolic Lorentzian manifold, 
let $P$ be a normally hyperbolic operator acting on sections in a  
vector bundle $E$ over $M$.
Let $G_+$ and $G_-$ be the advanced and retarded Green's operators for $P$
respectively.

Then $G_\pm:\DD(M,E) \to \Csc(M,E)$ and all maps in the complex
(\ref{eqExSeq}) are sequentially continuous. 
\end{prop}

\begin{proof}
The maps $P:\DD(M,E) \to \DD(M,E)$ and $P:\Csc(M,E) \to \Csc(M,E)$ are
sequentially continuous simply because $P$ is a differential operator.
It remains to show that $G:\DD(M,E) \to \Csc(M,E)$ is sequentially continuous.

Let $\phi_j,\phi\in\DD(M,E)$ and $\phi_j\to \phi$ in $\DD(M,E)$ for all $j$.
Then there exists a compact subset $K\subset M$ such that
$\supp(\phi_j),\supp(\phi) \subset K$.
Hence $\supp(G\phi_j),\supp(G\phi)\subset J^M(K)$ for all $j$.
From the proof of Theorem~\ref{globhypexist} we know that $G_+\phi$ coincides
with  the solution $u$ to the Cauchy problem $Pu=\phi$ with initial conditions
$u|_{S_-} = (\nabla_\mathfrak{n} u)|_{S_-}=0$ where $S_- \subset M$ is a
spacelike 
Cauchy hypersurface such that $K \subset I_+^M(S_-)$.
Theorem~\ref{cauchystetig}  tells us that if $\phi_j\to\phi$ in $\DD(M,E)$, 
then the solutions $G_+\phi_j\to G_+\phi$ in $C^\infty(M,E)$.
The proof for $G_-$ is analogous and the statement for $G$ follows.
\end{proof}

\Remark{
Green's operators need not exist for any normally hyperbolic operator on any
spacetime. 
For example consider a compact spacetime $M$ and the d'Alembert operator
acting on real functions.
Note that in this case $\DD(M,\R)=C^\infty(M)$.
If there existed Green's operators the d'Alembert operator would be
injective. 
But any constant function belongs to the kernel of the operator.
}

%%%%%%%%%%%%%%%%%%%%%%%%%%%%%%%%%%%%%%%%%%%%%%%%%%%%%%%%%%%%%%%%%%%%%%%%% 
\section{Non-globally hyperbolic manifolds}  
\label{seq:nonglobhyp} 
%%%%%%%%%%%%%%%%%%%%%%%%%%%%%%%%%%%%%%%%%%%%%%%%%%%%%%%%%%%%%%%%%%%%%%%%%

Globally hyperbolic Lorentzian manifolds turned out to form a good class for  
the solution theory of normally hyperbolic operators.
We have unique advanced and retarded fundamental solutions and Green's
operators.
The Cauchy problem is well-posed.
Some of these analytical features survive when we pass to more general
Lorentzian manifolds.
We will see that we still have existence (but not uniqueness) of fundamental 
solutions and Green's operators if the manifold can be embedded in a suitable
way as an open subset into a globally hyperbolic manifold such that the
operator extends.
Moreover, we will see that conformal changes of the Lorentzian metric
do not alter the basic analytical properties.
To illustrate this we construct Green's operators for the Yamabe operator on
the important anti-deSitter spacetime which is not globally hyperbolic.

\begin{prop}\label{propGreensubset}
Let $M$ be a timeoriented connected Lorentzian manifold.
Let $P$ be a normally hyperbolic operator acting on sections in a  
vector bundle $E$ over $M$.
Let $G_\pm$ be Green's operators for $P$.
Let $\Omega\subset M$ be a causally compatible connected open subset.

Define $\tilde G_\pm :\DD(\Omega,E) \to C^\infty(\Omega,E)$ by
$$
\tilde G_\pm (\varphi) := G_\pm(\varphi_{\mathrm{ext}})|_\Omega.
$$
Here $\DD(\Omega,E) \to \DD(M,E)$, $\varphi \mapsto \varphi_{\mathrm{ext}}$, 
denotes extension by zero.

Then $\tilde G_+$ and $\tilde G_-$ are advanced and retarded Green's operators
for the restriction of $P$ to $\Omega$ respectively.
\indexn{advanced Green's operator}
\indexn{retarded Green's operator}
\indexn{causally compatible subset}
\end{prop}

\begin{proof}
Denote the restriction of $P$ to $\Omega$ by $\tilde P$.
To show (i) in Definition~\ref{defGfunc} we check for 
$\varphi\in\DD(\Omega,E)$
$$
\tilde P\tilde G_\pm\varphi =
\tilde P(G_\pm(\varphi_{\mathrm{ext}})|_\Omega) =
P(G_\pm(\varphi_{\mathrm{ext}}))|_\Omega =
\varphi_{\mathrm{ext}}|_\Omega = 
\varphi .
$$
Similarly, we see for (ii)
$$
\tilde G_\pm \tilde P \varphi =
G_\pm ((\tilde P \varphi)_{\mathrm{ext}})|_\Omega =
G_\pm (P \varphi_{\mathrm{ext}})|_\Omega =
\varphi_{\mathrm{ext}}|_\Omega = 
\varphi .
$$
For (iii) we need that $\Omega$ is a causally compatible subset of $M$.
\begin{eqnarray*}
\supp(\tilde G_\pm \varphi) 
&=&
\supp(G_\pm(\varphi_{\mathrm{ext}})|_\Omega) \\
&=&
\supp(G_\pm(\varphi_{\mathrm{ext}}))\cap \Omega \\
&\subset&
J_\pm^M(\supp(\varphi_{\mathrm{ext}}))\cap \Omega \\
&=&
J_\pm^M(\supp(\varphi))\cap \Omega \\
&=&
J_\pm^\Omega(\supp(\varphi)).
\end{eqnarray*}
\end{proof}

\Example{
In Minkowski space every convex open subset $\Omega$ is causally compatible.
Proposition~\ref{propGreensubset} shows the existence of an advanced and a
retarded Green's operator for any normally hyperbolic operator on $\Omega$
which extends to a normally hyperbolic operator on $M$.

On the other hand, we have already noticed in Remark~\ref{konvexnichteind}
that on convex domains the advanced and retarded fundamental
solutions need not be unique.  
Thus the Green's operators $G_\pm$ are not unique in general.
} 

The proposition fails if we drop the condition on $\Omega$ to be a causally
compatible subset of $M$.

\Example{
For non-convex domains $\Omega$ in Minkowski space $M=\R^n$ causal
compatibility does not hold in general, see Figure~7 on page~\pageref{nichtkk}.
For any $\varphi\in\DD(\Omega,E)$ the proof of
Proposition~\ref{propGreensubset} shows  that
$\supp(\tilde G_\pm \varphi)\subset J_\pm^M(\supp(\varphi))\cap \Omega$.
Now, if $J_\pm^\Omega(p)$ is a proper subset of $J_\pm^M(p)\cap\Omega$ there
is no reason why $\supp(\tilde G_\pm \varphi)$ should be a subset of
$J^\Omega_\pm(\supp(\varphi))$.
Hence $\tilde G_\pm$ are not Green's operators in general.
}

\Example{\label{upperEinstein}
We consider the \defidx{Einstein cylinder} 
$M=\R\times S^{n-1}$ equipped with the
product metric $g=-dt^2 + \mathrm{can}_{S^{n-1}}$ where
$\mathrm{can}_{S^{n-1}}$ denotes the canonical Riemannian metric of constant
sectional curvature $1$ on the sphere. 
Since $S^{n-1}$ is compact, the Einstein cylinder is globally hyperbolic,
compare Example~\ref{ex:globhyp}.

We put $\Omega:=\R\times S^{n-1}_+$ where $S^{n-1}_+:=\{(z_1,\ldots,z_n)\in
S^{n-1}\mid z_n>0 \}$ denotes the northern hemisphere.
Let $p$ and $q$ be two points in $\Omega$ which can be joined by a causal
curve $c:[0,1]\to M$ in $M$.
We write $c(s)=(t(s),x(s))$ with $x(s)\in S^{n-1}$.
After reparametrization we may assume that the curve $x$ in $S^{n-1}$ is
parametrized proportionally to arclength,
$\mathrm{can}_{S^{n-1}}(x',x')\equiv\xi$ where $\xi$ is a nonnegative
constant.

Since $S^{n-1}_+$ is a geodesically convex subset of the Riemannian manifold
$S^{n-1}$ there is a curve $y:[0,1]\to S^{n-1}_+$ with the same end points as
$x$ and of length at most the length of $x$.
If we parametrize $y$ proportionally to arclength this means
$\mathrm{can}_{S^{n-1}}(y',y')\equiv\eta \leq \xi$.
The curve $c$ being causal means $0\geq g(c',c') = -(t')^2 +
\mathrm{can}_{S^{n-1}}(x',x')$, i.~e.,
$$
(t')^2 \geq \xi.
$$
This implies $(t')^2 \geq \eta$ which in turn is equivalent to the curve
$\tilde c :=(t,y)$ being causal.
Thus $p$ and $q$ can be joined by a causal curve which stays in $\Omega$.
Therefore $\Omega$ is a causally compatible subset of the Einstein
cylinder. 
}

Next we study conformal changes of the metric.
\indexn{conformal change of metrics>defemidx}
Let $M$ be a timeoriented connected Lorentzian manifold.
Denote the Lorentzian metric by $g$.
Let $f:M\to \R$ be a positive smooth function.
Denote the conformally related metric by $\tilde g := f\cdot g$.
This means that $\tilde g(X,Y)=f(p)\cdot g(X,Y)$ for all $X,Y \in T_pM$.
The causal type of tangent vectors and curves is unaffected by this change of
metric.
Therefore all causal concepts such as the chronological or causal future and
past remain unaltered by a conformal change of the metric.
Similarly, the causality conditions are unaffected.
Hence $(M,g)$ is globally hyperbolic if and only if $(M,\tilde g)$ is globally
hyperbolic. 

Let us denote by $g^*$ and $\tilde g^*$ the metrics on the cotangent bundle
$T^*M$ induced by $g$ and $\tilde g$ respectively.
Then we have $\tilde g^* = \frac{1}{f} g^*$.

Let $\tilde P$ be a normally hyperbolic operator with respect to $\tilde g$.
Put $P := f \cdot \tilde P$, more precisely, 
\begin{equation}
\label{PtildeP}
P(\varphi) = f\cdot\tilde P(\varphi)
\end{equation}
for all $\varphi$.
Since the principal symbol of $\tilde P$ is given by $\tilde g^*$, the
principal symbol of $P$ is given by $g^*$,
$$
\sigma_P(\xi) = f\cdot\sigma_{\tilde P}(\xi)  =
- f\cdot\tilde g^*(\xi,\xi)\cdot \id = - g^*(\xi,\xi)\cdot \id.
$$
Thus $P$ is normally hyperbolic for $g$.
Now suppose we have an advanced or a retarded Green's operator $G_+$ or $G_-$
for $P$. 
We define $\tilde G_\pm : \DD(M,E) \to C^\infty(M,E)$ by
\begin{equation}
\label{GtildeG}
\tilde G_\pm \varphi := G_\pm\left({f}\cdot\varphi\right).
\end{equation}
We see that
$$
\tilde G_\pm(\tilde P\varphi) = G_\pm(f\cdot \tfrac1f\cdot P\varphi)
= G_\pm(P\varphi) = \varphi
$$
and
$$
\tilde P(\tilde G_\pm\varphi) = \tfrac1f \cdot P(G_\pm(f\cdot\varphi)) = 
 \tfrac1f \cdot f \cdot \varphi = \varphi.
$$
Multiplication by a nowhere vanishing function does not change supports, hence
$$
\supp(\tilde G_\pm\varphi) = \supp(G_\pm(f\,\varphi)) 
\subset J_\pm^M(\supp(f\,\varphi)) = J_\pm^M(\supp(\varphi)).
$$
Notice again that $J_\pm^M$ is the same for $g$ and for $\tilde g$.
We have thus shown that $\tilde G_\pm$ is a Green's operator for $\tilde P$.
We summarize:

\begin{prop}\label{propGreenconformal}
Let $M$ be a timeoriented connected Lorentzian manifold with Lorentzian metric
$g$.
Let $f:M\to \R$ be a positive smooth function and
denote the conformally related metric by $\tilde g := f\cdot g$.

Then (\ref{PtildeP}) yields a 1-1-correspondence $P \leftrightarrow \tilde P$
between normally hyperbolic operators for $g$ and such operators for $\tilde
g$. 
Similarly, (\ref{GtildeG}) yields a 1-1-correspondence $G_\pm \leftrightarrow 
\tilde G_\pm$ for their Green's operators.
\hfill$\Box$
\end{prop}

This discussion can be slightly generalized.

\Remark{\label{nonnormally}
Let $(M,g)$ be a timeoriented connected Lorentzian manifold.
Let $P$ be a normally hyperbolic operator on $M$ for which advanced and
retarded Green's operators $G_+$ and $G_-$ exist.
Let $f_1,f_2:M\to\R$ be positive smooth functions. 
Then the operator $\tilde{P}:=\tfrac{1}{f_1}\cdot P\cdot \tfrac{1}{f_2}$,
given by 
\begin{equation}\label{doppelmult}
 \tilde{P}(\varphi)=\tfrac{1}{f_1}\cdot P(\tfrac{1}{f_2}\cdot\varphi) 
\end{equation}
for all $\varphi$, possesses advanced and retarded Green's operators
$\tilde{G}_\pm$.
They can be defined in analogy to (\ref{GtildeG}):
\[ \tilde{G}_\pm(\varphi):= {f_2}\cdot
G_\pm({f_1}\cdot\varphi). 
\]
As above one gets $\tilde{P}\tilde{G}_\pm(\varphi)=\varphi$ and
$\tilde{G}_\pm(\tilde{P}\varphi)=\varphi$ for all $\phi\in\DD(M,E)$.
Operators $\tilde{P}$ of the form (\ref{doppelmult}) are normally hyperbolic
with respect to the conformally related metric $\tilde g = f_1\cdot f_2\cdot
g$. 
}

Combining Propositions~\ref{propGreensubset} and
\ref{propGreenconformal} we get:

\begin{corollary}\label{conformeinschr}
Let $(\tilde{M},\tilde{g})$ be timeoriented connected
Lorentzian manifold which can be conformally embedded as a causally compatible
open subset $\Omega$ into the globally hyperbolic manifold $(M,g)$.
Hence on $\Omega$ we have $\tilde{g}=f\cdot g$ for some positive function
$f\in C^\infty(\Omega,\R)$.

Let $\tilde{P}$ be a normally hyperbolic operator on $(\tilde{M},\tilde{g})$
and let $P$ be the operator on $\Omega$ defined as in (\ref{PtildeP}).
Assume that $P$ can be extended to a normally hyperbolic operator on the whole
manifold $(M,g)$.
Then the operator $\tilde{P}$ possesses advanced and retarded Green's
operators. 
Uniqueness is lost in general. \qed
\end{corollary}

In the remainder of this section we will show that the preceding
considerations can be applied to an important example in general relativity:
\defem{anti-deSitter spacetime}. \indexn{anti-deSitter spacetime}
We will show that it can be conformally embedded into the Einstein cylinder.
The image of this embedding is the set $\Omega$ in Example~\ref{upperEinstein}.
Hence we realize anti-deSitter spacetime conformally as a causally compatible
subset of a globally hyperbolic Lorentzian manifold.

For an integer $n\geq 2$, one defines the $n$-dimensional
\defidx{pseudohyperbolic space}   
\[
H_1^n:=\{x\in\R^{n+1}\,\mid\,\lala x,x\rara=-1\},
\indexs{H*@$H_1^n$, pseudohyperbolic space}
\]
where $\lala x,y\rara:=-x_0y_0-x_1y_1+\sum_{j=2}^nx_jy_j$ for all
$x=(x_0,x_1,\ldots,x_n)$ and $y=(y_0,y_1,\ldots,y_n)$ in $\R^{n+1}$. 
With the induced metric (also denoted by $\lala\cdot\,,\cdot\rara$) 
$H_1^n$ becomes a connected Lorentzian manifold with constant sectional
curvature $-1$, see e.~g.\ \cite[Chap.~4, Prop.~29]{ONeill}. 

\begin{lemma}\label{lconfAdS}
There exists a conformal diffeomorphism 
\[ \Psi:
\left(S^1\times
  S_+^{n-1},-\mathrm{can}_{S^1}+\mathrm{can}_{S_+^{n-1}}\right) \to
 \left(H_1^n,\lala\cdot\,,\cdot\rara\right)
\] 
such that for any $(p,x)\in S^1\times S^{n-1}_+\subset S^1\times\R^n$ one has 
\[
(\psi^*\lala\cdot\,,\cdot\rara)_{(p,x)}=
\frac{1}{x_n^2}\left(-\mathrm{can}_{S^1} +\mathrm{can}_{S_+^{n-1}}\right).
\]
\end{lemma}
 
\begin{proof}
We first construct an isometry between the pseudohyperbolic space and
\[
\left(S^1\times
  H^{n-1},-y_1^2\,\mathrm{can}_{S^1}+\mathrm{can}_{H^{n-1}}\right),
\] 
where $H^{n-1}:=\{(y_1,\ldots,y_n)\in\R^n\,\mid\, y_1>0\textrm{ and
  }-y_1^2+\sum_{j=2}^ny_j^2=-1\}$ is the $(n-1)$-dimensional hyperbolic space.
The hyperbolic metric $\mathrm{can}_{H^{n-1}}$ is induced by the Minkowski
metric on $\R^n$.
Then $(H^{n-1},\mathrm{can}_{H^{n-1}})$ is a Riemannian manifold with constant
sectional curvature $-1$. 
Define the map
\be
\Phi:S^1\times H^{n-1}&\to&H_1^n,\\
(p=(p_0,p_1),y=(y_1,\ldots,y_n))&\mapsto&(y_1p_0,y_1p_1,y_2,\ldots,y_n)
 \in\R^{n+1}.    
\ee
This map is clearly well-defined because $-y_1^2
 (\underbrace{p_0^2+p_1^2}_{=1})+y_2^2+\ldots 
  y_n^2=-y_1^2+\sum_{j=2}^n y_j^2=-1$.
The inverse map is given by 
\[\Phi^{-1}(x)=\left(\left(\frac{x_0}{\sqrt{x_0^2+x_1^2}},
 \frac{x_1}{\sqrt{x_0^2+x_1^2}}\right),
 \left(\sqrt{x_0^2+x_1^2},x_2,\ldots,x_n\right)\right).\]       
Geometrically, the map $\Phi$ can be interpreted as follows: 
For any point $p=(p_0,p_1)\in S^1$, consider the hyperplane $\mathcal{H}_p$ of
$\R^{n+1}$ defined by 
\[
\mathcal{H}_p:=\R\cdot (p_0,p_1,0,\ldots,0)\oplus\R^{n-1},
\]
where $(p_0,p_1,0,\ldots,0)\in\R^{n+1}$ and $\R^{n-1}$ is identified with the
subspace $\{(0,0,w_2,\ldots,w_n)\,\mid\, w_j\in\R\}\subset\R^{n+1}$.
If $\{e_2,\ldots,e_n\}$ is the canonical basis of this $\R^{n-1}$, then 
\[\mathcal{B}_p:=\{e_1:=(p_0,p_1,0\ldots,0),e_2,\ldots,e_n\}\]
is a Lorentz orthonormal basis of $\mathcal{H}_p$ with respect to the metric
induced 
by $\lala\cdot\,,\cdot\rara$. 
Define $H^{n-1}(p)$ as the hyperbolic space of
$(\mathcal{H}_p,\lala\cdot\,,\cdot\rara)$ in this basis.
More precisely, $H^{n-1}(p)=\{\sum_{j=1}^n\eta_je_j \in\mathcal{H}_p\,|\,
\eta_1>0, -\eta_1^2 + \sum_{j=2}^n\eta_j^2=-1\}$.
Then  $y\mapsto\Phi(p,y)$ yields an isometry from Minkowski space to
$\mathcal{H}_p$ which restricts to an isometry $H^{n-1}\to H^{n-1}(p)$. 

\begin{center}
\input{fig-splitAdS}
\end{center}

Let now $(p,y)\in S^1\times H^{n-1}$ and $X=(X^1,X^{n-1})\in T_pS^1\oplus T_y
H^{n-1}$.
Then the differential of $\Phi$ at $(p,y)$ is given by
\[
d_{(p,y)}\Phi(X)=\left(y_1X^1+X_1^{n-1}p,X_2^{n-1},\ldots,X_n^{n-1}\right).
\]
Therefore the pull-back of the metric on $H_1^n$ via $\Phi$ can be computed to
yield 
\be
(\Phi^*\lala\cdot\,,\cdot\rara)_{(p,y)}(X,X)
&=&
\lala d_{(p,y)}\Phi(X),d_{(p,y)}\Phi(X)\rara\\ 
&=&
y_1^2\lala X^1,X^1\rara+(X_1^{n-1})^2\lala p,p\rara
+\sum_{j=2}^n(X_j^{n-1})^2\\ 
&=&
-y_1^2\,\{(X_0^1)^2+(X_1^1)^2\}-(X_1^{n-1})^2+\sum_{j=2}^n(X_j^{n-1})^2\\
&=&
-y_1^2\,\mathrm{can}_{S^1}(X^1,X^1)+\mathrm{can}_{H^{n-1}}(X^{n-1},X^{n-1}).
\ee
Hence $\Phi$ is an isometry
$$
\left(S^1\times
  H^{n-1},-y_1^2\,\mathrm{can}_{S^1}+\mathrm{can}_{H^{n-1}}\right) \to
 \left(H_1^n,\lala\cdot\,,\cdot\rara\right).
$$
The stereographic projection from the south pole
\be
\pi:S_+^{n-1}&\to& H^{n-1},\\
x=(x_1,\ldots,x_n)&\mapsto&\frac{1}{x_n}(1,x_1,\ldots,x_{n-1}),
\ee
is a conformal diffeomorphism.
It is easy to check that $\pi$ is a well-defined diffeomorphism with inverse
given by $y=(y_1,\ldots,y_n)\mapsto\frac{1}{y_1}(y_2,\ldots,y_{n},1)$.
For any $x\in S_+^{n-1}$ and $X\in T_xS_+^{n-1}$ the differential of $\pi$ at
$x$ is given by 
\begin{eqnarray*}
d_x\pi(X)
&=&
\frac{1}{x_n}(0,X_1,\ldots,X_{n-1})
-\frac{X_n}{x_n^2}(1,x_1,\ldots,x_{n-1})\\
&=&
\frac{1}{x_n^2}(-X_n,x_nX_1-x_1X_n, \ldots, x_nX_{n-1}-x_{n-1}X_n).
\end{eqnarray*}
Therefore we get for the pull-back of the hyperbolic metric
\be
(\pi^*\mathrm{can}_{H^{n-1}})_x(X,X)
&=&
\frac{1}{x_n^4}\Big\{-X_n^2 + \sum_{j=1}^{n-1}(x_nX_j-x_jX_n)^2 \Big\}\\
&=&
\frac{1}{x_n^4}\Big\{-X_n^2 + x_n^2\sum_{j=1}^{n-1}X_j^2
-2x_nX_n\underbrace{\sum_{j=1}^{n-1}x_jX_j}_{=-x_nX_n} 
+  X_n^2\underbrace{\sum_{j=1}^{n-1}x_j^2}_{=1-x_n^2}\Big\}\\
&=&
\frac{1}{x_n^2}\sum_{j=1}^{n-1}X_j^2+
\frac{X_n^2}{x_n^2}\\ 
&=&
\frac{1}{x_n^2}\sum_{j=1}^nX_j^2,
\ee
that is, $(\pi^*\mathrm{can}_{H^{n-1}})_x=\frac{1}{x_n^2}
(\mathrm{can}_{S_+^{n-1}})_x$.
We obtain an explicit diffeomorphism
\be
\Psi:=\Phi\circ (\id\times \pi): S^1\times S_+^{n-1}&\to&H_1^n,\\
(p=(p_0,p_1),x=(x_1,\ldots,x_n))&\mapsto&\frac{1}{x_n}(p_0,p_1,x_1,
\ldots,x_{n-1}),  
\ee
satisfying, for every $(p,x)\in S^1\times S_+^{n-1}$,
\begin{eqnarray}
(\psi^*\lala\cdot\,,\cdot\rara)_{(p,x)}
&=&
((\id\times\pi)^*(\Phi^*\lala\cdot\,,\cdot\rara))_x\nonumber\\  
&=&
((\id\times\pi)^*(-\pi(x)_1^2\,\mathrm{can}_{S^1}
+ \mathrm{can}_{H^{n-1}}))_x\nonumber\\  
&=&
-\pi(x)_1^2\,\mathrm{can}_{S^1}+\frac{1}{x_n^2}\,
\mathrm{can}_{S_+^{n-1}}\nonumber\\ 
&=&\frac{1}{x_n^2}\left(-\mathrm{can}_{S^1}
+\mathrm{can}_{S_+^{n-1}}\right). 
\label{adsconform}
\end{eqnarray}
This concludes the proof.
\end{proof}

Following \cite[Chap.~8, p.~228f]{ONeill}, one defines the $n$-dimensional
\defidx{anti-deSitter spacetime}  $\widetilde H^n_1$ 
\indexs{H*@$\protect\widetilde H^n_1$, anti-deSitter spacetime}
to be the universal covering
manifold of the pseudohyperbolic space $H^n_1$. 
For $\widetilde H^n_1$ the sectional curvature is identically $-1$ and the
scalar curvature equals $-n(n-1)$.
In physics, $\widetilde H^4_1$ is important because it provides a vacuum
solution to Einstein's field equation with cosmological constant
$\Lambda=-3$.

The causality properties of $\widetilde H^n_1$ are discussed in 
\cite[Chap.~14, Example~41]{ONeill}.
It turns out that $\widetilde H^n_1$ is not globally hyperbolic.
The conformal diffeomorphism constructed in Lemma~\ref{lconfAdS} lifts to a
conformal diffeomorphism of the universal covering manifolds:
\[ \widetilde{\Psi}:
  \left(\R\times S^{n-1}_+,-dt^2+\mathrm{can}_{S^{n-1}_+}
  \right) \to \left(\widetilde  H^n_1,\lala\cdot,\cdot\rara\right)
\]
such that for any $(t,x)\in \R^1\times S^{n-1}_+\subset \R^1\times\R^n$ one
has  
\[
(\psi^*\lala\cdot\,,\cdot\rara)_{(t,x)}=
\frac{1}{x_n^2}\left(-dt^2 +\mathrm{can}_{S_+^{n-1}}\right).
\]
Then $\widetilde H^n_1$ is conformally diffeomorphic to the causally
compatible subset $\R\times S^{n-1}_+$ of the globally hyperbolic Einstein
cylinder.
From the considerations above we will derive existence of Green's operators
for the Yamabe operator $Y_g$ on anti-deSitter spacetime $\widetilde H^n_1$. 

\begin{definition}\label{yamabe}
{\rm
Let $(M,g)$ be a Lorentzian manifold of dimension $n\geq 3$.
Then the \defidx{Yamabe operator} $Y_g$ acting on functions on $M$ is given
by 
\begin{equation}\label{yamadef}
 Y_g=4\frac{n-1}{n-2}\;\Box_g+ \mathrm{scal}_g
\end{equation}
where $\Box_g$ denotes the d'Alembert operator
\indexs{*@$\protect\Box_g$, d'Alembert operator for metric $g$}
and $\mathrm{scal}_g$
is the scalar curvature taken with respect to $g$.
}
\end{definition}

We perform a conformal change of the metric.
To simplify formulas we write the conformally related metric as
$\widetilde{g}=\varphi^{p-2}g$ where $p=\tfrac{2n}{n-2}$ and $\varphi$ is a
positive smooth function on $M$.
The Yamabe operators for the metrics $g$ and $\tilde g$ are related by 
\begin{equation}\label{yamachange}
Y_{\tilde{g}}u=\varphi^{1-p}\cdot \,Y_g\,(\varphi u),
\end{equation}
where $u\in C^\infty(M)$, see \cite[p.~43, Eq.~(2.7)]{LP}.
Multiplying $Y_g$ with $\frac{n-2}{4\cdot(n-1)}$ we obtain a normally
hyperbolic operator
\[ 
P_g=\Box_g+\frac{n-2}{4\cdot(n-1)}\cdot  \mathrm{scal}_g\;.
\]
Equation (\ref{yamachange}) gives for this operator
\begin{equation}\label{Pchange}
P_{\tilde{g}}u=\varphi^{1-p}\,\left(\,P_g\,(\varphi u)\right).
\end{equation}

Now we consider this operator $P_g$ on the Einstein cylinder $\R\times
S^{n-1}$. 
Since the Einstein cylinder is globally hyperbolic we get unique advanced and
retarded Green's operators $G_\pm$ for $P_g$.
From Example~\ref{upperEinstein} we know that $\R\times  S^{n-1}_+$ is a
causally compatible subset of the Einstein cylinder $\R\times S^{n-1}$.
By Proposition~\ref{propGreensubset} we have %(non-unique)
advanced and retarded Green's operators for $P_g$ on $\R\times  S^{n-1}_+$.
From Equation~(\ref{Pchange}) and Remark~\ref{nonnormally} we conclude

\begin{corollary}\label{AdSGreen}
On the anti-deSitter spacetime $\widetilde H^n_1$ the Yamabe operator 
possesses advanced and retarded Green's operators. 
\hfill$\Box$
\end{corollary}

\Remark{
It should be noted that the precise form of the zero order term of the Yamabe
operator given by the scalar curvature is crucial for our argument.
On $(\widetilde{H}^n_1,\tilde{g})$ the scalar curvature is constant,
$\mathrm{scal}_{\tilde g} = -n(n-1)$.
Hence the rescaled Yamabe operator is $P_{\tilde g} = \Box_{\tilde g} -
\frac{1}{4}n(n-2)=\Box_{\tilde g}-c$ with $c:=\tfrac{1}{4}n\cdot(n-2)$.
For the d'Alembert operator $\Box_{\tilde g}$
on $(\widetilde{H}^n_1,\tilde{g})$ we have for any $u\in
C^\infty(\widetilde{H}^n_1)$
\[
\Box_{\tilde{g}}u=P_{\tilde{g}}u+c\cdot u
=\varphi^{1-p}P_g(\varphi u)+c\cdot u
=\varphi^{1-p}\big(P_g+c\cdot\varphi^{p-2}\big)(\varphi u).
\]
The conformal factor $\varphi^{p-2}$ tends to infinity as one  approaches
the boundary of $\R\times S^{n-1}_+$ in $\R\times S^{n-1}$.
Namely, for $(t,x)\in\R\times S^{n-1}_+$ one has by (\ref{adsconform})
$\varphi^{p-2}(t,x)=x_n^{-2}$ where $x_n$ denotes the last component
of $x\in S^{n-1}\subset\R^n$.
Hence if one approaches the boundary, then $x_n\to 0$ and therefore
$\phi^{p-2}=x_n^{-2}\to\infty$.
Therefore one cannot extend the operator $P_g+c\cdot\varphi^{p-2}$ to an 
operator defined on the whole Einstein cylinder $\R\times S^{n-1}$.
Thus we cannot establish existence of Green's operators for the d'Alembert
operator on anti-deSitter spacetime with the methods developed here.  

%% It should be noted that anti-deSitter spacetime is a symmetric space.
%% \indexn{symmetric space>defemidx}
%% In this situation methods from harmonic analysis \indexn{harmonic analysis}
%% can be employed to construct fundamental solutions.
%% Since this approach applies only to invariant operators on spaces of large
%% symmetry we will not pursue it in the present book.
%% See \cite{Helgason94} and the references therein.
}

How about uniqueness of fundamental solutions for normally hyperbolic
operators on anti-deSitter spacetime? 
We note that Theorem~\ref{kernP} cannot be applied for anti-deSitter
spacetime because the time separation function $\tau$ is not finite.
This can be seen as follows:
We fix two points $x,y\in\R\times S^{n-1}_+$ with $x<y$ sufficiently far apart
such that there exists a timelike curve connecting them in $\{(p,x)\in\R\times
S^{n-1}\,|\,x_n\geq 0\}$ having a nonempty segment on the boundary $\{(p,x)\in\R\times
S^{n-1}\,|\,x_n= 0\}$.

\begin{center}
\input{fig-AdStau.tex}
\end{center}

By sliding the segment on the boundary slightly we obtain a
timelike curve in the upper half of the Einstein cylinder connecting $x$ and
$y$ whose length with respect to the metric 
$\frac{1}{x_n^2}\left(-\mathrm{can}_{S^1} +\mathrm{can}_{S_+^{n-1}}\right)$
in (\ref{adsconform}) can be made arbitrarily large.
This is due to the factor $\frac{1}{x_n^2}$ which is large if the segment is
chosen so that $x_n$ is small along it.

\begin{center}
\input{fig-AdStau1.tex}
\end{center}

A discussion as in Remark~\ref{konvexnichteind} considering supports (see 
picture below) shows that fundamental solutions for normally hyperbolic
operators are not unique on the upper half $\R\times S_+^{n-1}$ of the
Einstein cylinder.
The fundamental solution of a point $y$ in the lower half of the
Einstein cylinder can be added to a given fundamental solution of $x$ in the
upper half thus yielding a second fundamental solution of $x$ with the same
support in the upper half.

\begin{center}
\input{fig-AdSsupport.tex}
\end{center}

Since anti-deSitter spacetime and $\R\times S^{n-1}_+$  are conformally
equivalent we obtain distinct fundamental solutions for operators on
anti-deSitter spacetime as described in Corollary~\ref{conformeinschr}.  

%%%%%%%%%%%%%%%%%%%%%%%%%%%%%%%%%%%%%%%%%%%%%%%%%%%%%%%%%%%%%%%%%%%%%%%%%
\chapter{Quantization}\label{chapquant}
%%%%%%%%%%%%%%%%%%%%%%%%%%%%%%%%%%%%%%%%%%%%%%%%%%%%%%%%%%%%%%%%%%%%%%%%%

We now want to apply the analytical theory of wave equations and develop some
mathematical basics of field (or second) quantization.
We do not touch the so-called first quantization which is concerned with
replacing point particles by wave functions.
As in the preceding chapters we look at fields (sections in vector bundles)
which 
have to satisfy some wave equation (specified by a normally hyperbolic
operator) and now we want to quantize such fields.

We will explain two approaches.
In the more traditional approach one constructs a quantum field which is a
distribution satisfying the wave equation in the distributional sense.
This quantum field takes its values in selfadjoint operators on Fock space
which is the multi-particle space constructed out of the single-particle
space of wave functions.
This construction will however crucially depend on the choice of a Cauchy
hypersurface. 

It seems that for quantum field theory on curved spacetimes the 
approach of local quantum physics is more appropriate.
The idea is to associate to each (reasonable) spacetime region the algebra
of observables that can be measured in this region.
We will find confirmed the saying that ``quantization is a mystery, but second
quantization is a functor'' by mathematical physicist Edward Nelson.
One indeed constructs a functor from the category of globally hyperbolic
Lorentzian manifolds equipped with a formally selfadjoint normally hyperbolic
operator to the category of $C^*$-algebras.
We will see that this functor obeys the Haag-Kastler axioms of a local quantum
field theory.
This functorial interpretation of local covariant quantum field theory on
curved spacetimes was introduced in \cite{HW}, \cite{Ver}, and \cite{BFV}.

It should be noted that in contrast to what is usually done in the physics
literature there is no need to fix a wave equation and then quantize the
corresponding fields (e.~g.\ the Klein-Gordon field).
In the present book, both the underlying manifold as well as the normally
hyperbolic operator occur as variables in one single functor.

In Sections~\ref{sec:CAlg} and \ref{sec-CCR} we develop the theory of
$C^*$-algebras and CCR-representations in full detail to the extent that we
need. 
In the next three sections we construct the quantization functors and check
the Haag-Kastler axioms.
The last two sections are devoted to the construction of the Fock space and
the quantum field. 
We will see that the quantum field determines the CCR-algebras up to
isomorphism.
This relates the two approaches to quantum field theory on curved backgrounds.

%%%%%%%%%%%%%%%%%%%%%%%%%%%%%%%%%%%%%%%%%%%%%%%%%%%%%%%%%%%%%%%%%%%%%%%%%
\section{$C^*$-algebras}\label{sec:CAlg}
%%%%%%%%%%%%%%%%%%%%%%%%%%%%%%%%%%%%%%%%%%%%%%%%%%%%%%%%%%%%%%%%%%%%%%%%%

In this section we will collect those basic concepts and facts related to
$C^*$-algebras that we will need when we discuss the canonical commutator
relations in the subsequent section.
We give complete proofs.
Readers familiar with $C^*$-algebras may skip this section.
For more information on $C^*$-algebras see e.~g.\ \cite{BR1}.

\Definition{
Let $A$\ be an associative $\Co$-algebra, let $\|\cdot\|$\ be a norm on
the $\Co$-vector space $A$, and let $*:A\rightarrow A$, $a\mapsto a^*$, 
be a $\Co$-antilinear map. 
Then $(A,\|\cdot\|,*)$\ is called a
\indexn{C*Algebra@$C^*$-algebra>defemidx}\defem{$C^*$-algebra}, if 
$(A,\|\cdot\|)$\ is complete and we have for all $a$, $b\in A$: 
\indexs{A*@$(A,\|\cdot\|,*)$, $C^*$-algebra}
\begin{enumerate}
\item $a^{**} = a$ \qquad\qquad\text{{($*$\ is an involution)}}
\item $(ab)^* = b^*a^*$ 
\item $\|ab\| \le
  \|a\|\,\|b\|$\qquad\quad\text{{(submultiplicativity)}} 
\item $\|a^*\| = \| a\|$\qquad\qquad\text{{($*$\ is an isometry)}}
%\indexn{isometric element of a $C^*$-algebra}
\item $\| a^*a\| =
  \|a\|^2$\qquad\qquad\text{{($C^*$--property)}}.\indexn{C*p@$C^*$-property>defemidx}  
\end{enumerate}
A (not necessarily complete) norm on $A$ satisfying conditions (1) to (5) is
called a \indexn{C*Norm@$C^*$-norm>defemidx}\defem{$C^*$-norm}.
}

\Example{
\label{OperatorenaufHilbertraum}    
Let $(H,(\cdot,\cdot))$\ be a complex Hilbert space, let $A=\LL(H)$\ be the
algebra of bounded operators on $H$. 
Let $\|\cdot\|$\ be the \defidx{operator norm},
\indexs{L(H)@$\protect\LL(H)$,
bounded operators on Hilbert space $H$} i.~e., 
\[
\| a \| := \sup_{\ueber{x\in H}{\|x\|=1}} \| ax\|.
\]
Let $a^*$\ be the operator adjoint to $a$, i.~e.,
\[
(ax,y) = (x,a^*y)\qquad\text{ for all } x,\, y\in H.
\]
Axioms 1 to 4 are easily checked. 
Using Axioms 3 and 4 and the Cauchy-Schwarz inequality we see
\begin{align*}
  \|a\|^2 &= \sup_{\|x\|=1} \| ax\|^2 = \sup_{\|x\|=1} (ax,ax) =
  \sup_{\|x\|=1} (x,a^*ax) \\ 
&\le \sup_{\|x\|=1} \| x\|\cdot\| a^*ax\| = \| a^*a\|
  \overset{\mathrm{Axiom~3}}{\le} 
\|a^*\|\cdot\|a\| \overset{\mathrm{Axiom~4}}{=} \| a\|^2.
\end{align*}
This shows Axiom 5.
}

\Example{\label{stetigeFktimUnendlichenweg}
Let $X$\ be a locally compact Hausdorff space. 
Put
\begin{multline*}
A:= C_0(X) := \{ f:X\rightarrow \Co\text{ continuous}\mid 
\forall \varepsilon>0\,\exists K\subset X\text{ compact, so that } \\
\forall x\in X\setminus K: |f(x)|<\varepsilon\}.
\end{multline*} 
We call $C_0(X)$\ \indexs{C0X@$C_0(X)$, continuous functions vanishing at
infinity} the \defidx{algebra of continuous functions vanishing at infinity}.  
If $X$\ is compact, then $A=C_0(X) =\stetig(X)$. 
All $f\in C_0(X)$\ are bounded and we may define:
$$
  \| f\| := \sup_{x\in X} |f(x)|.
$$
Moreover let
$$
f^*(x) := \overline{f(x)}.
$$
Then $(C_0(X),\|\cdot\|,*)$\ is a commutative $C^*$--algebra.
}

\Example{\label{glatteFktimUnendlichenweg}
Let $X$\ be a differentiable manifold. 
Put 
\begin{equation*}
A:= C_0^\infty(X) := C^\infty(X)\cap C_0(X) .
\end{equation*} 
We call $C_0^\infty(X)$\ \indexs{C0999X@$C_0^\infty(X)$, smooth
functions vanishing at infinity} the \defidx{algebra of smooth
functions vanishing at infinity}. 
Norm and $*$ are defined as in the previous example. 
Then $(C_0^\infty(X),\|\cdot\|,*)$\ satisfies all axioms of a commutative
$C^*$-algebra except that $(A,\|\cdot\|))$ is not complete.
If we complete this normed vector space, then we are back to the previous
example. 
}

\Definition{
A subalgebra $A_0$ of a $C^*$-algebra $A$ is called a
\indexn{C*SubAlgebra@$C^*$-subalgebra>defemidx}\defem{$C^*$-subalgebra}
if it is a closed subspace and $a^*\in A_0$ for all $a\in A_0$.
}
Any $C^*$-subalgebra is a $C^*$-algebra in its own right.

\Definition{
Let $S$ be a subset of a $C^*$-algebra $A$.
Then the intersection of all $C^*$-subalgebras of $A$ containing $S$
is called the \indexn{C*SubAlgebrageneratedbyaset@$C^*$-subalgebra generated
  by a set>defemidx}\defem{$C^*$-subalgebra generated by $S$}.
}

\Definition{ 
An element $a$\ of a $C^*$-algebra is called
\defem{selfadjoint} if $a=a^*$. \indexn{selfadjoint 
element of a $C^*$-algebra>defemidx}
}

\Remark{
Like any algebra a $C^*$-algebra $A$\ has at most one unit $1$.
Namely, let $1'$\ be another unit, then
\[  
1= 1\cdot 1' = 1'.  
\]
Now we have for all $a\in A$
\[
1^*a= (1^*a)^{**} = (a^*1^{**})^* = (a^*1)^* = a^{**}= a
\]
and similarly one sees $a1^* = a$.  
Thus $1^*$\ is also a unit.
By uniqueness $1=1^*$, i.~e., the unit is selfadjoint. 
Moreover,
\[ \| 1\| = \| 1^*1\| = \|1\|^2,  \]
hence $\|1\| = 1$\ or $\|1\| =0$. 
In the second case $1=0$\ and therefore $A=0$. 
Hence we may (and will) from now on assume that $\|1\|=1$. 
}

\Example{
  \begin{enumerate}
  \item In Example~\ref{OperatorenaufHilbertraum}
    the algebra $A=\LL(H)$\ has a unit $1=\id_H$.
  \item The algebra $A=C_0(X)$\ has a unit $f\equiv 1$\ if and only if
    $C_0(X) = \stetig(X)$, i.~e., if and only if $X$\ is compact.
  \end{enumerate}
}

Let $A$\ be a $C^*$--algebra with unit $1$. 
We write $A^\times$\ for the set of invertible elements in $A$.  
\indexs{A<+@$A^\times$, invertible elements in $A$} 
If $a\in A^\times$, then also $a^*\in A^\times$ because
\[
a^*\cdot(a^{-1})^* = (a^{-1}a)^* = 1^* =1,
\]
and similarly $(a^{-1})^*\cdot a^* = 1$.
Hence $(a^*)^{-1}= (a^{-1})^*$.

\begin{lemma} \label{Verknuepfungenstetig}
Let $A$\ be a $C^*$-algebra. 
Then the maps
\begin{alignat*}{2}
A\times A &\rightarrow A, & (a,b) &\mapsto a+b,\\
\Co\times A &\rightarrow A, & (\alpha,a) &\mapsto \alpha a,\\
A\times A &\rightarrow A, & (a,b) &\mapsto a\cdot b,\\
A^\times &\rightarrow A^\times,\quad & a &\mapsto a^{-1},\\
A &\rightarrow A, \quad & a &\mapsto a^*,
\end{alignat*}
are continuous.
\end{lemma}

\begin{proof}
(a)
The first two maps are continuous for all normed vector spaces.
This easily follows from the triangle inequality and from homogeneity of
the norm.

(b) 
\textit{Continuity of multiplication. } 
Let $a_0$, $b_0\in A$. 
Then we have for all $a$, $b\in A$\ with $\|a - a_0\| < \varepsilon$\ and
$\|b-b_0\| < \varepsilon$: 
\begin{eqnarray*}
\|ab - a_0b_0\| 
&=& 
\| ab - a_0b + a_0b - a_0b_0\| \\
&\le& 
\|a-a_0\|\cdot \|b\| + \|a_0\|\cdot \|b-b_0\| \\
&\le& 
\varepsilon\big(\|b-b_0\| + \|b_0\|\big) + \|a_0\|\cdot \varepsilon \\
&\le&
\varepsilon\big(\varepsilon+\|b_0\|\big) + \|a_0\|\cdot\varepsilon .
\end{eqnarray*}
(c)
\textit{Continuity of inversion. } 
Let $a_0\in A^\times$. 
Then we have for all $a\in A^\times$\ with
$\|a-a_0\|<\varepsilon<\|a_0^{-1}\|^{-1}$:
\begin{eqnarray*}
\|a^{-1}-a_0^{-1}\|  
&=& 
\| a^{-1}(a_0-a)a_0^{-1}\|  \\
&\le& 
\|a^{-1}\|\cdot\|a_0-a\|\cdot \|a_0^{-1}\|  \\
&\le& 
\big(\|a^{-1}-a_0^{-1}\| + \|a_0^{-1}\|\big) \cdot\varepsilon\cdot\|a_0^{-1}\|.
\end{eqnarray*}
Thus
$$
\underbrace{\big(1-\varepsilon\|a_0^{-1}\|\big)}_{>0,\text{ since
  }\varepsilon<\|a_0^{-1}\|^{-1}} \|a^{-1} - a_0^{-1}\|  
      \le \varepsilon\cdot \|a_0^{-1}\|^2
$$
and therefore
$$
\|a^{-1} - a_0^{-1}\| \le \frac{\varepsilon}{1-\varepsilon\|a_0^{-1}\|} \cdot
\|a_0^{-1}\|^2.
$$

(d)
\textit{Continuity of $*$} is clear because $*$ is an isometry.
\end{proof}

\Remark{
If $(A,\|\cdot\|,*)$ satisfies the axioms of a $C^*$-algebra except that
$(A,\|\cdot\|)$ is not complete, then the above lemma still holds because
completeness has not been used in the proof.
Let $\bar A$ be the completion of $A$ with respect to the norm $\|\cdot\|$.
By the above lemma $+$, $\cdot$, and $*$ extend continuously to $\bar A$ thus  
making $\bar A$ into a $C^*$-algebra.
}

\Definition{
Let $A$\ be a $C^*$--algebra with unit $1$. 
For $a\in A$\ we call
\[
r_A(a) := \{\lambda\in\Co\mid \lambda\cdot 1 -a \in A^\times\}
\]
the \defem{\defidx{resolvent set} of $a$} 
\indexs{rAa@$r_A(a)$, resolvent set of $a\in A$}
and
\[
\sigma_A(a) := \Co \setminus r_A(a)
\]
the \defem{\defidx{spectrum} of $a$}. 
\indexs{sigmaAa@$\sigma_A(a)$, spectrum of $a\in A$}
For $\lambda\in r_A(a)$
\[
(\lambda\cdot 1 - a)^{-1} \in A
\]
is called the \defem{\defidx{resolvent} of $a$\ at $\lambda$}. 
Moreover, the number
\[
\rho_A(a) := \sup\{|\lambda|\mid \lambda\in\sigma_A(a)\}
\]
is called the \defem{\defidx{spectral radius} of $a$}.
\indexs{rhoAa@$\rho_A(a)$, spectral radius of $a\in A$}
}

\Example{
Let $X$\ be a compact Hausdorff space and let $A=\stetig(X)$. 
Then
\begin{eqnarray*}
  &A^\times = \{ f\in \stetig(X) \mid f(x)\not=0\text{ for all }x\in
  X\},& \\
  &\sigma_{\stetig(X)}(f) = f(X) \subset \Co,& \\
  &r_{\stetig(X)}(f) = \Co\setminus f(X),& \\
  &\rho_{\stetig(X)}(f) = \|f\|_\infty = \max_{x\in X}|f(x)|.&
\end{eqnarray*}
}

\begin{prop}\label{prop:spektrum}
Let $A$\ be a $C^*$--algebra with unit $1$ and let $a\in A$. 
Then $\sigma_A(a)\subset \Co$\ is a nonempty compact subset and the
resolvent
\[
r_A(a)\rightarrow A,\qquad \lambda\mapsto(\lambda\cdot 1-a)^{-1},
\]
is continuous. 
Moreover,
\[
\rho_A(a) = \lim_{n\rightarrow \infty} \|a^n\|^{\frac{1}{n}} =
\inf_{n\in\N} \| a^n\|^{\frac{1}{n}}\le \|a\|.
\]
\end{prop}

\begin{proof}
(a) 
Let $\lambda_0\in r_A(a)$. For $\lambda\in\Co$\ with
\begin{equation}\label{neumannklein}
|\lambda-\lambda_0| < \| (\lambda_0 1 - a)^{-1}\|^{-1}
\end{equation}
the Neumann series
\[
\sum_{m=0}^\infty (\lambda_0 - \lambda)^m(\lambda_01 - a)^{-m-1}
\]
converges absolutely because
\begin{align*}
\| (\lambda_0 - \lambda)^m(\lambda_01 - a)^{-m-1}\| &\le |\lambda_0 - \lambda|^m\cdot
\|(\lambda_01 - a)^{-1}\|^{m+1}\\
&= \|(\lambda_01-a)^{-1}\|\cdot\Big(
\underbrace{\frac{\|(\lambda_01-a)^{-1}\|}{|\lambda_0-\lambda|^{-1}}}_{<1\text{
    by (\ref{neumannklein})}}
\Big)^m.
\end{align*}
Since $A$ is complete the Neumann series converges in $A$.
It converges to the resolvent $(\lambda1 -a)^{-1}$ because
\begin{eqnarray*}
\lefteqn{(\lambda 1 - a)\sum_{m=0}^\infty (\lambda_0 - \lambda)^m(\lambda_01 - a)^{-m-1}}\\
&=&
[(\lambda-\lambda_0)1+(\lambda_0 1 - a)]
\sum_{m=0}^\infty (\lambda_0 - \lambda)^m(\lambda_01 - a)^{-m-1}\\
&=&
-\sum_{m=0}^\infty (\lambda_0 - \lambda)^{m+1}(\lambda_01 - a)^{-m-1}
+\sum_{m=0}^\infty (\lambda_0 - \lambda)^m(\lambda_01 - a)^{-m}\\
&=&
1.
\end{eqnarray*}
Thus we have shown $\lambda\in r_A(a)$\ for all $\lambda$ satisfying
(\ref{neumannklein}). 
Hence $r_A(a)$\ is open and $\sigma_A(a)$\ is closed.

(b) 
\textit{Continuity of the resolvent.}  
\label{Resolventestetig}  
  We estimate the difference of the resolvent of $a$\ at $\lambda_0$\ and
  at $\lambda$\ using the Neumann series.
  If $\lambda$ satisfies (\ref{neumannklein}), then
  \begin{align*}
    \left\|(\lambda1-a)^{-1}- (\lambda_01-a)^{-1}\right\|
    &= \Big\|\sum_{m=0}^\infty (\lambda_0-\lambda)^m(\lambda_01-a)^{-m-1} 
               - (\lambda_01-a)^{-1}\Big\| \\
    &\le \sum_{m=1}^\infty |\lambda_0-\lambda|^m\,
    \|(\lambda_01-a)^{-1}\|^{m+1} \\ 
    &= \|(\lambda_01-a)^{-1}\|\cdot
    \frac{|\lambda_0-\lambda|\cdot\|(\lambda_01-a)^{-1}\|}{1-|\lambda_0
    -\lambda|\cdot\|(\lambda_01-a)^{-1}\|} 
    \\ 
    &= |\lambda_0-\lambda|\cdot\frac{\|(\lambda_0
    1-a)^{-1}\|^2}{1-|\lambda_0-\lambda|\cdot\|(\lambda_01-a)^{-1}\|}  \\ 
&\to 0\quad\text{ for }\lambda\rightarrow \lambda_0.
  \end{align*}
Hence the resolvent is continuous.

(c) 
We show $\rho_A(a) \le \inf_{n}\|a^n\|^{\frac{1}{n}}
      \le \liminf_{n\rightarrow
      \infty}\|a^n\|^{\frac{1}{n}}${. } 
  Let $n\in \N$\ be fixed and let $|\lambda|^n >\|a^n\|$. 
  Each $m\in \N_0$\ can be written uniquely in the form
  $m=pn + q$, $p$, $q\in\N_0$, $0\le q \le n-1$. 
  The series
\[
\frac{1}{\lambda} \sum_{m=0}^\infty \Big(\frac{a}{\lambda}\Big)^m =
\frac{1}{\lambda}\sum_{q=0}^{n-1} \Big(\frac{a}{\lambda}\Big)^q
\sum_{p=0}^\infty\Big(\underbrace{\frac{a^n}{\lambda^n}}_{\|\cdot\|<1}\Big)^p
\]
converges absolutely.
Its limit is $(\lambda1-a)^{-1}$ because
\[
\big(\lambda1-a\big)\cdot\Big(\sum_{m=0}^\infty \lambda^{-m-1}a^m\Big) =
\sum_{m=0}^\infty \lambda^{-m}a^m - \sum_{m=0}^\infty \lambda^{-m-1}a^{m+1} =
1
\]
and similarly
\[
\Big(\sum_{m=0}^\infty \lambda^{-m-1}a^m\Big)\cdot\big(\lambda1-a\big) =1.
\]
Hence for $|\lambda|^n>\|a^n\|$\ the element $(\lambda1-a)$\ is invertible and
thus $\lambda\in r_A(a)$. 
Therefore
\[
\rho_A(a) \le \inf_{n\in\N}\|a^n\|^{\frac{1}{n}} \le
\liminf_{n\rightarrow \infty}\|a^n\|^{\frac{1}{n}}.
\]
(d) 
We show $\rho_A(a) \ge \limsup_{n\rightarrow\infty}\|a^n\|^{\frac{1}{n}}$.
We abbreviate
$\widetilde\rho(a):=\limsup_{n\rightarrow\infty}\|a^n\|^{\frac{1}{n}}$. 

\textit{Case 1:}
$\widetilde\rho(a) =0${. } 
  If $a$ were invertible, then
\[
1=\|1\| = \|a^n a^{-n}\| \le \|a^n\|\cdot\|a^{-n}\|
\]
would imply $1\le\widetilde\rho(a) \cdot\widetilde\rho(a^{-1}) = 0$,
which yields a contradiction.
Therefore $a\not\in A^\times$. 
Thus $0\in\sigma_A(a)$.
In particular, the spectrum of $a$\ is nonempty. 
Hence the spectral radius $\rho_A(a)$\ is bounded from below by $0$\ and thus
\[ \widetilde\rho(a) =0 \le \rho_A(a).  \]

\textit{Case 2:} 
$\widetilde{\rho}(a) >0${. } 
If $a_n\in A$\ are elements for which $R_n := (1-a_n)^{-1}$
exist, then
\[ 
a_n \rightarrow 0 \quad\Leftrightarrow \quad R_n\rightarrow 1.  
\]
This follows from the fact that the map $A^\times \rightarrow A^\times$, 
$a \mapsto a^{-1}$, is continuous by Lemma~\ref{Verknuepfungenstetig}.
Put
\[  
S:= \{ \lambda\in\Co\mid |\lambda| \ge \widetilde\rho(a)\}.  
\]
We want to show that $S\not\subset r_A(a)$ since then there exists
$\lambda\in\sigma_A(a)$\ such that $|\lambda| \ge \widetilde\rho(a)$ and hence
\[  
\rho_A(a) \ge |\lambda| \ge \widetilde\rho(a).  
\]
Assume in the contrary that $S\subset r_A(a)$. 
Let $\omega\in\Co$\ be an $n$--th root of unity, i.~e., $\omega^n=1$. 
For $\lambda\in S$\ we also have $\tfrac{\lambda}{\omega^k}\in S\subset
r_A(a)$. 
Hence there exists
\[  
\Big(\frac{\lambda}{\omega^k} 1-a\Big)^{-1} =
\frac{\omega^k}{\lambda}\Big(1-\frac{\omega^ka}{\lambda}\Big)^{-1}
\]
and we may define
\[  
R_n(a,\lambda) := 
\frac{1}{n} \sum_{k=1}^n \Big(1-\frac{\omega^k a}{\lambda}\Big)^{-1}.
\]
We compute
\begin{eqnarray*}
\Big(1-\frac{a^n}{\lambda^n}\Big) R_n(a,\lambda)
&=& 
\frac{1}{n}\sum_{k=1}^n \sum_{l=1}^n
\Big(\frac{\omega^{k(l-1)}a^{l-1}}{\lambda^{l-1}} -
\frac{\omega^{kl} a^l}{\lambda^l}\Big) 
\Big(1-\frac{\omega^ka}{\lambda}\Big)^{-1} \\
&=& 
\frac{1}{n} \sum_{k=1}^n\sum_{l=1}^n 
\frac{\omega^{k(l-1)}a^{l-1}}{\lambda^{l-1}} \\
&=&
\frac{1}{n}\sum_{l=1}^n \frac{a^{l-1}}{\lambda^{l-1}} 
\underbrace{\sum_{k=1}^n (\omega^{l-1})^k}_{=
\begin{cases} 0\:\text{ if }l\ge 2 \\ n\:\text{ if }l=1
\end{cases}} \\
&=& 
1.
\end{eqnarray*}
Similarly one sees $R_n(a,\lambda)\big(1-\frac{a^n}{\lambda^n}\big)=1$. 
Hence 
\[ 
R_n(a,\lambda) = \Big(1-\frac{a^n}{\lambda^n}\Big)^{-1}
\] 
for any $\lambda\in S\subset r_A(a)$.
Moreover for $\lambda\in S$ we have
\begin{eqnarray*}
\lefteqn{\Big\|\Big(1-\frac{a^n}{\widetilde\rho(a)^n}\Big)^{-1} -
\Big(1-\frac{a^n}{\lambda^n}\Big)^{-1}\Big\|} \\
&\le&
\frac{1}{n} \sum_{k=1}^n \Big\|\Big(1-\frac{\omega^k
a}{\widetilde\rho(a)}\Big)^{-1} -\Big(1-\frac{\omega^k
a}{\lambda}\Big)^{-1}\Big\| \\
&=&
\frac{1}{n} \sum_{k=1}^n \Big\|\Big(1-\frac{\omega^k
a}{\widetilde\rho(a)}\Big)^{-1} \Big(1-\frac{\omega^k a}{\lambda} -1 +
\frac{\omega^k a}{\widetilde\rho(a)}\Big)\Big(1-\frac{\omega^k
a}{\lambda}\Big)^{-1}\Big\| \\[1ex]
&=&
\frac{1}{n} \sum_{k=1}^n \Big\|\Big(\frac{\widetilde\rho(a)}{\omega^k}1
-a\Big)^{-1} \Big(-\frac{\widetilde\rho(a)a}{\omega^k} +
\frac{\lambda a}{\omega^k}\Big)\Big(\frac{\lambda}{\omega^k}1- a\Big)^{-1}
\Big\| \\[1ex]
&\le& 
|\widetilde\rho(a) -\lambda|\cdot \|a\| \cdot\sup_{z\in S}
\|(z1-a)^{-1}\|^2.
\end{eqnarray*}
The supremum is finite since $z\mapsto (z1-a)^{-1}$\ is continuous
on $r_A(a)\supset S$\ by part (b) of the proof and since for $|z|\ge
2\cdot \|a\|$\ we have
\[
\|(z1-a)^{-1}\| \le \frac{1}{|z|} \sum_{n=0}^\infty
\underbrace{\frac{\|a\|^n}{|z|^n}}_{\le(\frac{1}{2})^n} \le \frac{2}{|z|} 
\le \frac{1}{\|a\|}.
\]
\parbox[c]{.4\textwidth}{
Outside the annulus $\overline{B}_{2\|a\|}(0) -B_{\widetilde\rho(a)}(0)$\ the
expression 
$\|(z1-a)^{-1}\|$\ is bounded by $\frac{1}{\|a\|}$ and on the compact
annulus it is bounded by continuity.}
\qquad
\parbox[c]{{.6\textwidth}-2em}{
\psset{unit=1cm}
\psset{arrowsize=4pt 3, arrowinset=.6, arrowlength=1}
%\psset{showpoints=true}

\pspicture(-3,-2)(1,2)
\pscircle(0,0){2}
\pscircle(0,0){1}
\uput[270](0,0){\footnotesize$0$}
\SpecialCoor
\psline[linewidth=.5pt]{*->}(0,0)(2;140)
\uput{2pt}[225](1.25;140){\footnotesize$2\|a\|$}
\psline[linewidth=.5pt]{*->}(0,0)(1;40)
\uput{2pt}[315](.25;40){\footnotesize$\tilde{\rho}(a)$}
\endpspicture

\abb{$\|(z1-a)^{-1}\|$ is bounded}
}
Put
\[
C:=\| a\| \cdot\sup_{z\in S} \|(z1-a)^{-1}\|^2.
\]
We have shown
\[
\| R_n(a,\widetilde\rho(a)) - R_n(a,\lambda)\| 
\le C\cdot |\widetilde\rho(a) -\lambda|
\]
for all $n\in\N$\ and all $\lambda\in S$. 
Putting $\lambda = \widetilde\rho (a)+\frac{1}{j}$\ we obtain
$$
\Big\|\Big(1-\frac{a^n}{\widetilde\rho(a)^n}\Big)^{-1} -
  \underbrace{%
  \Big(1-\underbrace{%
  \frac{a^n}{(\widetilde\rho(a) + \frac{1}{j})^n}} _{\rightarrow 0\text{
    for }n\rightarrow\infty}\Big)^{-1}}_{\rightarrow 1\text{ for
    }n\rightarrow\infty} 
\Big\| \le \frac{C}{j},
$$
thus
$$
\limsup_{n\rightarrow \infty}
  \Big\|\Big(1-\frac{a^n}{\widetilde\rho(a)^n}\Big)^{-1} -1 \Big\| \le
  \frac{C}{j}
$$
for all $j\in\N$ and hence
$$
\limsup_{n\rightarrow \infty}
  \Big\|\Big(1-\frac{a^n}{\widetilde\rho(a)^n}\Big)^{-1} -1 \Big\| = 0.
$$
For $n\to\infty$ we get
$$
\Big(1-\frac{a^n}{\widetilde\rho(a)^n}\Big)^{-1}
\to 1 
$$
and thus
\begin{equation}
\frac{\|a^n\|}{\widetilde\rho(a)^n} \to 0.
\label{eq:adrhog0}
\end{equation}
On the other hand we have
\begin{align*}
  \|a^{n+1}\|^{\frac{1}{n+1}} 
  &\le \|a\|^{\frac{1}{n+1}}\cdot\|a^n\|^{\frac{1}{n+1}} \\
  &= \|a\|^{\frac{1}{n+1}} \cdot\|a^n\|^{-\frac{1}{n(n+1)}}\cdot
  \|a^n\|^{\frac{1}{n}}\\
  &\le \|a\|^{\frac{1}{n+1}} \cdot\|a\|^{-\frac{n}{n(n+1)}} \cdot
  \|a^n\|^{\frac{1}{n}}\\ 
  &= \|a^n\|^{\frac{1}{n}}.
\end{align*}
Hence the sequence $\left(\|a^n\|^{\frac{1}{n}}\right)_{n\in\N}$\
is monotonically nonincreasing and therefore
\[ 
  \widetilde\rho(a) =\limsup_{k\rightarrow \infty}\|a^k\|^{\frac{1}{k}} \le
  \|a^n\|^{\frac{1}{n}} \qquad\text{ for all } n\in \N.
\]
Thus $1\le \frac{\|a^n\|}{\widetilde\rho(a)^n}$ for all $n\in\N$,
in contradiction to (\ref{eq:adrhog0}).

(e)
\textit{The spectrum is nonempty. }
If $\sigma(a)=\emptyset$, then $\rho_A(a) = - \infty$ contradicting
$\rho_A(a) = \lim_{n\rightarrow \infty} \|a^n\|^{\frac{1}{n}} \ge 0$.
\end{proof}

\Definition{
Let $A$\ be a $C^*$-algebra with unit. 
Then $a\in A$\ is called
\begin{itemize}
\item
\defem{normal}\indexn{normal element of a $C^*$-algebra>defemidx}, 
if $aa^* = a^*a$,
\item
\defem{an isometry}\indexn{isometric element of a $C^*$--algebra>defemidx}, 
if $a^*a=1$, and 
\item
\defem{unitary}\indexn{unitary element of a $C^*$-algebra>defemidx}, if
$a^*a=aa^*=1$. 
\end{itemize}
}

\Remark{
In particular, selfadjoint elements are normal. 
In a commutative algebra all elements are normal.
}

\begin{prop} \label{EigenschaftenSpektrum}
Let $A$\ be a $C^*$--algebra with unit and let $a\in A$. 
Then the following holds: 
\begin{enumerate}
\item \label{sigmaquer}
$\sigma_A(a^*) =\overline{\sigma_A(a)}= \{\lambda\in\Co\,|\,
\ovl{\lambda}\in\sigma_A(a)\}$.
\item \label{sigmainvers}
If $a\in A^\times$, then $\sigma_A(a^{-1}) = \sigma_A(a)^{-1}$.
\item \label{normal}
If $a$\ is normal, then $\rho_A(a) = \|a\|$.
\item \label{sigmaisometrie}
If $a$\ is an isometry, then $\rho_A(a) =1$.
\item \label{unitaer}
If $a$\ is unitary, then $\sigma_A(a) \subset S^1\subset \Co$.
\item \label{sa}
If $a$\ is selfadjoint, then $\sigma_A(a) \subset
[-\|a\|,\|a\|]$\ and moreover $\sigma_A(a^2) \subset [0,\|a\|^2]$.
\item \label{Polynominvarianz}
If $P(z)$\ is a polynomial with complex coefficients and $a\in A$\ is
arbitrary, then 
\[  
\sigma_A\big(P(a)\big) = P\big(\sigma_A(a)\big) = \{P(\lambda)\,|\,
\lambda\in\sigma_A(a)\}.
\]
\end{enumerate}
\end{prop}

\begin{proof}
We start by showing assertion (\ref{sigmaquer}).
A number $\lambda$\ does not lie in the spectrum of $a$\ if and only if
$(\lambda1-a)$\ is invertible, i.~e., if and only if
$(\lambda1-a)^*=\overline\lambda 1-a^*$\ is invertible, i.~e., if and only
if $\overline\lambda$\ does not lie in the spectrum of $a^*$.

To see (\ref{sigmainvers}) let $a$\ be invertible.
Then $0$\ lies neither in the spectrum $\sigma_A(a)$\ of $a$\ nor in the
spectrum $\sigma_A(a^{-1})$\ of $a^{-1}$.
Moreover, we have for $\lambda\not=0$
$$
\lambda1-a = \lambda a(a^{-1} - \lambda^{-1}1)
$$
and
$$
\lambda^{-1} 1 -a^{-1} = \lambda^{-1}a^{-1}(a-\lambda1).
$$
Hence $\lambda1-a$\ is invertible if and only if
$\lambda^{-1}1-a^{-1}$\ is invertible.

To show (\ref{normal}) let $a$\ be normal. 
Then $a^*a$ is selfadjoint, in particular normal.
Using the $C^*$--property we obtain inductively
  \begin{align*}
    \|a^{2^n}\|^2 &= \|(a^{2^n})^*a^{2^n}\| = \|(a^*)^{2^n}a^{2^n}\|
    = \| (a^*a)^{2^n}\| \\
    &= \|({a^*a})^{2^{n-1}} (a^*a)^{2^{n-1}}\| = \| (a^*a)^{2^{n-1}}\|^2 \\
    &= \cdots = \|a^*a\|^{2^n} = \|a\|^{2^{n+1}}.
  \end{align*}
Thus
$$
\rho_A(a) = \lim_{n\rightarrow\infty} 
\|a^{2^n}\|^{\frac{1}{2^n}} = \lim_{n\rightarrow\infty}\| a\| = \|a\|.
$$

To prove (\ref{sigmaisometrie}) let $a$\ be an isometry. 
Then
$$
    \|a^n\|^2 = \|(a^n)^*a^n\| = \|(a^*)^n a^n\| = \| 1\| = 1.
$$
Hence
$$
\rho_A(a) = \lim_{n\rightarrow \infty} \|a^n\|^{\frac{1}{n}} = 1.
$$

For assertion (\ref{unitaer}) let $a$\ be unitary. 
On the one hand we have by (\ref{sigmaisometrie})
\[  \sigma_A(a) \subset \{\lambda\in\Co \mid |\lambda|\le 1\}.  \]
On the other hand we have
\[
\sigma_A(a) \overset{(\ref{sigmaquer})}{=} \overline{\sigma_A(a^*)} =
\overline{\sigma_A(a^{-1})} \overset{(2)}{=} \overline{\sigma_A(a)}^{-1}. 
\]
Both combined yield $\sigma_A(a) \subset S^1$.
 
To show (\ref{sa}) let $a$\ be selfadjoint. 
We need to show $\sigma_A(a)\subset\R$. 
Let $\lambda\in\R$\ with $\lambda^{-1}>\|a\|$.
Then $|-i\lambda^{-1}|=\lambda^{-1}>\rho(a)$ and hence $1+i\lambda a =
i\lambda(-i\lambda^{-1} + a)$\ is invertible. 
Put
\[ 
U:= (1-i\lambda a)(1+i\lambda a)^{-1}.  
\]
Then $U^*=((1+i\lambda a)^{-1})^*(1-i\lambda a)^* = (1-i\lambda
a^*)^{-1}\cdot(1+i\lambda a^*) = (1-i\lambda a)^{-1}\cdot(1+i\lambda a)$ and
therefore 
\begin{eqnarray*}
U^*U 
&=& 
(1-i\lambda a)^{-1}\cdot(1+i\lambda a)(1-i\lambda a)(1+i\lambda a)^{-1} \\
&=& 
(1-i\lambda a)^{-1}(1-i\lambda a)(1+i\lambda a)(1+i\lambda a)^{-1} \\
&=& 1 .
\end{eqnarray*}
Similarly $UU^* =1$, i.~e., $U$\ is unitary. 
By (\ref{unitaer}) $\sigma_A(U)\subset S^1$.
A simple computation with complex numbers shows that
\[
|(1-i\lambda \mu)(1+i\lambda\mu)^{-1}| 
= 1 \quad\Leftrightarrow\quad \mu \in \R.
\]
Thus $(1-i\lambda \mu)(1+i\lambda\mu)^{-1}\cdot 1-U$ is invertible if 
$\mu \in\Co\setminus\R$.
From
\begin{eqnarray*}
\lefteqn{(1-i\lambda\mu)(1+i\lambda\mu)^{-1}\cdot 1-U} \\
&=& 
(1+i\lambda\mu)^{-1}\big((1-i\lambda\mu)(1+i\lambda a)1 -
      (1+i\lambda\mu)(1-i\lambda a)\big)(1+i\lambda a)^{-1} \\
&=& 
2i\lambda(1+i\lambda\mu)^{-1}(a-\mu1)(1+i\lambda\mu)^{-1}
\end{eqnarray*}
we see that $a-\mu1$\ is invertible for all $\mu\in\Co\setminus\R$.
Thus $\mu\in r_A(a)$\ for all $\mu\in\Co\setminus\R$\ and hence
$\sigma_A(a)\subset \R$.
The statement about $\sigma_A(a^2)$\ now follows from 
part~(\ref{Polynominvarianz}).

Finally, to prove (\ref{Polynominvarianz}) decompose the polynomial $P(z)
-\lambda$\ into linear factors 
\[
P(z) -\lambda = \alpha\cdot \prod_{j=1}^n(\alpha_j-z),\qquad
\alpha,\alpha_j\in\Co. 
\]
We insert an algebra element $a\in A$:
\[  P(a) -\lambda1 = \alpha \cdot\prod_{j=1}^n (\alpha_j1-a).  \]
Since the factors in this product commute the product is invertible if and
only if all factors are invertible.\footnote{This is generally true in
  algebras with unit.  
Let $b=a_1\cdots a_n$\ with commuting factors. 
Then $b$\ is invertible if all factors are invertible: 
$b^{-1}=a_n^{-1}\cdots a_1^{-1}$. 
Conversely, if $b$\ is invertible, then $a_i^{-1} = b^{-1}
  \cdot\prod_{j\not=i}a_j$ where we have used that the factors commute.} 
In our case this means
\begin{eqnarray*}
\lambda\in\sigma_A\big(P(a)\big)
&\Leftrightarrow&
\text{at least one factor is noninvertible} \\
&\Leftrightarrow& 
\alpha_j\in\sigma_A(a) \text{ for some }j \\
&\Leftrightarrow&
\lambda=P(\alpha_j)\in P\big(\sigma_A(a)\big).
\end{eqnarray*}
\end{proof}

\begin{cor}
Let $(A,\|\cdot\|,*)$\ be a $C^*$--algebra with unit. 
Then the norm $\|\cdot\|$\ is uniquely determined by $A$\ and $*$.
\end{cor}

\begin{proof}
For $a\in A$ the element $a^*a$ is selfadjoint and hence
\[
\|a\|^2 = \|a^*a\|
\overset{\ref{EigenschaftenSpektrum} (\ref{normal})}{=} \rho_A(a^*a)
\]
depends only on $A$\ and $*$.
\end{proof}

\Definition{
Let $A$\ and $B$\ be $C^*$--algebras. 
An algebra homomorphism
\[ \pi:A\rightarrow B  \]
is called \defem{$*$--morphism}\indexn{*Morphismus@$*$--morphism>defemidx} 
if for all $a\in A$\ we have
\[  \pi(a^*) = \pi(a)^*.  \]
A map $\pi:A\rightarrow A$\ is called
\defem{$*$--automorphism}\indexn{*Automorphismus@$*$--automorphism>defemidx}
if it is an invertible $*$-morphism.
}

\begin{cor}\label{Chomo}
Let $A$\ and $B$\ be $C^*$--algebras with unit. 
Each unit-preserving $*$--morphism $\pi:A\rightarrow B$\ satisfies
\[  
\|\pi(a)\| \le \|a\|  
\]
for all $a\in A$. 
In particular, $\pi$\ is continuous.
\end{cor}

\begin{proof}
For $a\in A^\times$
\[  \pi(a)\pi(a^{-1}) = \pi(aa^{-1}) = \pi(1) = 1  \]
holds and similarly $\pi(a^{-1})\pi(a) = 1$. 
Hence $\pi(a) \in B^\times$\ with $\pi(a)^{-1}=\pi(a^{-1})$. 
Now if $\lambda\in r_{ A}(a)$, then
\[  \lambda1-\pi(a) = \pi(\lambda1-a) \in \pi( A^\times)\subset  B^\times,  \]
i.~e., $\lambda \in r_B(\pi(a))$. 
Hence $r_A(a) \subset r_B(\pi(a))$ and $\sigma_B(\pi(a)) \subset
\sigma_A(a)$. 
This implies the inequality
\[  \rho_B(\pi(a)) \le \rho_A(a).  \]
Since $\pi$\ is a $*$--morphism and $a^*a$ and $\pi(a)^*\pi(a)$ are 
selfadjoint we can estimate the norm as follows:
\begin{align*}
  \|\pi(a)\|^2 &= \|\pi(a)^*\pi(a)\| 
= \rho_B\big(\pi(a)^*\pi(a)\big) 
= \rho_B\big(\pi(a^*a)\big) \\
&\le \rho_A(a^*a) =\|a\|^2.
\end{align*}
\end{proof}

\begin{cor}\label{Cauto}
Let $A$\ be a $C^*$--algebra with unit. 
Then each unit-preserving $*$-automorphism $\pi:A\rightarrow A$\ satisfies for
all $a\in A$:
\[ \|\pi(a)\| = \|a\|  \]
\end{cor}

\begin{proof}
\[
\|\pi(a)\| \le \| a\| = \|\pi^{-1}\big(\pi(a)\big)\| \le \|\pi(a)\|.
\]
\end{proof}

We extend Corollary~\ref{Cauto} to the case where $\pi$ is injective but not
necessarily onto.
This is not a direct consequence of Corollary~\ref{Cauto} because it is not 
a priori clear that the image of a $*$-morphism is closed and hence a
$C^*$-algebra in its own right.

\begin{prop}\label{Cmono}
Let $A$\ and $B$\ be $C^*$--algebras with unit. 
Each injective unit-preserving $*$--morphism $\pi:A\rightarrow B$\ satisfies
\[  \|\pi(a)\| = \|a\|  \]
for all $a\in A$. 
\end{prop}

\begin{proof}
By Corollary~\ref{Chomo} we only have to show $\|\pi(a)\| \geq \|a\|$.
Once we know this inequality for selfadjoint elements it follows for all
$a\in A$ because
$$
\|\pi(a)\|^2 \,=\, \|\pi(a)^*\pi(a)\| \,=\, \|\pi(a^*a)\| \,\geq\, \|a^*a\|
\,=\, \|a\|^2 .
$$
Assume there exists a selfadjoint element $a\in A$ such that $\|\pi(a)\| <
\|a\|$.
By Proposition~\ref{EigenschaftenSpektrum} $\sigma_A(a) \subset
[-\|a\|,\|a\|]$ and $\rho_A(a) = \|a\|$, hence $\|a\|\in \sigma_A(a)$ or 
$-\|a\|\in \sigma_A(a)$.
Similarly, $\sigma_B(\pi(a)) \subset [-\|\pi(a)\|,\|\pi(a)\|]$. 

Choose a continuous function $f:[-\|a\|,\|a\|] \to \R$ such that $f$ vanishes
on $[-\|\pi(a)\|,\|\pi(a)\|]$ and $f(-\|a\|)=f(\|a\|)=1$.
By the Stone-Weierstrass theorem we can find polynomials $P_n$ such that 
$\|f-P_n\|_{C^0([-\|a\|,\|a\|])} \to 0$ as $n \to \infty$.
In particular, $\|P_n\|_{C^0([-\|\pi(a)\|,\|\pi(a)\|])}=
\|f-P_n\|_{C^0([-\|\pi(a)\|,\|\pi(a)\|])} \to 0$ as $n \to
\infty$.
We may and will assume that the polynomials $P_n$ are real.

From $\sigma_B(P_n(\pi(a))) = P_n(\sigma_B(\pi(a))) \subset
P_n([-\|\pi(a)\|,\|\pi(a)\|])$ we see
$$
\|P_n(\pi(a))\| = \rho_B(P_n(\pi(a))) \leq \max |P_n([-\|\pi(a)\|,\|\pi(a)\|])|
\stackrel{n\to\infty}{\longrightarrow} 0
$$
and thus
$$
\lim_{n\to \infty}P_n(\pi(a)) = 0.
$$
The sequence $(P_n(a))_n$ is a Cauchy sequence because
\begin{eqnarray*}
\|P_n(a)-P_m(a)\| &=& \rho_A(P_n(a)-P_m(a)) \\
&\leq&
\max |(P_n-P_m)([-\|a\|,\|a\|])|\\ 
&=& 
\|P_n-P_m\|_{C^0([-\|a\|,\|a\|])}\\
&\leq& 
\|P_n-f\|_{C^0([-\|a\|,\|a\|])} +
\|f-P_m\|_{C^0([-\|a\|,\|a\|])}.
\end{eqnarray*}
Denote its limit by $f(a)\in A$.
Since $\|a\|\in \sigma_A(a)$ or $-\|a\|\in \sigma_A(a)$ and since
$f(\pm\|a\|)=1$ we have 
$$
\|f(a)\| = \lim_{n\to\infty} \|P_n(a)\| = \lim_{n\to\infty} \rho_B(P_n(a))\geq
\lim_{n\to\infty} 
|P_n(\pm\|a\|)| = 1.
$$
Hence $f(a) \not=0$.
But $\pi(f(a)) = \pi(\lim_{n\to\infty} P_n(a)) = 
\lim_{n\to\infty}\pi( P_n(a)) = \lim_{n\to\infty}P_n(\pi(a)) = 0$.
This contradicts the injectivity of $\pi$.
\end{proof}

%%%%%%%%%%%%%%%%%%%%%%%%%%%%%%%%%%%%%%%%%%%%%%%%%%%%%%%%%%%%%%%%%%%%%%%%%
\section{The canonical commutator relations}\label{sec-CCR}\indexn{canonical
  commutator relations>defemidx} 
%%%%%%%%%%%%%%%%%%%%%%%%%%%%%%%%%%%%%%%%%%%%%%%%%%%%%%%%%%%%%%%%%%%%%%%%%

In this section we introduce Weyl systems and CCR-representations.
They formalize the ``canonical commutator relations'' from quantum field
theory in an ``exponentiated form'' as we shall see later.
The main result of the present section is Theorem~\ref{CCRunique} which
says that for each symplectic vector space there is an essentially unique
CCR-representation.
Our approach follows ideas in \cite{Man}.
A different proof of this result may be found in \cite[Sec.~5.2.2.2]{BR2}.

Let $(V,\omega)$ be a \defem{symplectic vector space}\indexn{symplectic vector
  space>defemidx}, i.~e., $V$ is a real 
vector space of finite or infinite dimension and $\omega : V\times V \to \R$
is an antisymmetric bilinear map such that $\omega(\phi,\psi)=0$ for all
$\psi\in V$ implies $\phi=0$.

\Definition{
A \defem{\defidx{Weyl system}} of $(V,\omega)$ consists of a $C^*$-algebra
$A$ with unit and a map 
$W : V \to A$ such that for all $\phi,\psi\in V$ we
have\indexs{Wphi@$W(\phi)$, Weyl-system} 
\begin{itemize}
\item[(i)]
$W(0)=1$,
\item[(ii)]
$W(-\phi) = W(\phi)^*$,
\item[(iii)]
$W(\phi)\cdot W(\psi) = e^{-i\omega(\phi,\psi)/2}\,W(\phi+\psi)$.
\end{itemize}
}

Condition~(iii) says that $W$ is a representation of the additive group $V$ in
$A$ up to the ``twisting factor'' $e^{-i\omega(\phi,\psi)/2}$.
Note that since $V$ is not given a topology there is no requirement on $W$
to be continuous.
In fact, we will see that even in the case when $V$ is finite-dimensional
and so $V$ carries a canonical topology $W$ will in general not be continuous.

\Example{\label{CCR1}
We construct a Weyl system for an arbitrary symplectic vector space
$(V,\omega)$.
Let $H = L^2(V,\Co)$ be the Hilbert space of square-integrable complex-valued
functions on $V$ with respect to the counting measure, i.~e., $H$ consists
of those functions $F:V\to\Co$ that vanish everywhere except for countably
many points and satisfy 
$$
\|F\|_{L^2}^2 := \sum_{\phi\in V} |F(\phi)|^2 < \infty.
$$
The Hermitian product on $H$ is given by
$$
(F,G)_{L^2} = \sum_{\phi\in V} \ovl{F(\phi)}\cdot G(\phi).
$$
Let $A:=\LL(H)$ be the $C^*$-algebra of bounded linear operators on $H$
as in Example~\ref{OperatorenaufHilbertraum}.
We define the map $W:V\to A$ by
$$
(W(\phi)F)(\psi) := e^{i\omega(\phi,\psi)/2}\,F(\phi+\psi).
$$
Obviously, $W(\phi)$ is a bounded linear operator on $H$ for any
$\phi\in V$ and $W(0)=\id_H=1$.
We check (ii) by making the substitution $\chi=\phi+\psi$:
\begin{eqnarray*}
(W(\phi)F,G)_{L^2} 
&=&
\sum_{\psi\in V} \ovl{(W(\phi)F)(\psi)}\, G(\psi)\\
&=&
\sum_{\psi\in V} \ovl{e^{i\omega(\phi,\psi)/2}\,F(\phi+\psi)}\, G(\psi)\\
&=&
\sum_{\chi\in V} \ovl{e^{i\omega(\phi,\chi-\phi)/2}\,F(\chi)}\, G(\chi-\phi)\\
&=&
\sum_{\chi\in V} \ovl{e^{i\omega(\phi,\chi)/2}}\cdot\ovl{F(\chi)}\cdot
G(\chi-\phi)\\ 
&=&
\sum_{\chi\in V} \ovl{F(\chi)}\cdot e^{i\omega(-\phi,\chi)/2}\cdot
G(\chi-\phi)\\
&=&
(F,W(-\phi)G)_{L^2}.
\end{eqnarray*}
Hence $W(\phi)^* = W(-\phi)$.
To check (iii) we compute
\begin{eqnarray*}
(W(\phi)(W(\psi)F))(\chi) 
&=&
e^{i\omega(\phi,\chi)/2}\,(W(\psi)F)(\phi+\chi)\\
&=&
e^{i\omega(\phi,\chi)/2}\,e^{i\omega(\psi,\phi+\chi)/2}\,F(\phi+\chi+\psi)\\
&=&
e^{i\omega(\psi,\phi)/2}\,e^{i\omega(\phi+\psi,\chi)/2}\,F(\phi+\chi+\psi)\\
&=&
e^{-i\omega(\phi,\psi)/2}\,(W(\phi+\psi)F)(\chi).
\end{eqnarray*}
Thus $W(\phi)W(\psi)=e^{-i\omega(\phi,\psi)/2}\,W(\phi+\psi)$.
Let $\CCR(V,\omega)$\indexs{CCRV@$\CCR(V,\omega)$, CCR-algebra of a symplectic
  vector space} be the $C^*$-subalgebra of $\LL(H)$ generated by the 
elements $W(\phi)$, $\phi\in V$. 
Then $\CCR(V,\omega)$ together with the map $W$ forms a Weyl-system for
$(V,\omega)$. 
}

\begin{prop}\label{Weyl1}
Let $(A,W)$ be a Weyl system of a symplectic vector space $(V,\omega)$.
Then
\begin{enumerate}
\item\label{Weyl1-1}
$W(\phi)$ is unitary for each $\phi\in V$,
\item\label{Weyl1-2}
$\|W(\phi)-W(\psi)\| = 2$ for all $\phi,\psi\in V$, $\phi\not=\psi$,
\item\label{Weyl1-3}
The algebra $A$ is not separable unless $V=\{0\}$,
\item\label{Weyl1-4}
the family $\{W(\phi)\}_{\phi\in V}$ is linearly independent.
\end{enumerate}
\end{prop}

\begin{proof}
From $W(\phi)^*W(\phi) = W(-\phi)\,W(\phi) = e^{i\omega(-\phi,\phi)}W(0)
= 1$ and similarly $W(\phi)\,W(\phi)^* = 1$ we see that $W(\phi)$ is 
unitary.

To show (\ref{Weyl1-2}) let $\phi,\psi\in V$ with $\phi\not=\psi$.
For arbitrary $\chi\in V$ we have
\begin{eqnarray*}
W(\chi)\, W(\phi-\psi)\, W(\chi)^{-1} 
&=&
W(\chi)\, W(\phi-\psi)\, W(\chi)^{*}\\
&=&
e^{-i\omega(\chi,\phi-\psi)/2}\,W(\chi+\phi-\psi)\, W(-\chi)\\
&=&
e^{-i\omega(\chi,\phi-\psi)/2}\,e^{-i\omega(\chi+\phi-\psi,-\chi)/2}\,W(\chi+\phi-\psi-\chi)\\ 
&=&
e^{-i\omega(\chi,\phi-\psi)}\,W(\phi-\psi).
\end{eqnarray*}
Hence the spectrum satisfies
$$
\sigma_A(W(\phi-\psi)) = \sigma_A(W(\chi)\, W(\phi-\psi)\, W(\chi)^{-1}) = 
e^{-i\omega(\chi,\phi-\psi)}\,\sigma_A(W(\phi-\psi)) .
$$
Since $\phi-\psi\not=0$ the real number $\omega(\chi,\phi-\psi)$ runs through
all of 
$\R$ as $\chi$ runs through $V$.
Therefore the spectrum of $W(\phi-\psi)$ is $\mathrm{U}(1)$-invariant.
By Proposition~\ref{EigenschaftenSpektrum}~(\ref{unitaer}) the spectrum
is contained in $S^1$ and by Proposition~\ref{prop:spektrum} it is nonempty.
Hence $\sigma_A(W(\phi-\psi)) = S^1$ and therefore
$$
\sigma_A(e^{i\omega(\psi,\phi)/2}\,W(\phi-\psi)) = S^1.
$$
Thus $\sigma_A(e^{i\omega(\psi,\phi)/2}\,W(\phi-\psi)-1)$ is the circle of
radius $1$ centered at $-1$.
Now Proposition~\ref{EigenschaftenSpektrum}~(\ref{normal}) says
$$
\|e^{i\omega(\psi,\phi)/2}\,W(\phi-\psi)-1\| = 
\rho_A\left(e^{i\omega(\psi,\phi)/2}\,W(\phi-\psi)-1\right) = 2.
$$
From $W(\phi)-W(\psi) = W(\psi)(W(\psi)^*\,W(\phi)-1) =
W(\psi)(e^{i\omega(\psi,\phi)/2}\,W(\phi-\psi)-1)$ we conclude
\begin{eqnarray*}
\lefteqn{\|W(\phi)-W(\psi)\|^2 }\\
&=&
\|(W(\phi)-W(\psi))^*(W(\phi)-W(\psi))\|\\
&=&
\|(e^{i\omega(\psi,\phi)/2}\,W(\phi-\psi)-1)^*\,W(\psi)^*\,W(\psi)\,
(e^{i\omega(\psi,\phi)/2}\,W(\phi-\psi)-1)\|\\
&=&
\|(e^{i\omega(\psi,\phi)/2}\,W(\phi-\psi)-1)^*\,
(e^{i\omega(\psi,\phi)/2}\,W(\phi-\psi)-1)\|\\
&=&
\|e^{i\omega(\psi,\phi)/2}\,W(\phi-\psi)-1\|^2\\
&=&
4.
\end{eqnarray*}
This shows (\ref{Weyl1-2}).
Assertion (\ref{Weyl1-3}) now follows directly since the balls of radius $1$ centered
at $W(\phi)$, $\phi\in V$, form an uncountable collection of mutually disjoint
open subsets.

We show (\ref{Weyl1-4}).
Let $\phi_j \in V$, $j=1,\ldots,n$, be pairwise different and let 
$\sum_{j=1}^n \alpha_j W(\phi_j)=0$.
We show $\alpha_1 = \ldots \alpha_n = 0$ by induction on $n$.
The case $n=1$ is trivial by (\ref{Weyl1-1}).
Without loss of generality assume $\alpha_n\not= 0$.
Hence
$$
W(\phi_n) = \sum_{j=1}^{n-1} \frac{-\alpha_j}{\alpha_n} W(\phi_j)
$$
and therefore
\begin{eqnarray*}
1 &=& W(\phi_n)^* W(\phi_n) \\
&=&
\sum_{j=1}^{n-1} \frac{-\alpha_j}{\alpha_n} W(-\phi_n) W(\phi_j) \\
&=&
\sum_{j=1}^{n-1} \frac{-\alpha_j}{\alpha_n} e^{-i\omega(-\phi_n,\phi_j)/2} 
W(\phi_j-\phi_n) \\
&=&
\sum_{j=1}^{n-1} \beta_j W(\phi_j-\phi_n)
\end{eqnarray*}
where we have put $\beta_j := \frac{-\alpha_j}{\alpha_n} 
e^{i\omega(\phi_n,\phi_j)/2}$.
For an arbitrary $\psi\in V$ we obtain
\begin{eqnarray*}
1 &=& W(\psi)\cdot 1 \cdot W(-\psi) \\
&=&
\sum_{j=1}^{n-1} \beta_j W(\psi) W(\phi_j-\phi_n) W(-\psi)\\
&=&
\sum_{j=1}^{n-1} \beta_j e^{-i\omega(\psi,\phi_j-\phi_n)} W(\phi_j-\phi_n) .
\end{eqnarray*}
From 
$$
\sum_{j=1}^{n-1} \beta_j W(\phi_j-\phi_n) =
\sum_{j=1}^{n-1} \beta_j e^{-i\omega(\psi,\phi_j-\phi_n)} W(\phi_j-\phi_n) 
$$
we conclude by the induction hypothesis
$$
\beta_j = \beta_j e^{-i\omega(\psi,\phi_j-\phi_n)}
$$
for all $j= 1,\ldots,n-1$.
%If all $\beta_j$ were $0$, then $\alpha_j=0$ for $j=1,\ldots,n-1$ and hence
%$\alpha_n=0$, a contradiction.
If some $\beta_j\not=0$, then $e^{-i\omega(\psi,\phi_j-\phi_n)}=1$, hence
$$
\omega(\psi,\phi_j-\phi_n) = 0
$$
for all $\psi\in V$.
Since $\omega$ is nondegenerate $\phi_j-\phi_n=0$, a contradiction.
Therefore all $\beta_j$ and thus all $\alpha_j$ are zero, a contradiction.
\end{proof}

\Remark{\label{pi1}
Let $(A,W)$ be a Weyl system of the symplectic vector space $(V,\omega)$.
Then the linear span of the $W(\phi)$, $\phi\in V$, is closed under
multiplication and under $*$. 
This follows directly from the properties of a Weyl system.
We denote this linear span by $\la W(V) \ra \subset A$\indexs{WV@$\la W(V)
  \ra$, linear span of the $W(\phi)$, $\phi\in V$}. 
Now if $(A',W')$ is another Weyl system of the same symplectic vector space
$(V,\omega)$, then there is a unique linear map $\pi: \la W(V) \ra \to \la
W'(V)\ra$ determined by $\pi(W(\phi))= W'(\phi)$.
Since $\pi$ is given by a bijection on the bases $\{W(\phi)\}_{\phi\in V}$ and
$\{W'(\phi)\}_{\phi\in V}$ it is a linear isomorphism.
By the properties of a Weyl system $\pi$ is a $*$-isomorphism.
In other words, there is a unique $*$-isomorphism such that the following
diagram commutes

{\hspace{6cm}
\xymatrix{
& \la W'(V) \ra\\
V \ar[r]^{W_1} \ar[ur]^{W_2} & \la W(V) \ra \ar[u]_\pi
}}
}

\Remark{
On $\la W(V) \ra$ we can define the norm
$$
\Big\|\sum_{\phi} a_\phi W(\phi)\Big\|_1 := \sum_{\phi} |a_\phi| .
\indexs{Norm1@$\|\cdot\|_1$, norm on $\la W(V) \ra$}
$$
This norm is not a $C^*$-norm but for every $C^*$-norm $\|\cdot\|_0$ on $\la
W(V) \ra$ we have by the triangle inequality and by
Proposition~\ref{Weyl1}~(\ref{Weyl1-1}) 
\begin{equation}
\|a\|_0 \leq \|a\|_1
\label{L1groesserC}
\end{equation}
for all $a\in \la W(V) \ra$.
}

\begin{lemma}
Let $(A,W)$ be a Weyl system of a symplectic vector space $(V,\omega)$.
Then 
$$
\| a \|_{\mathrm{max}} :=
\sup \{\|a\|_0\, |\, \|\cdot\|_0 \mbox{ is a $C^*$-norm on }\la W(V) \ra\}
\indexs{Normmax@$\|\cdot\|_{\protect\mathrm{max}}$, $C^*$-norm on $\la W(V)
  \ra$} 
$$
defines a $C^*$-norm on $\la W(V) \ra$.
\end{lemma}

\begin{proof}
The given $C^*$-norm on $A$ restricts to one on $\la W(V) \ra$, so the
supremum is not taken on the empty set.
Estimate (\ref{L1groesserC}) shows that the supremum is finite.
The properties of a $C^*$-norm are easily checked.
E.~g.\ the triangle inequality follows from
\begin{eqnarray*}
\|a+b\|_{\mathrm{max}} &=&
\sup\{\|a+b\|_0\, |\, \|\cdot\|_0 \mbox{ is a $C^*$-norm on }\la W(V) \ra\}\\
&\leq&
\sup\{\|a\|_0+\|b\|_0\, |\, \|\cdot\|_0 \mbox{ is a $C^*$-norm on }\la W(V)
\ra\}\\ 
&\leq&
\sup\{\|a\|_0\, |\, \|\cdot\|_0 \mbox{ is a $C^*$-norm on }\la W(V)
\ra\}\\
&&
+ \sup\{\|b\|_0\, |\, \|\cdot\|_0 \mbox{ is a $C^*$-norm on }\la W(V) \ra\}\\
&=&
\|a\|_{\mathrm{max}}+\|b\|_{\mathrm{max}} .
\end{eqnarray*}
The other properties are shown similarly.
\end{proof}

\begin{lemma}\label{einfach}
Let $(A,W)$ be a Weyl system of a symplectic vector space $(V,\omega)$.
Then the completion $\overline{\la W(V) \ra}^{\mathrm{max}}$ of $\la W(V) \ra$
with respect to $\|\cdot\|_{\mathrm{max}}$ is \defem{simple}\indexn{simple
  $C^*$--algebra>defemidx}, i.~e., it has no 
nontrivial closed twosided $*$-ideals. 
\end{lemma}

\begin{proof}
By Remark~\ref{pi1} we may assume that $(A,W)$ is the Weyl system constructed
in Example~\ref{CCR1}.
In particular, $\la W(V) \ra$ carries the $C^*$-norm
$\|\cdot\|_{\mathrm{Op}}$, the operator norm given by $\la W(V) \ra \subset
\LL(H)$ where $H=L^2(V,\Co)$. 

Let $I\subset \overline{\la W(V) \ra}^{\mathrm{max}}$ be a closed twosided
$*$-ideal. 
Then $I_0 := I \cap \Co\cdot W(0)$ is a (complex) vector subspace in $\Co\cdot
W(0) = \Co\cdot 1 
\cong \Co$ and thus $I_0=\{0\}$ or $I_0=\Co\cdot W(0)$.
If $I_0=\Co\cdot W(0)$, then $I$ contains $1$ and therefore $I=\overline{\la
  W(V) \ra}^{\mathrm{max}}$. 
Hence we may assume $I_0=\{0\}$.

Now we look at the projection map $P:\la W(V) \ra \to \Co\cdot W(0)$,
$P(\sum_{\phi} a_\phi W(\phi)) = a_0 W(0)$.
We check that $P$ extends to a bounded operator on $\overline{\la W(V)
\ra}^{\mathrm{max}}$.
Let $\delta_0\in L^2(V,\Co)$ denote the function given by $\delta_0(0)=1$
and $\delta_0(\phi)=0$ otherwise.
For $a=\sum_{\phi} a_\phi W(\phi)$ and $\psi\in V$ we have 
$$
(a\cdot \delta_0)(\psi) = (\sum_{\phi} a_\phi W(\phi) \delta_0)(\psi)
= (\sum_{\phi} a_\phi\, e^{i\omega(\phi,\psi)/2}\delta_0(\phi+\psi)
= a_{-\psi}\, e^{i\omega(-\psi,\psi)/2} = a_{-\psi}
$$
and therefore
$$
(\delta_0,a\cdot\delta_0)_{L^2} 
= \sum_{\psi\in V}\ovl{\delta_0(\psi)}(a\cdot \delta_0)(\psi)
= (a\cdot \delta_0)(0) = a_0.
$$
Moreover, $\|\delta_0\|=1$.
Thus
$$
\|P(a)\|_{\mathrm{max}} \,=\, \|a_0 W(0)\|_{\mathrm{max}} 
\,=\, |a_0| \,=\, |(\delta_0,a\cdot\delta_0)_{L^2}| 
\,\leq\, \|a\|_{\mathrm{Op}} \,\leq\, \|a\|_{\mathrm{max}}
$$
which shows that $P$ extends to a bounded operator on $\overline{\la W(V)
\ra}^{\mathrm{max}}$.

Now let $a\in I \subset \overline{\la W(V) \ra}^{\mathrm{max}}$.
Fix $\epsilon >0$.
We write 
$$
a = a_0W(0) + \sum_{j=1}^n a_j\, W(\phi_j) + r
$$
where the $\phi_j \not= 0$ are pairwise different and the remainder term $r$
satisfies $\|r\|_{\mathrm{max}} < \epsilon$.
For any $\psi\in V$ we have
$$
I \ni W(\psi)\, a\, W(-\psi) =
a_0W(0) + \sum_{j=1}^n a_j\, e^{-i\omega(\psi,\phi_j)} \,W(\phi_j) + r(\psi)
$$
where $\|r(\psi)\|_{\mathrm{max}} =\|W(\psi)rW(-\psi)\|_{\mathrm{max}} \leq
\|r\|_{\mathrm{max}} < \epsilon$.
If we choose $\psi_1$ and $\psi_2$ such that $e^{-i\omega(\psi_1,\phi_n)} = 
- e^{-i\omega(\psi_2,\phi_n)}$, then adding the two elements
\begin{eqnarray*}
a_0W(0) + \sum_{j=1}^n a_j\, e^{-i\omega(\psi_1,\phi_j)} \,W(\phi_j) +
r(\psi_1) 
&\in& I\\
a_0W(0) + \sum_{j=1}^n a_j\, e^{-i\omega(\psi_2,\phi_j)} \,W(\phi_j) +
r(\psi_2)  
&\in& I
\end{eqnarray*}
yields
$$
a_0W(0) + \sum_{j=1}^{n-1} a_j'\, W(\phi_j) + r_1\in I
$$
where $\|r_1\|_{\mathrm{max}}=\|\frac{r(\psi_1)+r(\psi_2)}{2}\|_{\mathrm{max}}
<\frac{\epsilon+\epsilon}{2} = \epsilon$. 
Repeating this procedure we eventually get
$$
a_0\, W(0) + r_n \in I
$$
where $\|r_n\|_{\mathrm{max}} < \epsilon$.
Since $\epsilon$ is arbitrary and $I$ is closed we conclude
$$
P(a) = a_0\, W(0) \in I_0,
$$
thus $a_0=0$.

For $a = \sum_{\phi}a_\phi\, W(\phi)\in I$ and arbitrary $\psi\in V$ we have
$W(\psi)a\in I$ as well, hence $P(W(\psi)a)=0$.
This means $a_{-\psi}=0$ for all $\psi$, thus $a=0$.
This shows $I=\{0\}$.
\end{proof}

\Definition{
A Weyl system $(A,W)$ of a symplectic vector space $(V,\omega)$ is called
a \defem{\defidx{CCR-representation}} of $(V,\omega)$ if $A$ is generated
as a $C^*$-algebra by the elements $W(\phi)$, $\phi\in V$.
In this case we call $A$ a \defem{\defidx{CCR-algebra}} of $(V,\omega)$
}

Of course, for any Weyl system $(A,W)$ we can simply replace $A$ by the
$C^*$-subalgebra generated by the elements $W(\phi)$, $\phi\in V$, and
we obtain a CCR-representation.

Existence of Weyl systems and hence CCR-representations has been established
in Example~\ref{CCR1}.
Uniqueness also holds in the appropriate sense.

\begin{thm}\label{CCRunique}
Let $(V,\omega)$ be a symplectic vector space and let $(A_1,W_1)$ and
$(A_2,W_2)$ be two CCR-representations of $(V,\omega)$.

Then there exists a unique $*$-isomorphism $\pi:A_1 \to A_2$ such that the
diagram

{\hspace{6cm}
\xymatrix{
& A_2\\
V \ar[r]^{W_1} \ar[ur]^{W_2} & A_1 \ar@{.>}[u]_\pi
}}

commutes.
\end{thm}

\begin{proof}
We have to show that the $*$-isomorphism $\pi:\la W_1(V) \ra \to \la W_2(V)
\ra$ as constructed in Remark~\ref{pi1} extends to an isometry
$(A_1,\|\cdot\|_1) \to (A_2,\|\cdot\|_2)$. 
Since the pull-back of the norm $\|\cdot\|_2$ on $A_2$ to $\la W_1(V)\ra$ via
$\pi$ is a $C^*$-norm we have $\|\pi(a)\|_2 \leq \|a\|_{\mathrm{max}}$ for all
$a\in \la W_1(V)\ra$.
Hence $\pi$ extends to a $*$-morphism $\overline{\la W_1(V)\ra}^{\mathrm{max}}
\to A_2$.
By Lemma~\ref{einfach} the kernel of $\pi$ is trivial, hence $\pi$ is
injective.
Proposition~\ref{Cmono} implies that $\pi : (\overline{\la W_1(V)\ra}^{\mathrm
{max}},\|\cdot\|_{\mathrm{max}}) \to (A_2,\|\cdot\|_2)$ is an isometry.

In the special case $(A_1,\|\cdot\|_1) = (A_2,\|\cdot\|_2)$ where $\pi$ is the
identity this yields $\|\cdot\|_{\mathrm{max}}=\|\cdot\|_1$.
Thus for arbitrary $A_2$ the map $\pi$ extends to an isometry
$(A_1,\|\cdot\|_1) \to (A_2,\|\cdot\|_2)$.
\end{proof}

From now on we will call $\CCR(V,\omega)$  as defined in Example~\ref{CCR1}
{\em the} CCR-algebra of $(V,\omega)$.
%\indexs{CCRV@$\CCR(V,\omega)$, CCR-algebra of a symplectic vector
%space>defemidx} 
 
\begin{cor}\label{cor:simple}
CCR-algebras of symplectic vector spaces are simple\indexn{simple
  $C^*$--algebra}, i.~e., all unit preserving 
$*$-morphisms to other $C^*$-algebras are injective.
\end{cor}

\begin{proof}
Direct consequence of Corollary~\ref{Chomo} and Lemma~\ref{einfach}.
\end{proof}

\begin{cor}\label{cfuncSYMPLCCR}
Let $(V_1,\omega_1)$ and $(V_2,\omega_2)$ be two symplectic vector spaces
and let $S:V_1\to V_2$ be a \defem{symplectic linear map}\indexn{symplectic
  linear map>defemidx}, i.~e., $\omega_2(S\phi,S\psi) = 
\omega_1(\phi,\psi)$ for all $\phi,\psi\in V_1$. 

Then there exists a unique injective $*$-morphism $\CCR(S): \CCR(V_1,\omega_1)
\to \CCR(V_2,\omega_2)$ such that the diagram 
\indexs{CCRS@$\CCR(S)$, $*$-morphism induced by a symplectic linear map}

{\hspace{3.5cm}
\xymatrixcolsep{5pc}
\xymatrix{
V_1 \ar[r]^S \ar[d]_{W_1} & V_2 \ar[d]^{W_2}\\
\CCR(V_1,\omega_1) \ar[r]^{\CCR(S)} & \CCR(V_2,\omega_2) 
}}

commutes.
\end{cor}

\begin{proof}
One immediately sees that $(\CCR(V_2,\omega_2),W_2 \circ S)$ is a Weyl system
of $(V_1,\omega_1)$.
Theorem~\ref{CCRunique} yields the result.
\end{proof}

 From uniqueness of the map $\CCR(S)$ we conclude that
$\CCR(\id_V)=\id_{\CCR(V,\omega)}$ and $\CCR(S_2\circ S_1) = \CCR(S_2)\circ
\CCR(S_1)$.
In other words, we have constructed a functor
\indexs{CCRfunctor@$\protect\CCR$, functor $\protect\SYMPL \to \protect\CALG$}
$$
\CCR : \SYMPL \to \CALG
$$
where $\SYMPL$ denotes the category whose objects are symplectic vector spaces
and whose morphisms are symplectic linear maps, i.~e., linear maps
$A:(V_1,\omega_1)\to (V_2,\omega_2)$ with $A^*\omega_2=\omega_1$.
\indexs{Sympl@${\protect\SYMPL}$, category of symplectic vector spaces and
  symplectic linear maps} 
By $\CALG$ we denote the category whose objects are $C^*$-algebras and whose
morphisms are {\em injective} unit preserving $*$-morphisms.
\indexs{Calg@$\protect{\CALG}$, category of $C^*$-algebras and injective
  $*$-morphisms} 
Observe that symplectic linear maps are automatically injective.

In the case $V_1 = V_2$ the induced $*$-automorphisms $\CCR(S)$ are called
\defem{\defidx{Bogoliubov transformation}} in the physics literature.

%%%%%%%%%%%%%%%%%%%%%%%%%%%%%%%%%%%%%%%%%%%%%%%%%%%%%%%%%%%%%%%%%%%%%%%%%
\section{Quantization functors}\label{sec:quantfunc}\indexn{functor}
%%%%%%%%%%%%%%%%%%%%%%%%%%%%%%%%%%%%%%%%%%%%%%%%%%%%%%%%%%%%%%%%%%%%%%%%%

In the preceding section we introduced the functor $\CCR$ from the
category $\SYMPL$ of symplectic vector spaces (with symplectic linear 
maps as morphisms)  to the category $\CALG$ of $C^*$-algebras (with unit
preserving $*$-monomorphisms as morphisms). 
We want to link these considerations to Lorentzian manifolds and the analysis
of normally hyperbolic operators.
In order to achieve this we introduce two further categories which are of
geometric-analytical nature.

So far we have treated real and complex vector bundles $E$ over the manifold
$M$ on an equal footing.
From now on we will restrict ourselves to real bundles.
This is not very restrictive since we can always forget a complex structure
and regard complex bundles as real bundles.
We will have to give the real bundle $E$ another piece of additional
structure.
We will assume that $E$ comes with a \defem{nondegenerate inner product}
$\la\cdot,\cdot\ra$\indexs{innerproduct@$\la\cdot,\cdot\ra$, (nondegenerate)
  inner product on a vector bundle}. 
This means that each fiber $E_x$ is equipped with a nondegenerate symmetric
bilinear form depending smoothly on the base point $x$.
In other words, $\la\cdot,\cdot\ra$ is like a Riemannian metric except that it
need not be positive definite.

We say that a differential operator $P$ acting on sections in $E$ is
\defem{formally selfadjoint}\indexn{formally selfadjoint differential
  operator>defemidx} with respect to to the inner product
$\la\cdot,\cdot\ra$ of $E$, if 
$$
\int_M \la P\phi,\psi\ra \dV = \int_M \la \phi,P\psi\ra \dV
$$ 
for all $\phi,\psi \in \DD(M,E)$.

\Example{\label{ex:Boxfsa}
Let $M$ be an $n$-dimensional timeoriented connected Lorentzian manifold with
metric $g$.
Let $E$ be a real vector bundle over $M$ with inner product
$\la\cdot,\cdot\ra$.
Let $\nabla$ be a connection on $E$.
We assume that $\nabla$\indexs{nabla@$\nabla$, connection on a vector bundle}
is \defem{metric}\indexn{metric connection on a vector bundle} with respect to
$\la\cdot,\cdot\ra$, i.~e.,  
$$
\partial_X \la \phi,\psi\ra 
= \la\nabla_X\phi,\psi\ra + \la\phi,\nabla_X\psi\ra
$$
for all sections $\phi,\psi\in C^\infty(M,E)$.
The inner product induces an isomorphism\footnote{In case $E=TM$ with the
  Lorentzian metric as the inner product we write $\flat$ instead of $\FB$ and
  $\sharp$ instead of $\FB^{-1}$.
Hence if $E=T^*M$ we have $\FB=\sharp$ and $\FB^{-1}=\flat$.} 
$\FB : E\to E^*$\indexs{finger@$\protect\FB$, isomorphism $E\rightarrow E^*$
  induced by an inner product}, $\phi\mapsto 
\la\phi,\cdot\ra$.
Since $\nabla$ is metric we get
\begin{eqnarray*}
\left(\nabla_X(\FB\phi)\right)\cdot\psi 
&=&
\partial_X((\FB\phi)\psi) - (\FB\phi)(\nabla_X\psi) =
\partial_X\la\phi,\psi\ra - \la\phi,\nabla_X\psi\ra \\
&=&
\la\nabla_X\phi,\psi\ra =
(\FB(\nabla_X\phi))\cdot\psi,
\end{eqnarray*}
hence
$$
\nabla_X(\FB\phi)=\FB(\nabla_X\phi).
$$
Equation (\ref{eq:nablastern}) says for all $\phi,\psi\in C^\infty(M,E)$
$$
(\FB\phi)\cdot(\Box^\nabla\psi)
= \sum_{i=1}^n\epsilon_i\nabla_{e_i}(\FB\phi)\cdot\nabla_{e_i}\psi - \div(V_1),
$$
thus 
$$
\la\phi,\Box^\nabla\psi\ra = \sum_{i=1}^n\epsilon_i\la\nabla_{e_i}\phi,
\nabla_{e_i}\psi\ra - \div(V_1),
$$
where $V_1$ is a smooth vector field with $\supp(V_1) \subset
\supp(\phi)\cap\supp(\psi)$ and $e_1,\ldots,e_n$ is a local Lorentz
orthonormal tangent frame, $\epsilon_i=g(e_i,e_i)$.
Interchanging the roles of $\phi$ and $\psi$ we get
$$
\la\Box^\nabla\phi,\psi\ra = \sum_{i=1}^n\epsilon_i\la\nabla_{e_i}\phi,
\nabla_{e_i}\psi\ra - \div(V_2),
$$
and therefore
$$
\la\phi,\Box^\nabla\psi\ra - \la\Box^\nabla\phi,\psi\ra = \div(V_2-V_1).
$$
If $\supp(\phi)\cap\supp(\psi)$ is compact we obtain
$$
\int_M\la\phi,\Box^\nabla\psi\ra\dV - \int_M\la\Box^\nabla\phi,\psi\ra\dV 
= \int_M\div(V_2-V_1)\dV = 0,
$$
thus $\Box^\nabla$ is formally selfadjoint.
If, moreover, $B\in C^\infty(M,\End(E))$ is selfadjoint with res\-pect to
$\la\cdot,\cdot\ra$, then the normally hyperbolic operator $P=\Box^\nabla +B$
is formally selfadjoint.

As a special case let $E$ be the trivial real line bundle.
In other words, sections in $E$ are simply real-valued functions.
The inner product is given by the pointwise product.
In this case the inner product is positive definite.
Then the above discussion shows that the d'Alembert operator $\Box$ is
formally selfadjoint and, more generally, $P=\Box + B$ is formally selfadjoint
where the zero-order term $B$ is a smooth real-valued function on $M$.
This includes the (normalized) Yamabe operator\indexn{Yamabe operator} $P_g$
discussed in Section~\ref{seq:nonglobhyp}, 
the \defidx{Klein-Gordon operator} $P=\Box+m^2$ and the
\defidx{covariant Klein-Gordon operator} $P=\Box+m^2+\kappa\,\mathrm{scal}$, 
where $m$ and $\kappa$ are real constants.
}

\Example{\label{ex:1forms}
Let $M$ be an $n$-dimensional timeoriented connected Lorentzian manifold.
Let $\Lambda^kT^*M$ be the bundle of $k$-forms on $M$.
The Lorentzian metric induces a nondegenerate inner product
$\la\cdot,\cdot\ra$ on $\Lambda^kT^*M$, which is indefinite if $1\leq k\leq
n-1$. 
Let $d:C^\infty(M,\Lambda^kT^*M) \to
C^\infty(M,\Lambda^{k+1}T^*M)$\indexn{exterior differential} denote exterior
differentiation. 
Let $\delta:C^\infty(M,\Lambda^kT^*M) \to C^\infty(M,\Lambda^{k-1}T^*M)$ be
the \defem{codifferential}\indexs{delta@$\delta$, codifferential} .
This is the unique differential operator formally adjoint to $d$, i.~e.,
$$
\int_M\la d\phi,\psi\ra \dV = \int_M \la\phi,\delta\psi\ra\dV
$$
for all $\phi\in \DD(M,\Lambda^kT^*M)$ and $\psi\in
\DD(M,\Lambda^{k+1}T^*M)$.
Then the operator 
$$
P=d\delta+\delta d:C^\infty(M,\Lambda^kT^*M) \to
C^\infty(M,\Lambda^kT^*M)
$$%%\indexn{Hodge\,-\,de Rham Laplacian}
is obviously formally selfadjoint. 
The Levi-Civita connection\indexn{Levi-Civita connection} induces a metric
connection $\nabla$ on the bundle 
$\Lambda^kT^*M$.
The \defidx{Weitzenb{\"o}ck formula} relates $P$ and $\Box^\nabla$,
$P=\Box^\nabla + B$ where $B$ is a certain expression in the curvature tensor
of $M$, see \cite[Eq.~(12.92')]{Bes}.
In particular, $P$ and $\Box^\nabla$ have the same principal symbol, hence $P$ 
is normally hyperbolic.

The operator $P$ appears in physics in different contexts.
Let $M$ be of dimension $n=4$.
Let us first look at the \defidx{Proca equation} describing a spin-1 particle
of mass $m>0$\indexs{m@$m$, mass of a spin-1 particle}.
The quantum mechanical wave function\indexn{quantum mechanical wave function>defemidx}
of such a particle is given by $A\in 
C^\infty(M,\Lambda^1T^*M)$ and satisfies
\begin{equation}
\delta dA + m^2A=0.
\label{eq:proca}
\end{equation}
Applying $\delta$ to this equation and using $\delta^2=0$ and $m\not=0$ we
conclude $\delta A=0$.
Thus the Proca equation (\ref{eq:proca}) is equivalent to 
$$
(P+m^2)A=0
$$
together with the constraint $\delta A=0$.

Now we discuss \defidx{electrodynamics}.
Let $M$ be a $4$-dimensional globally hyperbolic Lorentzian manifold and
assume that the second deRham cohomology vanishes, $H^2(M;\R)=\{0\}$.
By Poincar\'e duality, the second cohomology with compact supports also
vanishes, $H^2_c(M;\R)=\{0\}$.
See \cite{War} for details on deRham cohomology.

The electric and the magnetic fields can be combined to the \defem{field
strength}\indexn{field strength in electrodynamics>defemidx} $F\in
\DD(M,\Lambda^2T^*M)$.
The \defidx{Maxwell equations} are 
$$
dF=0 \quad\mbox{ and }\quad \delta F=J
$$
where $J\in\DD(M,\Lambda^1T^*M)$ is the \defem{current density}\indexn{current
density in electrodynamics>defemidx}.
From $H^2_c(M;\R)=0$ we have that $dF=0$ implies the existence of a
\defem{vector potential}\indexn{vector potential in electrodynamics>defemidx}
$A\in \DD(M,\Lambda^1T^*M)$ with $dA=F$.
Now $\delta A\in \DD(M,\R)$ and by Theorem~\ref{cauchyglobhyp} we can find
$f \in \Csc(M,\R)$ with $\Box f = \delta A$.
We put $A' := A - df$.
We see that $dA' = dA = F$ and $\delta A'=\delta A - \delta df = \delta A-\Box
f = 0$.
A vector potential satisfying the last equation $\delta A'=0$ is said to be in 
\defidx{Lorentz gauge}.
From the Maxwell equations we conclude $\delta dA' = \delta F = J$.
Hence
$$
PA'=J.
$$
}

\Example{\label{ex:Diracsquare}
Next we look at spinors and the Dirac operator.
These concepts are studied in much detail on general semi-Riemannian manifolds
in \cite{Baum}, see also \cite[Sec.~2]{BGM} for an overview. 

Let $M$ be an $n$-dimensional oriented and timeoriented connected Lorentzian
manifold.
Furthermore, we assume that $M$ carries a spin structure.
Then we can form the \defidx{spinor bundle} $\Sigma M$ over $M$.
This is a complex vector bundle of rank $2^{n/2}$ or $2^{(n-1)/2}$ depending
on whether $n$ is even or odd.
This bundle carries an indefinite Hermitian product $h$.

The \defidx{Dirac operator} $D: C^\infty(M,\Sigma M) \to C^\infty(M,\Sigma M)$
is a formally selfadjoint differential operator of first order.
The Levi-Civita connection induces a metric connection $\nabla$ on $\Sigma M$.
The Weitzenb\"ock formula\indexn{Weitzenb{\"o}ck formula} says
$$
D^2 = \Box^\nabla + \frac14 \mathrm{scal}.
$$
Thus $D^2$ is normally hyperbolic.
Since $D$ is formally selfadjoint so is $D^2$.

If we forget the complex structure on $\Sigma M$, i.~e., we regard $\Sigma M$
as a real bundle, and if we let
$\la\cdot,\cdot\ra$ be given by the real part of $h$, then the operator
$P=D^2$ is of the type under consideration in this section.
}

Now we define the category of globally hyperbolic manifolds equipped with
normally hyperbolic operators:

\Definition{ 
The category $\GLOBHY$ is defined as follows:
\indexs{GlobHyp@$\protect{\GLOBHY}$, category of globally hyperbolic manifolds
equipped with formally selfadjoint normally hyperbolic operators}
The objects of $\GLOBHY$ are triples $(M,E,P)$ 
where $M$ is a globally hyperbolic Lorentzian manifold,
$E\to M$ is a real vector bundle with nondegenerate inner product, and $P$ is
a formally selfadjoint normally hyperbolic operator acting on sections in $E$.

A morphism $(M_1,E_1,P_1)\to (M_2,E_2,P_2)$ in $\GLOBHY$ is a pair $(f,F)$
where $f:M_1\to M_2$ is a timeorientation preserving isometric embedding
so that $f(M_1) \subset M_2$ is a causally compatible open subset.
Moreover, $F:E_1\to E_2$ is a vector bundle homomorphism over $f$ which is
fiberwise an isometry.
In particular,
$$
\xymatrix{
E_1 \ar[r]^F \ar[d]& E_2 \ar[d]\\
M_1 \ar[r]^f& M_2
}
$$
commutes.
Furthermore, $F$ has to preserve the normally hyperbolic operators, i.~e.,
$$
\xymatrix{
\DD(M_1,E_1) \ar[r]^{P_1} \ar[d]^{\ext}& \DD(M_1,E_1)\ar[d]^{\ext} \\
\DD(M_2,E_2) \ar[r]^{P_2} & \DD(M_2,E_2)
}
$$
commutes where $\ext(\phi)$ \indexs{ext@$\ext$, extension of section}
denotes the extension of $F\circ\phi\circ f^{-1} \in
\DD(f(M_1),E_2)$ to all of $M_2$ by $0$.
}

Notice that a morphism between two objects
$(M_1,E_1,P_1)$ and
$(M_2,E_2,P_2)$ can exist only if $M_1$ and
$M_2$ have equal dimension and if $E_1$ and $E_2$ have the same rank.
The condition that $f(M_1) \subset M_2$ is causally compatible does not follow
from the fact that $M_1$ and $M_2$ are globally hyperbolic.
For example, consider $M_2 = \R \times S^1$ with metric
$-dt^2+\mathrm{can}_{S^1}$ and let $M_1 \subset M_2$ be a small strip about a
spacelike helix.
Then $M_1$ and $M_2$ are both intrinsically globally hyperbolic but $M_1$ is
not a causally compatible subset of $M_2$.

\begin{center}\indexn{causally compatible subset}
\input{fig-globhypnotcomp}
\end{center}

\begin{lemma}\label{schiefesymmetrie}\indexn{Green's operator}
Let $M$ be a globally hyperbolic Lorentzian manifold.
Let $E\to M$ be a real vector bundle with nondegenerate inner product
$\langle\cdot\,,\cdot\rangle$. 
Consider a formally selfadjoint normally hyperbolic operator $P$ with advanced
and retarded Green's operators $G_\pm$ as in Corollary~\ref{functorsolve}.
Then
\begin{equation}
\int_M\la G_\pm\varphi,\psi\ra\dV = \int_M\la\varphi,G_\mp\psi\ra\dV
\label{Gsametrisch}
\end{equation}
holds for all $\varphi,\psi\in\DD(M,E)$.
\end{lemma}

\begin{proof}
The proof is basically the same as that of Lemma~\ref{Greenadjungiert}.
Namely, for Green's operators we have $PG_\pm=\id_{\DD(M,E)}$ and therefore we
get 
\[ 
\int_M\la G_\pm\varphi,\psi\ra\dV = \int_M\la G_\pm\varphi,PG_\mp\psi\ra\dV  
= \int_M\la PG_\pm\varphi,G_\mp\psi\ra\dV = \int_M\la\varphi,G_\mp\psi\ra\dV
\]
where we have made use of the formal selfadjointness of $P$ in the second
equality.
Notice that $\supp(G_\pm\phi) \cap \supp(G_\mp\psi)\subset
J_\pm^M(\supp(\phi))\cap J_\mp^M(\supp(\psi))$ is compact in a globally
hyperbolic manifold so that the partial integration is justified.

Alternatively, we can also argue as follows.
The inner product yields the vector bundle isomorphism $\FB:E\to E^*$, $e
\mapsto \la e,\cdot\ra$, as noted in Example~\ref{ex:Boxfsa}. 
Formal selfadjointness now means that the operator $P$ corresponds to the dual
operator $P^*$ under $\FB$.
Now Lemma~\ref{schiefesymmetrie} is a direct consequence of
Lemma~\ref{Greenadjungiert}.
\end{proof}

If we want to deal with Lorentzian manifolds which are not globally hyperbolic
we have the problem that Green's operators need not exist and if they do they
are in general no longer unique.
In this case we have to provide the Green's operators as additional data.
This motivates the definition of a category of Lorentzian manifolds with
normally hyperbolic operators and global fundamental solutions.

\Definition{
Let $\LORFUND$ denote the category whose objects are 5-tuples
$(M,E,P,G_+,G_-)$ where $M$ is a timeoriented connected Lorentzian
manifold, $E$ is a real vector bundle over
$M$ with nondegenerate inner product, $P$ is a formally
selfadjoint\indexn{formally selfadjoint differential operator} normally 
hyperbolic operator acting on sections in $E$, and $G_{\pm}$ are advanced
and retarded Green's operators for $P$ respectively. 
Moreover, we {\em assume} that (\ref{Gsametrisch}) holds for all
$\varphi,\psi\in\DD(M,E)$.
\indexs{LorFund@$\protect{\LORFUND}$, category of timeoriented Lorentzian
manifolds equipped with formally selfadjoint normally hyperbolic operators and
Green's operators}

Let $X=(M_1,E_1,P_1,G_{1,+},G_{1,-})$ and
$Y=(M_2,E_2,P_2,G_{2,+},G_{2,-})$ be two objects in $\LORFUND$.
If $M_1$ is not globally hyperbolic, then we let the set of morphisms from 
$X$ to $Y$ be empty unless $X=Y$ in which case we put 
$\Mor(X,Y) := \{(\id_{M_1},\id_{E_1})\}$.
}

If $M_1$ is globally hyperbolic, then $\Mor(X,Y)$ consists of all pairs
$(f,F)$ with the same properties as those of the morphisms in $\GLOBHY$.
It then follows from Proposition~\ref{propGreensubset} and
Corollary~\ref{functorsolve} that we automatically have compatibility of the
Green's operators, i.~e., the diagram
$$
\xymatrix{
\DD(M_1,E_1) \ar[r]^{\ext} \ar[d]^{G_{1,\pm}}& \DD(M_2,E_2)
\ar[d]^{G_{2,\pm}}\\ 
C^\infty(M_1,E_1) & C^\infty(M_2,E_2)\ar[l]_{\res} 
}
$$
commutes. 
Here $\res$\indexs{res@$\protect\res$, restriction of a section} stands for
``restriction''. 
More precisely, $\res(\phi)=F^{-1}\circ \phi\circ f$.
Composition of morphisms is given by the usual composition of maps.

The definition of the category $\LORFUND$ is such that nontrivial morphisms
exist only if the source manifold $M_1$ is globally hyperbolic while there is
no such restriction on the target manifold $M_2$.
It will become clear in the proof of Lemma~\ref{functorSYMPL} why we restrict
to globally hyperbolic $M_1$.
 
By Corollary~\ref{functorsolve} there exist unique advanced and retarded
Green's operators $G_{\pm}$ for a normally hyperbolic operator on a globally
hyperbolic manifold.
Hence we can define
$$
\solve(M,E,P) := (M,E,P,G_+,G_-)
\indexs{SOLVE@$\protect\solve$, functor $\protect\GLOBHY\to\protect\LORFUND$}
$$
on objects of $\GLOBHY$ and
$$
\solve(f,F) := (f,F)
$$
on morphisms.

\begin{lemma}
This defines a functor $\solve:\GLOBHY\to\LORFUND$.
\end{lemma}

\begin{proof}
We only need to check that $\solve(f,F) = (f,F)$ is actually a morphism in
$\LORFUND$, i.~e., that $(f,F)$ is compatible with the Green's operators.
By uniqueness of Green's operators on globally hyperbolic manifolds it
suffices to show that $\res\circ G_{2,+}\circ\ext$ is an advanced Green's
operator on $M_1$ and similarly for $G_{2,-}$.
Since $f(M_1)\subset M_2$ is a causally compatible connected open subset this
follows from Proposition~\ref{propGreensubset}.
\end{proof}

Next we would like to use the Green's operators in order to construct a
symplectic vector space to which we can then apply the functor $\CCR$.
Let $(M,E,P,G_+,G_-)$ be an object of $\LORFUND$.
Using $G=G_+-G_- : \DD(M,E) \to C^\infty(M,E)$
we define
$$
\widetilde{\omega} : \DD(M,E) \times \DD(M,E) \to \R
\indexs{omegatilde@$\protect\widetilde{\omega}$, skew-symmetric bilinear form
  inducing the symplectic form $\omega$} 
$$
by
\begin{equation}\label{omegadef}
\widetilde{\omega}(\varphi,\psi)
:=
\int_M\la G\varphi,\psi\ra\dV.
\end{equation}
Obviously, $\widetilde\omega$ is $\R$-bilinear and by (\ref{Gsametrisch}) it is
skew-symmetric.
But $\widetilde\omega$ does not make $\DD(M,E)$ a symplectic vector space
because $\widetilde\omega$ is degenerate.
The null space is given by
$$
\ker(G) = \{\phi\in\DD(M,E)\, |\, G\phi=0\}
= \{\phi\in\DD(M,E)\, |\, G_+\phi=G_-\phi\}.
$$
This null space is infinite dimensional because it certainly contains
$P(\DD(M,E))$ by Theorem~\ref{thmExSeq}.
In the globally hyperbolic case this is precisely the null space,
$$
\ker(G) = P(\DD(M,E)),
$$
again by Theorem~\ref{thmExSeq}.
On the quotient space $V(M,E,G) :=
\DD(M,E)/\ker(G)$\indexs{VMEG@$V(M,E,G)=\protect\DD(M,E)/\protect\ker(G)$} the 
degenerate bilinear form $\widetilde\omega$ induces a symplectic form which we
denote by $\omega$.\indexs{omega@$\omega$, symplectic form}

\begin{lemma}\label{functorSYMPL}
Let $X=(M_1,E_1,P_1,G_{1,+},G_{1,-})$ and
$Y=(M_2,E_2,P_2,G_{2,+},G_{2,-})$ be two objects in
$\LORFUND$.
Let $(f,F)\in\Mor(X,Y)$ be a morphism.

Then $\ext :\DD(M_1,E_1) \to \DD(M_2,E_2)$ maps the null space $\ker(G_1)$ to
the null space $\ker(G_2)$ and hence induces a symplectic linear map
$$
V(M_1,E_1,G_1) \to V(M_2,E_2,G_2).
$$
\end{lemma}

\begin{proof}
If the morphism is the identity, then there is nothing to show.
Thus we may assume that $M_1$ is globally hyperbolic.
Let $\phi\in\ker(G_1)$.
Then $\phi = P_1\psi$ for some $\psi\in\DD(M_1,E_1)$ because $M_1$ is globally
hyperbolic.
From $G_2(\ext\phi) = G_2(\ext (P_1\psi)) = G_2(P_2(\ext\psi)) = 0$ we see
that $\ext(\ker(G_1)) \subset \ker(G_2)$.
Hence $\ext$ induces a linear map $V(M_1,E_1,G_1) \to V(M_2,E_2,G_2)$.
From
\begin{eqnarray*}
\widetilde\omega_2(\ext\phi,\ext\psi) 
&=&  \int_{M_2} \la G_2 \ext\phi,\ext\psi\ra \dV \\
&=&  \int_{M_1} \la \res G_2 \ext\phi, \psi\ra \dV \\
&=&  \int_{M_1} \la G_1\phi, \psi\ra \dV \\
&=& \widetilde\omega_1(\phi,\psi) 
\end{eqnarray*}
we see that this linear map is symplectic.
\end{proof}

We have constructed a functor from the category $\LORFUND$ to the category
$\SYMPL$ by mapping each object $(M,E,P,G_+,G_-)$ to
$V(M,E,G_+-G_-)$ and each morphism $(f,F)$ to the symplectic linear map
induced by $\ext$.
We denote this functor by SYMPL\indexs{SYMPL, functor
  $\protect\LORFUND\to\protect\SYMPL$}. 

We summarize the categories and functors we have defined so far in the
following scheme:

\begin{center}
\input{fig-quantenschema}
\end{center}

%%%%%%%%%%%%%%%%%%%%%%%%%%%%%%%%%%%%%%%%%%%%%%%%%%%%%%%%%%%%%%%%%%%%%%%%% 
\section{Quasi-local $C^*$-algebras} \label{sec:qlC*}
%%%%%%%%%%%%%%%%%%%%%%%%%%%%%%%%%%%%%%%%%%%%%%%%%%%%%%%%%%%%%%%%%%%%%%%%%

The composition of the functors $\CCR$ and $\SYM$ constructed in the
previous sections allows us to assign a $C^*$-algebra to each timeoriented
connected Lorentzian manifold equipped with a formally selfadjoint normally
hyperbolic operator and Green's operators.
Further composing with the functor $\solve$ we no longer need to provide
Green's operators if we are willing to restrict ourselves to globally
hyperbolic manifolds.
The elements of this algebra are physically interpreted as the observables
related to the field whose wave equation is given by the normally hyperbolic
operator.

``Reasonable'' open subsets of $M$ are timeoriented Lorentzian manifolds in
their own right and come equipped with the restriction of the normally
hyperbolic operator over $M$.
Hence each such open subset $O$ yields an algebra whose elements are
considered as the observables which can be measured in the spacetime region
$O$.
This gives rise to the concept of nets of algebras or quasi-local algebras.
A systematic exposition of quasi-local algebras can be found in \cite{BW}.

Before we define quasi-local algebras we characterize the systems that
parametrize the ``local algebras''.
For this we need the notion of directed sets with orthogonality relation.

\Definition{\label{directorth}
A set $I$ is called a \defem{directed set with orthogonality
relation}\indexn{orthogonality relation>defemidx}\indexn{directed
set>defemidx}\indexn{directed set!with orthogonality relation>defemidx} 
if it carries a partial order $\le$ and a symmetric
relation $\bot$\indexs{orthogonalityrelation@$\protect\bot$, orthogonality
  relation in a set} between its elements such that 
\begin{enumerate}
\item for all $\alpha,\beta\in I$ there exists a $\gamma\in I$ with
  $\alpha\le\gamma$ and $\beta\le\gamma$,
\item for every $\alpha\in I$ there is a $\beta\in I$ with
  $\alpha\bot\beta$, 
\item if $\alpha\le\beta$ and $\beta\bot\gamma$, then $\alpha\bot\gamma$,
\item if $\alpha\bot\beta$ and $\alpha\bot\gamma$, then there exists a
  $\delta\in I$ such that $\beta\le\delta$, $\gamma\le\delta$ and
  $\alpha\bot\delta$. 
\end{enumerate}
}

In order to handle non-globally hyperbolic manifolds we need to relax this
definition slightly:

\Definition{\label{Defquasilokal}
A set $I$ is called a \defem{directed set with weak orthogonality
  relation}\indexn{directed set!with weak orthogonality
  relation>defemidx} if it carries a partial order $\le$ and a symmetric
relation 
$\bot$ between its elements such that conditions (1), (2), and (3) in
Definition~\ref{directorth} are fulfilled.
}

Obviously, directed sets with orthogonality relation are automatically 
directed sets with weak orthogonality relation.
We use such sets in the following as index sets for nets of $C^*$-algebras. 

\Definition{\label{dqlocC*alg}
A \defem{(bosonic) quasi-local
  $C^*$-algebra}\indexn{C*Algebra@$C^*$-algebra!quasi-local>defemidx}\indexn{bosonic quasi-local 
  $C^*$-algebra>defemidx} is a
pair 
$\left(\A,\left\{\A_\alpha\right\}_{\alpha\in I}\right)$ of a $C^*$-algebra
$\A$\indexs{A@$\protect\A$, $\protect\A_\alpha$, $C^*$-algebra} and a family
$\left\{\A_\alpha\right\}_{\alpha\in I}$ of $C^*$-subalgebras, 
where $I$ is a directed set with orthogonality relation such that the
following holds:
\begin{enumerate}
\item $\A_\alpha\subset \A_\beta$ whenever $\alpha\le\beta$,
\item $\A=\overline{\,\bigcup\limits_\alpha\;\,\A_\alpha\,}\;$ where the
  bar denotes the closure with respect to the norm of $\A$.
\item the algebras $\A_\alpha$ have a common unit $1$,
\item if $\alpha\bot\beta$ the commutator of $\A_\alpha$ and $\A_\beta$
  is trivial:
  $\left[\A_\alpha,\A_\beta\right]=\{0\}$\indexs{commutator@$\left[\protect\A_\alpha,\protect\A_\beta\right]$,
  commutator of $\protect\A_\alpha$ and $\protect\A_\beta$}.  
\end{enumerate}
}

\Remark{
This definition is a special case of the one in \cite[Def.\ 2.6.3]{BR1} 
where there is in addition an involutive automorphism $\sigma$ of $\A$.
In our case $\sigma=\id$ which physically corresponds to a bosonic theory.
This is why one might call our version of quasi-local $C^*$-algebras
\emph{bosonic}.
}

\Definition{
A \defem{morphism}\indexn{morphism between (weak) quasi-local
  $C^*$-algebras>defemidx} between two quasi-local $C^*$-algebras 
$\big(\A,\{\A_\alpha\}_{\alpha\in I}\big)$ and 
$\big(\B,\{\B_\beta\}_{\beta\in J}\big)$  is a pair $(\varphi,\Phi)$
where $\Phi:\A\to\B$ is a unit-preserving $C^*$-morphism and
$\varphi:I\to J$ is a map such that:
\begin{enumerate}
\item $\varphi$ is monotonic, i.~e., if $\alpha_1\le\alpha_2$ in $I$ then
  $\varphi(\alpha_1)\le\varphi(\alpha_2)$ in $J$,
\item $\varphi$ preserves orthogonality, i.~e., if $\alpha_1\bot\alpha_2$
  in $I$, then $\varphi(\alpha_1)\bot\varphi(\alpha_2)$ in $J$, 
\item $\Phi(\A_\alpha)\subset\B_{\varphi(\alpha)}$ for all $\alpha\in I$.
\end{enumerate}
}

The composition of morphisms of quasi-local $C^*$-algebras is just the
composition of maps, and we obtain the category
$\QLA$\indexs{QLA@$\protect\QLA$, category of quasi-local $C^*$-algebras} of
quasi-local 
$C^*$-algebras.

\Definition{
A \defem{weak quasi-local $C^*$-algebra}\indexn{C*Algebra@$C^*$-algebra!weak
  quasi-local>defemidx} is a pair    
$\left(\A,\left\{\A_\alpha\right\}_{\alpha\in I}\right)$ of a $C^*$-algebra
$\A$ and a family $\left\{\A_\alpha\right\}_{\alpha\in I}$ of
$C^*$-subalgebras, where $I$ is a directed set with weak orthogonality
relation such that the same conditions as in Definition~\ref{dqlocC*alg}
hold. 
Morphisms between weak quasi-local $C^*$-algebras are defined in exactly the
same way as morphisms between quasi-local $C^*$-algebras.
}

This yields another category, the category of weak quasi-local $C^*$-algebras
$\CNET$\indexs{QLAW@$\protect\CNET$, category of weak quasi-local
  $C^*$-algebras}.  
We note that $\QLA$ is a full subcategory of $\CNET$.

Next we want to associate to any object $(M,E,P,G_+,G_-)$ in $\LORFUND$ a
weak quasi-local $C^*$-algebra.
For this we set
$$
I:=\{ O\subset M\,\mid \,O \textrm{\small{ is open, relatively
    compact, causally compatible, globally hyperbolic}}\}\cup\{\emptyset, M\}.
$$
On $I$ we take the inclusion $\subset$ as the partial order $\le$ and define
the orthogonality relation by 
\[
O\perp O' :\Leftrightarrow  J^M(\ovl{O})\cap \ovl{O'}=\emptyset.
\]
This means that two elements of $I$ are orthogonal if and only if they are
\defem{causally independent}\indexn{causally independent subsets>defemidx}
subsets of $M$ in the sense that there are no causal curves connecting a point
in $\ovl{O}$ with a point in $\ovl{O'}$. 
Of course, this relation is symmetric.

\begin{lemma}\label{lIdirortho}
The set $I$ defined above is a directed set with weak orthogonality relation.
\end{lemma}
\begin{proof}
Condition (1) in Definition~\ref{directorth} holds with $\gamma=M$ and (2)
with $\beta=\emptyset$. 
Property (3) is also clear because $O\subset O'$ implies $J^M(\ovl{O})\subset
J^M(\ovl{O'})$. 
\end{proof}

\begin{lemma}\label{lIdirorth}
Let $M$ be globally hyperbolic.
Then the set $I$ is a directed set with (non-weak) orthogonality relation.
\end{lemma}

\begin{proof} 
In addition to Lemma~\ref{lIdirortho} we have to show Property (4) of
Definition~\ref{directorth}. 
Let $O_1,O_2,O_3\in I$ with $J^M(\ovl{O_1})\cap \ovl{O_2}=\emptyset$ and
$J^M(\ovl{O_1})\cap \ovl{O_3}=\emptyset$. 
We want to find $O_4\in I$ with $O_2\cup O_3\subset O_4$ and
$J^M(\ovl{O_1})\cap \ovl{O_4}=\emptyset$.

Without loss of generality let $O_1,O_2,O_3$ be non-empty. 
Now none of $O_1$, $O_2,$ and $O_3$ can equal $M$. 
In particular, $O_1$, $O_2,$ and $O_3$ are relatively compact. 
Set $\Omega:=M\setminus J^M(\ovl{O_1})$. 
By Lemma~\ref{lMminusJAglobhyp} the subset $\Omega$ of $M$ is causally
compatible and globally hyperbolic.
The hypothesis $J^M(\ovl{O_1})\cap \ovl{O_2}=\emptyset=J^M(\ovl{O_1})\cap
\ovl{O_3}$ implies $\ovl{O_2}\cup \ovl{O_3}\subset\Omega$. 
Applying Proposition~\ref{lIdir} with $K:=\ovl{O_2}\cup\ovl{O_3}$ in the
globally hyperbolic manifold $\Omega$, we obtain a relatively compact causally
compatible globally hyperbolic open subset $O_4\subset\Omega$ containing
$O_2$ and $O_3$. 
Since $\Omega$ is itself causally compatible in $M$, the subset $O_4$ is
causally compatible in $M$ as well. 
By definition of $\Omega$ we have $\ovl{O_4}\subset\Omega=M\setminus
J^M(\ovl{O_1})$, i.~e., $J^M(\ovl{O_1})\cap \ovl{O_4}=\emptyset$. 
This shows Property (4) and concludes the proof of Lemma~\ref{lIdirorth}. 
\end{proof}

\Remark{
If $M$ is globally hyperbolic, the proof of Proposition~\ref{lIdir} shows
that the index set $I$ would also be directed if we removed $M$ from it in its
definition.
Namely, for all elements $O_1,O_2\in I$ different from $\emptyset$ and $M$, 
the element $O$ from Proposition~\ref{lIdir} applied to
$K:=\ovl{O_2}\cup\ovl{O_3}$ belongs to $I$.
}

Now we are in the situation to associate  a weak quasi-local $C^*$-algebra to
any object $(M,E,P,G_+,G_-)$ in $\LORFUND$.

We consider the index set $I$ as defined above.
For any non-empty $O\in I$ we take the restriction $E|_O$ and the
corresponding restriction of the operator $P$ to sections of $E|_O$. 
Due to the causal compatibility of $O\subset M$ the restrictions of the
Green's operators $G_+$, $G_-$ to sections over $O$ yield the Green's operators
$G_+^O$, $G_-^O$ for $P$ on $O$, see Proposition~\ref{propGreensubset}. 
Therefore we get an object $(O,E|_O,P,G_+^O,G_-^O)$ for each $O\in I$,
$O\ne\emptyset$.

For $\emptyset\ne O_1\subset O_2$ the inclusion induces
a morphism $\iota_{O_2,O_1}$\indexs{iota@$\iota_{O_2,O_1}$, morphism in
  $\protect\LORFUND$ induced by inclusion $O_1\subset O_2$} in the category
$\LORFUND$. 
This morphism is given by the embeddings $O_1 \hookrightarrow O_2$ and
$E|_{O_1} \hookrightarrow E|_{O_2}$. 
Let $\alpha_{O_2,O_1}$\indexs{alphaO1O2@$\alpha_{O_2,O_1}$, morphism
  $\CCR\circ\SYM\!(\iota_{O_2,O_1})$ in $\protect\CALG$} denote the morphism
$\CCR\circ\SYM\!(\iota_{O_2,O_1})$ 
in $\CALG$.   
Recall that $\alpha_{O_2,O_1}$ is an injective unit-preserving $*$-morphism.

We set for $\emptyset\ne O\in I$
\[
(V_O,\omega_O):=\SYM (O,E|_O,P,G_+^O,G_-^O)
\]
and for $O\in I$, $O\ne\emptyset$, $O\ne M$,
\[ 
\A_O:=\alpha_{M,O}\left(\CCR(V_O,\omega_O) \right) .
\]

Obviously, for any  $O\in I$, $O\ne\emptyset$, $O\ne M$ the algebra $\A_O$
is a $C^*$-subalgebra of $\CCR(V_M,\omega_M)$.
For $O=M$ we define $\A_M$ as the $C^*$-subalgebra of $\CCR(V_M,\omega_M)$
generated by all $\A_O$, 
\[ 
\A_M:=C^*\Big(\bigcup_{ O\in I \atop O\ne\emptyset, O\ne M }\A_O
\Big).
\] 
Finally, for $O=\emptyset$ we set $\A_\emptyset=\Co\cdot 1$.
We have thus defined a family $\left\{ \A_O\right\}_{O\in I}$ of
$C^*$-subalgebras of $\A_M$.

\begin{lemma}
Let $(M,E,P,G_+,G_-)$ be an object in $\LORFUND$. 
Then $\big(\A_M, \left\{ \A_O\right\}_{O\in I}\big)$ is a weak
quasi-local $C^*$-algebra.
\end{lemma}

\begin{proof}
We know from Lemma~\ref{lIdirortho} that $I$ is a directed set with weak
orthogonality relation. 

It is clear that $\A_M=\ovl{\buil{\bigcup}{O\in I}\A_O}$ because $M$ belongs
to $I$.
By construction it is also clear that all algebras $\A_O$ have the common unit
$W(0)$, $0\in V_M$.
Hence Conditions~(2) and (3) in Definition~\ref{dqlocC*alg} are obvious.

By functoriality we have the following commutative diagram 
\begin{center}
\hspace{.5cm}\xymatrix{
\CCR(V_O,\omega_O) \ar[r]^{\alpha_{M,O}} \ar[d]_{\alpha_{O',O}}
&\CCR(V_M,\omega_M)\\ 
\CCR(V_{O'},\omega_{O'}) \ar[ur]_{\alpha_{M,O'}} & 
}
\end{center}
Since $\alpha_{O',O}$ is injective we have $\A_{O}\subset\A_{O'}$.
This proves Condition~(1) in Definition~\ref{dqlocC*alg}.

Let now $O,O'\in I$ be causally independent. 
Let $\phi\in\DD(O,E)$ and $\psi\in\DD(O',E)$. 
It follows from $\supp(G\phi)\subset J^M(O)$ that
$\supp(G\phi)\cap\supp(\psi)=\emptyset$, hence 
\[\int_M \la G\phi\,,\psi\ra \dV=0.\]
For the symplectic form $\omega$ on $\DD(M,E)/\ker(G)$ this means
$\omega(\phi,\psi)=0$, where we denote the equivalence class in
$\DD(M,E)/\ker(G)$ of the extension to $M$ by zero of $\phi$ again by $\phi$
and similarly for $\psi$.
This yields by Property~(iii) of a Weyl-system 
\[ 
W(\phi)\cdot W(\psi)=W(\phi+\psi)=W(\psi)\cdot W(\phi),
\]
i.~e., the generators of $\A_O$ commute with those of $\A_{O'}$.
Therefore $\left[\A_O,\A_{O'}\right]=0$.
This proves (4) in Definition~\ref{dqlocC*alg}.
\end{proof}

Next we associate a morphism in $\CNET$ to any morphism $(f,F)$ in $\LORFUND$
beween two objects $(M_1,E_1,P_1,{G_1}_+,{G_1}_-)$ and $(M_2,E_2,P_2,{G_2}
_+,{G_2} _-)$.
Recall that in the case of distinct objects such a morphism only exists if
$M_1$ is globally hyperbolic.
Let $I_1$ and $I_2$ denote the index sets for the two objects as above and let
$\big(\A_{M_1}, \left\{\A_O\right\}_{O\in I_1}\big)$ and $\big(\B_{M_2},
\left\{ \B_O\right\}_{O\in I_2}\big)$\indexs{B@$\protect\B$, $\protect\B_O$,
  $C^*$-algebra} denote the corresponding weak 
quasi-local $C^*$-algebras.
Then $f$ maps any $O_1\in I_1$, $O_1\ne M_1$, to $f(O_1)$ which is an element
of $I_2$ by definition of $\LORFUND$.
We get a map $\varphi:I_1\to I_2$ by $M_1\mapsto M_2$ and $O_1\mapsto
f(O_1)$ if $O_1\ne M_1$. 
Since $f$ is an embedding such that $f(M_1)\subset M_2$ is causally compatible,
the map $\phi$ is monotonic and preserves causal independence.

Let
$\Phi:\CCR(V_{M_1},\omega_{M_1})\to\CCR(V_{M_2},\omega_{M_2})$\indexs{Phi@$\Phi$,
  morphism $\protect\CCR\circ\protect\SYM(f,F)$ in $\protect\CALG$} be the 
morphism $\Phi=\CCR\circ\SYM(f,F)$.
From the commutative diagram of inclusions and embeddings
\begin{center}
\hspace{.5cm}\xymatrix{
O_1 \ar@{^{(}->}[r] \ar[d]_{f|_{O_1}} & O_2 \ar[d]^{f|_{O_2}}\\ 
f(O_1) \ar@{^{(}->}[r] & M_2
}
\end{center}
we see 
\begin{eqnarray*}
\Phi(\A_{O_1}) &=&
\CCR(\SYM(f,F))\circ \CCR(\SYM(\iota_{M,O_1}))(\CCR(V_{O_1},\omega_{O_1}))\\ 
&=&
\CCR(\SYM(\iota_{M_2,f(O_1)}))\circ \CCR(\SYM(f|_{O_1},F|_{E|_{O_1}}))
(\CCR(V_{O_1},\omega_{O_1}))\\
&\subset&
\alpha_{M_2,f(O_1)}(\CCR(V_{f(O_1)},\omega_{f(O_1)}))\\
&=&
\B_{f(O_1)}.
\end{eqnarray*}
This also implies $\Phi(\A_{M_1}) \subset \B_{M_2}$.
Therefore the pair $(\phi,\Phi|_{\A_{M_1}})$ is a morphism in $\CNET$.
We summarize

\begin{thm}\label{funktorcnet}
The assignments $(M,E,P,G_+,G_-)\mapsto\big(\A_M, \left\{ \A_O\right\}_{O\in
  I}\big)$ and $(f,F) \mapsto (\phi,\Phi|_{\A_{M_1}})$ yield a functor
  $\LORFUND\to\CNET$. 
\end{thm}

\begin{proof}
If $f=\id_M$ and $F=\id_E$, then $\phi=\id_I$ and $\Phi=\id_{\A_M}$.
Similarly, the composition of two morphisms in $\LORFUND$ is mapped to the
composition of the corresponding two morphisms in $\CNET$.
\end{proof}

\begin{cor}
The composition of $\solve$ and the functor from Theorem~\ref{funktorcnet}
yields a functor $\GLOBHY\to\QLA$. 
One gets the following commutative diagram of functors:
\end{cor}
\begin{center}
\input{fig-quantenschema2}
\end{center}

\begin{proof}
Let $(M,E,P)$ be an object in $\GLOBHY$.
Then we know from Lemma~\ref{lIdirorth} that the index set $I$ associated to
$\solve(M,E,P)$ is a directed set with (non-weak) orthogonality relation, and
the corresponding weak quasi-local $C^*$-algebra is in fact a quasi-local
$C^*$-algebra.
This concludes the proof since $\QLA$ is a full subcategory of $\CNET$.
\end{proof}

\begin{lemma}\label{globhypAM}
Let $(M,E,P)$ be an object in $\GLOBHY$, and denote by $\big(\A_M,\big\{\A_O
\big\}_{O\in I}\big)$ the corresponding quasi-local $C^*$-algebra.
Then
\[ 
\A_M=\CCR\circ\SYM\circ\solve\big(M,E,P \big).
\]
\end{lemma}
\begin{proof}
Denote the right hand side by $\ovl{\A}$.
By definition of $\A_M$ we have $\A_M\subset\ovl{\A}$.

In order to prove the other inclusion write $(M,E,P,G_+,G_-):=\solve(M,E,P)$.
Then $\SYM(M,E,P,G_+,G_-)$ is given by $V_M=\DD(M,E)/\ker(G)$ with symplectic
form $\omega_M$ induced by $G$.
Now $\ovl{\A}$ is generated by
\[ 
\mathcal{E}=\left\{W([\varphi])\,\big|\,\varphi\in\DD(M,E)\right\} 
\] 
where $W$ is the Weyl system from Example~\ref{CCR1} and $[\phi]$ denotes the
equivalence class of $\phi$ in $V_M$.
For given $\phi\in\DD(M,E)$ there exists a relatively compact globally
hyperbolic causally compatible open subset $O\subset M$ containing the compact
set $\supp(\phi)$ by Proposition~\ref{lIdir}.
For this subset $O$ we have $W([\varphi])\in\A_O$.
Hence we get $\mathcal{E}\subset\buil{\bigcup}{O\in I}\A_O\subset\A_M$ which
implies $\ovl{\A}\subset\A_M$. 
\end{proof}

\Example{
Let $M$ be globally hyperbolic.
All the operators listed in Examples~\ref{ex:Boxfsa} to \ref{ex:Diracsquare}
give rise to quasi-local $C^*$-algebras.
These operators include the d'Alembert operator, the Klein-Gordon operator,
the Yamabe operator, the wave operators for the electro-magnetic potential and
the Proca field as well as the square of the Dirac operator.
\indexn{d'Alembert operator}
\indexn{Klein-Gordon operator}
\indexn{Yamabe operator}
\indexn{Dirac operator}
\indexn{Proca equation}
\indexn{field strength in electrodynamics}
}

\Example{
Let $M$ be the anti-deSitter spacetime.\indexn{anti-deSitter spacetime}
Then $M$ is not globally hyperbolic but as we have seen in
Section~\ref{seq:nonglobhyp} we can get Green's operators for the (normalized)
Yamabe operator $P_g$ by embedding $M$ conformally into the Einstein cylinder.
This yields an object $(M,M\times\R, P_g,G_+,G_-)$ in $\LORFUND$.
Hence there is a corresponding weak quasi-local $C^*$-algebra over $M$.
}

%%%%%%%%%%%%%%%%%%%%%%%%%%%%%%%%%%%%%%%%%%%%%%%%%%%%%%%%%%%%%%%%%%%%%%%%%
\section{Haag-Kastler axioms}\label{sec:Haag-Kastler}
%%%%%%%%%%%%%%%%%%%%%%%%%%%%%%%%%%%%%%%%%%%%%%%%%%%%%%%%%%%%%%%%%%%%%%%%%

We now check that the functor that assigns to each object in $\LORFUND$ a
quasi-local $C^*$-algebra as constructed in the previous section satisfies the
Haag-Kastler axioms of a quantum field theory. 
These axioms have been proposed in \cite[p.~849]{HK} for Minkowski space.
Dimock \cite[Sec.~1]{Dimock1} adapted them to the case of globally hyperbolic
manifolds.
He also constructed the quasi-local $C^*$-algebras for the Klein-Gordon
operator.

\begin{thm}\label{tHaagKastler}
The functor $\LORFUND\to\CNET$ from Theorem~\ref{funktorcnet} satisfies the
\defidx{Haag-Kastler axioms}, that is, for every object $(M,E,P,G_+,G_-)$ in
$\LORFUND$ the corresponding weak quasi-local $C^*$-algebra
$\left(\A_M,\left\{\A_O\right\}_{O\in I}\right)$ satisfies: 
\begin{enumerate}
\item \label{HK:isotony}
If $O_1\subset O_2$, then $\A_{O_1}\subset\A_{O_2}$ for all $O_1,O_2\in I$.
\item \label{HK:local}
$\A_M=\ovl{\buil{\cup}{O\in I\atop O\neq\emptyset,\, O\neq M}\A_O}$.
\item \label{HK:simple}
If $M$ is globally hyperbolic, then $\A_M$ is simple. 
\item \label{HK:unit}
The $\A_O$'s have a common unit $1$.
\item \label{HK:causality1}
For all $O_1,O_2\in I$ with $J(\ovl{O_1})\cap\ovl{O_2}=\emptyset$ the
subalgebras $\A_{O_1}$ and $\A_{O_2}$ of $\A_M$ commute: 
$[\A_{O_1},\A_{O_2}]=\{0\}$.
\item (\defem{Time-slice axiom})\indexn{time-slice axiom>defemidx}
  \label{HK:time-slice}  
Let $O_1\subset O_2$ be nonempty elements of $I$ admitting a common Cauchy
hypersurface. 
Then $\A_{O_1}=\A_{O_2}$.
\item \label{HK:causality2}
Let $O_1,O_2\in I$ and let the Cauchy development $D(O_2)$ be
relatively compact in $M$.
If $O_1\subset D(O_2)$, then $\A_{O_1}\subset \A_{O_2}$.
\end{enumerate}
\end{thm}

\Remark{
It can happen that the Cauchy development $D(O)$ of a causally compatible
globally hyperbolic subset $O$ in a globally hyperbolic manifold $M$ is not
relatively compact even if $O$ itself is relatively compact. 
See the following picture for an example where $M$ and $O$ are ``lens-like''
globally hyperbolic subsets of Minkowski space:

\begin{center}
\input{fig-DOnrelkompakt}
\end{center}

This is why we {\em assume} in (\ref{HK:causality2}) that $D(O_2)$ is
relatively compact. 
}

\Remark{
Instead of (\ref{HK:simple}) one often finds the requirement that $\A_M$
should be
\defem{primitive}\indexn{primitive $C^*$--algebra>defemidx} for globally
hyperbolic $M$.
This means that there exists a faithful irreducible representation of $\A_M$
on a Hilbert space.
We know by Lemma~\ref{globhypAM} and Corollary~\ref{cor:simple} that $\A_M$ is
simple, i.~e., that (\ref{HK:simple}) holds.
Simplicity implies primitivity because each $C^*$-algebra has irreducible
representations \cite[Sec.~2.3.4]{BR1}.
}

\begin{proof}[Proof of Theorem~\ref{tHaagKastler}]
Only axioms (\ref{HK:time-slice}) and (\ref{HK:causality2}) require a proof.
First note that axiom~(\ref{HK:causality2}) follows from axioms
(\ref{HK:isotony}) and (\ref{HK:time-slice}):

Let $O_1,O_2\in I$, let the Cauchy development $D(O_2)$ be
relatively compact in $M$, and let $O_1\subset D(O_2)$.
By Theorem~\ref{globhyp} there is a smooth spacelike Cauchy hypersurface
$\Sigma \subset O_2$.
It follows from the definitions that $D(O_2)=D(\Sigma)$.
Since $O_2$ is causally compatible in $M$ the hypersurface $\Sigma$ is acausal
in $M$.
By Lemma~\ref{lem:DS} $D(\Sigma)$ is a causally compatible 
and globally hyperbolic open subset of $M$. 
Since $D(O_2)=D(\Sigma)$ is relatively compact by assumption we have
$D(O_2)\in I$. 

Axiom~(\ref{HK:time-slice}) implies $\A_{O_2} = \A_{D(O_2)}$.
By axiom (\ref{HK:isotony}) $\A_{O_1} \subset \A_{D(O_2)} = \A_{O_2}$.

It remains to show the time-slice axiom. 
We prepare the proof by first deriving two lemmas.
The first lemma is of technical nature while the second one is essentially
equivalent to the time-slice axiom.

\begin{lemma}\label{lexistrho}
Let $O$ be a causally compatible globally hyperbolic open subset of a globally
hyperbolic manifold $M$.
Assume that there exists a Cauchy hypersurface $\Sigma$ of $O$ which is also a
Cauchy hypersurface of $M$. 
Let $h$ be a Cauchy time-function on $O$ as in Corollary~\ref{cexisttimefctn}
(applied to $O$).
Let $K\subset M$ be compact.
Assume that there exists a $t\in\R$ with $K\subset I_+^M(h^{-1}(t))$.

Then there is a smooth function $\rho:M\to[0,1]$ such that
\beit
\item $\rho=1$ on a neighborhood of $K$,
\item $\supp(\rho)\cap J_-^M(K)\subset M$ is compact, and
\item $\ovl{\{x\in M\,|\, 0<\rho(x)<1\}}\cap J_-^M(K)$ is compact and
  contained in $O$. 
\eeit
\end{lemma}

\Remark{
Similarly, if instead of $K\subset I_+^M(h^{-1}(t))$ we have $K\subset
I_-^M(h^{-1}(t))$ for some $t$, then we can find  a smooth function
$\rho:M\to[0,1]$ such that 
\beit
\item $\rho=1$ on a neighborhood of $K$,
\item $\supp(\rho)\cap J_+^M(K)\subset M$ is compact, and
\item $\ovl{\{x\in M\,|\, 0<\rho(x)<1\}}\cap J_+^M(K)$ is compact and
  contained in $O$. 
\eeit
}

\begin{proof}[Proof of Lemma~\ref{lexistrho}]
By assumption there exist real numbers $t_-<t_+$ in the range of $h$ such that
$K\subset I_+^M(S_{t_+})$, hence also $K\subset I_+^M(S_{t_-})$, where
$S_t:=h^{-1}(t)$. 
Since $S_t$ is a Cauchy hypersurface of $O$ and since $O$ and $M$ admit a
common Cauchy hypersurface, it follows from Lemma~\ref{lcommonCauchhyp} that
$S_{t_-}$ and $S_{t_+}$ are also Cauchy hypersurfaces of $M$.
Since $J_+^M(S_{t_+})$ and $J_-^M(S_{t_-})$ are disjoint closed subsets of $M$
there exists a smooth function $\rho:M\to[0,1]$ such that
$\rho_{|_{J_+^M(S_{t_+})}}=1$ and $\rho_{|_{J_-^M(S_{t_-})}}=0$. 

\begin{center}
\input{fig-constructrho}
\end{center}

We check that $\rho$ has the three properties stated in Lemma~\ref{lexistrho}.
The first one follows from $K\subset I_+^M(S_{t_+})$.

Since $\rho_{|_{I_-^M(S_{t_-})}}=0$, we have $\supp(\rho)\subset
J_+^M(S_{t_-})$.
It follows from Lemma~\ref{pastcompact} applied to the past-compact subset
$J_+^M(S_{t_-})$ of $M$ that $J_+^M(S_{t_-})\cap J_-^M(K)$ is relatively
compact, hence compact by Lemma~\ref{lJKompaktumabgeschl}.
Therefore the second property holds.

The closed set $\ovl{\{0<\rho<1\}}\cap J_-^M(K)$ is contained in the compact
set $\supp(\rho)\cap J_-^M(K)$, hence compact itself.

The subset $\ovl{\{0<\rho<1\}}$ of $M$ lies in $J_+^M(S_{t_-})\cap
J_-^M(S_{t_+})$.
We claim that $J_+^M(S_{t_-})\cap J_-^M(S_{t_+})\subset O$ which will then
imply $\ovl{\{0<\rho<1\}}\cap J_-^M(K)\subset O$ and hence conclude the proof. 

Assume that there exists $p\in J_+^M(S_{t_-})\cap J_-^M(S_{t_+})$ but
$p\not\in O$. 
Choose a future directed causal curve $c:[s_-,s_+] \to M$ from $S_{t_-}$ to
$S_{t_+}$ through $p$. 
Extend this curve to an inextendible future directed causal curve $c:\R\to M$.
Let $I'$ be the connected component of $c^{-1}(O)$ containing $s_-$.
Then $I'\subset \R$ is an open interval and $c|_{I'}$ is an inextendible
causal curve in $O$.
Since $p\not\in O$ the curve leaves $O$ before it reaches $S_{t_+}$, hence
$s_+\not\in I'$.
But $S_{t_+}$ is a Cauchy hypersurface in $O$ and so there must be an $s\in I'$
with $c(s)\in S_{t_+}$.
Therefore the curve $c$ meets $S_{t_+}$ at least twice (namely in $s$ and in
$s_+$) in contradiction to $S_{t_+}$ being a Cauchy hypersurface in $M$.
\end{proof}

\begin{lemma}\label{ltimeslice}
Let $(M,E,P)$ be an object in $\GLOBHY$ and let $O$ be a causally compatible
globally hyperbolic open subset of $M$.
Assume that there exists a Cauchy hypersurface $\Sigma$ of $O$ which is also a
Cauchy hypersurface of $M$. 
Let $\phi\in\DD(M,E)$.

Then there exist $\psi,\chi\in\DD(M,E)$ such that $\supp(\psi)\subset O$ and 
\[\phi=\psi+P\chi.\]
\end{lemma}

\begin{proof}[Proof of Lemma~\ref{ltimeslice}]
Let $h$ be a time-function on $O$ as in Corollary~\ref{cexisttimefctn}
(applied to $O$). 
Fix $t_-<t_+\in\R$ in the range of $h$. 
By Lemma~\ref{lcommonCauchhyp} the subsets $S_{t_-}:=h^{-1}(t_-)$ and
$S_{t_+}:=h^{-1}(t_+)$ are also Cauchy hypersurfaces of $M$. 
Hence every inextendible timelike curve in $M$ meets $S_{t_-}$ and $S_{t_+}$.
Since $t_-<t_+$ the set $\{I_+^M(S_{t_-}),I_-^M(S_{t_+})\}$ is a finite open
cover of $M$.

Let $\{f_+,f_-\}$ be a partition of unity subordinated to this cover.
In particular, $\supp(f_\pm)\subset I_\pm^M(S_{t_\mp})$. 
Set $K_\pm:=\supp(f_\pm\varphi)=\supp(\varphi)\cap\supp(f_\pm)$. 
Then $K_\pm$ is a compact subset of $M$ satisfying $K_\pm\subset
I_\pm^M(S_{t_\mp})$. 
Applying Lemma~\ref{lexistrho} we obtain two smooth functions
$\rho_+,\rho_-:M\rightarrow [0,1]$ satisfying: 

\beit
\item $\rho_\pm=1$ in a neighborhood of $K_\pm$,
\item $\supp(\rho_\pm)\cap J_\mp^M(K_\pm)\subset M$ is compact, and
\item $\ovl{\{0<\rho_\pm<1\}}\cap J_\mp^M(K_\pm)$ is compact and contained in
  $O$.
\eeit

Set $\chi_\pm:=\rho_\pm G_\mp(f_\pm\phi)$, $\chi:=\chi_++\chi_-$ and
$\psi:=\phi-P\chi$. 
By definition, $\chi_\pm$, $\chi$, and $\psi$ are smooth sections in $E$ over
$M$.
Since $\supp(G_\mp(f_\pm\phi))\subset J_\mp^M(\supp(f_\pm\phi))\subset
J_\mp^M(K_\pm)$, the support of $\chi_\pm$ is contained in
$\supp(\rho_\pm)\cap J_\mp^M(K_\pm)$, which is compact by the second property
of $\rho_\pm$. 
Therefore $\chi\in\DD(M,E)$.

It remains to show that $\supp(\psi)$ is compact and contained in $O$.
By the first property of $\rho_\pm$ one has $\chi_\pm=G_\mp(f_\pm\phi)$ in a
neighborhood of $K_\pm$.
Moreover, $f_\pm\phi=0$ on ${\{\rho_\pm=0\}}$.
Hence $P\chi_\pm=f_\pm\phi$ on $\{\rho_\pm=0\}\cup\{\rho_\pm=1\}$. 
Therefore $f_\pm\phi-P\chi_\pm$ vanishes outside $\ovl{\{0<\rho_\pm<1\}}$,
i.~e., $\supp(f_\pm\phi-P\chi_\pm)\subset\ovl{\{0<\rho_\pm<1\}}$. 
By the definitions of $\chi_\pm$ and $f_\pm$ one also has
$\supp(f_\pm\phi-P\chi_\pm)\subset K_\pm\cup J_\mp^M(K_\pm)=J_\mp^M(K_\pm)$,
hence $\supp(f_\pm\phi-P\chi_\pm)\subset \ovl{\{0<\rho_\pm<1\}}\cap
J_\mp^M(K_\pm)$, which is compact and contained in $O$ by the third property
of $\rho_\pm$.
Therefore the support of $\psi=f_+\phi-P\chi_++f_-\phi-P\chi_-$ is compact and
contained in $O$. 
This shows Lemma~\ref{ltimeslice}.
\end{proof}

{\em End of proof of Theorem~\ref{tHaagKastler}.}
The time-slice axiom in Theorem~\ref{tHaagKastler} follows directly from
Lemma~\ref{ltimeslice}.
Namely, let $O_1\subset O_2$ be nonempty causally compatible globally
hyperbolic open subsets of $M$ admitting a common Cauchy hypersurface. 
Let $[\phi]\in V_{O_2}:=\DD(O_2,E)/\ker(G_{O_2})$. 
Lemma~\ref{ltimeslice} applied to $M:=O_2$ and $O:=O_1$ yields
$\chi\in\DD(O_2,E)$ and $\psi\in\DD(O_1,E)$ such that $\phi=\ext\psi+P\chi$. 
Since $P\chi\in\ker(G_{O_2})$ we have $[\phi]=[\ext\psi]$, that is,
$[\phi]$ is the image of the symplectic linear map $V_{O_1}\to V_{O_2}$
induced by the inclusion $O_1 \hookrightarrow O_2$, compare
Lemma~\ref{functorSYMPL}.
We see that this symplectic map is surjective, hence an isomorphism. 
Symplectic isomorphisms induce isomorphisms of $C^*$-algebras, hence the
inclusion $\A_{O_1}\subset\A_{O_2}$ is actually an equality. 
This proves the time-slice axiom and concludes the proof of
Theorem~\ref{tHaagKastler}. 
\end{proof}

%%%%%%%%%%%%%%%%%%%%%%%%%%%%%%%%%%%%%%%%%%%%%%%%%%%%%%%%%%%%%%%%%%%%%%%%%
\section{Fock space}\label{sec:Fock}\indexn{Fock space} 
%%%%%%%%%%%%%%%%%%%%%%%%%%%%%%%%%%%%%%%%%%%%%%%%%%%%%%%%%%%%%%%%%%%%%%%%% 

In quantum mechanics a particle is described by its wave function which
mathematically is a solution $u$ to an equation $Pu=0$.
We consider normally hyperbolic operators $P$ in this text.
The passage from single particle systems to multi particle systems is known
as \defem{second quantization}\indexn{second quantization>defemidx} in the
physics literature. 
Mathematically it requires the construction of the quantum field which we
will do in the subsequent section.
In this section we will describe some functional analytical underpinnings,
namely the construction of the bosonic Fock space.

We start by describing the symmetric tensor product of Hilbert spaces.
Let $H$ denote a complex vector space.
We will use the convention that the Hermitian scalar product
$(\cdot,\cdot)$\indexs{Hermscalprod@$(\cdot,\cdot)$, Hermitian scalar
  product}\indexn{Hermitian scalar product>defemidx} 
on $H$ is anti-linear in the first argument.
Let $H_n$ be the vector space freely generated by $H \times \cdots \times H$
($n$ copies), i.~e., the space of all finite formal linear combinations of
elements of $H \times \cdots \times H$.
Let $V_n$ be the vector subspace of $H_n$ generated by all elements of the form
$(v_1, \ldots, cv_k, \ldots, v_n) - c\cdot(v_1, \ldots, v_k, \ldots, v_n)$,
$(v_1, \ldots, v_k+v_k', \ldots, v_n) - (v_1, \ldots, v_k, \ldots, v_n) - (v_1,
\ldots, v_k', \ldots, v_n)$ and $(v_1,\ldots, v_n) - (v_{\sigma(1)}, \ldots,
v_{\sigma(n)})$ where $v_j, v_j' \in H$, $c\in\Co$ and $\sigma$ a permutation.

\Definition{
The vector space $\bigodot^n_\alg H := H_n/
V_n$\indexs{symmprodalg@$\protect\bigodot^n_{\protect\alg} H$, algebraic
  $n^{th}$ symmetric tensor product of $H$}\indexn{algebraic symmetric tensor
  product>defemidx} is 
called the \defem{algebraic $n^{th}$ symmetric tensor product} of $H$.
By convention, we put $\bigodot^0_\alg H:=\Co$.
}

For the equivalence class of $(v_1,\ldots, v_n)\in H_n$ in $\bigodot^n_\alg H$
we write $v_1 \odot \cdots \odot v_n$.
The map $\gamma:H \times \cdots \times H \to \bigodot^n_\alg H$ given by
$\gamma(v_1,\ldots,v_n) = v_1 \odot \cdots \odot v_n$ is multilinear and
symmetric. 
The algebraic symmetric tensor product has the following universal
property\indexn{universal property of the symmetric tensor product>defemidx}. 

\begin{lemma}\label{univpropstalg}
For each complex vector space $W$ and each symmetric multilinear map $\alpha
: H \times \cdots \times H \to W$ there exists one and only one linear map
$\bar\alpha : \bigodot^n_\alg H \to W$ such that the diagram

{\hspace{4cm}
\xymatrix{
H \times \cdots \times H \ar[d]_{\gamma} \ar[dr]^{\alpha}&\\
\bigodot^n_\alg H \ar[r]^{\bar\alpha} & W 
}}

commutes.
\end{lemma}

\begin{proof}
Uniqueness of $\bar\alpha$ is clear because
\begin{equation}\label{symtensorunivprop}
\bar\alpha(v_1 \odot \cdots \odot v_n) = \alpha(v_1,\ldots,v_n)
\end{equation}
and the elements $v_1 \odot \cdots \odot v_n$ generate $\bigodot^n_\alg H$.

To  show existence one defines $\bar\alpha$ by
Equation~\ref{symtensorunivprop} and checks easily that this is well-defined.
\end{proof}

The algebraic symmetric tensor product $\bigodot^n_\alg H$ inherits a scalar
product from $H$ characterized by
$$
(v_1 \odot \cdots \odot v_n,w_1 \odot \cdots \odot w_n) = 
\sum_\sigma (v_1,w_{\sigma(1)})\cdots(v_n,w_{\sigma(n)})
$$
where the sum is taken over all permutations $\sigma$ on $\{1,\ldots,n\}$.

\Definition{
The completion of $\bigodot^n_\alg H$ with respect to this scalar product is
called the \defem{$n^{th}$ symmetric tensor product}\indexn{symmetric tensor
  product>defemidx} of the Hilbert space $H$ 
and is denoted $\bigodot^n H$\indexs{symmprodH@$\protect\bigodot^n H$,
  $n^{th}$ symmetric tensor product of $H$}. 
In particular, $\bigodot^0 H=\Co$. 
}

\Remark{
If $\{e_j\}_{j\in \JJ}$ is an orthonormal system of $H$ where $\JJ$ is some
ordered index set, then $\{e_{j_1}\odot \cdots
\odot e_{j_n}\}_{j_1 \leq \cdots \leq j_n}$ forms an orthogonal system of
$\bigodot^n H$.
For each ordered multiindex $J=(j_1,\ldots,j_n)$ there is a corresponding
partition of $n$, $n=k_1 + \cdots + k_l$, given by
$$
{j_1}= \cdots = {j_{k_1}} < {j_{k_1+1}} = \cdots = {j_{k_1+k_2}} <
\cdots < {j_{k_1 + \cdots + k_{l-1}+1}} = \cdots = {j_{n}}.
$$
We compute
\begin{eqnarray*}
\|e_{j_1}\odot \cdots \odot e_{j_n}\|^2
&=&
\sum_\sigma (e_{j_1},e_{j_{\sigma(1)}}) \cdots (e_{j_n},e_{j_{\sigma(n)}})\\
&=&
\# \{\sigma\,|\, (j_{\sigma(1)},\ldots,j_{\sigma(n)})=(j_1,\ldots,j_n)\}\\
&=&
k_1! \cdots k_l! .
\end{eqnarray*}
In particular,
$$
1 \leq \|e_{j_1}\odot \cdots \odot e_{j_n}\| \leq \sqrt{n!}.
$$
}

%% The corresponding universal property of the completed symmetric tensor
%% product 
%% $\bigodot^n H$ is
%% }

%% \begin{lemma}\label{univprophilbert}
%% For each complex Banach space $W$ and each continuous symmetric multilinear
%% map 
%% $\alpha : H \times \cdots \times H \to W$ there exists one and only one
%% continuous linear map $\bar\alpha : \bigodot^n H \to W$ such that the
%% diagram

%% {\hspace{4cm}
%% \xymatrix{
%% H \times \cdots \times H \ar[d]_{\gamma} \ar[dr]^{\alpha}&\\
%% \bigodot^n H \ar[r]^{\bar\alpha} & W 
%% }}

%% commutes.
%% \end{lemma}

%% \begin{proof}
%% By Lemma~\ref{univpropstalg} the map $\bar\alpha$ is uniquely determined by
%% $\alpha$ on $\bigodot^n_\alg H$.
%% By continuity it is also determined on the completion $\bigodot^n H$.

%% Again by Lemma~\ref{univpropstalg}, given $\alpha$, there exists
%% $\bar\alpha$ 
%% on $\bigodot^n_\alg H$ making the diagram commute.
%% It only remains to show that $\bar\alpha$ is continuous since it then
%% extends continuously to the completion $\bigodot^n H$.
%% Since $\alpha$ is continuous there exists $C>0$ such that 
%% $$
%% \|\alpha(v_1,\ldots,v_n)\| \leq C\cdot \|v_1\| \cdots \|v_n\|
%% $$ 
%% for all $v_j\in H$.
%% Hence

%% \bf WEITER!
%% \end{proof}

The algebraic direct sum $\FF_\alg(H):=\bigoplus_{\alg,n=0}^\infty \bigodot^n
H$ 
carries a natural scalar product, namely
$$
((w_0,w_1,w_2,\ldots),(u_0,u_1,u_2,\ldots)) = \sum_{n=0}^\infty (w_n,u_n)
$$
where $w_n,u_n\in \bigodot^n H$.

\Definition{
We call $\FF_\alg(H)$\indexs{FalgH@$\protect\FF_{\protect\alg}(H)$, algebraic
  symmetric Fock space of $H$} the \defidx{algebraic symmetric Fock space} of
$H$. 
The completion of $\FF_\alg(H)$ with respect
to this scalar product is denoted $\FF(H)$\indexs{FH@$\protect\FF(H)$,
  symmetric Fock space of $H$} and 
is called the {\em bosonic} or {\em symmetric} \defidx{Fock space} of $H$.
The vector $\Omega:=1\in \Co =\bigodot^0 H \subset \FF_\alg(H) \subset
\FF(H)$\indexs{Omegavac@$\Omega$, vacuum vector} is 
called the \defidx{vacuum vector}. 
}

The elements of the Hilbert space $\FF(H)$ are therefore sequences
$(w_0,w_1,w_2,\ldots)$ with $w_n\in \bigodot^n H$ such that 
$$
\sum_{n=0}^\infty \|w_n\|^2 < \infty .
$$

Fix $v\in H$.
The map $\alpha : H \times \cdots \times H \to \bigodot^{n+1} H$, 
$\alpha(v_1,\ldots,v_n) = v \odot v_1 \odot \cdots \odot v_n$, is symmetric
multilinear 
and induces a linear map $\bar\alpha:\bigodot^{n}_\alg H \to \bigodot^{n+1}
H$, $v_1 \odot \cdots \odot v_n \mapsto v \odot v_1 \odot \cdots \odot v_n$, by
Lemma~\ref{univpropstalg}. 
We compute the operator norm of $\bar\alpha$.
Without loss of generality we can assume $\|v\|=1$.
We choose the orthonormal system $\{e_j\}_{j\in \JJ}$ of $H$ such that $v$
belongs to it.
If $v$ is perpendicular to all $e_{j_1},\ldots,e_{j_n}$, then 
$$
\|\bar\alpha(e_{j_1}\odot\cdots\odot e_{j_n})\| = 
\|v\odot e_{j_1}\odot\cdots\odot e_{j_n}\| =
\|e_{j_1}\odot\cdots\odot e_{j_n}\| .
$$
If $v$ is one of the $e_{j_\mu}$, say $v=e_{j_1}$, then
\begin{eqnarray*}
\|\bar\alpha(e_{j_1}\odot\cdots\odot e_{j_n})\|^2
&=&
(k_1+1)!\,k_2!\cdots k_l!\\
&=&
(k_1+1)\|e_{j_1}\odot\cdots\odot e_{j_n}\|^2.
\end{eqnarray*}
Thus in any case
$$
\|\bar\alpha(e_{j_1}\odot\cdots\odot e_{j_n})\| \leq 
\sqrt{n+1}\|e_{j_1}\odot\cdots\odot e_{j_n}\|
$$
and equality holds for $e_{j_1}=\cdots =e_{j_n} = v$.
Dropping the assumption $\|v\|=1$ this shows
$$
\|\bar\alpha\| = \sqrt{n+1}\|v\|.
$$
Hence $\bar\alpha$ extends to a bounded linear map
$$
a^*(v) : {\bigodot}^{n} H \to {\bigodot}^{n+1} H, \quad
a^*(v)(v_1 \odot \cdots \odot v_n)= v \odot v_1 \odot \cdots \odot v_n
\indexs{a*@$a^*$, creation operator}
$$
with 
\begin{equation}
\|a^*(v)\|=\sqrt{n+1}\|v\|.
\label{erzeugernorm}
\end{equation}
For the vacuum vector this means $a^*(v)\Omega = v$.
The map $a^*(v)$ is naturally defined as a linear map
$\FF_\alg(H) \to \FF_\alg(H)$.
By (\ref{erzeugernorm}) $a^*(v)$ is unbounded on $\FF_\alg(H)$ unless $v=0$ and
therefore does not extend continuously to $\FF(H)$.
Writing $v=v_0$ we see
\begin{eqnarray*}
\lefteqn{(a^*(v)(v_1\odot\cdots\odot v_n),w_0\odot w_1 \odot \cdots\odot
  w_n)}\\
&=&
(v_0\odot v_1\odot\cdots\odot v_n,w_0\odot w_1 \odot \cdots\odot w_n)\\
&=&
\sum_\sigma (v_0,w_{\sigma(0)}) (v_1,w_{\sigma(1)}) \cdots
(v_n,w_{\sigma(n)})\\
&=&
\sum_{k=0}^n \sum_{\sigma\,\,\mathrm{ with}\atop \sigma(0)=k}
(v,w_k)(v_1,w_{\sigma(1)}) \cdots (v_n,w_{\sigma(n)})\\
&=&
\left(v_1\odot\cdots\odot v_n,\sum_{k=0}^n(v,w_k)w_0\odot\cdots\odot \hat{w}_k
  \odot \cdots\odot w_n\right)
\end{eqnarray*}
where $\hat{w}_k$ indicates that the factor $w_k$ is left out.
Hence if we define $a(v) : \bigodot^{n+1}_\alg H \to\bigodot^n_\alg
H$\indexs{a@$a$, annihilation operator}
by  
$$
a(v)(w_0\odot\cdots\odot w_n) = 
\sum_{k=0}^n(v,w_k)w_0\odot\cdots\odot \hat{w}_k\odot \cdots\odot w_n
$$
and $a(v)\Omega=0$, then we have
\begin{equation}\label{vernerzadjungiert}
(a^*(v)\omega,\eta) = (\omega,a(v)\eta)  
\end{equation}
for all $\omega\in\bigodot^n_\alg H$ and $\eta\in \bigodot^{n+1}_\alg H$.
The operator norm of $a(v)$ is easily determined:
\begin{eqnarray*}
\|a(v)\| 
&=&
\sup_{\eta\in\bigodot^{n+1}_\alg H\atop \|\eta\|=1} \|a(v)\eta\|
\quad=\quad
\sup_{{\eta\in\bigodot^{n+1}_\alg H\atop\omega\in\bigodot^n_\alg H}
  \atop\|\eta\|=\|\omega\|=1} (\omega,a(v)\eta)\\ 
&=&
\sup_{{\eta\in\bigodot^{n+1}_\alg H\atop\omega\in\bigodot^n_\alg H}
  \atop\|\eta\|=\|\omega\|=1} (a^*(v)\omega,\eta)
\quad=\quad
\sup_{\omega\in\bigodot^n_\alg H\atop \|\omega\|=1} \|a^*(v)\omega\|\\
&=&
\|a^*(v)\|
\quad=\quad
\sqrt{n+1}\,\|v\| .
\end{eqnarray*}
Thus $a(v)$ extends continuously to a linear operator $a(v) : \bigodot^{n+1} H
\to\bigodot^n H$ with 
\begin{equation}
\|a(v)\|=\sqrt{n+1}\,\|v\| .
\label{vernichternorm}
\end{equation}

We consider both $a^*(v)$ and $a(v)$ as unbounded linear operators in Fock
space $\FF(H)$ with $\FF_\alg(H)$ as invariant domain of definition.
In the physics literature, $a(v)$ is known as \defem{annihilation
  operator}\indexn{annihilation operator>defemidx} and 
$a^*(v)$ as \defem{creation operator}\indexn{creation operator>defemidx} for $v\in H$.

\begin{lemma}\label{CCR4vernerz}
Let $H$ be a complex Hilbert space and let $v,w\in H$.
Then the canonical commutator relations (CCR)\indexn{canonical commutator
  relations} hold, i.~e., 
$$
[a(v),a(w)] = [a^*(v),a^*(w)] = 0,
$$
$$
[a(v),a^*(w)] = (v,w)\id.
$$
\end{lemma}

\begin{proof}
From 
\begin{eqnarray*}
a^*(v)a^*(w)v_1\odot\cdots\odot v_n 
&=&
v\odot w\odot v_1\odot\cdots\odot v_n\\
&=&
w\odot v\odot v_1\odot\cdots\odot v_n\\
&=&
a^*(w)a^*(v)v_1\odot\cdots\odot v_n 
\end{eqnarray*}
we see directly that $[a^*(v),a^*(w)] = 0$.
By (\ref{vernerzadjungiert}) we have for all $\omega,\eta\in
\FF_\alg(H)$ 
$$
(\omega,[a(v),a(w)]\eta) =
([a^*(w),a^*(v)]\omega,\eta) = 0.
$$
Since $\FF_\alg(H)$ is dense in $\FF(H)$ this
implies $[a(v),a(w)]\eta=0$.
Finally, subtracting
\begin{eqnarray*}
a(v)a^*(w)v_1\odot\cdots\odot v_n 
&=&
a(v)w\odot v_1\odot\cdots\odot v_n\\
&=&
(v,w)v_1\odot\cdots\odot v_n\\ 
&&+ \sum_{k=1}^n (v,v_k) w \odot v_1 \odot\cdots\odot 
\hat{v}_k\odot\cdots\odot v_n
\end{eqnarray*}
and
\begin{eqnarray*}
a^*(w)a(v)v_1\odot\cdots\odot v_n 
&=&
w\odot \sum_{k=1}^n (v,v_k) v_1 \odot\cdots\odot
\hat{v}_k\odot\cdots\odot v_n\\
&=&
\sum_{k=1}^n (v,v_k)w\odot v_1 \odot\cdots\odot
\hat{v}_k\odot\cdots\odot v_n
\end{eqnarray*}
yields $[a(v),a^*(w)](v_1\odot\cdots\odot v_n) = (v,w)v_1\odot\cdots\odot v_n$.
\end{proof}

\Definition{
Let $H$ be a complex Hilbert space and let $v\in H$.
We define the \defidx{ Segal field} as the unbounded operator 
$$
\theta(v) := \frac{1}{\sqrt{2}} (a(v) + a^*(v))
\indexs{theta@$\theta$, Segal field}
$$
in $\FF(H)$ with $\FF_\alg(H)$ as domain of definition.
}

Notice that $a^*(v)$ depends $\Co$-linearly on $v$ while $v \mapsto a(v)$ is
anti-linear.
Hence $\theta(v)$ is only $\R$-linear in $v$.

\begin{lemma}
Let $H$ be a complex Hilbert space and let $v\in H$.
Then the Segal operator $\theta(v)$ is essentially selfadjoint.
\end{lemma}
 
\begin{proof}
Since $\theta(v)$ is symmetric by (\ref{vernerzadjungiert}) and densely defined
it is closable.
The domain of definition $\FF_\alg(H)$ is
invariant for $\theta(v)$ and hence all powers $\theta(v)^m$ are defined on it.
By Nelson's theorem is suffices to show that all vectors in
$\FF_\alg(H)$ are analytic, see Theorem~\ref{thm:nelson}.
Since all vectors in the domain of definition are finite linear combinations
of vectors in $\bigodot^n H$ for various $n$ we only need to show that 
$\omega\in\bigodot^n H$ is analytic.
By (\ref{erzeugernorm}), (\ref{vernichternorm}), and the fact that
$a^*(v)\omega \in\bigodot^{n+1}H$ and $a(v)\omega \in\bigodot^{n-1}H$ are
perpendicular we have
\begin{eqnarray*}
\|\theta(v)\omega\|^2
&=&
\frac12 \|a^*(v)\omega + a(v)\omega\|^2\\
&=&
\frac12 (\|a^*(v)\omega\|^2+\|a(v)\omega\|^2 )\\
&\leq&
\frac12 ((n+1)\|v\|^2\|\omega\|^2+n\|v\|^2\|\omega\|^2)\\
&\leq&
(n+1)\|v\|^2\|\omega\|^2,
\end{eqnarray*}
hence
\begin{equation}
\|\theta(v)^m\omega\| \leq \sqrt{(n+1)(n+2)\cdots(n+m)}\|v\|^m\|\omega\|.
\label{WAbsch}
\end{equation}
For any $t>0$
\begin{eqnarray*}
\sum_{m=0}^\infty \frac{t^m}{m!}\|\theta(v)^m\omega\|
&\leq&
\sum_{m=0}^\infty
\frac{t^m}{m!}\sqrt{(n+1)(n+2)\cdots(n+m)}\,\|v\|^m\|\omega\|\\ 
&=&
\sum_{m=0}^\infty
\frac{t^m}{\sqrt{m!}}\sqrt{\frac{n+1}{1}\cdot\frac{n+2}{2}\cdots\frac{n+m}{m}}
\,\|v\|^m\|\omega\|\\
&\leq&
\sum_{m=0}^\infty \frac{t^m}{\sqrt{m!}}\sqrt{n+1}^m\,\|v\|^m\|\omega\|\\
&<& 
\infty
\end{eqnarray*}
because the power series $\sum_{m=0}^\infty \frac{x^m}{\sqrt{m!}}$ has
infinite radius of convergence.
Thus $\omega$ is an analytic vector.
\end{proof}

\begin{lemma}\label{thetakommutator}
Let $H$ be a complex Hilbert space and let $v,v_j,w\in H$, $j=1,2,\ldots$
Let $\eta\in \FF_\alg(H)$.
Then the following holds:
\begin{enumerate}
\item \label{TK:commutator}
$(\theta(v)\theta(w)-\theta(w)\theta(v))\eta = i\Im(v,w)\eta.$
\item \label{TK:stetig}
If $\|v-v_j\|\to 0$, then $\|\theta(v)\eta-\theta(v_j)\eta\|\to 0$ as
$j\to\infty$.
\item \label{TK:dicht}
The linear span of the vectors $\theta(v_1)\cdots\theta(v_n)\Omega$ where
$v_j\in H$ and $n\in\N$ is dense in $\FF(H)$.
\end{enumerate}
\end{lemma}

\begin{proof}
We see (\ref{TK:commutator}) by Lemma~\ref{CCR4vernerz}
\begin{eqnarray*}
\theta(v)\theta(w)\eta
&=&
\frac12 (a^*(v)+a(v))(a^*(w)+a(w))\eta\\
&=&
\frac12 (a^*(v)a^*(w)+a^*(v)a(w)+a(v)a^*(w)+a(v)a(w))\eta\\
&=&
\frac12 (a^*(w)a^*(v)+a(w)a^*(v)-(w,v)\id+a(w)a^*(w)+(v,w)\id\\
&&+a(w)a(v))\eta\\
&=&
\theta(w)\theta(v)\eta + i\Im(v,w)\eta .
\end{eqnarray*}

For (\ref{TK:stetig}) it suffices to prove the statement for
$\eta\in\bigodot^n H$. 
By (\ref{erzeugernorm}) and (\ref{vernichternorm}) we see
\begin{eqnarray*}
\|\theta(v)\eta-\theta(v_j)\eta\|
&=&
2^{-1/2}\|a^*(v)\eta+a(v)\eta-a^*(v_j)\eta-a(v_j)\eta\|\\
&\leq&
2^{-1/2}(\|a^*(v)\eta-a^*(v_j)\eta\|+\|a(v)\eta-a(v_j)\eta\|)\\
&=&
2^{-1/2}(\|a^*(v-v_j)\eta\|+\|a(v-v_j)\eta\|)\\
&\leq&
2^{-1/2}(\sqrt{n+1}+\sqrt{n})\|v-v_j\|\|\eta\|
\end{eqnarray*}
which implies the statement.

For (\ref{TK:dicht})
one can easily see by induction on $N$ that the span of the vectors 
$\theta(v_1)\cdots\theta(v_n)\Omega$, $n\leq N$, and the span of the vectors
$a^*(v_1)\cdots a^*(v_n)\Omega$, $n\leq N$, both coincide with
$\bigoplus_{n=0}^N \bigodot_\alg^n H$.
This is dense in $\bigoplus_{n=0}^N \bigodot^n H$ and the assertion follows.
\end{proof}

Now we can relate this discussion to the canonical commutator
relations\indexn{canonical commutator relations} as 
studied in Section~\ref{sec-CCR}.
Denote the (selfadjoint) closure of $\theta(v)$ again by $\theta(v)$ and 
denote the domain of this closure by $\dom(\theta(v))$.
Look at the unitary operator $W(v) := \exp(i\theta(v))$.
Recall that for the analytic vectors $\omega\in \FF_\alg(H)$ we have the series
$$
W(v)\omega = \sum_{m=0}^\infty \frac{i^m}{m!} \theta(v)^m\omega
$$
converging absolutely.

\begin{prop}\label{prop:WT}
For $v,v_j,w \in H$ we have
\begin{enumerate}
\item \label{WT:commutator}
The domain $\dom(\theta(w))$ is preserved by $W(v)$, i.~e.,
$W(v)(\dom(\theta(w)))=\dom(\theta(w))$ and 
$$
W(v)\theta(w)\omega = \theta(w) W(v)\omega - \Im(v,w)W(v)\omega
$$
for all $\omega\in\dom(\theta(w))$.
\item \label{WT:weyl}
The map $W:H \to \LL(\FF(H))$ is a Weyl system\indexn{Weyl system} of the
symplectic vector space 
$(H,\Im(\cdot,\cdot))$\indexs{Im@$\protect\Im$, imaginary part}.
\item \label{WT:stetig}
If $\|v-v_j\|\to0$, then $\|(W(v)-W(v_j))\eta\| \to 0$
for all $\eta\in \FF(H)$.
\end{enumerate}
\end{prop}

\begin{proof}
We first check the formula (\ref{WT:commutator}) for $\omega\in \FF_\alg(H)$.
From Lemma~\ref{thetakommutator}~(\ref{TK:commutator}) we get inductively
$$
\theta(v)^m\theta(w)\omega = 
\theta(w)\theta(v)^m\omega + i\cdot m\cdot\Im(v,w)\,\theta(v)^{m-1}\omega.
$$
Since $\theta(w)\omega\in \FF_\alg(H)$ we have
\begin{eqnarray}
\theta(w)W(v)\omega 
&=&
\sum_{m=0}^\infty \frac{i^m}{m!}\theta(w)\theta(v)^m\omega\nonumber\\
&=&
\sum_{m=0}^\infty \frac{i^m}{m!}\left(\theta(v)^m\theta(w) -
im\,\,\Im(v,w)\,\theta(v)^{m-1}\right)\omega\nonumber\\
&=&
W(v)\theta(w)\omega - \sum_{m=1}^\infty
\frac{i^{m+1}\Im(v,w)}{(m-1)!}\,\theta(v)^{m-1}\omega\nonumber\\
&=&
W(v)\theta(w)\omega + \Im(v,w) W(v)\omega.
\label{PhiWKommutator}
\end{eqnarray}
In particular, $W(v)\omega$ is an analytic vector for $\theta(w)$ and we have
for $\omega,\eta\in \FF_\alg(H)$
\begin{eqnarray}
\|\theta(w)W(v)(\omega-\eta)\| 
&=&
\|(W(v)\theta(w) + \Im(v,w) W(v))(\omega-\eta)\|\nonumber\\
&\leq&
\|W(v)\theta(w)(\omega-\eta)\| + |\Im(v,w)| \|W(v)(\omega-\eta)\|\nonumber\\
&\leq&
\left(\|\theta(w)(\omega-\eta)\| +
|\Im(v,w)|\cdot\|\omega-\eta\|\right)\cdot\|v\|. 
\label{domainabsch}
\end{eqnarray}
Now let $\omega\in\dom(\theta(w))$.
Then there exist $\omega_j\in \FF_\alg(H)$ such
that $\|\omega-\omega_j\| \to 0$ and
$\|\theta(w)(\omega)-\theta(w)(\omega_j)\|\to 
0$ as $j\to \infty$.
Since $W(v)$ is bounded we have $\|W(v)(\omega)-W(v)(\omega_j)\|\to 0$
and by (\ref{domainabsch}) $\{\theta(w)W(v)(\omega_j)\}_j$ is a Cauchy sequence
and therefore convergent as well.
Hence $W(v)(\omega)\in\dom(\theta(w))$ and the validity of
(\ref{PhiWKommutator}) extends to all $\omega\in\dom(\theta(w))$.

We have shown $W(v)(\dom(\theta(w)))\subset\dom(\theta(w))$.
Replacing $W(v)$ by $W(-v)=W(v)^{-1}$ yields
$W(v)(\dom(\theta(w)))=\dom(\theta(w))$.

For (\ref{WT:weyl}) observe
 $W(0) = \exp(0) = \id$ and $W(-v)=\exp(i\theta(-v))=\exp(-i\theta(v))
=\exp(i\theta(v))^*=W(v)^*$.
We fix $\omega\in \FF_\alg(H)$ and look at the
smooth curve
$$
x(t) := W(tv)W(tw)W(-t(v+w))\omega
$$
in $\FF(H)$.
We have $x(0)=\omega$ and 
\begin{eqnarray*}
\dot{x}(t) &=&
\frac{d}{dt} W(tv)W(tw)W(-t(v+w))\omega\\
&=&
iW(tv)\theta(v)W(tw)W(-t(v+w))\omega
+ iW(tv)W(tw)\theta(w)W(-t(v+w))\omega\\
&&
+ iW(tv)W(tw)\theta(-v-w)W(-t(v+w))\omega\\
&=&
iW(tv)\theta(v)W(tw)W(-t(v+w))\omega
+ iW(tv)W(tw)\theta(-v)W(-t(v+w))\omega\\
&\stackrel{(\ref{WT:commutator})}{=}&
iW(tv)W(tw)(\theta(v)+\Im(tw,v))W(-t(v+w))\omega\\
&&
+ iW(tv)W(tw)\theta(-v)W(-t(v+w))\omega\\
&=&
i\cdot t\cdot\Im(w,v)\cdot x(t) .
\end{eqnarray*}
Thus $x(t)=e^{i\Im(w,v)t^2/2}\omega$ and $t=1$ yields 
$W(v)W(w)W(-(v+w))\omega = e^{i\Im(w,v)/2}\omega$.
By continuity this equation extends to all $\omega\in \FF(H)$ and shows that 
$W$ is a Weyl system for the symplectic form $\Im(\cdot,\cdot)$.

For (\ref{WT:stetig})
let $\eta\in\bigodot^n H$ and let $\|v-v_j\|\to 0$ as $j\to \infty$.
Then
\begin{eqnarray*}
\|(W(v)-W(v_j))\eta\|
&=&
\|W(v)(\id-W(-v)W(v_j))\eta\|\\
&\leq&
\|(\id-W(-v)W(v_j))\eta\|\\
&\leq&
\|(1-e^{i\Im(v,v_j)/2})\eta\| + \|(e^{i\Im(v,v_j)/2}-W(-v)W(v_j))\eta\|\\
&\stackrel{(\ref{WT:weyl})}{=}&
\|(1-e^{i\Im(v,v_j)/2})\eta\| 
+ \|(e^{i\Im(v,v_j)/2}-e^{i\Im(v,v_j)/2}W(v_j-v))\eta\|\\
&\leq&
|1-e^{i\Im(v,v_j)/2}|\cdot\|\eta\| + \|(\id-W(v_j-v))\eta\| .
\end{eqnarray*}
Since $\Im(v,v_j) = \frac12 \Im(v-v_j,v+v_j) \to 0$ it suffices to show
$\|(\id-W(v_j-v))\eta\|\to 0$.
This follows from 
\begin{eqnarray*}
\sum_{m=1}^\infty \frac{1}{m!}\|\theta(v_j-v)^m\eta\|
&\stackrel{(\ref{WAbsch})}{\leq}&
\sum_{m=1}^\infty \sqrt{\frac{(n+1)^m}{m!}}\|v_j-v\|^m\,\|\eta\|\\
&=&
\|v_j-v\|\sum_{m=0}^\infty \sqrt{\frac{(n+1)^{m+1}}{(m+1)!}}\|v_j-v\|^m\,
\|\eta\|
\end{eqnarray*}
and $\sum_{m=0}^\infty \sqrt{\frac{(n+1)^{m+1}}{(m+1)!}}\|v_j-v\|^m<\infty$
uniformly in $j$.
We have seen that $\|(W(v)-W(v_j))\eta\| \to 0$ for all
$\eta\in\bigodot^n H$ ($n$ fixed) hence for all
$\eta\in \FF_\alg(H)$. 

Finally, let $\eta\in \FF(H)$ be arbitrary.
Let $\epsilon>0$.
Choose $\eta'\in \FF_\alg(H)$ such that
$\|\eta-\eta'\|<\epsilon$.
For $j\gg 0$ we have $\|(W(v)-W(v_j))\eta'\|<\epsilon$.
Hence
\begin{eqnarray*}
\|(W(v)-W(v_j))\eta\|
&\leq&
\|(W(v)-W(v_j))(\eta-\eta')\| + \|(W(v)-W(v_j))\eta'\|\\
&\leq&
2\|\eta-\eta'\| + \epsilon\\
&<&
3\epsilon.
\end{eqnarray*}
This concludes the proof.
\end{proof}

%%%%%%%%%%%%%%%%%%%%%%%%%%%%%%%%%%%%%%%%%%%%%%%%%%%%%%%%%%%%%%%%%%%%%%%%%
\section{The quantum field defined by a Cauchy hypersurface}
\indexn{quantum field} 
%%%%%%%%%%%%%%%%%%%%%%%%%%%%%%%%%%%%%%%%%%%%%%%%%%%%%%%%%%%%%%%%%%%%%%%%% 

In this final section we construct the quantum field.
This yields a formulation of the quantized theory on Fock space which is
closer to the traditional presentations of quantum field theory than the
formulation in terms of quasi-local $C^*$-algebras given in
Sections~\ref{sec:qlC*} and \ref{sec:Haag-Kastler}. 
It has the disadvantage however of depending on a choice of Cauchy
hypersurface.
Even worse from a physical point of view, this quantum field has all the
properties that one usually requires except for one, the ``microlocal spectrum
condition''. 
We do not discuss this condition in the present book, see the remarks at the
end of this section and the references mentioned therein.
The construction given here is nevertheless useful because it illustrates how
the abstract algebraic formulation of quantum field theory relates to more
traditional ones.

Let $(M,E,P)$ be an object in the category $\GLOBHY$, i.~e.,
$M$ is a globally hyperbolic Lorentzian manifold, $E$ is a real vector bundle
over $M$ with nondegenerate inner product, and $P$ is a formally selfadjoint
normally hyperbolic operator acting on sections in $E$.
We need an additional piece of structure.

\Definition{
Let $k\in\N$.
A \defidx{twist structure of spin $k/2$} on $E$ is a smooth section $Q\in
C^\infty(M,\Hom(\bigodot^kTM,\End(E)))$\indexs{Q@$Q$, twist structure} with
the following properties: 
\begin{enumerate}
%% \item \label{twist:sym1}
%% $Q$ is totally symmetric in the vector arguments, i.~e.,
%% $$
%% Q(X_1,\ldots,X_k)e=Q(X_{\sigma(1)},\ldots,X_{\sigma(k)})e
%% $$ 
%% for all $X_j\in T_pM$, $e\in E_p$, $p\in M$, and all permutations $\sigma$
  %% on 
%% $k$ elements. 
\item \label{twist:sym2}
$Q$ is symmetric with respect to the inner product on $E$, i.~e.,
$$
\la Q(X_1\odot\cdots\odot X_k)e,f\ra = \la e,Q(X_1\odot\cdots\odot X_k)f\ra
$$ for all $X_j\in T_pM$, $e,f\in E_p$, and $p\in M$.
\item \label{twist:posdef}
If $X$ is future directed timelike, then the bilinear form
$\la\cdot,\cdot\ra_X$ defined by
$$
\la f,g\ra_X := \la Q(X\odot\cdots\odot X)f,g\ra
$$
is positive definite.
\end{enumerate}
}

Note that the bilinear form $\la\cdot,\cdot\ra_X$ is symmetric by
(\ref{twist:sym2}) so that (\ref{twist:posdef}) makes sense.
From now on we write $Q_X:=Q(X\odot\cdots\odot X)$ for brevity.
If $X$ is past directed timelike, then $\la\cdot,\cdot\ra_X$ is positive or
negative definite depending on the parity of $k$.
Note furthermore, that $Q_X$ is a field of {\em isomorphisms} of $E$ in case
that $X$ is timelike since otherwise $\la\cdot,\cdot\ra_X$ would be
degenerate.

\Examples{
a)
Let $E$ be a real vector bundle over $M$ with Riemannian metric.
In the case of the d'Alembert, the Klein-Gordon, and the Yamabe operator we
are in this situation.\indexn{d'Alembert operator}\indexn{Klein-Gordon
  operator}\indexn{Yamabe operator}
We take $k=0$ and $Q:\bigodot^0TM=\R \to \End(E)$, $t\mapsto t\cdot\id$.
By convention, $Q_X=\id$ and hence $\la\cdot,\cdot\ra_X=\la\cdot,\cdot\ra$ for
a timelike vector $X$ of unit length. 

b)
Let $M$ carry a spin structure and let $E=\Sigma M$ be the spinor bundle.
The Dirac operator and its square act on sections in $\Sigma M$.\indexn{Dirac
  operator}\indexn{spinor bundle}
As explained in \cite[Sec.~3.3]{Baum} and \cite[Sec.~2]{BGM} there is a
natural indefinite Hermitian product $(\cdot,\cdot)$ on $\Sigma M$ such that 
for future directed timelike $X$ the sesquilinear form $(\cdot,\cdot)_X$
defined by 
$$
(\phi,\psi)_X = (\phi,X\cdot\psi)
$$
is symmetric and positive definite where ``$\cdot$'' denotes Clifford
multiplication.
Hence if we view $\Sigma M$ as a real bundle and put $\la\cdot,\cdot\ra :=
\Re(\cdot,\cdot)$, $k:=1$, and $Q(X)\phi:=X\cdot\phi$, then we have a twist
structure of spin $1/2$ on the spinor bundle.

c)
On the bundle of $p$-forms $E=\Lambda^pT^*M$ there is a natural indefinite
inner product $\la\cdot,\cdot\ra$ characterized by 
$$
\la \alpha,\beta\ra = \sum_{0\leq i_1 < \cdots < i_p \leq n}
\epsilon_{i_1}\cdots\epsilon_{i_p} \cdot
\alpha(e_{i_1},\ldots,e_{i_p}) \cdot\beta(e_{i_1},\ldots,e_{i_p})
$$
where $e_1,\ldots,e_n$ is an orthonormal basis with $\epsilon_i=\la
e_i,e_i\ra=\pm 1$.
We put $k:=2$ and 
$$
Q(X\odot Y)\alpha := X^\flat\wedge \iota_Y\alpha + Y^\flat\wedge\iota_X\alpha
-  
\la X,Y\ra\cdot\alpha
$$
where $\iota_X$ denotes insertion of $X$ in the first argument, $\iota_X\alpha
= \alpha(X,\cdot,\ldots,\cdot)$, and $X\mapsto X^\flat$ is the natural
isomorphism $TM\to T^*M$ induced by the Lorentzian metric.
It is easy to check that $Q$ is a twist structure of spin $1$ on
$\Lambda^pT^*M$.
Recall that the case $p=1$ is relevant for the wave equation in
electrodynamics and for the Proca equation.
\indexn{electrodynamics}\indexn{Proca equation}

The physically oriented reader will have noticed that in all these examples
$k/2$ indeed coincides with the spin of the particle under consideration.
}

\Remark{
If $Q$ is a twist structure on $E$, then $Q^*$ is a twist structure on $E^*$
of the same spin where $Q^*(X_1\odot\cdots\odot X_k) = Q(X_1\odot\cdots\odot
X_k)^*$ is given by the adjoint map.
On $E^*$, we will always use this induced twist structure without further
comment. 
}

Let us return to the construction of the quantum field for the object
$(M,E,P)$ in $\GLOBHY$.
As an additional data we fix a twist structure $Q$ on $E$.
As usual let $E^*$ denote the dual bundle and $P^*$ the adjoint operator
acting on sections in $E^*$.
Let $G^*_\pm$ denote the Green's operators for $P^*$ and $G^*=G^*_+-G^*_-$.

We choose a spacelike smooth Cauchy hypersurface $\Sigma\subset M$.
We denote by $L^2(\Sigma,E^*)$\indexs{L2SigmaER@$L^2(\Sigma,E^*)$, Hilbert
  space of square integrable sections in $E^*$ over $\Sigma$}the real Hilbert
space of square 
integrable sections in $E^*$ over $\Sigma$ with scalar product 
\indexs{HermscalprodSigma@$(u,v)_\Sigma:={\protect\int_\Sigma \protect\la
    Q_\mathfrak{n} u,v \protect\ra \protect\dA}$} 
$$
(u,v)_\Sigma
:= \int_\Sigma \la u,v \ra_\mathfrak{n} \dA = \int_\Sigma \la Q^*_\mathfrak{n}
u,v \ra \dA 
$$
where $\mathfrak{n}$ denotes the future directed (timelike) unit normal to
$\Sigma$. 
Here $\la\cdot,\cdot\ra$ denotes the inner product on $E^*$ inherited
from the one on $E$.
Let
$H_\Sigma:=L^2(\Sigma,E^*)\otimes_\R\Co$
\indexs{HSigma@$H_\Sigma:=L^2(\Sigma,E^*)\otimes_\R\Co$}  
be the complexification of this 
real Hilbert space and extend $(\cdot,\cdot)_\Sigma$ to a Hermitian scalar
product on $H_\Sigma$ thus turning $H_\Sigma$ into a complex Hilbert space.
We use the convention that $(\cdot,\cdot)_\Sigma$ is conjugate linear in the
first argument.

We construct the symmetric Fock space $\FF(H_\Sigma)$\indexn{Fock space} as in
the previous section. 
Let $\theta$ be the corresponding Segal field.

Given $f\in\DD(M,E^*)$ the smooth section $G^*f$ is contained in
$C_{\mathrm{sc}}^\infty(M,E^*)$, i.~e., there exists a compact subset $K\subset
M$ such that $\supp(G^*f) \subset J^M(K)$, see Theorem~\ref{thmExSeq}.
It thus follows from Corollary~\ref{cJ+Spastcompact} that the intersection
$\supp(G^*f) \cap \Sigma$ is compact 
and $G^*f|_{\Sigma}\in \DD(\Sigma,E^*)\subset L^2(\Sigma,E^*)\subset
H_\Sigma$.
Similarly, $\nabla_\mathfrak{n}(G^*f)\in \DD(\Sigma,E^*)\subset
L^2(\Sigma,E^*)\subset H_\Sigma$.
We can therefore define
$$
\Phi_\Sigma(f) := \theta(i(G^*f)|_{\Sigma} -
(Q^*_\mathfrak{n})^{-1}\nabla_\mathfrak{n}(G^*f)). 
$$\indexs{PhiSigma@$\Phi_\Sigma$, quantum field defined by $\Sigma$}

\Definition{
The map $\Phi_\Sigma$ from $\DD(M,E^*)$ to the set of selfadjoint operators on
Fock 
space $\FF(H_\Sigma)$ is called the \defidx{quantum field} (or the
\defidx{field operator}) for $P$ defined by $\Sigma$.
}

Notice that $\Phi_\Sigma$ depends upon the choice of the Cauchy hypersurface
$\Sigma$.
One thinks of $\Phi_\Sigma$ as an operator-valued distribution on $M$.
This can be made more precise.

\begin{prop}\label{Phistetig}
Let $(M,E,P)$ be an object in the category $\GLOBHY$ with a twist structure
$Q$. 
Choose a spacelike smooth Cauchy hypersurface $\Sigma\subset M$.
Let $\Phi_\Sigma$ be the quantum field for $P$ defined by $\Sigma$.

Then for every $\omega\in \FF_\alg(H_\Sigma)$ the map
$$
\DD(M,E^*) \to \FF(H_\Sigma), \quad\quad
f\mapsto \Phi_\Sigma(f)\omega,
$$
is continuous.
In particular, the map
$$
\DD(M,E^*) \to \Co, \quad\quad
f\mapsto (\eta,\Phi_\Sigma(f)\omega),
$$
is a distributional section in $E$ for any $\eta,\omega\in \FF_\alg(H_\Sigma)$.
\end{prop}

\begin{proof}
Let $f_j \to f$ in $\DD(M,E^*)$.
Then $G^*f_j \to G^*f$ in $\Csc(M,E^*)$ by Proposition~\ref{PGstetig}.
Thus $G^*f_j|_{\Sigma} \to G^*f|_{\Sigma}$ and
$(Q^*_\mathfrak{n})^{-1}\nabla_\mathfrak{n} G^*f_j \to 
(Q^*_\mathfrak{n})^{-1}\nabla_\mathfrak{n} G^*f$ in $\DD(\Sigma,E^*)$.
Hence $G^*f_j|_{\Sigma} \to G^*f|_{\Sigma}$ and
$(Q^*_\mathfrak{n})^{-1}\nabla_\mathfrak{n} G^*f_j 
\to (Q^*_\mathfrak{n})^{-1}\nabla_\mathfrak{n} G^*f$ in $H_\Sigma$.
The proposition now follows from Lemma~\ref{thetakommutator}~(\ref{TK:stetig}).
\end{proof}

The quantum field satisfies the equation $P\Phi_\Sigma=0$ in the distributional
sense.
More precisely, we have 

\begin{prop}
Let $(M,E,P)$ be an object in the category $\GLOBHY$ with a twist structure
$Q$. 
Choose a spacelike smooth Cauchy hypersurface $\Sigma\subset M$.
Let $\Phi_\Sigma$ be the quantum field for $P$ defined by $\Sigma$.

For every $f\in\DD(M,E^*)$ one has
$$
\Phi_\Sigma(P^*f) = 0.
$$
\end{prop}

\begin{proof}
This is clear from $G^*P^*f=0$ and $\theta(0)=0$.
\end{proof}

To proceed we need the following reformulation of Lemma~\ref{greenformel}.

\begin{lemma}\label{greenformel2}
Let $(M,E,P)$ be an object in the category $\GLOBHY$, let $G_\pm$ be the
Green's operators\indexn{Green's operator} for $P$ and let $G=G_+-G_-$.
Furthermore, let $\Sigma\subset M$ be a spacelike Cauchy hypersurface with
future directed (timelike) unit normal vector field $\mathfrak{n}$. 

Then for all $f,g\in\mathcal{D}(M,E)$,
\[
\int_M\la f,Gg\ra\dV
=
\int_\Sigma\left(\la\nabla_\mathfrak{n}(Gf),Gg\ra-\la
  Gf,\nabla_\mathfrak{n}(Gg)\ra\right)\dA .
\] 
\end{lemma}

\begin{proof}
Since $J_+^M(\Sigma)$ is past compact and $J_-^M(\Sigma)$ is future
compact Lemma~\ref{greenformel} applies.
After identification of $E^*$ with $E$ via the inner product
$\la\cdot\,,\cdot\ra$ the assertion follows from Lemma~\ref{greenformel} with
$u=Gg$.
\end{proof}

The quantum field satisfies the following commutator relation.

\begin{prop} \label{prop:Phi-comm}
Let $(M,E,P)$ be an object in the category $\GLOBHY$ with a twist structure $Q$.
Choose a spacelike smooth Cauchy hypersurface $\Sigma\subset M$.
Let $\Phi_\Sigma$ be the quantum field for $P$ defined by $\Sigma$.

Then for all $f,g\in\DD(M,E^*)$ and all $\eta\in \FF_\alg(H_\Sigma)$ one
has 
$$
[\Phi_\Sigma(f),\Phi_\Sigma(g)]\eta 
= i\cdot\int_M\la G^*f,g\ra\dV\,\cdot\eta.
$$
\end{prop}

\begin{proof}
Using Lemma~\ref{thetakommutator} and the fact that $(\cdot,\cdot)_\Sigma$ is
the complexification of a real scalar product we compute
\begin{eqnarray*}
\lefteqn{[\Phi_\Sigma(f),\Phi_\Sigma(g)]\eta}\\
&=&
\left[\theta\left(i(G^*f)|_{\Sigma} -
  (Q^*_\mathfrak{n})^{-1}\nabla_\mathfrak{n}(G^*f)\right), 
\theta\left(i(G^*g)|_{\Sigma} - (Q^*_\mathfrak{n})^{-1}
\nabla_\mathfrak{n}(G^*g)\right)\right]\eta\\ 
&=&
i\Im\left(i(G^*f)|_{\Sigma}-(Q^*_\mathfrak{n})^{-1}\nabla_\mathfrak{n}(G^*f),
i(G^*g)|_{\Sigma}-(Q^*_\mathfrak{n})^{-1}
\nabla_\mathfrak{n}(G^*g)\right)_\Sigma\eta\\ 
&=&
-i\Im\left(i(G^*f)|_{\Sigma},(Q^*_\mathfrak{n})^{-1}
\nabla_\mathfrak{n}(G^*g)\right)_\Sigma\eta 
-i\Im\left((Q^*_\mathfrak{n})^{-1}
\nabla_\mathfrak{n}(G^*f),i(G^*g)|_{\Sigma}\right)_\Sigma\eta\\ 
&=&
i\left((G^*f)|_{\Sigma},(Q^*_\mathfrak{n})^{-1}
\nabla_\mathfrak{n}(G^*g)\right)_\Sigma\cdot\eta 
-i\left((Q^*_\mathfrak{n})^{-1}\nabla_\mathfrak{n}(G^*f),
(G^*g)|_{\Sigma}\right)_\Sigma\cdot\eta\\  
&=&
i\cdot\int_\Sigma\la(G^*f)|_{\Sigma},\nabla_\mathfrak{n}(G^*g)\ra\dV\cdot\eta
-i\cdot\int_\Sigma\la\nabla_\mathfrak{n}(G^*f),(G^*g)|_{\Sigma}\ra\dV\cdot\eta.
\end{eqnarray*}
Lemma~\ref{greenformel2} applied to $P^*$ concludes the proof.
\end{proof}

\begin{cor}
Let $(M,E,P)$ be an object in the category $\GLOBHY$ with a twist structure
$Q$. 
Choose a spacelike smooth Cauchy hypersurface $\Sigma\subset M$.
Let $\Phi_\Sigma$ be the quantum field for $P$ defined by $\Sigma$.
If the supports of  $f$ and $g\in\DD(M,E^*)$ are causally independent,
then
$$
[\Phi_\Sigma(f),\Phi_\Sigma(g)]=0.
$$
\end{cor}

\begin{proof}
If the supports of  $\supp(f)$ and $\supp(g)$ are causally independent,
then $\supp(G^*f) \subset J^M(\supp(f))$ and $\supp(g)$ are disjoint.
Hence 
$$
[\Phi_\Sigma(f),\Phi_\Sigma(g)]
= i\cdot\int_M\la G^*f,g\ra\dV=0.
$$
\end{proof}

\begin{prop}
Let $(M,E,P)$ be an object in the category $\GLOBHY$ with a twist structure
$Q$. 
Choose a spacelike smooth Cauchy hypersurface $\Sigma\subset M$.
Let $\Phi_\Sigma$ be the quantum field for $P$ defined by $\Sigma$.
Let $\Omega$ be the vacuum vector in $\FF(H_\Sigma)$.

Then the linear span of the vectors
$\Phi_\Sigma(f_1)\cdots\Phi_\Sigma(f_n)\Omega$ is 
dense in $\FF(H_\Sigma)$ where $f_j\in\DD(M,E^*)$ and $n\in\N$.
\end{prop}

\begin{proof}
By Lemma~\ref{thetakommutator}~(\ref{TK:dicht}) the span of vectors of the form
$\theta(v_1)\cdots \theta(v_n)\Omega$, $v_j\in H_\Sigma$, $n\in\N$, is dense
in $\FF(H_\Sigma)$.
It therefore suffices to approximate vectors of the form $\theta(v_1)\cdots
\theta(v_n)\Omega$ by vectors of the form
$\Phi_\Sigma(f_1)\cdots\Phi_\Sigma(f_n)\Omega$.
Any $v_j\in H_\Sigma$ is of the form $v_j= w_j + iz_j$ with $w_j,z_j\in
L^2(\Sigma,E^*)$.
Since $\DD(\Sigma,E^*)$ is dense in $L^2(\Sigma,E^*)$ we may assume
without loss of generality that $w_j,z_j\in\DD(\Sigma,E^*)$ by
Proposition~\ref{Phistetig}. 

By Theorem~\ref{cauchyglobhyp} there exists a solution $u_j\in
\Csc(M,E^*)$ to the Cauchy problem $Pu_j=0$ with initial conditions
$u_j|_\Sigma = z_j$ and $\nabla_\mathfrak{n} u_j = -Q^*_\mathfrak{n} w_j$.
By Theorem~\ref{thmExSeq} there exists $f_j\in\DD(M,E^*)$ with
$G^*f_j=u_j$.
Then $\Phi_\Sigma(f_j) =
\theta(-(Q^*_\mathfrak{n})^{-1}\nabla_\mathfrak{n}(G^*f_j)+i(G^*f_j)|_\Sigma) 
= \theta(-(Q^*_\mathfrak{n})^{-1}\nabla_\mathfrak{n}(u_j)+iu_j|_\Sigma) =
\theta(w_j + iz_j) = 
\theta(v_j)$.
This concludes the proof.
\end{proof}

\Remark{
In the physics literature one usually also finds that the quantum field should
satisfy
\begin{equation}
\Phi_\Sigma(\bar f) = \Phi_\Sigma(f)^*.
\label{eq:QFreal}
\end{equation}
This simply expresses the fact that we are dealing with a real theory and that
the quantum field takes its values in self-adjoint operators.
Recall that we have assumed $E$ to be a {\em real} vector bundle.
Of course, one could complexify $E$ and extend $\Phi_\Sigma$ complex linearly
such that (\ref{eq:QFreal}) holds.
}

We relate the quantum field constructed in this section to the CCR-algebras
studied earlier.

\begin{prop} \label{prop:QF-CCR}
Let $(M,E,P)$ be an object in the category $\GLOBHY$ and let $Q$ be a twist
structure on $E^*$.
Choose a spacelike smooth Cauchy hypersurface $\Sigma\subset M$.
Let $\Phi_\Sigma$ be the quantum field for $P$ defined by $\Sigma$.

Then the map
$$
W_\Sigma:\DD(M,E^*) \to \LL(\FF(H_\Sigma)), \quad
W_\Sigma(f)=\exp(i\,\Phi_\Sigma(f)),
$$
yields a Weyl system\indexn{Weyl system} of the symplectic vector space
$\SYM\circ\solve(M,E^*,P^*)$. 
\end{prop}

\begin{proof}
Recall that the symplectic vector space $\SYM\circ\solve(M,E^*,P^*)$ is given
by $V(M,E^*,G^*)=\DD(M,E^*)/\ker(G^*)$ with symplectic form induced by 
$\tilde\omega(f,g)=\int_M\la G^*f,g\ra\dV$.
By definition $W_\Sigma(f)=1$ holds for any $f\in\ker(G^*)$, hence $W_\Sigma$
descends to a map $V(M,E^*,G^*)\to\LL(\FF(H_\Sigma))$.

Let $f,g\in\DD(M,E^*)$.
Set $u:=i(G^*f)|_{\Sigma} - (Q^*_\mathfrak{n})^{-1}\nabla_\mathfrak{n}(G^*f)$
and 
$v:=i(G^*g)|_{\Sigma} - (Q^*_\mathfrak{n})^{-1}\nabla_\mathfrak{n}(G^*g)\in
H_\Sigma$ so that 
$\Phi_\Sigma(f)=\theta(u)$ and $\Phi_\Sigma(g)=\theta(v)$.
Then by Lemma~\ref{thetakommutator}~(\ref{TK:commutator}) and by
Proposition~\ref{prop:Phi-comm} we have
$$
i\, \Im(u,v)_\Sigma \cdot\id
=
[\theta(u),\theta(v)]
=
[\Phi_\Sigma(f),\Phi_\Sigma(g)]
=
i\, \int_M \la G^*f,g\ra \dV \cdot\id,
$$
hence
$$
\Im(u,v)_\Sigma = \int_M \la G^*f,g\ra \dV = \tilde\omega(f,g).
$$
Now the result follows from Proposition~\ref{prop:WT}~(\ref{WT:weyl}).
\end{proof}

\begin{cor}
Let $(M,E,P)$ be an object in the category $\GLOBHY$ and let $Q$ be a twist
structure on $E^*$.
Choose a spacelike smooth Cauchy hypersurface $\Sigma\subset M$.
Let $W_\Sigma$ be the Weyl system defined by $\Phi_\Sigma$.

Then the CCR-algebra generated by the $W_\Sigma(f)$, $f\in \DD(M,E^*)$, is
isomorphic to $\CCR(\SYM(\solve(M,E^*,P^*)))$.\indexn{canonical commutator
  relations}  
\end{cor} 

\begin{proof}
This is a direct consequence of Proposition~\ref{prop:QF-CCR} and of
Theorem~\ref{CCRunique}. 
\end{proof}

The construction of the quantum field on a globally hyperbolic Lorentzian
manifold goes back to \cite{Ish}, \cite{Haj}, \cite{Dimock1}, and others in the
case of scalar fields, i.~e., if $E$ is the trivial line bundle.
See also the references in \cite{Ful} and \cite{Wa2}.
In \cite{Dimock1} the formula $W_\Sigma(f)=\exp(i\,\Phi_\Sigma(f))$ in
Proposition~\ref{prop:QF-CCR} was used to {\em define} the CCR-algebra.
It should be noted that this way one does not get a true quantization functor
$\GLOBHY \to \CALG$ because one determines the $C^*$-algebra {\em up to
isomorphism only}.
This is caused by the fact that there is no canonical choice of Cauchy
hypersurface.
It seems that the approach based on algebras of observables as developed in
Sections~\ref{sec:quantfunc} to \ref{sec:Haag-Kastler} is more natural in the
context of curved spacetimes than the more traditional approach via the Fock
space. 

Wave equations for sections in nontrivial vector bundles also appear
frequently.
The approach presented in this book works for linear wave equations in general
but often extra problems have to be taken care of.
In \cite{Dimock3} the electromagnetic field\indexn{electromagnetic field>defemidx} is
studied. 
Here one has to take the gauge freedom into account.
For the Proca equation\indexn{Proca equation} as studied e.~g.\ in \cite{Fur}
the extra constraint $\delta A=0$ must be considered, compare
Example~\ref{ex:1forms}. 
If one wants to study the Dirac equation\indexn{Dirac operator} itself rather
than its square as we did in Example~\ref{ex:Diracsquare}, then one has to use
the canonical anticommutator relations (CAR)\indexn{canonical anticommutator
relations>defemidx} instead of the CCR, see e.~g.\ \cite{Dimock2}. 

In the physics papers mentioned above the authors fix a wave
equation, e.~g.\ the Klein-Gordon equation, and then they set up a functor
$\GLOBHYN \to \CALG$.
Here $\GLOBHYN$\indexs{GlobHypnaked@$\protect{\GLOBHYN}$, category of globally
  hyperbolic manifolds without further structure} is the category whose
objects are globally hyperbolic Lorentzian manifolds without any further
structure and the morphisms are the timeorientation preserving isometric
embeddings $f:M_1\to M_2$ such that $f(M_1)$ is a causally compatible open
subset of $M_2$. 
The relation to our more universal functor $\CCR\circ\SYM\circ\solve :
\GLOBHY \to \CALG$ is as follows:

There is the forgetful functor $\forget:\GLOBHY \to \GLOBHYN$ given by
$\forget(M,E,P)=M$ and $\forget(f,F)=f$.
A \defidx{geometric normally hyperbolic operator} is a functor $\mathrm{GOp}:
\GLOBHYN \to \GLOBHY$\indexs{GeometricOperator@$\protect{\mathrm{GOp}}$,
  geometric normally hyperbolic operator} such that
$\forget\circ\mathrm{GOp}=\id$. 

For example, the Klein-Gordon equation for fixed mass $m$ yields such a
functor.
One puts $\mathrm{GOp}(M):=(M,E,P)$ where $E$ is the trivial real line bundle
over $M$ with the canonical inner product and $P$ is the Klein-Gordon
operator\indexn{Klein-Gordon operator} $P=\Box +m^2$.
On the level of morphisms, one sets $\mathrm{GOp}(f):=(f,F)$ where $F$ is the
embedding $M_1\times\R \hookrightarrow M_2\times\R$ induced by $f:M_1
\hookrightarrow M_2$.
Similarly, the Yamabe operator, the wave equations for the electromagnetic
field\indexn{electromagnetic field} and for the Proca field\indexn{Proca
  equation} yield geometric normally hyperbolic operators. 

The square of the Dirac operator\indexn{Dirac operator} does not yield a
geometric normally hyperbolic operator because the construction of the spinor
bundle depends on the additional choice of a spin structure.
One can of course fix this by incorporating the spin structure into yet
another category, the category of globally hyperbolic Lorentzian manifolds
equipped with a spin structure, see \cite[Sec.~3]{Ver}.

In any case, given a geometric normally hyperbolic operator $\mathrm{GOp}$,
then $\CCR\circ\SYM\circ\solve\circ\mathrm{GOp} : \GLOBHYN \to \CALG$ is a
\defidx{locally covariant quantum field theory} in the sense of
\cite[Def.~2.1]{BFV}. 

For introductions to quantum field theory on curved spacetimes from the
physical point of view the reader is referred to the books
\cite{BD}, \cite{Ful}, and \cite{Wa2}.

The passage from the abstract quantization procedure yielding quasi-local
$C^*$-algebras to the more familiar concept based on Fock space and
quantum fields requires certain choices (Cauchy hypersurface) and additional
structures (twist structure) and is therefore not canonical.
Furthermore, there are many more Hilbert space representations than the Fock
space representations constructed here
and the question arises which ones are physically relevant. 
A criterion in terms of micro-local analysis was found in \cite{Rad}.
As a matter of fact, the Fock space representations constructed here turn out
not to satisfy this criterion and are therefore nowadays regarded as
unphysical. 
A good geometric understanding of the physical Hilbert space
representations on a general globally hyperbolic spacetime is still
missing. 

Radzikowski's work was developed further in \cite{BFK} and applied in
\cite{BF} to interacting fields.
The theory of interacting quantum fields, in particular their
renormalizability, currently forms an area of very active research.

%%%%%%%%%%%%%%%%%%%%%%%%%%%%%%%%%%%%%%%%%%%%%%%%%%%%%%%%%%%%%%%%%%%%%%%%%
\begin{appendix}  
\chapter{Background material}
%%%%%%%%%%%%%%%%%%%%%%%%%%%%%%%%%%%%%%%%%%%%%%%%%%%%%%%%%%%%%%%%%%%%%%%%%

In Sections~\ref{app:categories} to \ref{app:diffops} the necessary
terminology and basic facts from such diverse fields of mathematics as
category theory, functional analysis, differential geometry, and differential
operators are presented.
These sections are included for the convenience of the reader and are not
meant to be a substitute for a thorough introduction to these topics.

Section~\ref{app:lorgeo} is of a different nature.
Here we collect advanced material on Lorentzian geometry which is needed in
the main text.
In this section we give full proofs.
Partly due to the technical nature of many of these results they have not been
included in the main text in order not to distract the reader.

%%%%%%%%%%%%%%%%%%%%%%%%%%%%%%%%%%%%%%%%%%%%%%%%%%%%%%%%%%%%%%%%%%%%%%%%%%%%
\section{Categories}
\label{app:categories}
%%%%%%%%%%%%%%%%%%%%%%%%%%%%%%%%%%%%%%%%%%%%%%%%%%%%%%%%%%%%%%%%%%%%%%%%%%%%

We start with basic definitions and examples from category theory (compare
\cite[Ch.~1, \S~11]{Lang}). 
A nice introduction to further concepts related to categories can be found in 
\cite{Mac}.

\Definition{\label{category}
A \defem{category}\indexn{category>defemidx} $\AAA$\indexs{AAA@$\protect\AAA$,
  $\protect\BBB$, categories} consists of the following data: 
\begin{itemize} 
\item 
a class $\Obj(\AAA)$ whose members are called
\defem{objects}\indexs{ObjA@$\protect\Obj(\protect\AAA)$, class of objects in
  a category $\protect\AAA$}\indexn{object in a category>defemidx} 
\item
for any two objects $A,B\in\Obj(\AAA)$ there is a (possibly empty) set
$\Mo(A,B)$\indexs{MorAB@$\protect\Mo(A,B)$, set of morphisms from $A$ to $B$}
whose elements are called \defem{morphisms}\indexn{morphism in a
  category>defemidx}, 
\item
for any three objects $A,B,C\in\Obj(\AAA)$ there is a map (called the
\defem{composition} of morphisms)\indexn{composition of morphisms in a
  category>defemidx}\indexs{compose@$\circ$, composition of morphisms in a
  category} 
\[ 
\Mo(B,C)\times\Mo(A,B)\to\Mo(A,C)\;,\;(f,g)\mapsto f\circ g, 
\] 
\end{itemize}
such that the following axioms are fulfilled:
\begin{enumerate}
\item 
If two pairs of objects $(A,B)$ and $(A',B')$ are not equal, then the
sets $\Mo(A,B)$ and $\Mo(A',B')$ are disjoint. 
\item 
For every $A\in\Obj(\AAA)$ there exists an element
$\id_A\in\Mo(A,A)$\indexs{idA@$\protect\id_A$, identity morphism of
  $A$}\indexn{identity morphism>defemidx} 
(called the \defem{identity morphism} of $A$)
such that for all $B\in\Obj(\AAA)$, for all $f\in\Mo(B,A)$ and all
$g\in\Mo(A,B)$ one has
\[ 
\id_A\circ f=f\quad\mbox{ and }\quad g\circ \id_A=g.
\]
\item 
The law of composition is \defem{associative}\indexn{associativity
  law>defemidx}, i.~e., for any 
$A,B,C,D\in\Obj(\AAA)$ and for any $f\in\Mo(A,B)$, $g\in\Mo(B,C)$,
$h\in\Mo(C,D)$ we have
\[ 
(h\circ g)\circ f=h\circ(g\circ f).
\]
\end{enumerate}
}

\Examples{
a) In the category of sets $\SET$ the class of objects $\Obj(\SET)$ consists
of all sets, and for any two sets $A,B\in\Obj(\SET)$ the set $\Mo(A,B)$
consists of all maps from $A$ to $B$.
Composition is the usual composition of maps.

b) The objects of the category $\TOP$ are the topological spaces, and the
morphisms are the continuous maps.

c) In the category of groups $\GROUP$ one considers the class $\Obj(\GROUP)$
of all groups, and the morphisms are the group homomorphisms. 

d) In $\AB$, the category of abelian groups, $\Obj(\AB)$ is the class of all
abelian groups, and again the morphisms are the group homomorphisms.
}

\Definition{\label{subcategory}
Let $\AAA$ and $\BBB$ be two categories.
Then $\AAA$ is called a \defidx{full subcategory} of $\BBB$ provided
\begin{enumerate}
\item $\Obj(\AAA)\subset\Obj(\BBB)$,
\item for any $A,B\in\Obj(\AAA)$ the set of morphisms of $A$ to $B$ are the
  same in both categories $\AAA$ and $\BBB$,
\item for all $A,B,C\in\Obj(\AAA)$, any $f\in\Mo(A,B)$ and any
  $g\in\Mo(B,C)$ the composites $g\circ f$ coincide in $\AAA$ and $\BBB$,
\item for $A\in\Obj(\AAA)$ the identity morphism $\id_A$ is the same in both
  $\AAA$ and $\BBB$. 
\end{enumerate}
}

\Examples{
a) $\TOP$ is not a full subcategory of $\SET$ because there are non-continuous
maps between topological spaces.

b) $\AB$ is a full subcategory of $\GROUP$.
}

\Definition{\label{functor}
Let $\AAA$ and $\BBB$ be categories.
A \defem{(covariant) functor}\indexn{functor>defemidx}\indexn{covariant
  functor>defemidx} $T$ from $\AAA$ to $\BBB$ consists of a map 
$T:\Obj(\AAA)\to\Obj(\BBB)$ and maps $T:\Mo(A,B)\to\Mo(TA,TB)$ for every
$A,B\in\Obj(\AAA)$ such that
\begin{enumerate}
\item the composition is preserved, i.~e., for all $A,B,C\in\Obj(\AAA)$, for any
  $f\in\Mo(A,B)$ and for any $g\in\Mo(B,C)$ one has 
  \[T(g\circ f)=T(g)\circ T(f), \]
\item $T$ maps identities to identities, i.~e., for any $A\in\Obj(\AAA)$ we
  get 
 \[ T(\id_A)=\id_{TA}.\]
\end{enumerate}
In symbols one writes $T:\AAA\to\BBB$.
}

\Examples{
a) For every category $\AAA$ one has the {\em identity functor} $\Id:\AAA\to\AAA$
which is defined by $\Id (A)=A$ for all $A\in\Obj(\AAA)$ and $\Id(f)=f$ for all
$f\in\Mo(A,B)$ with $A,B\in\Obj(\AAA)$.

b) There is a functor $F:\TOP\to\SET$ which maps each topological space to the
underlying set and $F(g)=g$ for all $A,B\in\Obj(\TOP)$ and all $g\in\Mo(A,B)$.
This functor $F$ is called the \defidx{forgetful functor} because it forgets the
topological structure.

c) Let $\AAA$ be a category. 
We fix an object $C\in\Obj(\AAA)$. 
We define $T:\AAA\to\SET$ by $T(A)=\Mo(C,A)$ for all $A\in\Obj(\AAA)$ and by 
\begin{eqnarray*} 
\Mo(A,B)&\to&\Mo\big(\Mo(C,A),\Mo(C,B) \big),\\
f&\mapsto&\Big(g\mapsto f\circ g\Big),
\end{eqnarray*}
for all $A,B\in\Obj(\AAA)$.
It is easy to check that $T$ is a functor.
}

%%%%%%%%%%%%%%%%%%%%%%%%%%%%%%%%%%%%%%%%%%%%%%%%%%%%%%%%%%%%%%%%%%%%%%%%%%%%%%%

\section{Functional analysis}
\label{app:funcana}

%%%%%%%%%%%%%%%%%%%%%%%%%%%%%%%%%%%%%%%%%%%%%%%%%%%%%%%%%%%%%%%%%%%%%%%%%%%%%%%

In this section we give some background in functional analysis.
More comprehensive expositions can be found e.~g.\ in \cite{RS}, \cite{RS2},
and \cite{Ru}. 

\Definition{\label{banachspace}
A \defem{Banach space}\indexn{Banach space>defemidx} is a real or complex vector space $X$ equipped with a
norm $\|\cdot\|$ such that every Cauchy sequence in $X$ has a limit. 
}

\Examples{ \label{ex:CkBanach}
a)
Consider $X=C^0([0,1])$\indexs{Cup0X@$C^0(X)$, space of continuous functions
  on $X$}, the space of continuous functions on the 
unit interval $[0,1]$.
We pick the \defidx{supremum norm}\indexn{C0norm@$C^0$-norm>defemidx}:
For $f\in C^0([0,1])$ one puts
\[ \|f\|_{C^0([0,1])}:=\sup_{t\in[0,1]}\big|f(t)\big|.\]
With this norm $X$ is Banach space.
In this example the unit interval can be replaced by any compact topological
space. 

b)
More generally, let $k\in\N$ and let $X=C^k([0,1])$\indexs{CupkX@$C^0(X)$,
  space of $k$ times continuously differentiable functions on $X$}, the space
of $k$ times continuously differentiable functions on the unit interval
$[0,1]$. 
The \defem{$C^k$-norm}\indexn{Cknorm@$C^k$-norm} is defined by
\[ 
\|f\|_{C^k([0,1])}:=\max_{\ell=0,\ldots,k} \|f^{(\ell)}\|_{C^0([0,1])} 
\]
where $f^{(\ell)}$ denotes the $\ell^\mathrm{th}$ derivative of $f\in X$.
Then $X=C^k([0,1])$ together with the $C^k$-norm is a Banach space.
}

Now let $H$ be a complex vector space, and let $(\cdot\,,\cdot)$ be a
(positive definite) Hermitian scalar product. 
The scalar product induces a norm on $H$,
\[ 
\|x\|:=\sqrt{(x,x)}\quad\mbox{ for all }\,x\in H.
\]

\Definition{\label{hilbertspace}\indexn{Hilbert space>defemidx}
A complex vector space $H$ endowed with Hermitian scalar product
$(\cdot\,,\cdot)$ is called a \defidx{Hilbert space} if $H$ together with the
norm induced by $(\cdot\,,\cdot)$ forms a Banach space.
}

\Example{
Consider the space of \defem{square integrable functions}\indexn{square
  integrable function>defemidx} on $[0,1]$: 
\[
\mathscr{L}^2([0,1]):=\left\{f:[0,1]\to\Co\;\Big|\;\;f\;\mbox{ measurable and }
\int_0^1\left|f(t)\right|^2 \dt<\infty \right\}.
\]\indexs{l201@$\protect\mathscr{L}^2(X)$, space of square integrable functions on $X$}
On $\mathscr{L}^2([0,1])$ one gets  a natural sesquilinear form
$(\cdot,\cdot)$ by $(f,g):=\int_0^1\overline{f(t)}\cdot g(t)\dt$ for
all $f,g\in\mathscr{L}^2([0,1])$.
Then $\mathscr{N}:=\left\{f\in\mathscr{L}^2([0,1])\,|\,(f,f)=0\right\}$ is a
linear subspace, and one denotes the quotient vector space by 
\[
L^2([0,1]):=\mathscr{L}^2([0,1])\,/\,\mathscr{N}.\indexs{L201@$L^2(X)$, space
  of classes of square integrable functions on $X$} 
\]
The sesquilinear form $(\cdot\,,\cdot)$ induces a Hermitian scalar product on
$L^2([0,1])$.
The Riesz-Fisher theorem\indexn{Riesz-Fisher theorem>defemidx} \cite[Example 2,
  p.~29]{RS} states that $L^2([0,1])$ 
equipped with this Hermitian scalar product is a Hilbert space.
}

\Definition{
A \defidx{semi-norm} on a $\K$-vector space $X$, $\K=\R$ or $\Co$, is a map
$\rho:X\to[0,\infty)$ such that
\begin{enumerate}
\item 
$\rho(x+y)\le\rho(x)+\rho(y)$ for any $x,y\in X$,
\item 
$\rho(\alpha\,x)=|\alpha|\,\rho(x)$ for any $x\in X$ and
$\alpha\in\K$.
\end{enumerate}
A family of semi-norms $\{\rho_i\}_{i\in I}$ is said to \defem{separate
  points}\indexn{separating points (semi-norms)>defemidx} 
if 
\begin{enumerate}
\item[(3)] $\rho_i(x)=0$ for all $i\in I$ implies $x=0$.
\end{enumerate}
}

Given a countable family of seminorms $\{\rho_k\}_{k\in\N}$ separating points
one defines a metric $d$ on $X$ by setting for $x,y\in X$:
\begin{equation}\label{frechetvollst}
 d(x,y):= \sum_{k=0}^\infty\;\frac{1}{2^k}\cdot\max\big(1,\rho_k(x,y) \big) .
\end{equation}

\Definition{\label{frechet}
A \defidx{Fr\'echet space} is a $\K$-vector space $X$ equipped with a
countable family of semi-norms $\{\rho_k\}_{k\in\N}$ separating points such
that the metric $d$ given by (\ref{frechetvollst}) is complete.
The \defem{natural topology}\indexn{topology of a Fr\'echet space>defemidx} of
a Fr\'echet space is the one induced by this metric $d$. 
}

\Example{
Let $C^\infty([0,1])$\indexs{C0999X@$C^\infty(X)$, space of smooth functions
  on $X$} be the space of smooth functions on the interval $[0,1]$. 
A countable family of semi-norms is given by the $C^k$-norms as defined in
Example~\ref{ex:CkBanach}~b). 
In order to prove that this family of (semi-)norms turns $C^\infty([0,1])$
into a Fr\'echet space we will show that $C^\infty([0,1])$ equipped with the
metric $d$ given by (\ref{frechetvollst}) is complete.

Let $(g_n)_n$ be a Cauchy sequence in $C^\infty([0,1])$ with respect to the
metric $d$.
Then for any $k\ge 0$ the sequence $(g_n)_n$ is Cauchy with respect to the
$C^k$-norm. 
Since $C^k([0,1])$ together with the $C^k$-norm is a Banach space there
exists a unique $h_k\in C^k([0,1]$ such that $(g_n)_n$ converges to $h_k$ in
the $C^k$-norm. 
From the estimate $\|\cdot\|_{C^k([0,1])}\le\|\cdot\|_{C^\ell([0,1])}$ for
$k\le\ell$ we conclude that $h_k$ and $h_\ell$ coincide.
Therefore, putting $h:=h_0$ we obtain $h\in C^\infty([0,1])$ and $d(h,g_n)\to
0$ for $n\to\infty$. 
This shows the completeness of $C^\infty([0,1])$.
} 

If one wants to show that linear maps between Fr\'echet spaces are
homeomorphisms, the following theorem is very helpful.

\begin{thm}[Open Mapping Theorem]\indexn{open mapping theorem>defemidx}
Let $X$ and $Y$ be Fr\'echet spaces, and let $f:X\to Y$ be a continuous linear 
surjection. 
Then $f$ is open, i.~e., $f$ is a homeomorphism.
\end{thm}

\begin{proof}
See \cite[Cor.~2.12., p.~48]{Ru} or \cite[Thm.~V.6, p.~132]{RS}
\end{proof}

From now on we fix a Hilbert space $H$.
A continuous linear map $H\to H$ is called \defidx{bounded operator} on $H$.
But many operators occuring in analysis and mathematical physics are not
continuous and not even defined on the whole Hilbert space.
Therefore one introduces the concept of unbounded operators.

\Definition{
Let $\dom(A)\subset H$\indexs{domA@$\protect\dom(A)$, domain of a linear
  operator $A$} be a linear subspace of $H$. 
A linear map $A:\dom(A)\to H$ is called an \defem{unbounded
  operator}\indexn{unbounded operator>defemidx} in $H$ 
with \defem{domain}\indexn{domain of a linear operator>defemidx} $\dom(A)$.
One says that $A$ is \defem{densely defined}\indexn{densely defined
  operator>defemidx} if $\dom(A)$ is a dense subspace of 
$H$. 
}

\Example{
One can represent elements of $L^2(\R)$ by functions. 
The space of smooth functions with compact support $C^\infty_c(\R)$ is 
regarded as a linear subspace, $C^\infty_c(\R)\subset L^2(\R)$.
Then one can consider the differentiation operator $A:=\tfrac{d}{dt}$ as an
unbounded operator in $L^2(\R)$ with domain $\dom(A)=C^\infty_c(\R)$,
and $A$ is densely defined.
}

\Definition{
Let $A$ be an unbounded operator on $H$ with domain $\dom(A)$.
The \defem{graph}\indexs{GammaA@$\Gamma(A)$, graph of $A$}\indexn{graph of an operator>defemidx} of $A$ is the set
\[ \Gamma(A):=\left\{(x,Ax)\,\big|\,x\in \dom(A) \right\}\subset H\times H. \]
The operator $A$ is called a \defem{closed} operator\indexn{closed
  operator>defemidx} if its graph 
$\Gamma(A)$ is a closed subset of $H\times H$.
}

\Definition{
Let $A_1$ and $A_2$ be operators on $H$.
If $\dom(A_1)\supset \dom(A_2)$ and $A_1x=A_2x$ for all $x\in \dom(A_2)$ ,
then $A_1$ is said to be an \defem{extension}\indexn{extension of an
  operator>defemidx} of $A_2$.  
One then writes $A_1\supset A_2$.
}

\Definition{
Let $A$ be an operator on $H$.
An operator $A$ is \defem{closable}\indexn{closable operator>defemidx} if it
possesses a closed extension. 
In this case the closure $\ovl{\Gamma(A)}$ of $\Gamma(A)$ in $H\times H$ is
the graph of an operator called the \defem{closure}\indexn{closure of an
  operator>defemidx} of $A$. 
}

\Definition{
Let $A$ be a densely  defined operator on $H$.
Then we put 
\[
\dom(A^*):=\left\{x\in H\,\big|\,\mbox{there exists a }y\in H\;\mbox{with
}(Az,x)=(z,y)\;\mbox{ for all }z\in \dom(A) \right\} .
\]
For each $x\in \dom(A^*)$ we define $A^*x:=y$ where $y$ is uniquely determined
by the requirement $(Az,x)=(z,y)$ for all $z\in \dom(A)$.
Uniqueness of $y$ follows from $\dom(A)$ being dense in $H$.
We call $A^*$ the \defem{adjoint}\indexs{A*@$A^*$, adjoint of the operator
  $A$}\indexn{adjoint operator>defemidx} of $A$.  
}

\Definition{
A densely defined operator $A$ on $H$ is called
\defem{symmetric}\indexn{symmetric operator>defemidx} if $A^*$ is an 
extension of $A$, i.~e., if $\dom(A)\subset \dom(A^*)$ and $Ax=A^*x$ for all
$x\in \dom(A)$.
The operator $A$ is called \defem{selfadjoint}\indexn{selfadjoint
  operator>defemidx} if $A=A^*$, that is, if $A$ is 
symmetric and $\dom(A)=\dom(A^*)$. 
}
Any symmetric operator is closable with closure $\overline{A}=A^{**}$.

\Definition{
A symmetric operator $A$ is called \defem{essentially
  selfadjoint}\indexn{essentially selfadjoint operator>defemidx} if its 
closure $\overline{A}$ is selfadjoint.
}

We conclude this section by stating a criterion for 
essential selfadjointness of a symmetric operator.

\Definition{
Let $A$ be an operator on a Hilbert space $H$.
Then one calls the set $C^\infty(A):=\bigcap_{n=1}^\infty
\dom(A^n)$\indexs{C0999A@$C^\infty(A)$, set of $C^\infty$-vectors for $A$} the
set of 
\defem{$C^\infty$-vectors}\indexn{C0999v@$C^\infty$-vector>defemidx} for $A$.
A vector $\varphi\in C^\infty(A)$ is called an \defem{analytic
  vector}\indexn{analytic vector>defemidx} for $A$ if  
\[ 
\sum_{n=0}^\infty \frac{\|A^n\varphi \|}{n!}t^n\,<\,\infty
\]
for some $t>0$.
}

\begin{thm}[Nelson's Theorem]\label{thm:nelson}\indexn{Nelson's theorem>defemidx}
Let $A$ be a symmetric operator on a Hilbert space $H$. 
If $\dom(A)$ contains a set of analytic vectors which is dense in $H$, then
$A$ is essentially selfadjoint.
\end{thm}
\begin{proof}
See \cite[Thm.~X.39, p.~202]{RS2}.
\end{proof}

If $A$ is a selfadjoint operator and $f:\R\to\Co$ is a bounded
Borel-measurable function, then one can define the bounded operator
$f(A)$\indexs{fofA@$f(A)$, function $f$ applied to selfadjoint operator $A$} in
a natural manner.
We use this to get the unitary operator $\exp(i\,A)$ in Section~\ref{sec:Fock}.
If $\phi$ is an analytic vector, then
$$
\exp(iA)\phi = \sum_{n=0}^\infty \frac{i^n}{n!}A^n\phi .
$$

%%%%%%%%%%%%%%%%%%%%%%%%%%%%%%%%%%%%%%%%%%%%%%%%%%%%%%%%%%%%%%%%%%%%%%%%%%%%

\section{Differential geometry}
\label{app:diffgeo}

%%%%%%%%%%%%%%%%%%%%%%%%%%%%%%%%%%%%%%%%%%%%%%%%%%%%%%%%%%%%%%%%%%%%%%%%%%%%

In this section we introduce the basic geometrical objects such as manifolds
and vector bundles which are used throughout the text.
A detailed introduction can be found e.~g.\ in \cite{Sp1} or in \cite{Nic}.

\subsection{Differentiable manifolds}\indexn{differentiable manifold>defemidx}

We start with the concept of a manifold.
Loosely speaking, manifolds are spaces which look locally like $\R^n$.

\Definition{
Let $n$ be an integer.
A topological space $M$ is called an \defem{$n$-dimensional topological
  manifold}\indexn{manifold!topological>defemidx}\indexn{dimension of a
  manifold>defemidx} if and only if 
\ben
\item 
its topology is Hausdorff and has a countable basis, and
\item 
it is locally homeomorphic to $\R^n$, i.~e., for every $p\in M$ there exists an
open neighborhood $U$ of $p$ in $M$ and a homeomorphism $\phi:U\to\phi(U)$,
where $\phi(U)$ is an open subset of $\R^n$.
\een
}

Any such homeomorphism $\phi:U\to\phi(U)\subset\R^n$ is called a
\defem{(local) chart}\indexn{chart>defemidx} of $M$. 
The coordinate functions $\phi^j:U\to \R$ of $\phi=(\phi^1,\ldots,\phi^n)$ are
called the \defem{coordinates}\indexn{coordinates in a manifold>defemidx} of
the local chart. 
An \defem{atlas}\indexn{atlas>defemidx} of $M$ is a family of local charts
$(U_i,\phi_i)_{i\in I}$ of 
$M$ which covers all of $M$, i.~e., $\buil{\cup}{i\in I}U_i=M$.

\Definition{
Let $M$ be a topological manifold.
A \defem{smooth atlas}\indexn{smooth atlas>defemidx} of $M$ is an atlas
$(U_i,\phi_i)_{i\in I}$ such that  
\[
\varphi_i\circ\varphi_j^{-1}:\varphi_j(U_i\cap U_j)\to\varphi_i(U_i\cap U_j)
\]
is a smooth map (as a map between open subsets of $\R^n$) whenever $U_i\cap
U_j\neq\emptyset$.
}

\begin{center}
\input{fig-defmfd}
\end{center}

Not every topological manifold admits a smooth atlas.
We shall only be interested in those topological manifolds that do.
Moreover, topological manifolds can have essentially different smooth atlases
in the sense that they give rise to non-diffeomorphic smooth manifolds.
Hence the smooth atlas is an important additional piece of structure.

\Definition{
A \defem{smooth manifold}\indexn{manifold!smooth>defemidx} is a topological
manifold $M$ together with a maximal smooth atlas. 
}

Maximality means that there is no smooth atlas on $M$ containing all local
charts of the given atlas except for the given atlas itself.
Every smooth atlas is contained in a unique maximal smooth atlas.

In the following ``manifold'' will always mean ``smooth manifold''.
The smooth atlas will usually be suppressed in the notation.

\Examples{
a)
Every nonempty open subset of $\R^n$ is an $n$-dimensional manifold.
More generally, any nonempty open subset of an $n$-dimensional manifold is
itself an $n$-dimensional manifold. 

b)
The product\indexn{product manifold>defemidx} of any $m$-dimensional manifold
with any $n$-dimensional manifold 
is canonically an $(m+n)$-dimensional manifold.

c)
Let $n\leq m$.
An $n$-dimensional \defidx{submanifold} $N$ of an $m$-dimensional manifold $M$
is a nonempty subset $N$ of $M$ such that for every $p\in N$ there exists a local
chart $(U,\phi)$ of $M$ about $p$ with 
\[\phi(U\cap N)=\phi(U)\cap\R^n,\]
where we identify $\R^n\cong\R^n\times\{0\}\subset\R^m$.
Any submanifold is canonically a manifold.
In the case $n=m-1$ the submanifold $N$ is called \defidx{hypersurface} of $M$.
}

As in the case of open subsets of $\R^n$, we have the concept of
\emph{differentiable map} between manifolds:

\Definition{\label{ddefdiffmap}
Let $M$ and $N$ be manifolds and let $p\in M$.
A continuous map $f:M\to N$ is said to be \defem{differentiable} at the point
  $p$\indexn{differentiable map>defemidx} if there exist local charts
$(U,\phi)$ and $(V,\psi)$ about $p$ 
in $M$ and about $f(p)$ in $N$ respectively, such that $f(U)\subset V$ and
\[
\psi\circ f\circ\phi^{-1}:\phi(U)\to\psi(V)
\]
is differentiable at $\phi(p)\in\phi(U)$.
The map $f$ is said to be differentiable on $M$ if it is
differentiable at every point of $M$.

Similarly, one defines $C^k$-maps\indexn{map between manifolds>defemidx} between smooth
manifolds, $k\in\N \cup 
\{\infty\}$.
A $C^\infty$-map is also called a \defidx{smooth map}. 
\begin{center}
\input{fig-defdiffmap}
\end{center}
}

Note that, if $\psi\circ f\circ\phi^{-1}$ is ($C^k$-)differentiable for
\emph{some} local charts $\phi,\psi$ as in Definition~\ref{ddefdiffmap}, then so is
$\psi'\circ f\circ\phi'^{-1}$ for \emph{any} other pair of local charts
$\phi',\psi'$ obeying the same conditions.
This is a consequence of the fact that the atlases of $M$ and $N$ have been
assumed to be smooth. 

In order to define the differential of a differentiable map between manifolds,
we need the concept of \emph{tangent space}:

\Definition{
Let $M$ be a manifold and $p\in M$.
Consider the set $\mathcal{T}_p$ of differentiable curves $c:I\to M$ with
$c(0)=p$ where $I$ is an open interval containing $0\in\R$.
The \defidx{tangent space} of $M$ at $p$ is the quotient
\[
T_pM:=\mathcal{T}_p/\sim,
\indexs{TpM@$T_pM$, tangent space of $M$ at $p$}
\]
where ``$\sim$'' is the equivalence relation defined as follows: 
Two smooth curves $c_1,c_2\in\mathcal{T}_p$ are equivalent if and only if
there exists a local chart about p such that $(\phi\circ c_1)'(0)=(\phi\circ
c_2)'(0)$.
}

\begin{center}
\input{fig-deftangentspace}
\end{center}

One checks that the definition of the equivalence relation does not depend on
the choice of local chart:
If $(\phi\circ c_1)'(0)=(\phi\circ c_2)'(0)$ for one local chart $(U,\phi)$
with $p\in U$, then $(\psi\circ c_1)'(0)=(\psi\circ c_2)'(0)$ for all local
charts $(V,\psi)$ with $p\in V$.

Let $n$ denote the dimension of $M$.
Denote the equivalence class of $c\in\mathcal{T}_p$ in $T_pM$ by $[c]$.
It can be easily shown that the map
\be
\Theta_\phi:T_pM&\to&\R^n,\\
{}[c]&\mapsto&(\phi\circ c)'(0),
\indexs{Thetaphi@$\Theta_\phi$, bijective map $T_pM\rightarrow\protect\R^n$
  induced by a chart $\phi$} 
\ee
is a well-defined bijection.
Hence we can introduce a vector space structure on $T_pM$ by {\em declaring}
$\Theta_\phi$ to be a linear isomorphism.
This vector space structure is independent of the choice of local chart
because for two local charts $(U,\phi)$ and $(V,\psi)$ containing $p$ the map
$\Theta_\psi\circ\Theta_\phi^{-1}=d_{\phi(p)}(\psi\circ\phi^{-1})$ is linear.

By definition, the \defem{tangent bundle}\indexn{tangent bundle of a
  manifold>defemidx}\indexs{TM@$TM$, tangent bundle of $M$} of $M$ is the
disjoint union of all the tangent spaces of $M$, 
\[
TM:=\buil{\bui{\cup}{\centerdot}}{p\in M}T_pM.
\]

\Definition{\label{def:diff}
Let $f:M\to N$ be a differentiable map between manifolds and let $p\in M$.
The \defem{differential}\indexn{differential of a map>defemidx} of $f$ at $p$
(also called the \defidx{tangent map} of $f$ at $p$) is the map 
\indexs{dfp@$d_pf$, differential of $f$ at $p$}
$$
d_pf:T_pM \to T_{f(p)}N, \quad
[c] \mapsto [f\circ c].
$$
The \emph{differential} of $f$ is the map $df:TM\to TN$,\quad
$df_{|_{T_pM}}:=d_pf$.\indexs{df@$df$, differential of $f$}
}

The map $d_pf$ is well-defined and linear.
The map $f$ is said to be an \defidx{immersion} or a \defidx{submersion} if $d_pf$
is injective or surjective for every $p\in M$ respectively.
A \defidx{diffeomorphism} between manifolds is a smooth bijective map whose
inverse is also smooth. 
An \defidx{embedding} is an immersion $f:M\to N$ such that $f(M)\subset N$ is a
submanifold of $N$ and $f:M\to f(M)$ is a diffeomorphism. 

Using local charts basically all local properties of differential calculus on
$\R^n$ can be translated to manifolds.
For example, we have the \defidx{chain rule}
$$
d_p(g\circ f) = d_{f(p)}g \circ d_pf,
$$
and the \defidx{inverse function theorem} which states that if $d_pf :T_pM \to
T_{f(p)}N$ is a linear isomorphism, then $f$ maps a neighborhood of $p$
diffeomorphically onto a neighborhood of $f(p)$.

\subsection{Vector bundles}

We can think of the tangent bundle as a family of pairwise disjoint vector
spaces parametrized by the points of the manifold.
In a suitable sense these vector spaces depend smoothly on the base point.
This is formalized by the concept of a vector bundle.

\Definition{\label{ddefvectorbundle}
Let $\K=\R$ or $\Co$. 
Let $E$ and $M$ be manifolds of dimension $m+n$ and $m$ respectively.
Let $\pi:E\to M$ be a surjective smooth map.
Let the fiber $E_p:=\pi^{-1}(p)$\indexs{Ep@$E_p$, fiber of $E$ above $p$}
carry a structure of $\K$-vector space 
$\VV_p$  for each $p\in M$.
The quadruple $(E,\pi,M,\{\VV_p\}_{p \in M})$\indexs{pi@$\pi$, projection from the total space of a vector bundle onto its base} is called a  
\defem{$\K$-vector bundle}\indexn{vector bundle>defemidx} if for every $p\in
M$ there exists an open 
neighborhood $U$ of $p$ in $M$ and a diffeomorphism $\Phi:\pi^{-1}(U)\to
U\times\mathbb{K}^n$ such that the following diagram 
\begin{equation}
\xymatrix{
\pi^{-1}(U)\ar[dr]_\pi\ar[rr]^\Phi& &U\times\mathbb{K}^n\ar[dl]^{\pi_1}\\
& U &
}
\label{cd:loctriv}
\end{equation}
commutes and for every $q\in U$ the map $\pi_2\circ \Phi|_{E_q} : E_q \to
\K^n$ is a vector space isomorphism.
Here $\pi_1:U\times \K^n\to U$ denotes the projection onto the first factor
$U$ and $\pi_2:U\times \K^n\to \K^n$ is the projection onto the second factor
$\K^n$. 
}

\begin{center}
\input{fig-defvectorbundle}
\end{center}

Such a map $\Phi:\pi^{-1}(U)\to U\times\mathbb{K}^n$ is called a \defidx{local
trivialization}\indexn{trivialization>defemidx} of the vector bundle.
The manifold $E$ is called the \defem{total space}\indexn{total space of a
  vector bundle>defemidx}, $M$ the \defem{base}\indexn{base of a vector
  bundle>defemidx}, and 
the number $n$ the \defem{rank}\indexn{rank of a vector bundle>defemidx} of
the vector bundle. 
Often one simply speaks of the vector bundle $E$ for brevity.

A vector bundle is said to be \defem{trivial}\indexn{trivial vector
  bundle>defemidx} if it admits a global 
trivialization, that is, if there exists a diffeomorphism as in
(\ref{cd:loctriv}) with $U=M$.

\Examples{
a)
The tangent bundle\indexn{tangent bundle of a manifold} of any $n$-dimensional
manifold $M$ is a real vector bundle 
of rank $n$.
The map $\pi$ is given by the canonical map $\pi(T_pM)=\{p\}$ for all $p\in M$.

b)
Most operations from linear algebra on vector spaces can be carried out
fiberwise on vector bundles to give new vector bundles. 
For example, for a given vector bundle $E$ one can define the \defidx{dual
  vector bundle} $E^*$\indexs{E*@$E^*$, dual vector bundle}. 
Here one has by definition $(E^*)_p=(E_p)^*$.
Similarly, one can define the exterior and the symmetric powers of $E$.
For given $\K$-vector bundles $E$ and $F$ one can form the direct sum $E\oplus
F$, the tensor product $E\otimes F$, the bundle
$\mathrm{Hom}_{\mathbb{K}}(E,F)$\indexs{EtensorF@$E\otimes F$, tensor product
  of vector
  bundles}\indexs{HomEF@$\protect\mathrm{Hom}_{\protect\mathbb{K}}(E,F)$,
  bundle of homomorphisms between two bundles} etc.   

c)
The dual vector bundle of the tangent bundle is called the \defidx{cotangent
bundle} and is denoted by $T^*M$\indexs{T*M@$T^*M$, cotangent bundle of $M$}. 

d)
Let $n$ be the dimension of $M$ and $k\in\{0,1,\ldots,n\}$. The
$k^{\textrm{th}}$ exterior power of $T^*M$ is the \defem{bundle of $k$-linear
skew-symmetric forms}\indexn{bundle of $k$-forms on $M$>defemidx} on $TM$ and
is denoted by $\Lambda^kT^*M$\indexs{LambdakT*M@$\Lambda^kT^*M$, bundle of
  $k$-forms on $M$}. 
It is a real vector bundle of rank $\frac{n!}{k!(n-k)!}$.
By convention $\Lambda^0T^*M$ is the trivial real vector bundle of rank $1$. 
}

\Definition{
A \defem{section}\indexn{section in a vector bundle>defemidx} in a vector
bundle $(E,\pi,M,\{\VV_p\}_{p \in M})$ is a map 
$s:M\to E$ such that 
\[\pi\circ s=\id_M.\]
\begin{center}
\input{fig-defsection}
\end{center}
}

Since $M$ and $E$ are smooth manifolds we can speak about $C^k$-sections,
$k\in\N \cup \{\infty\}$.
The set $C^k(M,E)$\indexs{CkME@$C^k(M,E)$, space of $C^k$-sections in $E$} of
$C^k$-sections of a given $\K$-vector bundle forms a 
$\K$-vector space, and a module over the algebra $C^k(M,\K)$ as well because
multiplying \emph{pointwise} any $C^k$-section with any $C^k$-function one
obtains a new $C^k$-section.

In each vector bundle there exists a canonical smooth section, namely the
\defidx{zero section} defined by $s(x):=0_x\in E_x$. 
However, there does not in general exist any smooth \emph{nowhere vanishing}
section. 
Moreover, the existence of $n$ everywhere linearly independent smooth sections
in a vector bundle of rank $n$ is equivalent to its
\emph{triviality}\indexn{trivial vector bundle}. 

\Examples{
a)
Let $E=M\times\mathbb{K}^n$ be the trivial $\mathbb{K}$-vector bundle of rank
$n$ over $M$. 
Then the sections of $E$ are essentially just the $\mathbb{K}^n$-valued
functions on $M$.

b)
The sections in $E=TM$ are called the \defem{vector fields}\indexn{vector
  field>defemidx} on $M$. 
If $(U,\phi)$ is a local chart of the $n$-dimensional manifold $M$, then
for each $j=1,\ldots,n$ the curve $c(t) = \phi^{-1}(\phi(p)+te_j)$ represents
a tangent vector $\frac{\del}{\del
  x^j}(p)$\indexs{delxj@$\frac{\protect\del}{\protect\del x^j}$, local vector
  field provided by a chart} where $e_1,\ldots,e_n$ denote the standard basis
of $\R^n$. 
The vector fields $\frac{\del}{\del x^1},\ldots,\frac{\del}{\del x^n}$ are
smooth on $U$ and yield a basis of $T_pM$ for every $p\in U$,
\[
T_pM=\mathrm{Span}_\R\pa{\frac{\del}{\del x^1}(p),\ldots,\frac{\del}{\del
    x^n}(p)}.
\]

c)
The sections in $E=T^*M$ are called the \defem{$1$-forms}\indexn{oneform@$1$-form>defemidx}. 
Let $(U,\phi)$ be a local chart of $M$.
Denote the basis of $T^*_pM$ dual to $\frac{\del}{\del
  x^1}(p),\ldots,\frac{\del}{\del x^n}(p)$ by
$dx^1(p),\ldots,dx^n(p)$\indexs{dxj@$dx^j$, local $1$-form provided by a
  chart}. 
Then $dx^1,\ldots,dx^n$ are smooth $1$-forms on $U$.

d)
Fix $k\in\{0,\ldots,n\}$.
Sections in $E=\Lambda^kT^*M$ are called
\defem{$k$-forms}\indexn{kform@$k$-form>defemidx}\indexn{differential
  form>defemidx}. 
Given a local chart $(U,\phi)$ we get smooth $k$-forms which pointwise yield
a basis of $\Lambda^kT^*M$ by
$$
dx^{i_1}\wedge \ldots \wedge dx^{i_k}, \quad
1\leq i_1 < \ldots < i_k \leq n.
$$
In particular, for $k=n$ the bundle $\Lambda^nT^*M$ has rank $1$ and a local
chart yields the smooth local section $dx^1\wedge\ldots\wedge dx^n$.
Existence of a global smooth section in $\Lambda^nT^*M$ is equivalent to $M$
being orientable\indexn{orientable vector bundle>defemidx}\indexn{orientable
  manifold>defemidx}. 

e)
For each $p\in M$ let $|\Lambda M|_p$\indexs{absLambda@$|\Lambda M|$, bundle
  of densities over $M$} be the set of all functions 
$v:\Lambda^nT^*_pM \to \R$ with $v(\lambda X) = |\lambda|\cdot v(X)$ for all
$X\in \Lambda^nT^*_pM$ and all $\lambda\in\R$.
Now $|\Lambda M|_p$ is a $1$-dimensional real vector space and yields a vector
bundle $|\Lambda M|$ of rank $1$ over $M$.
Sections in $|\Lambda M|$ are called \defem{densities}\indexn{density on a
  manifold>defemidx}. 

Given a local chart $(U,\phi)$ there is a smooth density
$|dx|$\indexs{absdx@$|dx|$, density provided by a local chart} 
defined on $U$ and characterized by 
$$
|dx|(dx^1\wedge\ldots\wedge dx^n)=1.
$$
The bundle $|\Lambda M|$ is always trivial.
Its importance lies in the fact that densities can be integrated.
There is a unique linear map
$$
\int_M : \DD(M,|\Lambda M|) \to \R,
$$
called the \defem{integral}\indexn{integration on a manifold>defemidx}, such
that for any local chart $(U,\phi)$ and any 
$f\in\DD(U,\R)$ we have
$$
\int_M f\, |dx| = \int_{\phi(U)} (f\circ\phi^{-1})(x^1,\ldots,x^n)\,
dx^1\cdots dx^n
$$
where the right hand side is the usual integral of functions on $\R^n$ and
$\DD(M,E)$ denotes the set of smooth sections with compact support.

f)
Let $E$ be a real vector bundle.
Smooth sections in $E^*\otimes E^*$ which are pointwise \emph{nondegenerate
symmetric} bilinear forms are called \defem{semi-Riemannian
  metrics}\indexn{semi-Riemannian metric>defemidx} or 
\defem{inner products}\indexn{inner product on a vector bundle>defemidx} on $E$.
An inner product on $E$ is called \defem{Riemannian metric}\indexn{Riemannian
  metric>defemidx} if it is pointwise positive definite. 
An inner product on $E$ is called a \defem{Lorentzian
  metric}\indexn{Lorentzian metric>defemidx} if it has pointwise signature
$(-\,+\,\ldots\,+)$.  
In case $E=TM$ a Riemannian or Lorentzian metric on $E$ is also called a
Riemannian or Lorentzian metric on $M$ respectively.
A \defem{Riemannian}\indexn{manifold!Riemannian>defemidx} or \defem{Lorentzian
  manifold}\indexn{manifold!Lorentzian>defemidx} is a manifold $M$ together 
with a Riemannian or Lorentzian metric on $M$ respectively.

Any semi-Riemannian metric on $T^*M$ yields a nowhere vanishing smooth density
$\dV$\indexs{dV@$\protect\dV$, semi-Riemannian volume density} on $M$.
In local coordinates, write the semi-Riemannian metric as
$$
\sum_{i,j=1}^n g_{ij}dx^i \otimes dx^j.
$$
Then the induced density is given by 
$$
\dV = \sqrt{|\det(g_{ij})|}\, |dx|.
$$
Therefore there is a canonical way to form the integral $\int_M f\dV$ of any
function $f\in\DD(M)$ on a Riemannian or Lorentzian manifold.

g)
If $E$ is a complex vector bundle a \defidx{Hermitian metric} on $E$ is by
definition a smooth section of $E^*{\otimes}_{\R} E^*$\indexs{EtensorRF@$E
  \otimes_{\protect\R} F$, real tensor product of complex vector bundles} (the
\emph{real} tensor 
product of $E^*$ with itself) which is a Hermitian scalar product on each
fiber. 
}

\Definition{
Let $(E,\pi,M,\{\VV_p\}_{p \in M})$ and $(E',\pi',M',\{\VV'_{p'}\}_{p' \in M'})$
be $\K$-vector bundles. 
A \defidx{vector-bundle-homomorphism} from $E$ to $E'$ is a pair $(f,F)$ where
\begin{enumerate}
\item $f:M\to M'$ is a smooth map,
\item $F:E\to E'$ is a smooth map such that the diagram
\[
\xymatrix{
E\ar[d]_{\pi}\ar[r]^F &E'\ar[d]^{\pi'}\\
M\ar[r]^f& M'
}\]
commutes and such that $F_{|_{E_p}}:E_p\to E'_{f(p)}$ is $\mathbb{K}$-linear
for every $p\in M$.  
\end{enumerate}
}

If $M=M'$ and $f=\id_M$ a vector-bundle-homomorphism is simply a smooth
section in $\mathrm{Hom}_{\mathbb{K}}(E,E')\to M$.

%% \begin{prop}[pull-back of a vector bundle]\label{ppullbackVB}
%% Let $E\bui{\to}{\pi}M'$ be a vector bundle and $f:M\to M'$ be a map.
%% There exists a unique vector bundle on $M$, denoted by $f^*E$ and called
%% \emph{the pull-back of $E$ through $f$}, such that $(f^*E)_p=E_{f(p)}$ for
%% all $p\in M$.\\ 
%% \end{prop}

%% \begin{proof}
%% The result follows from the local description of vector bundles through
%% trivializations, see e.g. \cite[Chap.~2, Def. 5.3 and Chap.~3,
%% Prop.~3.1]{Hus}. 
%% \end{proof}
%% $ $\\

%% Of course $(\id_M)^*E=E$ and $(g\circ f)^*E=f^*(g^*E)$ for maps $f:M''\to
%% M'$ and $g:M'\to M$. 
%% Furthermore, Proposition~\ref{ppullbackVB} allows to identify a
%% vector-bundle-homomorphism $(f,F)$ from $E\bui{\to}{\pi}M$ to
%% $E'\bui{\to}{\pi'}M'$ with a section of
%% $\mathrm{Hom}_{\mathbb{K}}(E,f^*E')$.\\ 

\Remark{
Let $E$ be a real vector bundle with inner product $\la\cdot\,,\cdot\ra$.
Then we get a vector-bundle-isomorphism $\FB: E\to E^*$, $\FB(X)=\la
X,\cdot\ra$.

In particular, on a Riemannian or Lorentzian manifold $M$ with $E=TM$ one can
define the \defem{gradient} of a differentiable function $f:M\to \R$ by
$\grad f:=\FB^{-1}(df)=(df)^\sharp$\indexn{gradient of a function>defemidx}\indexs{gr*@$\protect\grad f$, gradient of function $f$} . 
The differential $df$ is a $1$-form defined independently of the metric while
the gradient $\grad f$ is a vector field whose definition does depend on
the semi-Riemannian metric.
}

\subsection{Connections on vector bundles}

For a differentiable function $f:M\to \R$ on a (smooth) manifold $M$ its
derivative in direction $X\in C^\infty(M,TM)$ is defined by
$$
\partial_Xf := df(X).
\indexs{delXf@$\partial_Xf:=df(X)$} 
$$
We have defined the concept of differentiability of a section $s$ in a vector
bundle.
What is the derivative of $s$?

Without further structure there is no canonical way of defining this.
A rule for differentiation of sections in a vector bundle is called a
connection. 

\Definition{
Let $(E,\pi,M,\{\VV_p\}_{p \in M})$ be a $\mathbb{K}$-vector bundle,
$\mathbb{K}:=\R$ or $\mathbb{C}$. 
A \defem{connection}\indexn{connection on a vector bundle>defemidx} (or
\defidx{covariant derivative}) on $E$ is a $\mathbb{R}$-bilinear map 
\be
\nabla:C^\infty(M,TM)\times C^\infty(M,E)&\to& C^\infty(M,E),\\
(X,s)&\mapsto&\nabla_Xs,
\ee
with the following properties:
\begin{enumerate}
\item 
The map $\nabla$ is \emph{$C^\infty(M)$-linear} in the first argument, i.~e.,
\[
\nabla_{fX}s=f\nabla_Xs
\]
holds for all $f\in C^\infty(M)$, $X\in C^\infty(M,TM)$ and $s\in C^\infty(M,E)$.
\item 
The map $\nabla$  is a \defem{derivation}\indexn{derivation>defemidx} with respect to the second argument, i.~e., it is $\K$-bilinear and 
\[
\nabla_X(f\cdot s)=\partial_Xf\cdot s+f\cdot \nabla_Xs
\]
holds for all $f\in C^\infty(M)$, $X\in C^\infty(M,TM)$ and $s\in C^\infty(M,E)$.
\end{enumerate}
}
The properties of a connection imply that the value of $\nabla_Xs$ at a given
point $p\in M$ depends only on $X(p)$ and on the values of $s$ on a curve
representing $X(p)$.

Let $\nabla$ be a connection on a vector bundle $E$ over $M$.
Let $c:[a,b]\to M$ be a smooth curve.
Given $s_0\in E_{c(a)}$ there is a unique smooth solution $s:[a,b]\to E$,
$t\mapsto s(t)\in E_{c(t)}$, satisfying $s(a)=s_0$ and 
\begin{equation}
\nabla_{\dot c}s =0 .
\label{eq:parallel}
\end{equation}
This follows from the fact that in local coordinates (\ref{eq:parallel}) is a
linear ordinary differential equation of first order.
The map
$$
\Pi_c : E_{c(a)}\to E_{c(b)}, \quad
s_0 \mapsto s(b),
\indexs{Pic@$\Pi_c$, parallel transport along the curve $c$}
$$
is called \defidx{parallel transport}.
It is easy to see that $\Pi_c$ is a linear isomorphism.
This shows that a connection allows us via its parallel transport to
``connect'' different fibers of the vector bundle.
This is the origin of the term ``connection''.
Be aware that in general the parallel transport $\Pi_c$ does depend on the
choice of curve $c$ connecting its endpoints.

Any connection $\nabla$ on a vector bundle $E$ induces a connection, also
denoted by $\nabla$, on the dual vector bundle $E^*$ by 
\[
(\nabla_X\theta)(s):=\partial_X\pa{\theta(s)}-\theta\pa{\nabla_Xs}
\]
for all $X\in C^\infty(M,TM)$, $\theta\in C^\infty(M,E^*)$ and $s\in
C^\infty(M,E)$. 
Here $\theta(s)\in C^\infty(M)$ is the function on $M$ obtained by 
pointwise evaluation of $\theta(p)\in E^*_p$ on $s(p)\in E_p$.

Similarly, tensor products, exterior and symmetric products, and direct sums
inherit connections from the connections on the vector bundles out of which
they are built.
For example, two connections $\nabla$ and $\nabla'$ on $E$ and $E'$
respectively induce a connection $D$ on $E\otimes
E'$\indexn{connection induced on powers of vector bundles>defemidx} by 
\[
D_X(s\otimes s'):=(\nabla_Xs)\otimes s'+s\otimes(\nabla'_X s')
\]
and a connection $\tilde D$ on $E\oplus E'$ by
$$
\tilde D_X(s \oplus s') := (\nabla_Xs)\oplus (\nabla'_X s')
$$
for all $X\in C^\infty(M,TM)$, $s\in C^\infty(M,E)$ and $s'\in
C^\infty(M,E')$.  

If a vector bundle $E$ carries a semi-Riemannian or Hermitian metric
$\la\cdot\,,\cdot\ra$, then a connection $\nabla$ on $E$ is called
\defem{metric}\indexn{metric connection on a vector bundle>defemidx} if the
following Leibniz rule\indexn{Leibniz rule>defemidx} holds: 
\[
\del_X(\la s,s'\ra)=\la\nabla_Xs,s'\ra+\la s,\nabla_Xs'\ra
\]
for all $X\in C^\infty(M,TM)$ and $s,s'\in C^\infty(M,E)$.

Given two vector fields $X,Y \in C^\infty(M,TM)$ there is a unique vector
field $[X,Y]\in C^\infty(M,TM)$ characterized by
$$
\partial_{[X,Y]}f = \partial_X \partial_Y f - \partial_Y \partial_X f
$$\indexs{Liebracket@$[X,Y]$, Lie bracket of $X$ with $Y$}
for all $f\in C^\infty(M)$.
The map $[\cdot,\cdot] : C^\infty(M,TM) \times C^\infty(M,TM) \to
C^\infty(M,TM)$ is called the \defidx{Lie bracket}.
It is $\R$-bilinear, skew-symmetric and satisfies the \defidx{Jacobi identity}
$$
[[X,Y],Z] + [[Y,Z],X] + [[Z,X],Y] =0.
$$

\Definition{
Let $\nabla$ be a connection on a vector bundle $E$. 
The \defidx{curvature tensor} of $\nabla$ is the map
$$
R:C^\infty(M,TM)\times C^\infty(M,TM)\times
C^\infty(M,E)\to C^\infty(M,E),
$$\indexs{R@$R$, curvature tensor of a connection}
$$
(X,Y,s)\mapsto R(X,Y)s
:=
\nabla_X(\nabla_Ys)-\nabla_Y(\nabla_Xs)-\nabla_{[X,Y]}s.
$$
}

One can check that the value of $R(X,Y)s$ at any point $p\in M$ depends only
on $X(p)$, $Y(p)$, and $s(p)$.
Thus the curvature tensor can be regarded as a section, $R\in
C^\infty(M,\Lambda^2T^*M \otimes \Hom_\K(E,E))$.

Now let $M$ be an $n$-dimensional manifold with semi-Riemannian metric $g$ on
$TM$.
It can be shown that there exists a unique metric connection $\nabla$ on $TM$
satisfying 
\[
\nabla_XY-\nabla_YX=[X,Y]
\]
for all vector fields $X$ and $Y$ on $M$. 
This connection is called the \defem{Levi-Civita}\indexn{Levi-Civita
  connection>defemidx} connection of the 
semi-Riemannian manifold $(M,g)$.
Its curvature tensor $R$ is the \defem{Riemannian curvature
  tensor}\indexn{Riemannian curvature tensor>defemidx} of $(M,g)$. 
The \defidx{Ricci curvature} $\ric \in C^\infty(M,T^*M\otimes
T^*M)$\indexs{ric@$\protect\ric$, Ricci curvature} is defined 
by 
$$
\ric(X,Y) := \sum_{j=1}^n \epsilon_j\,g(R(X,e_j)e_j,Y)
$$
where $e_1,\ldots,e_n$ are smooth locally defined vector fields which are
pointwise orthonormal with respect to $g$ and $\epsilon_j=g(e_j,e_j)=\pm
1$\indexs{epsilonj@$\varepsilon_j:=g(e_j,e_j)=\pm 1$}. 
It can easily be checked that this definition is independent of the choice of
the vector fields $e_1,\ldots,e_n$.
Similarly, the \defidx{scalar curvature} is the function $\scal\in
C^\infty(M,\R)$ defined by
$$
\scal := \sum_{j=1}^n \epsilon_j\,\ric(e_j,e_j).
$$\indexs{scal@$\protect\scal$, scalar curvature}

%%%%%%%%%%%%%%%%%%%%%%%%%%%%%%%%%%%%%%%%%%%%%%%%%%%%%%%%%%%%%%%%%%%%%%%%%%%%%%
\section{Differential operators}
\label{app:diffops}
%%%%%%%%%%%%%%%%%%%%%%%%%%%%%%%%%%%%%%%%%%%%%%%%%%%%%%%%%%%%%%%%%%%%%%%%%%%%%%

In this section we explain the concept of linear differential operators and
we define the principal symbol. 
A detailed introduction to the topic can be found e.~g.\ in \cite[Ch.~9]{Nic}.
As before we write $\mathbb{K}=\R$ or $\mathbb{C}$.

\Definition{\label{ddefdiffop}
Let $E$ and $F$ be $\mathbb{K}$-vector bundles of rank $n$ and $m$ respectively over
a $d$-dimensional manifold $M$.
A \defem{linear differential operator of order at most
  $k$}\indexn{differential operator>defemidx} from $E$ to $F$ 
is a $\mathbb{K}$-linear map 
\[L:C^\infty(M,E)\to C^\infty(M,F)\]
which can locally be described as follows: 
For every $p\in M$ there exists an open coordinate-neighborhood $U$ of $p$ in
$M$ on which $E$ and $F$ are trivialized and there are smooth maps
$A_\alpha:U\to \mathrm{Hom}_\mathbb{K}(\mathbb{K}^n,\mathbb{K}^m)$ such that
on $U$
\[
Ls=\sum_{|\alpha|\leq k}A_\alpha\frac{\partial^{|\alpha|} s}{\partial
  x^\alpha}.
\]
Here summation is taken over all multiindices
 $\alpha=(\alpha_1,\ldots,\alpha_d)\in\mathbb{N}^d$ with
 $|\alpha|:=\sum_{r=1}^d\alpha_r\leq k$.
Moreover, $\frac{\partial^{|\alpha|}}{\partial
 x^\alpha}:=\frac{\partial^{{}^{\alpha_1+\cdots+\alpha_d}}}
{\partial x_1^{\alpha_1}\cdots\,\partial x_d^{\alpha_d}}$. 
In this definition we have used the local trivializations to identify sections
 in $E$ with $\mathbb{K}^n$-valued functions and sections in $F$ with
 $\mathbb{K}^m$-valued functions on $U$.
If $L$ is a linear differential operator of order at most $k$, but not of
 order at most $k-1$, then we say that $L$ is of \defem{order $k$}.
}
\indexn{order of a differential operator>defemidx}
%It can be shown that the local form of a differential operator does not depend
%on the local chart and on the local trivializations chosen. 
%If one picks other coodinates on $M$ and other local trivializations for $E$
%and $F$, then one obtains a similar local formula for $L$.
Note that zero-order differential operators are nothing but sections of
$\mathrm{Hom}_\mathbb{K}(E,F)$, i.~e., they are
 vector-bundle-homomorphisms\indexn{vector-bundle-homomorphism} from $E$ to
 $F$.  

\Definition{
Let $L$ be a linear differential operator of order $k$ from $E$ to $F$. 
The \defem{principal symbol}\indexn{principal symbol of a differential operator>defemidx} of $L$ is the map
\indexs{sigmaL@$\sigma_L$, principal symbol of $L$}
\[
\sigma_L:T^*M\to\mathrm{Hom}_\mathbb{K}(E,F)
\]
defined locally as follows: 
For a given $p\in M$ write $L=\sum_{|\alpha|\leq
  k}A_\alpha\frac{\partial^{|\alpha|}}{\partial x^\alpha}$ in a coordinate
neighborhood of $p$ with respect to local trivializations of $E$ and $F$ as in
Definition~\ref{ddefdiffop}.
For every $\xi=\sum_{r=1}^d \xi_r\cdot dx^r\in T_p^*M$ we have with respect to
these trivializations,
\[\sigma_L(\xi):=\sum_{|\alpha|=k}\xi^\alpha A_\alpha(p)\]
where $\xi^\alpha:=\xi_1^{\alpha_1}\cdots\xi_d^{\alpha_d}$.
Here we have used the local trivializations of $E$ and $F$ to identify
$\mathrm{Hom}_\mathbb{K}(E,F)$ with
$\mathrm{Hom}_\mathbb{K}(\mathbb{K}^n,\mathbb{K}^m)$.
}

One can show that the principal symbol of a differential operator is
well-defined, that is, it is independent of the choice of the local
coordinates and trivializations. 
Moreover, the principal symbol of a differential operator of order $k$ is, by
definition, a homogeneous polynomial of degree $k$ on $T^*M$.

\Example{
The gradient\indexn{gradient of a function} is a linear differential operator of first order
$$
\grad: C^\infty(M,\R) \to C^\infty(M,TM)
$$
with principal symbol
$$
\sigma_{\grad}(\xi)f = f\cdot \xi^\sharp.
$$
}

\Example{
The divergence\indexn{divergence of a vector field} yields a first order
linear differential operator  
$$
\div: C^\infty(M,TM) \to C^\infty(M,\R)
$$
with principal symbol
$$
\sigma_{\div}(\xi)X = \xi(X).
$$
}

\Example{
For each $k\in\N$ there is a unique linear first order differential operator 
$$
d : C^\infty(M,\Lambda^kT^*M) \to C^\infty(M,\Lambda^{k+1}T^*M),
$$
called \defidx{exterior differential},\indexs{d@$d$, exterior differentiation}
such that  
\begin{enumerate}
\item 
for $k=0$ the exterior differential coincides with the differential defined in
Definition~\ref{def:diff}, after the canonical identification $T_y\R=\R$,
\item
$d^2=0:C^\infty(M,\Lambda^kT^*M) \to C^\infty(M,\Lambda^{k+2}T^*M)$ for all
$k$,
\item
$d(\omega\wedge\eta) = (d\omega)\wedge\eta + (-1)^k\omega\wedge d\eta$ for all
$\omega\in C^\infty(M,\Lambda^kT^*M)$ and $\eta\in C^\infty(M,\Lambda^lT^*M)$.
\end{enumerate}
Its principal symbol is given by 
$$
\sigma_d(\xi)\,\omega = \xi \wedge \omega.
$$
}

\Example{
A connection\indexn{connection on a vector bundle} $\nabla$ on a vector bundle
$E$ can be considered as a first 
order linear differential operator
$$
\nabla : C^\infty(M,E) \to C^\infty(M,T^*M\otimes E).
$$
Its principal symbol is easily be seen to be
$$
\sigma_\nabla(\xi)\,e = \xi\otimes e.
$$
}

\Example{
If $L$ is a linear differential operator of order $0$, i.~e., $L\in
C^\infty(M,\Hom(E,F))$, then 
$$
\sigma_L(\xi) = L.
$$
}

\Remark{
If $L_1:C^\infty(M,E)\to C^\infty(M,F)$ is a linear differential operator of
order $k$ and $L_2:C^\infty(M,F)\to C^\infty(M,G)$ is a linear differential
operator of order $l$, then $L_2\circ L_1$ is a linear differential
operator of order $k+l$.
The principal symbols satisfy
$$
\sigma_{L_2\circ L_1}(\xi) = \sigma_{L_2}(\xi) \circ \sigma_{L_1}(\xi).
$$
}

%%%%%%%%%%%%%%%%%%%%%%%%%%%%%%%%%%%%%%%%%%%%%%%%%%%%%%%%%%%%%%%%%%%%%%%%%%%%%%

\section{More on Lorentzian geometry}
\label{app:lorgeo}
%%%%%%%%%%%%%%%%%%%%%%%%%%%%%%%%%%%%%%%%%%%%%%%%%%%%%%%%%%%%%%%%%%%%%%%%%%%%%%

This section is a rather heterogeneous collection of results on Lorentzian
manifolds.
We give full proofs.
This material has been collected in this appendix in order not to overload
Section~\ref{sec:lorgeo} with technical statements.

Throughout this section $M$ denotes a Lorentzian manifold.\indexn{manifold!Lorentzian}

\begin{lemma}\label{lJKompaktumabgeschl}
Let the causal relation $\leq$ on $M$ be closed, i.~e., for all
convergent sequences $p_n{\to}p$ and $q_n{\to}q$ in $M$ with $p_n\leq q_n$ we
have $p\leq q$. 

Then for every compact subset $K$ of $M$ the subsets $J_+^M(K)$ and $J_-^M(K)$
are closed. 
\end{lemma}

\begin{proof}
Let $(q_n)_{n\in\mathbb{N}}$ be any sequence in $J_+^M(K)$ converging in $M$
and  $q\in M$ be its limit.  
By definition, there exists a sequence $(p_n)_{n\in\mathbb{N}}$ in $K$ with
$p_n\leq q_n$ for every $n$.  
Since $K$ is compact we may assume, after to passing to a subsequence, that
$(p_n)_{n\in\mathbb{N}}$ converges to some $p\in K$.  
Since $\leq$ is closed we get $p\leq q$ , hence $q\in J_+^M(K)$. 
This shows that $J_+^M(K)$ is closed.
The proof for $J_-^M(K)$ is the same.
\end{proof}

\Remark{
If $K$ is only assumed to be \emph{closed} in Lemma~\ref{lJKompaktumabgeschl},
then $J_\pm^M(K)$ need not be closed. 
The following picture shows a curve $K$, closed as a subset and asymptotic to
a lightlike line in 2-dimensional Minkowski space.
Its causal future $J_+^M(K)$ is the open half plane bounded by this lightlike
line. 

\begin{center}
\input{fig-J+Anabgschl}
\end{center}
}

\begin{lemma}\label{pastcompact}
Let $M$ be a timeoriented Lorentzian manifold.
Let $K\subset M$ be a compact subset.
Let $A\subset M$ be a subset such that, for every $x\in M$, the intersection
$A\cap J_-^M(x)$ is relatively compact in $M$.

Then $A\cap J_-^M(K)$ is a relatively compact subset of $M$.
Similarly, if $A\cap J_+^M(x)$ is relatively compact for every $x\in M$, then
$A\cap J_+^M(K)$ is relatively compact. 
\end{lemma}

\begin{center}
\input{fig-pastcompact2}
\end{center}

\begin{proof}
It suffices to consider the first case.
The family of open sets $I_-^M(x)$, $x\in M$, is an open covering of $M$.
Since $K$ is compact it is covered by a finite number of such sets,
$$ 
K\subset I_-^M(x_1)\cup\ldots\cup I_-^M(x_l).
$$
We conclude
$$
J_-^M(K)\subset 
J_-\left(I_-^M(x_1)\cup\ldots\cup I_-^M(x_l)\right)
\subset
J_-^M(x_1)\cup\ldots\cup J_-^M(x_l) .
$$
Since each $A\cap J_-^M(x_j)$ is relatively compact, we have that $A\cap
J_-^M(K)\subset \bigcup_{j=1}^{\; l}\left(A\cap 
J_-^M(x_j) \right)$ is contained in a compact set.
\end{proof}

\begin{cor}\label{cJ+Spastcompact}
Let $S$ be a Cauchy hypersurface in a globally hyperbolic Lorentzian manifold
$M$ and let $K\subset M$ be compact.
Then $J_\pm^M(K)\cap S$ and $J_\pm^M(K)\cap J_\mp^M(S)$ are compact.
\end{cor}

\begin{proof}
The causal future of every Cauchy hypersurface is past-compact.
This follows e.~g.\ from \cite[ Chap.~14, Lemma~40]{ONeill}.
Applying Lemma~\ref{pastcompact} to $A:=J_+^M(S)$ we conclude that
$J_-^M(K)\cap J_+^M(S)$ is relatively compact in $M$.
By \cite[Chap.~14, Lemma 22]{ONeill} the subsets $J_\pm^M(S)$ and the causal
relation ``$\leq$'' are closed.
By Lemma~\ref{lJKompaktumabgeschl} $J_-^M(K)\cap J_+^M(S)$ is closed, hence
compact. 

Since $S$ is a closed subset of $J_+^M(S)$ we also have that  $J_-^M(K)\cap S$
is compact. 
 
The statements on $J_+^M(K)\cap J_-^M(S)$ and on $J_+^M(K)\cap S$ are
analogous.
\end{proof}

\begin{lemma}\label{kleinergleichabgeschl}
Let $M$ be a timeoriented convex domain.\indexn{convex domain}
Then the causal relation ``$\leq$'' is closed.
In particular, the causal future and the causal past of each point are closed
subsets of $M$.
\end{lemma}

\begin{proof}
Let $p,p_i,q,q_i \in M$ with $\lim_{i\to\infty}p_i = p$, $\lim_{i\to\infty}q_i
= q$, and $p_i \leq q_i$ for all $i$.
We have to show $p\leq q$.

Let $x\in T_pM$ be the unique vector such that $q=\exp_p(x)$ and,
similarly, for each $i$ let $x_i\in T_{p_i}M$ be such that
$\exp_{p_i}(x_i)=q_i$. 
Since $p_i\leq q_i$ and since $\exp_{p_i}$ maps $J_+(0)\cap
\exp_{p_i}^{-1}(M)$ in $T_{p_i}M$ diffeomorphically onto $J_+^M(p_i)$, we have
$x_i\in J_+(0)$, hence $\la x_i,x_i\ra\leq 0$.
From $\lim_{i\to\infty}p_i = p$ and $\lim_{i\to\infty}q_i = q$
we conclude $\lim_{i\to\infty}x_i= x$ and therefore $\la x,x\ra\leq 0$.
Thus $x\in J_+(0)\cup J_-(0) \subset T_pM$.

Now let $T$ be a smooth vector field on $M$ representing the timeorientation.
In other words, $T$ is timelike and future directed.
Then $\la T,x_i\ra \leq 0$ because $x_i$ is future directed and so $\la T,x\ra
\leq 0$ as well.
Thus $x\in J_+(0) \subset T_pM$ and hence $p\leq q$. 
\end{proof}

\begin{lemma}\label{hyperflaechekausal} 
Let $M$ be a timeoriented Lorentzian manifold and let $S\subset M$ be
a spacelike hypersurface.
Then for every point $p$ in $S$, there exists a basis of open neighborhoods
$\Omega$ of $p$ in $M$ such that $S\cap\Omega$ is a Cauchy hypersurface in
$\Omega$.  
\end{lemma} 

\begin{proof}
Let $p\in S$.
Since every spacelike hypersurface is locally acausal there exists an open
neighborhood $U$ of $p$ in $M$ such that $S\cap U$ is an acausal spacelike
hypersurface of $U$. 
Let $\Omega$ be the Cauchy development of $S\cap U$ in $U$.

\begin{center}
\input{fig-cauchyentwick}
\end{center}

Since $\Omega$ is the Cauchy development of an acausal hypersurface containing
$p$, it is an open neighborhood of $p$ in $U$ and hence also in $M$.  
It follows from the definition of the Cauchy development that
$\Omega\cap S=S\cap U$ and that $S\cap \Omega$ is a Cauchy hypersurface of
$\Omega$. 

Given any neighborhood $V$ of $p$ the neighborhood $U$ from above can be
chosen to be contained in $V$.
Hence $\Omega$ is also contained in $V$.
Therefore we get a basis of neighborhoods $\Omega$ with the required
properties.
\end{proof}

On globally hyperbolic manifolds the relation $\leq$ is always closed
\cite[Chap.~14, Lemma 22]{ONeill}. 
The statement that the sets $J_+^M(p)\cap J_-^M(q)$ are compact can be
strengthened as follows:

\begin{lemma}\label{lJ+KJ-K'cpct}
Let $K,K'\subset M$ two compact subsets of a globally hyperbolic
Lorentzian manifold $M$. 
Then $J_+^M(K)\cap J_-^M(K')$ is compact.
\end{lemma}

\begin{proof}
Let $p$ in $M$. 
By the definition of global-hyperbolicity, the subset $J_+^M(p)$ is past
compact in $M$. 
It follows from Lemma~\ref{pastcompact} that $J_+^M(p)\cap J_-^M(K')$ is
relatively compact in $M$. 
Since the relation $\leq$ is closed on $M$, the sets $J_+^M(p)$ and
$J_-^M(K')$ are closed by Lemma~\ref{lJKompaktumabgeschl}.
Hence $J_+^M(p)\cap J_-^M(K')$ is actually compact. 
This holds for every $p\in M$, i.~e., $J_-^M(K')$ is future compact in $M$. 
It follows again from Lemma~\ref{pastcompact} that $J_+^M(K)\cap J_-^M(K')$ is
relatively compact in $M$, hence compact by Lemma~\ref{lJKompaktumabgeschl}.
\end{proof}

\begin{lemma}\label{lsuffcondGHCC}
Let $\Omega\subset M$ be a nonempty open subset of a timeoriented
Lorentzian manifold $M$. 
Let $J_+^M(p)\cap J_-^M(q)$ be contained in $\Omega$ for all $p,q\in \Omega$.
Then $\Omega$ is causally compatible.\indexn{causally compatible subset}

If furthermore $M$ is globally hyperbolic, then $\Omega$ is globally
hyperbolic as well.\indexn{globally hyperbolic manifold} 
\end{lemma} 

\begin{proof}
We first show that $J_\pm^M(p)\cap \Omega = J_\pm^\Omega(p)$ for all
$p\in\Omega$. 
Let $p\in\Omega$ be fixed.  
The inclusion $J_\pm^\Omega(p) \subset J_\pm^M(p)\cap \Omega$ is obvious.
To show the opposite inclusion let $q\in J_+^M(p)\cap\Omega$.
Then there exists a future directed causal curve
$c:[0,1]\rightarrow M$ with $c(0)=p$ and $c(1)=q$. 
For every $z\in c([0,1])$ we have $z\in J_+^M(p)\cap J_-^M(q)\subset \Omega$,
i.~e., $c([0,1])\subset \Omega$. 
Therefore $q\in J_+^\Omega(p)$.
Hence $J_+^M(p)\cap\Omega\subset J_+^\Omega(p)$ and $J_-^M(p)\cap\Omega\subset
J_-^\Omega(p)$ can be seen similarly.

We have shown $J_\pm^M(p)\cap\Omega= J_\pm^\Omega(p)$, i.~e., $\Omega$ is a
causally compatible subset of $M$.
Let now $M$ be globally hyperbolic.
Then since for any two points $p,q\in \Omega$ the intersection $J_+^M(p)\cap
J_-^M(q)$ is contained in $\Omega$ the subset
\[
J_+^\Omega(p)\cap J_-^\Omega(q)=J_+^M(p)\cap J_-^M(q)\cap\Omega=J_+^M(p)\cap
J_-^M(q)
\]
is compact.
Remark~\ref{rem:globhypOmega} concludes the proof.
\end{proof}

\begin{lemma}\label{lem:DS}
For any acausal hypersurface $S$ in a
timeoriented Lorentzian manifold the Cauchy development $D(S)$\indexn{Cauchy
  development of a subset} is a causally 
compatible and globally hyperbolic open subset of $M$.
\end{lemma}

\begin{proof}
Let $S$ be an acausal hypersurface in a timeoriented Lorentzian manifold $M$.
By \cite[Chap.~14, Lemma 43]{ONeill} $D(S)$ is an open and globally hyperbolic
subset of $M$.
Let $p,q\in D(S)$.
Let $z\in J_+^M(p)\cap J_-^M(q)$.
We choose a future directed causal curve $c$ from $p$ through $z$ to $q$.
Extend $c$ to an inextendible causal curve in $M$, again denoted by $c$. 
Since $p\in D(S)$ the curve $c$ must meet $S$.
Since $S$ is acausal this intersection point is unique.

Now let $\tilde c$ be any inextendible causal curve through $z$.
If $c$ intersects $S$ before $z$, then look at the inextendible curve obtained
by first following $\tilde c$ until $z$ and then following $c$.
This is an inextendible causal curve through $q$.
Since $q\in D(S)$ this curve must intersect $S$.
This intersection point must come before $z$, hence lie on $\tilde c$.
Similarly, if $c$ intersects $S$ at or after $z$, then look at the
inextendible curve obtained by first following $c$ until $z$ and then
following $\tilde c$. 
Again, this curve is inextendible causal and goes through $p\in D(S)$.
Hence it must hit $S$ and this intersection  point must come before or at $z$,
thus it must again lie on $\tilde c$.

In any case $\tilde c$ intersects $S$.
This shows $z\in D(S)$.
We have proved $J_+^M(p)\cap J_-^M(q) \subset D(S)$.
By Lemma~\ref{lsuffcondGHCC} $D(S)$ is causally compatible in $M$.
\end{proof}

Note furthermore that, by the definition of $D(S)$, the acausal subset $S$
is a Cauchy hypersurface of $D(S)$.

\begin{lemma}\label{lcommonCauchhyp}
Let $M$ be a globally hyperbolic Lorentzian manifold.
Let $\Omega\subset M$ be a causally compatible and globally hyperbolic open
subset.  
Assume that there exists a Cauchy hypersurface $\Sigma$ of $\Omega$ which is
also a Cauchy hypersurface of $M$.

Then every Cauchy hypersurface of $\Omega$ is also a Cauchy hypersurface of
$M$. 
\end{lemma} 

\begin{proof}
Let $S$ be any Cauchy hypersurface of $\Omega$. 
Since $\Omega$ is causally compatible in $M$, achronality of $S$ in $\Omega$
implies achronality of $S$ in $M$.

Let $c:I\rightarrow M$ be any inextendible timelike curve in $M$. 
Since $\Sigma$ is a Cauchy hypersurface of $M$ there exists some $t_0\in
I$ with $c(t_0)\in\Sigma\subset\Omega$.
Let $I'\subset I$ be the connected component of $c^{-1}(\Omega)$ containing
$t_0$.
Then $I'$ is an open interval and $c|_{I'}$ is an inextendible timelike curve in
$\Omega$.
Therefore it must hit $S$.
Thus $S$ is a Cauchy hypersurface in $M$.
\end{proof}

Given any compact subset $K$ of a globally hyperbolic manifold $M$, one can
construct a causally compatible globally hyperbolic open subset of $M$ which
is ``causally independent'' of $K$:\indexn{causally independent subsets} 

\begin{lemma}\label{lMminusJAglobhyp} 
Let $K$ be a compact subset of a globally hyperbolic Lorentzian manifold $M$. 
Then the subset $M\setminus J^M(K)$ is, when nonempty, a causally
compatible globally hyperbolic open subset of $M$.  
\end{lemma}

\begin{proof}
Since $M$ is globally hyperbolic and $K$ is compact it follows from
Lemma~\ref{lJKompaktumabgeschl} that $J^M(K)$ is closed in $M$, hence
$M\setminus J^M(K)$ is open.
Next we show that $J_+^M(x)\cap J_-^M(y)\subset M\setminus J^M(K)$ for any two
points $x,y\in M\setminus J(K)$.
It will then follow from Lemma~\ref{lsuffcondGHCC} that $M\setminus J^M(K)$ is
causally compatible and globally hyperbolic.

Let $x,y\in M\setminus J^M(K)$ and pick $z\in J_+^M(x)\cap J_-^M(y)$. 
If $z\notin M\setminus J^M(K)$, then $z\in J_+^M(K)\cup J_-^M(K)$.
If $z\in J_+^M(K)$, then $y\in J_+^M(z)\subset J_+^M(J_+^M(K))=J_+^M(K)$ 
in contradiction to $y\not\in J^M(K)$.
Similarly, if $z\in J_-^M(K)$ we get $x\in J_-^M(K)$, again a contradiction.
Therefore $z\in M\setminus J^M(K)$.
This shows $J_+^M(x)\cap J_-^M(y)\subset M\setminus J^M(K)$.
\end{proof}

Next we prove the existence of a causally compatible globally hyperbolic
neighborhood of any compact subset in any globally hyperbolic manifold. 
First we need a technical lemma.

\begin{lemma}\label{lI+AI-Bglobhyp}
Let $A$ and $B$ two nonempty subsets of a globally hyperbolic Lorentzian
manifold $M$. 
Then $\Omega:=I_+^M(A)\cap I_-^M(B)$\indexn{causally compatible
  subset}\indexn{globally hyperbolic manifold} is a globally hyperbolic 
causally compatible open subset of $M$.

Furthermore, if $A$ and $B$ are relatively compact in $M$, then so is $\Omega$.
\end{lemma}

\begin{proof}
Since the chronological future and past of any subset of $M$ are open, so is
 $\Omega$.
For any $x,y\in \Omega$ we have $J_+^M(x)\cap J_-^M(y)\subset\Omega$ because
$J_+^M(x)\subset J_+^M(I_+^M(A))=I_+^M(A)$ and $J_-^M(y)\subset
 J_-^M(I_-^M(B))=I_-^M(B)$. 
Lemma~\ref{lsuffcondGHCC} implies that $\Omega$ is globally hyperbolic and
causally compatible. 

If furthermore $A$ and $B$ are relatively compact, then 
\[
\Omega\subset J_+^M(A)\cap J_-^M(B)\subset J_+^M(\ovl{A})\cap J_-^M(\ovl{B})
\]
and $J_+^M(\ovl{A})\cap J_-^M(\ovl{B})$ is compact by
Lemma~\ref{lJ+KJ-K'cpct}.
Hence $\Omega$ is relatively compact in $M$.
\end{proof}

\begin{prop}\label{lIdir}
Let $K$ be a compact subset of a globally hyperbolic Lorentzian manifold $M$.
Then there exists a relatively compact causally compatible globally hyperbolic
open subset $O$ of $M$ containing $K$. 
\end{prop}

\begin{proof}
Let $h:M\to \R$ be a Cauchy time-function as in Corollary~\ref{cexisttimefctn}.
The level sets $S_t:=h^{-1}(\{t\})$ are Cauchy hypersurfaces for each $t\in
h(M)$.
Since $K$ is compact so is $h(K)$.
Hence there exist numbers $t_+,t_-\in h(M)$ such that  
\[
S_{t_\pm}\cap J_\mp^M(K)=\emptyset,
\]
that is, such that $K$ lies in the past of $S_{t_+}$ and in the future
of $S_{t_-}$. 
We consider the open set
\[O:=I_-^M\pa{J_+^M(K)\cap S_{t_+}}\cap\, I_+^M\pa{J_-^M(K)\cap S_{t_-}}.\] 

\begin{center}
\input{fig-constructIdir}
\end{center}

We show $K\subset O$.
Let $p\in K$. 
By the choice of ${t_\pm}$ we have $K\subset J_\mp^M(S_{t_\pm})$. 
Choose any inextendible future directed timelike curve starting at $p$.
Since $S_{t_+}$ is a Cauchy hypersurface, it is hit exactly once by
this curve at a point $q$.
Therefore $q\in I_+^M(K)\cap S_{t_+}$ hence $p\in I_-^M(q)\subset
I_-^M(I_+^M(K)\cap S_{t_+})\subset I_-^M(J_+^M(K)\cap S_{t_+})$. 
Analogously $p\in I_+^M\pa{J_-^M(K)\cap S_{t_-}}$. 
Therefore $p\in O$.

It follows from Lemma~\ref{lI+AI-Bglobhyp} that $O$ is a causally compatible
globally hyperbolic open subset of $M$. 
Since every Cauchy hypersurface is future and past compact, the subsets
$J_-^M(K)\cap S_{t_-}$ and $J_+^M(K)\cap S_{t_+}$ are relatively compact by
Lemma~\ref{pastcompact}.
According to Lemma~\ref{lI+AI-Bglobhyp} the subset $O$ is also
relatively compact. 
This finishes the proof. 
\end{proof}

\begin{lemma}\label{lem:globhypcyl}
Let $(S,g_0)$ be a connected Riemannian manifold.
Let $I\subset \R$ be an open interval and let $f:I\to \R$ be a smooth positive
function. 
Let $M=I\times S$ and $g=-\dt^2+f(t)^2 g_0$.
We give $M$ the timeorientation with respect to which the vector field
$\frac{\partial}{\partial t}$ is future directed.

Then $(M,g)$ is globally hyperbolic if and only if $(S,g_0)$ is complete.
\end{lemma}

\begin{proof}
Let  $(S,g_0)$ be complete.
Each slice $\{t_0\}\times S$ in $M$ is certainly achronal, $t_0\in I$.
We show that they are Cauchy hypersurfaces by proving that each inextendible
causal curve meets all the slices.

Let $c(s)=(t(s),x(s))$ be a causal curve in $M=I\times S$.
Without loss of generality we may assume that $c$ is future directed, i.~e., $t'(s)>0$.
We can reparametrize $c$ and use $t$ as the curve parameter, i.~e.,
$c$ is of the form $c(t)=(t,x(t))$.
 
Suppose that $c$ is inextendible.
We have to show that $c$ is defined on all of $I$.
Assume that $c$ is defined only on a proper subinterval $(\alpha,\beta)\subset
I$ with, say, $\alpha \in I$.
Pick $\epsilon>0$ with $[\alpha-\epsilon,\alpha+\epsilon]\subset I$.
Then there exist constants $C_2 > C_1>0$ such that $C_1 \leq f(t)\leq C_2$ for
all $t\in [\alpha-\epsilon,\alpha+\epsilon]$.
The curve $c$ being causal means $0\geq g(c'(t),c'(t)) = -1 + f(t)^2
\|x'(t)\|^2$ where $\|\cdot\|$ is the norm induced by $g_0$.
Hence $\|x'(t)\| \leq \frac{1}{f(t)} \leq \frac{1}{C_1}$ for all $t\in
(\alpha,\alpha+\epsilon)$.

Now let $(t_i)_i$ be a sequence with $t_i \searrow \alpha$.
For sufficiently large $i$ we have $t_i\in (\alpha,\alpha+\epsilon)$.
For $j>i\gg0$ the length of the part of $x$ from $t_i$ to $t_j$ is bounded
from above by $\frac{t_j-t_i}{C_1}$.
Thus we have for the Riemannian distance
$$
\dist(x(t_j),x(t_i)) \leq \frac{t_j-t_i}{C_1}.
$$
Hence $(x(t_i))_i$ is a Cauchy sequence and since $(S,g_0)$ is complete it
converges to a point $p\in S$.
This limit point $p$ does not depend on the choice of Cauchy sequence because
the union of two such Cauchy sequences is again a Cauchy sequence with a
unique limit point.
This shows that the curve $x$ can be extended continuously by putting
$x(\alpha):=p$.
We extend $x$ in an arbitrary fashion beyond $\alpha$ to a piecewise
$C^1$-curve with $\|x'(t)\| \leq \frac{1}{C_2}$ for all
$t\in(\alpha-\epsilon,\alpha)$.
This yields an extension of $c$ with
$$
g(c'(t),c'(t)) = -1 + f(t)^2\|x'(t)\|^2 \leq -1 + \frac{f(t)^2}{C_2^2} \leq 0.
$$
Thus this extension is causal in contradiction to the inextendibility of $c$.

Conversely, assume that $(M,g)$ is globally hyperbolic.
We fix $t_0\in I$ and choose $\epsilon>0$ such that
$[t_0-\epsilon,t_0+\epsilon] \subset I$.
There is a constant $\eta>0$ such that $\frac{1}{f(t)}\geq\eta$ for all $t\in
[t_0-\epsilon,t_0+\epsilon]$.
Fix $p\in S$.
For any $q\in S$ with $\dist(p,q)\leq \frac{\epsilon\eta}{2}$ there is a
smooth curve $x$ in $S$ of length at most $\epsilon\eta$ joining $p$ and $q$.
We may parametrize $x$ on $[t_0,t_0+\epsilon]$ such that $x(t_0)=q$,
$x(t_0+\epsilon)=p$ and $\|x'\|\leq \eta$. 
Now the curve $c(t):=(t,x(t))$ is causal because
$$
g(c',c') = -1 + f^2\, \|x'\|^2 \leq -1 + f^2\,\eta^2 \leq 0.
$$
Moreover, $c(t_0)=(t_0,q)$ and $c(t_0+\epsilon)=(t_0+\epsilon,p)$.
Thus $(t_0,q) \in J_-^M(t_0+\epsilon,p)$.
Similarly, one sees $(t_0,q) \in J_+^M(t_0-\epsilon,p)$.
Hence the closed ball $\ovl{B}_{r}(p)$ in $S$ is contained in the compact set
$J_+^M(t_0-\epsilon,p)\cap J_-^M(t_0+\epsilon,p)$ and therefore compact
itself, where $r=\frac{\eta\epsilon}{2}$. 
We have shown that all closed balls of the fixed radius $r>0$ in $S$ are
compact.

Every metric space with this property is complete.
Namely, let $(p_i)_i$ be a Cauchy sequence.
Then there exists $i_0>0$ such that $\dist(p_i,p_j) \leq r$ whenever $i,j\geq
i_0$.
Thus $p_j\in \ovl{B}_r(p_{i_0})$ for all $j\geq i_0$.
Since any Cauchy sequence in the compact ball $\ovl{B}_r(p_{i_0})$ must
converge we have shown completeness. 
\end{proof}

\end{appendix}

%%%%%%%%%%%%%%%%%%%%%%%%%%%%%%%%%%%%%%%%%%%%%%%%%%%%%%%%%%%%%%%%%%%%%%%%%

\backmatter 

%%%%%%%%%%%%%%%%%%%%%%%%%%%%%%%%%%%%%%%%%%%%%%%%%%%%%%%%%%%%%%%%%%%%%%%%%

%%%%%%%%%%%%%%%%%%%%%%%%%%%%%%%%%%%%%%%%%%%%%%%%%%%%%%%%%%%%%%%%%%%%%%%%%

\printindex{Abbildungen}{Figures}

\printindex{Symbole}{Symbols}

\printindex{DerIndex}{Index}

%%%%%%%%%%%%%%%%%%%%%%%%%%%%%%%%%%%%%%%%%%%%%%%%%%%%%%%%%%%%%%%%%%%%%%%%%
\end{document}